\documentclass{amsart}

\usepackage{amsmath,amsthm,amssymb}
\usepackage{graphicx}
\usepackage{hyperref}

\title{Semigroups, groups, and algebras of dynamical origin}
\author{Volodymyr Nekrashevych}

\newtheorem{theorem}{Theorem}[section]
\newtheorem{proposition}[theorem]{Proposition}
\newtheorem{corollary}[theorem]{Corollary}
\newtheorem{lemma}[theorem]{Lemma}

\theoremstyle{definition}
\newtheorem{definition}{Definition}[section]
\newtheorem{example}{Example}[section]

\newcommand{\mapdown}[1]%
{\Big\downarrow\rlap{$\vcenter{\hbox{$\scriptstyle#1$}}$}}

\newcommand{\arr}{\longrightarrow}
\newcommand{\alb}{\mathsf{X}}
\newcommand{\xs}{\alb^*}
\newcommand{\xo}{\alb^\omega}
\newcommand{\xmo}{\alb^{-\omega}}
\newcommand{\img}[1]{\mathop{\mathrm{IMG}}\left(#1\right)}
\newcommand{\limg}[1][G]{\mathcal{X}_{#1}}
\newcommand{\til}{\mathcal{T}}
\newcommand{\bim}{\mathsf{B}}
\newcommand{\essent}{I_{\mathrm{ess}}}

\newcommand{\C}{\mathbb{C}}
\newcommand{\R}{\mathbb{R}}
\newcommand{\Z}{\mathbb{Z}}

\newcommand{\Gr}{\mathfrak{G}}
\newcommand{\Hr}{\mathfrak{H}}
\newcommand{\Fr}{\mathfrak{F}}
\newcommand{\be}{\mathsf{s}}
\newcommand{\en}{\mathsf{r}}
\newcommand{\X}{\mathcal{X}}
\newcommand{\F}{\mathcal{F}}
\newcommand{\M}{\mathcal{M}}
\newcommand{\nuke}{\mathcal{N}}

\newcommand{\G}{\Gamma}
\newcommand{\si}{\sigma}
\newcommand{\wt}{\widetilde}

\begin{document}

\maketitle \tableofcontents

\section{Introduction}

Inverse semigroups are natural tools for studying structures with rich partial symmetries and self-similarities not captured by globally defined symmetries (see~\cite{lawson:inversesemigroups}). Classical examples of such structures are quasi-crystals,  self-similar fractals, transversals of foliations, and locally invertible dynamical systems (e.g., one-sided subshifts of finite type).
Closely related to inverse semigroups are \emph{\'etale groupoids}, usually defined as \emph{groupoids of germs} of elements of inverse semigroups. They are used when local action is more important than particular choice of the domains of partial symmetries. 

Groupoids and inverse semigroups are also used as tools to study groups. When a group is defined by its action on a topological space or as a group naturally associated with a topological dynamical system, it is often easier to describe local properties of the action and the structure of its orbits, i.e., the groupoid of germs of the action. Therefore, it is important to understand the connection between the properties of the groupoid (or the associated inverse semigroup) and the properties of the group. This approach is especially fruitful in the study of asymptotic properties of groups such as amenability and growth. When the relation is well understood, then inverse semigroups and groupoids can be used, for example, to construct groups and algebras with interesting prescribed properties. 

Our paper is a survey of recent results and techniques concerning groups, inverse semigroups, groupoids, and related structures (convolution algebras and topological full groups) associated with topological dynamical systems. It is based on a minicourse given by the author at the workshop ``Semigroups, groupoids, and $C^*$-algebras,'' that took place at KIAS in July 1-5 of 2024.

Let us give a short overview of the structure of the paper.
In the second section, we give a survey of the basic notions of the theory of \'etale groupoids and inverse semigroups. In particular, we discuss compactly generated groupoids and their Cayley graphs, and the notions of \emph{minimality} and \emph{expansivity}, which illustrate how conditions from the theory of topological dynamical systems are reflected in algebraic conditions (simplicity and finite generation). We also discuss a homology theory for ample groupoids, following H.~Matui.

In the next section ``Algebras,'' we consider the convolution algebras (\emph{Steinberg algebras}) of ample groupoids and show how their properties are related to the dynamical properties of the groupoid. We show that expansivity is equivalent to the finite generation of the algebra, show a relation between the lattice of ideals of the Steinberg algebra and the set of invariant open subsets for the groupoid. We also discuss relations between asymptotical properties of the groupoid and the algebras (growth, complexity, and amenability), following~\cite{nek:gk}.

The section ``Full groups'' is an overview of results about topological full groups of inverse semigroups and ample groupoids (following papers of H.~Matui and~\cite{nek:fullgr}). In particular, we defined the \emph{alternating full group} of an ample groupoid. Its properties are parallel to the properties of the Steinberg algebras. For example, they are simple if the groupoid is minimal, and finitely generated if it is expansive. The same statements are true for the Steinberg algebras (except that simplicity of Steinberg algebras may fail if the groupoid is not Hausdorff). Alternating full groups can be used to construct simple finitely generated groups with special properties such as amenability and sub-exponential growth. We also discuss dynamical conditions on ample groupoids (such as the \emph{comparison property}, \emph{almost finiteness}, and being \emph{purely infinite}) under which the difference between the full group and the alternating full group is well understood. 

 One of the central topics of our paper are semigroups and groups associated with hyperbolic dynamical systems.
Hyperbolic dynamical systems such as contracting iterated function systems, subshifts of finite type, Anosov diffeomorphisms, Smale spaces, post-critically finite complex rational functions, etc., induce families of expanding and contracting homeomorphisms between subsets of a topological space, thus are well suited for analysis via inverse semigroups. Algebraic structures associated with hyperbolic dynamical systems are implicitly present already in the classical methods of symbolic dynamics developed for hyperbolic dynamical systems in the works of J.~Hadamard, M.~Morse, G.A.~Hedlund and others. One of the features of hyperbolic dynamical systems is their \emph{structural stability}, which implies that they are completely determined (up to topological conjugacy) by a discrete structure. In most cases, this structure can be naturally defined as a finitely generated inverse semigroup. 

One of the most well studied, from this point of view, classes of hyperbolic dynamical systems are expanding self-coverings (such as hyperbolic complex rational functions restricted to their Julia set). They are uniquely encoded by the associated \emph{iterated monodromy group}, which is a finitely generated group $G$ together with a \emph{self-similarity} of its action. A self-similar group acts on the space of one-sided infinite sequences $\xo=\{x_1x_2\ldots\colon x_i\in\alb\}$ in such a way that the action is  ``invariant'' in some sense with respect to the shift $\si(x_1x_2\ldots)=x_2x_3\ldots$. The shift is a non-invertible map, so it does not naturally come from a group action. However, it is a local homeomorphism and has local inverses $S_x(x_1x_2\ldots)=xx_1x_2\ldots$, which generate an inverse semigroup. An action of a group $G$ on $\xo$ is said to be self-similar if for every $g\in G$ and $x\in\alb$ there exist $h\in G$ and $y\in\alb$ such that 
\[gS_x=S_yh.\]
In particular, this implies that the groupoid of germs of the action of $G$ on $\xo$ is invariant under conjugation by the shift.

The self-similar group $G$ and the maps $S_x$ generate an inverse semigroup, which is a natural object in the setting of iterated monodromy groups. For example, the dynamical system defining the iterated monodromy group can be reconstructed from the inverse semigroup as its natural boundary. We give a short  overview (with a semigroup theory flavor) of the theory of self-similar groups in Section~5 of the paper.

Self-similar groups is an active area of research with many interesting results, especially in relation to growth of groups and amenability. The definition of self-similar groups almost immediately suggests a generalization to inverse semigroups.
Such a generalization was formulated in~\cite{bgn}. It was shown in~\cite{nek:smale} that self-similar inverse semigroups are naturally associated with a class of hyperbolic dynamical systems, thus are generalizations of the iterated monodromy groups. See~\cite{WaltonWhittaker:tilings} where examples associated with self-similar tilings are studied.

Basic theory of contracting self-similar inverse semigroups is outlined in Section~6. In particular, we repeat (using a more general approach) the construction associating a self-similar inverse semigroup to a Smale space given in~\cite{nek:smale}. We give a groupoid-theoretic definition of a contracting self-similar semigroup, thus showing that self-similar inverse semigroups, as defined in~\cite{bgn,nek:smale}, appear in natural situations as semigroups generated by a ``nucleus'' of an ample groupoid invariant under conjugation by a one-sided shift of finite type.

Section~7 describes three explicit examples of self-similar inverse semigroups: the semigroup generated by a golden mean rotation, the semigroup associated with the Penrose tiling (coinciding with the tiling semigroup, defined by J.~Kellendonk in~\cite{kellendonk:noncom,kellendonk:coinvar}), and an inverse semigroup used to construct the first example of a simple group of intermediate growth in~\cite{nek:burnside}.

In the last section we prove some general properties of contracting self-similar inverse semigroups. We prove some technical conditions that are useful in the study of their full groups and $C^*$-algebras: that the associated groupoids of germs are amenable and almost finite. The latter result is new even for the case of self-similar groups (though easy to check in all concrete examples). We give a finite presentation for the Steinberg algebra of a contracting self-similar inverse semigroup. We also discuss their Matui's homology, and give examples of its computation.

The work was partially supported by NSF grant DMS2204379.

\section{Groupoids and inverse semigroups}

\subsection{Groupoids}

We give here the overview of the classical definitions of the theory of topological groupoids, in order to fix the notation and terminology. For more, see~\cite{paterson:gr,renault:groupoids,williams:groupoids} and~\cite[Chapter~3]{nek:dyngroups}.

A \emph{groupoid} is a small category of isomorphisms. Equivalently, it is a set $\Gr$ with a partially defined operation of multiplication and everywhere defined operation of taking inverse such that the following conditions are satisfied.

\begin{enumerate}
\item If $g_1g_2$ and $g_2g_3$ are defined, then $(g_1g_2)g_3=g_1(g_2g_3)$ and both products are defined.

\item For every $g\in \Gr$ the products $gg^{-1}$ and $g^{-1}g$ are defined. If $g_1g_2$ is defined, then $g_1^{-1}g_1g_2=g_2$ and $g_1g_2g_2^{-1}=g_1$.
\end{enumerate}

Elements of the form $gg^{-1}$ are called \emph{units} of the groupoid. Equivalently, they are elements $u$ such that $u^2$ is defined and is equal to $u$. We denote the set of units by $\Gr^{(0)}$. We also denote $\be(g)=g^{-1}g$ and $\en(g)=gg^{-1}$. We call $\be, \en\colon \Gr\arr\Gr^{(0)}$ the \emph{source} and the \emph{range maps}, respectively. A product $g_1g_2$ is defined if and only if $\en(g_2)=\be(g_1)$. We have $\be(g_1g_2)=\be(g_2)$ and $\en(g_1g_2)=\en(g_1)$.

We denote by $\Gr^{(2)}\subset\Gr\times\Gr$ the set of pairs $(g_1, g_2)$ such that the product $g_1g_2$ is defined.

\begin{example}
A group $G$ is a groupoid with a single unit and everywhere defined multiplication. Conversely, every groupoid with a single unit is a group.

In the other extreme, a \emph{trivial groupoid} is a groupoid consisting of units only, a.k.a.\ a set.
\end{example}

We say that two units $x_1, x_2\in\Gr^{(0)}$ \emph{belong to the same $\Gr$-orbit} if there exists $g\in\Gr$ such that $x_1=\be(g)$ and $x_2=\en(g)$. Belonging to one orbit is an equivalence relation, and we call the equivalence classes \emph{orbits}.

A set $A\subset\Gr^{(0)}$ is said to be $\Gr$-invariant if it is a union of $\Gr$-orbits, i.e., if $\be(g)\in A$ implies $\en(g)\in A$.

For a set $A\subset\Gr^{(0)}$, we denote by $\Gr|_A$ the sub-groupoid $\{g\in\Gr\colon \be(g), \en(g)\in A\}$, and call it the \emph{restriction} of $\Gr$ to $A$.
If $A\subset\Gr^{(0)}$ is $\Gr$-invariant, then $\Gr$ is equal to the disjoint union of the sub-groupoids $\Gr|_A$ and $\Gr|_{\Gr^{(0)}\setminus A}$.

An element $g\in\Gr$ is \emph{isotropic} if $\be(g)=\en(g)$. For $x\in\Gr^{(0)}$, the restriction $\Gr|_{\{x\}}$ is a group, which we call the \emph{isotropy group} of $x$.

A groupoid is said to be \emph{principal} if all its isotropy groups are trivial. An abstract principal groupoid is just an equivalence relation. Its elements $g$ are uniquely determined by a pair $(\be(g), \en(g))$ of points belonging to one orbit.

A \emph{topological groupoid} is a groupoid $\Gr$ together with a topology on it such that the operation of taking inverse $\Gr\arr\Gr$ and the multiplication $\Gr^{(2)}\arr\Gr$ (with respect to the relative topology on $\Gr^{(2)}\subset\Gr\times\Gr$) are continuous. 

We assume that all our groupoids are locally compact and that the unit space $\Gr^{(0)}$ is Hausdorff. We do not assume that the groupoid itself is Hausdorff.

\begin{definition}
A topological groupoid $\Gr$ is said to be \emph{\'etale} if the maps $\be, \en\colon \Gr\arr\Gr^{(0)}$ are local homeomorphisms. 

A subset $U\subset\Gr$ is a \emph{bisection} if $\be\colon U\arr\be(U)$ and $\en\colon U\arr\en(U)$ are homeomorphisms.
\end{definition}

Thus, a topological groupoid $\Gr$ is \'etale if and only if it has a basis of topology consisting of open bisections.

Since we assume that the unit spaces of our groupoids are Hausdorff, \'etale groupoids will be locally Hausdorff.

\begin{definition}
A topological groupoid $\Gr$ is \emph{ample} if has a basis of topology consisting of compact open bisections.

All ample groupoids in our paper are assumed to be second countable.
\end{definition}

Equivalently, a groupoid is ample if and only if it is \'etale and the space of units $\Gr^{(0)}$ is totally disconnected.

\begin{definition}
A topological groupoid $\Gr$ is said to be \emph{proper} if the map $(\be, \en)\colon\Gr\arr\Gr^{(0)}\times\Gr^{(0)}$ is proper, i.e., if for every pair of compact sets $A, B\subset\Gr^{(0)}$ the set of elements $g\in\Gr$ such that $\be(g)\in A, \en(g)\in B$ is compact.
\end{definition}

\begin{definition}
Let $\Gr_1, \Gr_2$ be groupoids. A map $\phi\colon \Gr_1\arr\Gr_2$ is called a \emph{functor} if it is a functor of the corresponding categories, i.e., if $\phi(g_1g_2)=\phi(g_1)\phi(g_2)$ for all $(g_1, g_2)\in\Gr_1^{(2)}$ and $\phi(g)=\phi(g)^{-1}$ for all $g\in\Gr_1$. 

A functor $\phi\colon \Gr_1\arr\Gr_2$ between topological groupoids is called an \emph{isomorphism} if it is a continuous functor such that inverse map $\phi^{-1}$ exists and is a continuous functor.
\end{definition}

\subsection{Inverse semigroups}
\label{sss:inversesemigroups}

If $F$ is a bisection, then $F^{-1}$ is also a bisection, since $\be(g^{-1})=\en(g)$ and $\en(g^{-1})=\be(g)$. It is also not hard to show that if $F_1$ and $F_2$ are open (resp.\ compact) bisections, then $F_1F_2$ is also an open (resp.\ compact) bisection.

It follows that the set of open bisections is an \emph{inverse semigroup}.
Similarly, if $\Gr$ is an ample groupoid, then the set $\mathcal{B}(\Gr)$ of compact open bisections is an inverse semigroup (in both cases, we include the empty bisection, which is a zero of the semigroup).

In the other direction, \'etale groupoids are naturally constructed from inverse semigroups of local homeomorphisms of a topological space $\X$ as the \emph{groupoids of germs} in the following way.

Let $\mathcal{G}$ be an \emph{inverse semigroup of local homeomorphisms} of a topological space $\mathcal{X}$, i.e., a set of homeomorphisms between open subsets of $\mathcal{X}$ closed under composition and taking inverses. 

For $F\in\mathcal{G}$ and $x$ in the domain of $F$, the \emph{germ} $[F, x]$ of $F$ at $x$ is the equivalence class of the pair $(F, x)$, where two pairs $(F_1, x)$ and $(F_2, x)$, are considered equivalent if there exists a neighborhood $U$ of $x$ such that the maps $F_1|_U\colon U\arr F_1(U)$ and $F_2|_U\colon U\arr F_2(U)$ coincide.

The set of germs of elements of $\mathcal{G}$ is naturally a groupoid with multiplication
\[[F_1, x][F_2, y]=[F_1F_2, y],\]
where the product is defined if and only if $F_2(y)=x$. The inverse is given by $[F, x]^{-1}=[F^{-1}, F(x)]$.

The groupoid of germs is an \'etale groupoid with a basis of topology consisting of sets $U_F=\{[F, x] \colon  x\in\mathop{\mathrm{Dom}(F)}\}$. We call this groupoid the \emph{groupoid of germs} of the inverse semigroup. The set $U_F$ is an open bisection and we have $U_{F_1}U_{F_2}=U_{F_1F_2}$. Thus, the inverse semigroup $\mathcal{G}$ is naturally isomorphic to an inverse subsemigroup of the semigroup of open bisections of the groupoid of germs.

If $\Gr$ is an arbitrary \'etale groupoid, then every open $\Gr$-bisection $F$ defines a homeomorphism $\be(F)\arr\en(F)$ acting by the rule $\be(g)\arr\en(g)$ for $g\in F$. It is the composition of the homeomorphism $\en\colon F\arr \en(F)$ with the inverse of the map $\be\colon F\arr\be(F)$. The set of such local homeomorphisms is an inverse semigroup $\mathcal{G}$ of local homeomorphisms of $\Gr^{(0)}$. The map $g\mapsto [F, g]$, where $F$ is an open bisection containing $g$ is a well defined surjective functor from $\Gr$ to the groupoid of germs of $\mathcal{G}$ preserving the source and range maps. It is not an isomorphism in general, i.e., it is possible that $g\in\Gr$ is not a unit, but the germ $[F, g]$ is. 

\begin{definition}
\label{def:effective}
An \'etale groupoid $\Gr$ is said to be \emph{effective} if it coincides with its groupoid of germs, i.e., if for every $g\in\Gr\setminus\Gr^{(0)}$ and every neighborhood $U$ of $g$ there exists a non-isotropic element $h\in U$.
\end{definition}

\begin{example}
Let $G$ be a discrete group acting by homeomorphisms on a topological space $\mathcal{X}$. The \emph{action groupoid} is the space $G\times\mathcal{X}$ with the multiplication
\[(g_1, x_1)(g_2, x_2)=(g_1g_2, x_2),\]
where the product is defined if and only if $g_2(x_2)=x_1$. 

The action groupoid is effective if and only if all germs of non-trivial elements of $G$ are not units, i.e., if and only if for every $g\in G, g\ne 1$, the set of fixed points of $g$ has empty interior. Such actions are usually called \emph{topologically free}.
\end{example}

An abstract \emph{inverse semigroup} is a semigroup $\mathcal{H}$ such that every $g\in\mathcal{H}$ has a unique element $g^{-1}\in\mathcal{H}$ such that $g=gg^{-1}g$ and $g^{-1}=g^{-1}gg^{-1}$. Consequences of these axioms are the identities  $(g^{-1})^{-1}=g$ and $(g_1g_2)^{-1}=g_2^{-1}g_1^{-1}$.

Let us show how an ample groupoid $\Gr$ can be reconstructed from the abstract inverse semigroup $\mathcal{B}(\Gr)$ of compact open bisections and some of its inverse subsemigroups. The general case of an \'etale groupoid is more complicated than the ample case, see~\cite{exelpardo:tight}.

Let $\mathcal{H}$ be an inverse semigroup. Denote by $\mathcal{E}$ the set of idempotents of $H$ (i.e., such elements $g$ that $g^2=g$). It follows from the axioms of inverse semigroups that for all $g, h\in\mathcal{E}$ we have $g=g^{-1}$ and $gh=hg\in\mathcal{E}$.

The set $\mathcal{E}$ has a natural order defined by the condition $a\le b$ if $ab=a$. 
A subset $\xi\subset\mathcal{E}$ is a \emph{filter} if the following conditions hold:
\begin{enumerate}
\item if $a, b\in\xi$, then $ab\in\xi$;
\item if $a\in\xi$ and $b\ge a$, then $b\in\xi$.
\end{enumerate}

A proper (i.e., not equal to $\mathcal{E}$) filter is called an \emph{ultrafilter} if it is maximal in the set of all proper filter with respect to inclusion.
The set of all filters has a natural topology as a subset of the direct product space $2^{\mathcal{H}}$ of all subsets of $\mathcal{H}$. 

Let $\xi$ be an ultrafilter on $\mathcal{E}$, and let $g\in\mathcal{H}$ be such that $g^{-1}g\in\xi$. A \emph{germ} $[g, \xi]$ is defined as an equivalence class of the pair $(g, x)$ with respect to the equivalence relation
\[[g_1, \xi]=[g_2, \xi]\Longleftrightarrow\exists a\in\xi\colon g_1a=g_2a.\]

The set of germs of $\mathcal{H}$ is naturally a groupoid with respect to the multiplication
\[[g_1, \xi_1][g_2, \xi_2]=[g_1g_2, \xi_2],\]
where the product is defined if and only if $g_2\xi_2 g_2^{-1}=\xi_1$.

Suppose that $\Gr$ is an ample groupoid, and let $\mathcal{H}$ be a subsemigroup of $\mathcal{B}(\Gr)$ such that $\mathcal{H}$ is a basis of topology on $\Gr$. We naturally identify compact open subsets of $\Gr^{(0)}$ with the corresponding bisections (contained in the unit space $\Gr^{(0)}$).

For every ultrafilter $\xi$ on $\mathcal{E}(\mathcal{H})$ the intersection of all elements $U\in\xi$ is non-empty (in $\Gr^{(0)}$) by compactness of the elements of $\mathcal{E}(\mathcal{H})$. Since $\mathcal{H}$ is a basis of topology on $\Gr$, and $\Gr$ is \'etale,  the elements of $\mathcal{E}(\mathcal{H})$ form a basis of topology on $\Gr^{(0)}$. Consequently, the intersection of the elements of $\xi$ is a singleton. 

Conversely, for every point $x\in\Gr^{(0)}$, the set of all elements of $\mathcal{E}(\mathcal{H})$ containing it is an ultrafilter. We get a natural bijection between the set of ultrafilters on $\mathcal{E}(\mathcal{H})$ and the unit space $\Gr^{(0)}$. 

\begin{lemma}
The constructed bijection is a homeomorphism.
\end{lemma}

\begin{proof}
The bijection maps an element $U\in\mathcal{E}(\mathcal{H})$ seen as a subset of $\Gr^{(0)}$ to the set of ultrafilters containing it. This set is clopen by the definition of the topology on the set of ultrafilters. Since the set of elements of $\mathcal{E}(\mathcal{H})$ is a basis of the topology on $\Gr^{(0)}$ the bijection maps open sets to open sets, and hence it is a homeomorphism (as the space of ultrafilters and the space $\Gr^{(0)}$ are compact).
\end{proof}

Let $[g, \xi]$ be a germ for $g\in\mathcal{H}$, as above. Let $x\in\Gr^{(0)}$ be the point corresponding to $\xi$. Then $g\cdot x$ is an element of $\Gr$ and $g\in\mathcal{H}$ is a $\Gr$-bisection containing it. Two germs $[g_1, \xi]$ and $[g_2, \xi]$ are equal if and only if there exists $U\in\mathcal{E}(\mathcal{H})$ such that $x\in U$ and $g_1U=g_2U$. If $U$ is a subset of $\Gr^{(0)}$ such that $g_1U=g_2U$ and $x\in U$, then we obviously have $g_1x=g_1Ux=g_2Ux=g_2x$. Conversely, if $g_1x=g_2x$, then we can find $U\in\mathcal{E}(\mathcal{H})$ such that $U\subset=\be(g_1\cap g_2)$, and get $g_1U=g_2U$.
Consequently, the groupoid of germs of $\mathcal{H}$ is isomorphic to $\Gr$. 

We have shown that if an inverse subsemigroup $\mathcal{H}$ of $\mathcal{B}(\Gr)$ is a basis of topology on $\Gr$, then the groupoid $\Gr$ can be naturally reconstructed from the abstract inverse semigroup $\mathcal{H}$ as its groupoid of germs.

\subsection{Actions of groupoids and Morita equivalence}

Let $\Gr$ be a topological groupoid. A \emph{left action} of $\Gr$ on a space $\X$ over an \emph{anchor} $P\colon \X\arr\Gr^{(0)}$ is a map $(g, x)\mapsto g\cdot x$ from the space $\Gr\times_P\X=\{(g, x)\in\Gr\times P :\be(g)=P(x)\}$ to $\X$ satisfying the conditions $P(x)\cdot x=x$ and $g_1\cdot(g_2\cdot x)=(g_1g_2)\cdot x$ for all $g_1, g_2\in\Gr$ and $x\in\X$ such that $P(x)=\be(g_2)$ and $\be(g_1)=\en(g_2)$. 

Right actions are defined in a similar way.

The \emph{action groupoid} is the space $\Gr\times_P\X$ with respect to the multiplication
\[(g_1, x_1)(g_2, x_2)=(g_1g_2, x_2),\]
where the product is defined if and only if $g_2\cdot x_2=x_1$.

For example, the left action of $\Gr$ on itself by multiplication is naturally defined over the anchor $\en\colon \Gr\arr\Gr^{(0)}$.

\begin{definition}
A left action of $\Gr$ on $\X$ is said to be free if $g\cdot x=x$, for $g\in\Gr$ and $x\in\X$, implies that $g$ is a unit.

It is called \emph{proper} if the groupoid of the action is proper. Equivalently, it is proper if for every compact subset $C\subset\X$ the set of elements $g\in\Gr$ such that $g\cdot x\in C$ for some $x\in C$ is compact.
\end{definition}

For a pair of groupois $\Gr_1, \Gr_2$, a \emph{bi-action} consists of a space $\X$, a left action of $\Gr_1$ over an anchor $P_l\colon \X\arr\Gr_1^{(0)}$, and a right action of $\Gr_2$ over an anchor $P_r\colon \X\arr\Gr_2^{(0)}$ such that the actions commute, i.e., 
\[(g_1\cdot x)\cdot g_2=g_1\cdot (x\cdot g_2)\]
for all $g_i\in\Gr_i$ and $x\in\X$ for which the points $g_1\cdot x$ and $x\cdot g_2$ are defined.

The following definition of equivalence of groupoids, introduced in~\cite{muhlyrenault:equiv}, is closely related to the classical Morita equivalence of algebras. For this reason, it is called sometimes \emph{Morita equivalence} of groupoids.

\begin{definition}
An \emph{equivalence} between groupoids $\Gr_1$ and $\Gr_2$ is given by a bi-action $\Gr_1\curvearrowright\X\curvearrowleft\Gr_2$ such that the following conditions hold.
\begin{enumerate}
\item Both actions are free and proper.
\item The action of one groupoid $\Gr_i$ is transitive on the fibers of the anchor map of the action of the other groupoid.
\item The anchor maps are onto and open.
\end{enumerate}
\end{definition}

We have the following alternative description of equivalence of groupoids (see, for example,~\cite[Proposition~3.2.31]{nek:dyngroups}).

\begin{proposition}
\label{pr:ambientgroupoid}
Topological groupoids $\Gr_1, \Gr_2$ are equivalent if and only if there exists a topological groupoid $\mathfrak{H}$ and functors $\phi_i\colon \Gr_i\arr\mathfrak{H}$ such that the following conditions hold.
\begin{enumerate}
\item The functors $\phi_i\colon\Gr_i\arr\phi_i(\Gr_i)$ are isomorphisms of topological groupoids.
\item The groupoids $\phi_i(\Gr_i)$ are equal to restrictions of $\mathfrak{H}$ to $\phi_i(\Gr_i^{(0)})$.
\item The sets $\phi_i(\Gr_i^{(0)})$ are locally closed in $\mathfrak{H}^{(0)}$ and intersect every $\mathfrak{H}$-orbit.
\end{enumerate}

If the groupoids $\Gr_i$ are \'etale, then we may assume that $\Hr$ is \'etale, $\phi_i(\Gr_i^{(0)})$ are clopen disjoint subsets of $\Hr^{(0)}$, and $\Hr^{(0)}=\phi_1(\Gr_1^{(0)})\cup\phi_2(\Gr_2^{(0)})$.
\end{proposition}

A set is called locally closed if it is an intersection of an open set with a closed set.

Groupoids are often considered (see, for example,~\cite{connes:noncomg}) as representations of the quotient of a topological space under the action of a group or inverse semigroup in the cases when the quotient topology is degenerate (e.g., trivial). From this point of view, one has to consider groupoids up to the equivalence, and all ``natural'' definitions have to be invariant under equivalence of groupoids.

Equivalence of groupoids can be defined using functors between them.
Every functor $\phi\colon \Gr\arr\Hr$ naturally defines a bi-action  $\Gr\curvearrowright\X\curvearrowleft\Hr$ in the following way. The space $\X$ is equal to the subset $\{(x, h)\in\Gr^{(0)}\times\Hr\colon \en(h)=\phi(x)\}$ of the direct product $\Gr^{(0)}\times\Hr$, and the actions are given by
\[g_1\cdot(x, h)\cdot h_1=(\en(g_1), \phi(g_1)hh_1),\]
over the anchors $(x, h)\mapsto x$ for the left $\Gr$-action and $(x, h)\mapsto \be(h)$ for the right $\Hr$-action.

The following characterization of equivalences defined by functors is proved, for example, in~\cite[Proposition~3.2.29]{nek:dyngroups}.

\begin{proposition}
\label{prop:equivalencefunctor}
The bi-action defined by a functor $\phi\colon \Gr\arr\Hr$ is an equivalence if and only if the following conditions are satisfied:
\begin{enumerate}
\item The map $\phi\colon \Gr^{(0)}\arr\Hr^{(0)}$ is open and surjective.
\item If $x, y\in\Gr^{(0)}$ and $h\in\Hr$ are such that $\phi(x)=\be(h)$ and $\phi(y)=\en(h)$, then there exists a unique $g\in\Gr$ such that $\be(g)=x, \en(g)=y$, and $\phi(g)=h$.
\end{enumerate}
\end{proposition}

\begin{example}
Let $\Gr$ be an \'etale groupoid, and let $f\colon \X\arr\Gr^{(0)}$ be a surjective local homeomorphism, where $\X$ is a topological space. The \emph{pull-back} $f^*(\Gr)$ is the groupoid consisting of triples $(x_2, g, x_1)$, where $x_1, x_2\in\X$, $g\in\Gr$ are such that $f(x_1)=\be(g)$ and $f(x_2)=\en(g)$. The topology is the relative topology of a subset of $\X\times\Gr\times\X$. We define the source and the range maps by
\[\be(x_2, g, x_1)=x_1,\qquad\en(x_2, g, x_1)=x_2,\]
and the multiplication by 
\[(x_3, g_2, x_2)(x_2, g_1, x_1)=(x_3, g_2g_1, x_1).\]

Then $\phi(x_2, g, x_1)=g$ is a functor from $f^*(\Gr)$ to $\Gr$ satisfying the conditions of Proposition~\ref{prop:equivalencefunctor}, hence $f^*(\Gr)$ and $\Gr$ are equivalent. 
\end{example}

\begin{example}
A particular case of the previous example is equivalence between the groupoid of the action of the fundamental group on the universal covering of a topological space with the trivial groupoid of the space.

More generally, any principal proper groupoid $\Gr$ is equivalent to the trivial groupoid on the space $\Gr^{(0)}/\Gr$ of orbits of $\Gr$. The equivalence is defined, for example, by the functor $\Gr\arr\Gr^{(0)}/\Gr$ mapping an element $g\in\Gr$ to the orbit of $\be(g)$.
\end{example}

\begin{definition}
An \emph{orbispace} is defined by a proper \'etale Hausdorff groupoid.
\end{definition}

We consider orbispaces up to equivalence of groupoids.

\subsection{Compactly generated and expansive groupoids}

\subsubsection{Compact generation and Cayley graphs}
\label{sss:cayleygraphs}

We say that an ample groupoid $\Gr$ with compact unit space is \emph{compactly generated} if it is generated by a compact subset. Equivalently, it is compactly generated if it is equal to the union of elements of a finitely generated inverse subsemigroup of $\mathcal{B}(\Gr)$.

We say that a finite set $\mathcal{S}\subset\mathcal{B}(\Gr)$ of compact open bisections generates $\Gr$ if the union of the elements of $\mathcal{S}$ generates $\Gr$.

Let $\Gr$ be an ample groupoid generated by a compact open set $S\subset\Gr$. For $x\in\Gr^{(0)}$, the \emph{Cayley graph} $\Gamma_x(\Gr, S)$ is the directed graph (\emph{digraph}) with the set of vertices $\be^{-1}(x)$ in which whenever $g_1, g_2\in\be^{-1}(x)$ and $s\in S$ are such that $g_2=sg_1$, we have an arrow starting in $g_1$ and ending in $g_2$.

If $\mathcal{S}$ is a finite set of compact open bisections generating $\Gr$, then we define $\Gamma_x(\Gr, S)$ as the graph with the set of vertices $\be^{-1}(x)$ in which for every $g\in\be^{-1}(x)$ and $F\in\mathcal{S}$ such that $\en(g)\in\be(F)$ we have an arrow starting in $g$, ending in $Fg$ and labeled by $F$.

We have the following behavior of compact generation with respect to the Morita equivalence of groupoids (see~\cite[Corollary~2.3.6]{nek:hyperbolic} and Proposition~\ref{pr:ambientgroupoid}).

\begin{proposition}
\label{pr:quasiisometric}
Suppose that $\Gr_1, \Gr_2$ are two ample groupoids with compact unit spaces. If $\Gr_1$ is compactly generated, then so is $\Gr_2$. Moreover, if $x\in\Gr_1^{(0)}$ and $y\in\Gr_2^{(0)}$ are related by the equivalence (i.e., if they are images of the same point under the anchor maps of the bi-action defining the equivalence), then the Cayley graphs based at $x$ and $y$ of $\Gr_1$ and $\Gr_2$, respectively, are quasi-isometric.
\end{proposition}

\begin{example}
\label{ex:Itinerary}
Let $f\colon \X\arr\X$ be a homeomorphism of a totally disconnected compact space $\X$. Let $\{U_1, U_2, \ldots, U_n\}$ be a partition of $\X$ into disjoint clopen subsets. Let $\mathfrak{F}$ be the groupoid of germs of the action of $\Z$ on $\X$ generated by $f$. Let $F_i$ be the restrictions of $f$ to the subsets $U_i$, seen as $\mathfrak{F}$-bisections. Then $\mathcal{F}=\{F_1, F_2, \ldots, F_n\}$ generates $\mathfrak{F}$. The Cayley graph $\G_x(\mathfrak{F}, \mathcal{F})$ is a bi-infinite chain of labeled arrows $\cdots\stackrel{F_{x_{-2}}}{\arr} f^{-1}(x)\stackrel{F_{x_{-1}}}{\arr} x\stackrel{F_{x_0}}{\arr} f(x)\stackrel{F_{x_1}}{\arr} \cdots$, where $x_k\in\{1, 2, \ldots, n\}$ are given by the condition $f^k(x)\in U_{x_k}$.
\end{example}

\subsubsection{Graph shift}
\label{sss:graphshift}

Let $S$ be a finite set, and let $\G$ be a directed graph together with a labeling of the edges of $\G$ by elements of $S$. We say that $\G$ is \emph{well labeled} if for every vertex $v$ of $\G$ and for every $s\in S$ there exists at most one edge labeled by $s$ starting in $v$ and at most one edge labeled by $s$ ending in $v$.

Let $\Fr_S$ be the set of all isomorphism classes of triples $(\G, v_1, v_0)$, where $\G$ is a (weakly) connected graph well labeled by $S$ and $v_0, v_1$ are two vertices of $\G$. Two triples are isomorphic if there exists an isomorphism of the corresponding directed labeled graphs preserving the marking of the two roots $v_1, v_0$. We introduce a metric on $\Fr_S$ with respect to which the distance between two marked labeled graphs $(\G, v_1, v_0)$ and $(\G', v_1', v_0')$ is equal to $2^{-R}$, where $R$ is the maximal radius for which there exists a label-preserving isomorphism of the ball in $\G$ of radius $R$ with center in $v_0$ with the ball in $\G'$ of radius $R$ with center in $v_0'$ mapping $v_1$ to $v_1'$. (If such a ball does not exist, we set $R=0$.) The set $\Fr_S$ will be then a totally disconnected compact metric space.

The space $\Fr_S$ has a natural structure of a groupoid with respect to the multiplication $(\G, v_2, v_1)(\G, v_1, v_0)=(\G, v_2, v_0)$. The multiplication is well defined, since the stabilizer of any vertex of a connected well labeled graph is trivial. It is also easy to check that $\Fr_S$ is an ample Hausdorff groupoid.

We call the groupoid $\Fr_S$ the \emph{full graph shift}. For every $s\in S$ we have the associated bisection $F_s$, consisting of all marked graphs $(\G, v_1, v_0)$ such that there is an arrow labeled by $s$ starting in $v_0$ and ending in $v_1$. The set of such bisections generates $\Fr_S$. 

Let $\Gr$ be an ample groupoid generated by a finite set $\mathcal{S}\subset\mathcal{B}(\Gr)$. We have a natural functor $\gamma\colon \Gr\arr\Fr_{\mathcal{S}}$ mapping an element $g\in\Gr$ to the triple $(\G_{\be(g)}(\Gr, \mathcal{S}), g, \be(g))$. 

This functor is not continuous in general for non-Hausdorff. A ball of the Cayley graph $\G_x(\Gr, \mathcal{S})$ with center in $x\in\Gr^{(0)}$ is described by a finite number of conditions of the form $x\in \be(F)$, $x\notin\be(F)$, $F_1x=F_2x$, $F_1x\ne F_2x$ for elements $F, F_1, F_2$ of the inverse semigroup generated by $\mathcal{S}$. The first three types of conditions, if true, are satisfied on an open neighborhood of $x$. An inequality may be satisfied on a non-open neighborhood of $x$, if $\Gr$ is not Hausdorff. However, if $\Gr$ is Hausdorff, then all such conditions are satisfied on open subsets of $\Gr^{(0)}$, and the functor $\gamma$ is continuous.

\subsubsection{Expansivity}

Let $\Gr$ be an ample groupoid with compact space of units. We say that a finite set of compact open $\Gr$-bisections is an \emph{expansive generating set} if the inverse semigroup generated by it is a basis of topology of $\Gr$. We say that the groupoid $\Gr$ is \emph{expansive} if it has an expansive generating set.

As we have see in~\ref{sss:inversesemigroups}, the groupoid $\Gr$ can be reconstructed from the abstract inverse semigroup $H$ generated by the expansive set of bisections as the groupoid of germs of the elements of $H$ in ultrafilters of the semi-lattice of idempotents of $H$. Hence, expansivity is a more natural notion of finite generation for ample groupoids than being compactly generated.

The following characterization of the expansivity was proved in~\cite[Proposition~5.5]{nek:fullgr}.

\begin{proposition}
\label{pr:expansive}
Let $\Gr$ be an ample groupoid and let $\mathcal{S}\subset\mathcal{B}(\Gr)$ be a finite set of compact open bisections generating $\Gr$. Let $H$ be the inverse semigroup generated by $\mathcal{S}$. Then the following conditions are equivalent:
\begin{enumerate}
\item The set $\mathcal{S}$ is expansive, i.e., $H$ is a basis of topology of $\Gr$.
\item For any two points $x, y\in\Gr^{(0)}$ there exist idempotents $U, V\in H$ such that $x\in U, y\in V$ and $U\cap V\ne\emptyset$.
\item The set of idempotents of $H$ is a basis of topology of $\Gr^{(0)}$.
\end{enumerate}
\end{proposition}

Expansivity of groupoids is related to expansivity of group actions in the following sense (see a proof in~\cite[Proposition~5.7]{nek:fullgr}).

\begin{proposition}
\label{pr:expansivegroupaction}
Let $G$ be a finitely generated group acting by homeomorphisms on a Cantor set $\mathcal{X}$. Choose a metric $d$ on $\mathcal{X}$. Let $G\ltimes\X$ and $\Gr$ be the groupoid of the action and the groupoid of germs, respectively. Then the following conditions are equivalent.
\begin{itemize}
\item[(i)] The groupoid $G\ltimes\X$ is expansive.
\item[(ii)] The groupoid $\Gr$ is expansive.
\item[(iii)] There exists $\delta>0$ such that, for any $x, y\in\X$, if $d(g(x), g(y))<\delta$ for all $g\in G$, then $x=y$.
\item[(iv)] The action of $G$ on $\X$ is topologically conjugate to the action of $G$ on a $G$-invariant closed subset of $A^G$ for some finite set $A$ (i.e., to a \emph{subshift}).
\end{itemize}
\end{proposition}

\begin{example}
\label{ex:itineraryexpansive}
In Example~\ref{ex:Itinerary}, the generating set $\mathcal{F}$ is expansive if and only if any two points $x, y\in\X$ such that $f^n(x)$ and $f^n(y)$ belong to the same element of the partition $\{U_1, U_2, \ldots, U_n\}$ for all $n\in\Z$ are necessarily equal. Then the map $I\colon \X\arr\{1, 2, \ldots, n\}^\omega$ mapping a point $x\in\X$ to its itinerary with respect to the partition is a homeomorphic embedding, and $f\colon \X\arr\X$ is topologically conjugate to the action of the shift on the image of $I$.
\end{example}

\subsection{Homology theory of ample groupoids}
\label{ss:homology}

We describe here a homology theory for ample groupoids, introduced by H.~Matui in~\cite{matui:etale} (as a specialization to ample groupoids of a more general theory due to M.~Crainic and I.~Moerdijk~\cite{CrainicMoerdki:homology}).

Let $\Gr$ be an ample groupoid. Denote by $\Gr^{(n)}$ the subspace of the direct power $\Gr^d$ consisting of sequences $(g_1, g_2, \ldots, g_n)$ such that the product $g_1g_2\cdots g_n$ is defined, i.e., such that 
$\be(g_i)=\en(g_{i+1})$ for all $i=1, \ldots, n-1$. We keep the notation $\Gr^{(0)}$ for the space of units.

If $(F_1, F_2, \ldots, F_n)$ is a sequence of compact open $\Gr$-bisections such that $\be(F_i)=\en(F_{i+1})$ for all $i=1, \ldots, n-1$, then the corresponding set of elements $\{(g_1, g_2, \ldots, g_n)\in\Gr^{(n)}\colon g_i\in F_i\}$ is a compact open set, which we will also denote by $(F_1, F_2, \ldots, F_n)$. We call such sequences \emph{ordered multisections}. The set of ordered multisections is a basis of topology of $\Gr^{(n)}$. 

Let us define maps $d_i\colon \Gr^{(n)}\arr\Gr^{(n-1)}$, for $i=0, 1, \ldots, n$ by
\[d_0(g_1, g_2, \ldots, g_n)=(g_2, g_3, \ldots, g_n),\quad d_n(g_1, g_2, \ldots, g_n)=(g_1, g_2, \ldots, g_{n-1}),\]
and
\[d_i(g_1, g_2, \ldots, g_n)=(g_1, \ldots, g_ig_{i+1}, \ldots, g_n)\]
for $i=1, 2, \ldots, n-1$. We also define the maps $d_i\colon \Gr^{(1)}\arr\Gr^{(0)}$ for $i=1, 0$ by
\[d_0(g)=\en(g), \qquad d_1(g)=\be(g).\]
It is not hard to check that the maps $d_i$ are local homeomorphisms. Namely, $d_i$ maps an ordered multisection $U\subset\Gr^{(n)}$ homeomorphically to an ordered multisection $d_i(U)\subset\Gr^{(n-1)}$.

Let $A$ be an abelian group. Denote by $C_c(\Gr^{(n)}; A)$ the subgroup of the group $A^{\Gr^{(n)}}$ of functions $\Gr^{(n)}\arr A$ generated by functions $f_{U, a}$ equal to a constant $a\in A$ on an ordered multisection $U\subset\Gr^{(n)}$ and to $0$ outside of $U$. If $\Gr$ is Hausdorff, then $C^{(n)}(\Gr; A)$ is equal to the group of continuous compactly supported functions $\Gr^{(n)}\arr A$, where $A$ is equipped with the discrete topology.

Denote by $(d_i)_*$ the homomorphisms $C_c(\Gr^{(n)}; A)\arr C_c(\Gr^{(n-1)}; A)$ given by
\[(d_i)_*(f)(g_1, g_2, \ldots, g_{n-1})=\sum_{d_i(h_1, h_2, \ldots, h_n)=(g_1, g_2, \ldots, g_{n-1})}f(h_1, h_2, \ldots, h_n).\]
In other words, if $U\subset\Gr^{(n)}$ is an ordered multiseciton, then $(d_i)_*$ maps the function supported on $U$ and equal to $a\in A$ on it to the function supported on $d_i(U)$ and equal to $a$ on it. Since restrictions of $d_i$ to multisections are homeomorphisms, and multisections form a basis of topology, this condition uniquely determines the homomorphism $(d_i)_*$. 

Define then the homomorphisms $\delta_n\colon C_c(\Gr^{(n)}; A)\arr C_c(\Gr^{(n-1)}; A)$ by
\[\delta_n=\sum_{i=0}^n(-1)^i(d_i)_*.\]

Then
\[0\stackrel{\delta_0}{\longleftarrow} C_c(\Gr^{(0)}; A)
\stackrel{\delta_1}{\longleftarrow} C_c(\Gr^{(1)}; A)
\stackrel{\delta_2}{\longleftarrow} C_c(\Gr^{(2)}; A)
\stackrel{\delta_3}{\longleftarrow} \cdots\]
is a chain complex, and we define
\[H_n(\Gr; A)=\mathop{\mathrm{Ker}}\delta_n/\mathop{\mathrm{Im}}\delta_{n+1}.\]
We denote $H_n(\Gr)=H_n(\Gr;\Z)$.

Homology groups $H_n(\Gr; A)$ are naturally invariant under equivalence of groupoids, see~\cite[Theorem~3.6]{matui:etale}.

The group $H_0(\Gr)$ has the following natural description.

\begin{proposition}
\label{pr:dimensiongroup}
The group $H_0(\Gr)$ is isomorphic to the abelian group defined by the following presentation. Its generators $1_U$ are in a bijective correspondence $U\mapsto 1_U$ with compact open subsets of $\Gr^{(0)}$. The defining relations are:
\begin{enumerate}
\item $1_{U_1\cup U_2\cup\cdots\cup U_m}=1_{U_1}+1_{U_2}+\cdots+1_{U_m}$ whenever $U_i$ are pairwise disjoint;
\item $1_{\be(F)}=1_{\en(F)}$ for every compact open $\Gr$-bisection $F$.
\end{enumerate}
\end{proposition}

Denote by $H_0^+(\Gr)$ the image in $H_0(\Gr)$ of the subsemigroup of $C_c(\Gr^{(0)}; \Z)$ generated by the indicators of clopen subsets of $\Gr^{(0)}$, i.e., the sub-semigroup of non-negative functions. 

Assume that $\Gr^{(0)}$ is compact. The the triple $(H_0(\Gr), H_0^+(\Gr), [1])$, where $[1]$ is the image of the constant function $1$ on $\Gr^{(0)}$, is called the \emph{dimension group} of $\Gr$. 

A \emph{state} on $(H_0(\Gr), H_0^+(\Gr), [1])$ is a homomorphism $f\colon H_0(\Gr)\arr\R$ such that $f(a)\ge 0$ for every $a\in H^+(\Gr)$, and $f([1])=1$.

We say that a Borel measure $\mu$ on $\Gr^{(0)}$ is \emph{$\Gr$-invariant} if $\mu(\be(F))=\en(F)$ for every $\Gr$-bisection $F$.

For a proof of the following theorem (which is essentially the same as~\cite[Theorem~3.5]{putn:bratel}), see~\cite[Theorem~5.4.4]{nek:dyngroups}.

\begin{theorem}
Let $\Gr$ be an ample groupoid with a compact space of units. Let $M(\Gr)$ be the simplex of $\Gr$-invariant Borel probability measures. Then the map transforming $\mu\in M(\Gr)$ into the map $[U]\mapsto\mu(U)$ from $H_0(\Gr)$ to $\R$ is an affine isomorphism from $M(\Gr)$ to the affine simplex of states on $(H_0(\Gr), H_0^+(\Gr), [1])$.
\end{theorem}

\begin{example}
Consider the groupoid $\Gr$ with the space of units $\{0, 1\}^\omega=\{x_1x_2\ldots : x_i\in\{0, 1\}\}$ generated by  the germs of the transformations \[v_1\{0, 1\}^\omega\arr v_2\{0, 1\}^\omega\colon v_1w\mapsto v_2w\] for $v_1, v_2\in\{0, 1\}^*$ of equal lengths. By Proposition~\ref{pr:dimensiongroup}, the images in $H_0(\Gr)$ of the indicators of cylindrical sets $v\{0, 1\}^\omega$ for a given length $|v|=n$ are equal to each other, and are twice the images of the indicators of $v\{0, 1\}^\omega$ for $|v|=n+1$. Consequently, $H_0(\Gr)$ is isomorphic to the inductive limit of the sequence
$\Z\arr\Z\arr\cdots$ of the homomorphisms $x\mapsto 2x$. It follows that the dimension group $(H_0(\Gr), H_0^+(\Gr), [1])$ is isomorphic to $(\Z[1/2], \Z[1/2]\cap[0, \infty), 1)$, and that there exists a unique $\Gr$-invariant probability measure corresponding to the identity map $\Z[1/2]\arr\R$.
\end{example}

\section{Algebras}

\subsection{Steinberg algebras}
\label{ss:Steingergalgebras}

If $F\subset\Gr$ is an open bisection, then for every unit $x\in\be(F)$ the set $\be^{-1}(x)\cap F$ consists of a single element. It follows that if $\Gr$ is \'etale, then for every $x\in\Gr^{(0)}$ the set $\be^{-1}(x)$ is a closed discrete subset of $\Gr$. Similarly, $\en^{-1}(x)$ is closed and discrete. (Groupoids with this property are called sometimes \emph{$r$-discrete}.) 

Consequently, if $C\subset\Gr$ is compact, then for every $x\in\Gr^{(0)}$ the sets $\be^{-1}(x)\cap C$ and $\en^{-1}(x)\cap C$ are finite (as closed, hence compact, discrete sets). 

We assume for the rest of this section that $\Gr$ is ample.

Let $\Bbbk$ be a field, and let $f\colon \Gr\arr\Bbbk$ be a map. We say that $f$ is \emph{compactly supported} if $f^{-1}(\Bbbk\setminus\{0\})$ is contained in a compact subset of $\Gr$ (we do not assume here that $f$ is continuous).

The \emph{convolution} $f_1*f_2$ of functions $f_1, f_2\colon \Gr\arr\Bbbk$ is defined by
\[f_1*f_2(g)=\sum_{g=g_1g_2}f_1(g_1)f_2(g_2).\]
If $f_1$ is compactly supported, then we have $\en(g_1)=\en(g)$ on the right-hand side, hence the sum has finitely many non-zero summmands. Similarly, since $\be(g_2)=\be(g)$, the sum on the right-hand side is finite if $f_2$ is compactly supported.

In particular, convolution of two compactly supported functions is defined. One can show that product of two compact subsets of an \'etale groupoid is contained in a compact set. Consequently, convolution of two compactly supported functions is compactly supported.

From now on, we will often write the convolution $f_1*f_2$ just as $f_1f_2$.

The following is straightforward.

\begin{lemma}
If $F_1, F_2\subset\Gr$ are compact open bisections, then $1_{F_1}1_{F_2}=1_{F_1F_2}$.
\end{lemma}

It follows that the linear span $C_c(\Gr; \Bbbk)$ of the indicators of compact open bisections is closed under the convolution, hence is a $\Bbbk$-algebra. We denote it $\Bbbk\Gr$ and call it the \emph{Steinberg algebra} of $\Gr$, see~\cite{steinberg:groupoidapproach}.

In the case of the field of complex numbers, $\C\Gr$ has an additional structure of a $*$-algebra with involution
\[\left(\sum_k a_k1_{F_k}\right)^*=\sum_k\overline{a_k}1_{F_k^{-1}},\]
where $\overline z$ is the complex conjugation.

\subsection{Ideals in Steinberg algebras}

If a set $U\subset\Gr^{(0)}$ is open, then every open $\Gr|_U$-bisection $F$ is also an open $\Gr$-bisection. It follows that $\Bbbk\Gr|_U$ is a sub-algebra of $\Bbbk\Gr$. 

If $U$ is open and $\Gr$-invariant, then the subalgebra $\Bbbk\Gr|_U$ is an ideal in $\Bbbk\Gr$. Let us denote it by $\mathcal{I}_0(U)$. It is equal to the ideal generated by the set of functions $1_V$ for all compact open subsets $V\subset U$.

For any $\Gr$-invariant subset $U\subset\Gr^{(0)}$, denote by $\mathcal{I}(U)$ the set of functions $f\in\Bbbk\Gr$ such that $f(g)=0$ for all $g\in\Gr|_{\Gr^{(0)}\setminus U}$.  It is also an ideal in $\Bbbk\Gr$.

We obviously have $\mathcal{I}_0(U)\subset\mathcal{I}(U)$ for every open $\Gr$-invariant set $U\subset\Gr^{(0)}$.

\begin{lemma}
If $\Gr$ is Hausdorff, then for every open $\Gr$-invariant set $U$ we have $\mathcal{I}(U)=\mathcal{I}_0(U)$.
\end{lemma}

\begin{proof}
We need to prove that $\mathcal{I}(U)\subset\mathcal{I}_0(U)$.
Let $f\in\mathcal{I}(U)$. Since $f$ is continuous, it is locally constant, hence we can write $f$ as a linear combination with non-zero coefficients of indicators of disjoint compact open bisections. For each of these bisections $F$ we have $f(g)\ne 0$ for all $g\in F$, which implies that $F\subset\Gr|_U$. Consequently, $f\in\mathcal{I}_0(U)$.
\end{proof}

We have associated with every open $\Gr$-invariant set $U$ two ideals. Let us conversely, associate with every ideal in $\Bbbk\Gr$ two $\Gr$-invariant subsets.

\begin{definition}
Let $I\subset\Bbbk\Gr$ be an ideal. Its \emph{support} $S(I)$ is the set of units $x\in\Gr^{(0)}$ such that there exist $g\in\Gr$ and $f\in I$ such that $\be(g)=x$ and $f(g)\ne 0$.

The \emph{internal support} $S_0(I)$ is the union of all open compact sets $U\subset\Gr^{(0)}$ such that $1_U\in I$.
\end{definition}

We obviously have $S_0(I)\subset S(I)$. 

\begin{lemma}
Let $I$ be an ideal of $\Bbbk\Gr$. 
\begin{enumerate}
\item The sets $S(I)$ and $S_0(I)$ are $\Gr$-invariant.

\item An open compact subset $U\subset\Gr^{(0)}$ is contained in $S_0(I)$ if and only if $1_U\in I$. 

\item If $\Gr$ is Hausdorff, then $S(I)$ is open.
\end{enumerate}
\end{lemma}

\begin{proof}
Let us show that $S(I)$ is $\Gr$-invariant. Let $x\in S(I)$. Then there exists $f\in I$ and $g\in\Gr$ such that $f(g)\ne 0$ and $\be(g)=x$. Let $F$ be a compact open bisection such that $g^{-1}\in F$. Consider the product $f\cdot 1_F\in I$ and let us evaluate it at $\en(g)=gg^{-1}$. If $(h_1, h_2)\in\Gr^{(2)}$ are such that $h_1h_2=\en(g)$ and $h_2\in F$, then $\be(h_2)=\en(g)$, and since $\be\colon F\arr\be(F)$ is a homeomorphism and $\be(g^{-1})=\en(g)$, we get $h_2=g^{-1}$. Then $h_1h_2=\en(g)$ implies that $h_1=g$. Consequently, the only non-zero summand in the definition of the value of the convolution $f\cdot 1_F$ at $\en(g)$ is $f(g)\cdot 1_F(g^{-1})=f(g)$. Since $f\cdot 1_F\in I$ and $f(g)\ne 0$, this implies that $\en(g)\in S(I)$.

Let us prove the statement (2) before proving that $S_0(I)$ is $\Gr$-invariant.
If $U_1, U_2\subset\Gr^{(0)}$ are compact open sets, then 
\begin{equation}
\label{eq:1operations}
1_{U_1\cap U_2}=1_{U_1}1_{U_2},\qquad 1_{U_1\cup U_2}=1_{U_1}+1_{U_2}-1_{U_1}1_{U_2}
\end{equation}

Let $\mathcal{I}$ be the set of all compact open sets $U\subset\Gr^{(0)}$ such that $1_U\in I$. Equations~\ref{eq:1operations} imply that $\mathcal{I}$ is closed under taking finite unions and passing to a clopen subset. 

The ``if'' direction of the statement (2) of the lemma follows directly from the definition of $S_0(I)$. Suppose that $U$ is a compact open set of units contained in $S_0(I)$. Then, by compactness of $U$, we can cover it by finitely many elements of $\mathcal{I}$. Since $\mathcal{I}$ is closed under taking finite unions and passing to clopen subsets, this implies that $U\in\mathcal{I}$.

Let us prove now that $S_0(I)$ is $\Gr$-invariant.
Suppose that $x\in S_0(I)$, and let $g\in\Gr$ be such that $\be(g)=x$. Let $U\in\mathcal{I}$ be such that $x\in U$, and let $F$ be a compact open bisection containing $g$. Denote $V_1=\be(F)\cap U$, and let $V_2=\en(FV_1)$ be the image of $V_2$ under the map $\alpha_F$. Then $V_1\in\mathcal{I}$, since it is contained in $U\in\mathcal{I}$. We have
\[1_{V_2}=1_F1_{V_1}1_{F^{-1}},\]
hence $1_{V_2}\in I$, i.e., $V_2\in\mathcal{I}$. Consequently, $\en(g)\in S_0(I)$. This shows that the $S_0(I)$ is $\Gr$-invariant.

Let us prove statement (3). Suppose that $\Gr$ is Hausdorff. Then every element of $\Bbbk\Gr$ is continuous (locally constant). Let $x\in S(I)$. Then there exists $f\in I$ and $g\in\Gr$  such that $f(g)\ne 0$ and $\be(g)=x$. Since $f$ is continuous, there exists a neighborhood $F$ of $g$ such that $f(h)\ne 0$ for every $h\in F$. Consequently, the neighborhood $\be(F)$ of $x$ is contained in $S(I)$. It follows that $S(I)$ is open if $\Gr$ is Hausdorff.
\end{proof}

\begin{proposition}
\label{pr:idealbetween}
For every ideal $I\subset\Bbbk\Gr$ we have \[\mathcal{I}_0(S_0(I))\subset I\subset \mathcal{I}(S(I)).\]
\end{proposition}

\begin{proof}
The inclusion $I\subset\mathcal{I}(S(I))$ follows directly from the definitions. The inclusion $\mathcal{I}_0(S_0(I))$ follows from statement (2) of the lemma.
\end{proof}

Denote by $\mathfrak{S}$ the set of points of $\Gr^{(0)}$ with non-trivial isotropy group. We call points of $\mathfrak{S}$ \emph{singular}. We say that a groupoid $\Gr$ is \emph{essentially principal} if the set $\Gr^{(0)}\setminus\mathfrak{S}$ of points with trivial isotropy group is dense.

\begin{proposition}
A second countable effective \'etale groupoid is essentially principal.
\end{proposition}

\begin{proof}
An effective \'etale groupoid is, by definition, the groupoid of germs of an inverse semigroup $\mathcal{H}$ of homeomorphisms between open subsets of a topological space $\X$. If it is second countable, then we may assume that the inverse semigroup is countable (take the inverse semigroup generated by a countable set of open bisections of the groupoid forming a basis of topology of the groupoid). For every $F\in\mathcal{H}$, the set of points $x$ in the domain of $F$ such that the germs $[F, x]$ is isotropic but not a unit is nowhere dense. It follows that the set of points of $\X$ with non-trivial isotropy groups is a countable union of nowhere dense sets. By the Baire category theorem, its complement $\mathfrak{S}$ is dense.
\end{proof}

\begin{proposition}
\label{pr:smins0singular}
For every ideal $I$ of $\Bbbk\Gr$ we have $S(I)\setminus S_0(I)\subset\mathfrak{S}$.
\end{proposition}

\begin{proof}
Let $f=\sum_{i=1}^m\alpha_i1_{F_i}$ for some compact open bisections $F_i\subset\Gr$ and $\alpha_i\in\Bbbk$.  Let $A$ be the set of indices $i=1, 2, \ldots, m$ such that $g\in F_i$. Then $f(v)=\sum_{i\in A}\alpha_i$.

Let $x_0\in S(I)$ be a point with trivial isotropy group. Then $\en(F_ix_0)\ne\en(g)$ for $i\notin A$. Let $V_0$ be a small neighborhood of $x_0$ contained in $\cap_{i\in A}\be(F_i)$. Then $F=F_iV_0$ does not depend on $i\in A$. Denote $V_1\en(F)$. 

By Hausdorffness of $\Gr^{(0)}$, if $V_0$ is sufficiently small, then $V_1$ is disjoint with every set $\en(F_jV_0)$ for $j\notin A$.

Consider the element  $1_{V_1}\cdot f\cdot 1_{V_0}$ of $\Bbbk\Gr$. If $i\in A$, then $1_{V_1}1_{F_i}1_{V_0}=1_F$. Otherwise $1_{V_1}1_{F_i}1_{V_0}=0$, since $\en(F_iV_0)$ and $V_1$ are disjoint. It follows that $1_{V_1}\cdot f\cdot 1_{V_0}=f(g)1_F$. Since $f(g)\ne 0$, this implies that $1_{V_0}=f(g)^{-1}1_{F^{-1}}f(g)1_F\in I$, i.e., that $V_0\subset S_0(I)$.
\end{proof}

\begin{theorem}
\label{th:idealbijection}
If $\Gr$ is Hausdorff and essentially principal, then for every ideal $I$ there exist open $\Gr$-invariant sets $U_1\subset U_2$ such that $U_1$ is dense in $U_2$ and $\Bbbk\Gr|_{U_1}\subset I\subset\Bbbk\Gr|_{U_2}$.

If $\Gr$ is principal, then the map $U\mapsto\Bbbk\Gr|_U$ is a bijection between the set of open $\Gr$-invariant subsets and the set of ideals of $\Bbbk\Gr$.
\end{theorem}

\begin{proof}
For the first statement, set $U_1=S_0(I)$ and $U_2=S(I)$. Then $\mathcal{I}_0(U_1)\subset I\subset\mathcal{I}(U_2)=\mathcal{I}_0(U_2)$. Suppose that $S_0(I)$ is not dense in $S(I)$. Then there exists an open set $V\subset S(I)$ disjoint from $S_0(I)$.  Since the set of points with trivial isotropy is dense, $V$ must contain a point with trivial isotropy. But this contradicts Proposition~\ref{pr:smins0singular}.

Note that every principal groupoid is Hausdorff.
If $\Gr$ is principal, then the set $\mathfrak{S}$ is empty, so $S(I)=S_0(I)$ for every ideal $I\subset\Bbbk\Gr$, which proves that $I=\Bbbk\Gr|_{S(I)}$.
\end{proof}

Let $\Gr$ be an essentially principal groupoid.
Denote $\essent=\mathcal{I}(\mathfrak{S})$. Since $\mathfrak{S}$ is nowhere dense, any function $f\in\essent$ is equal to 0 on a dense subset of $\Gr^{(0)}$. In particular, if $\Gr$ is Hausdorff, then $\essent=\{0\}$.

The following result was proved in~\cite{steinberg:simplicity}, see also~\cite[Proposition~4.1]{nek:simplegrowth}.

\begin{theorem}
\label{th:simplealgebra}
If $\Gr$ is minimal and essentially principal, then every proper ideal of $\Bbbk\Gr$ is contained in $\essent$. In particular, if $\Gr$ is additionally Hausdorff, then $\Bbbk\Gr$ is simple.
\end{theorem}

\begin{proof}
Let $I$ be a proper ideal in $\Bbbk\Gr$. Since $S_0(I)$ is open, $\Gr$-invariant, and $\mathcal{I}_0(S_0(I))\subset I$, the set $S_0(I)$ must be empty. Otherwise, it is equal to $\Gr^{(0)}$, and $\mathcal{I}_0(S_0(I))=\Bbbk\Gr$.

But then the condition $S(I)\setminus S_0(I)\subset\mathfrak{S}$ implies that $S(I)\subset\mathfrak{S}$, so $I\subset\mathcal{I}(S(I))\subset\essent$.
\end{proof}

\subsection{Morita equivalence}

For a ring $R$, we denote by $M_n(R)$ the ring of $n\times n$ matrices with entries in $R$. Let us show that equivalence of ample groupoids implies Morita equivalence of their Steinberg algebras. 

\begin{theorem}
Let $\Gr_1, \Gr_2$ be ample groupoids with compact unit spaces. If they are Morita equivalent, then there exists $n\ge 1$ and a full idempotent $p\in M_n(\Bbbk\Gr_1)$ such that $\Bbbk\Gr_2$ is isomorphic to $pM_n(\Bbbk\Gr_1)p$. Consequently, the rings $\Bbbk\Gr_1$ and $\Bbbk\Gr_2$ are Morita equivalent.
\end{theorem}

An idempotent $p$ in a ring $R$ is called \emph{full} if the two-sided ideal $RpR$ generated by it is equal to $R$. See~\cite[Proposition~18.33]{lam:modulesandrings} for the characterization of Morita equivalence used in our theorem.

\begin{proof}
By Proposition~\ref{pr:ambientgroupoid}, there exists an ample groupoid $\Hr$ such that $\Hr^{(0)}$ is equal to the disjoint union of the spaces $\Gr_1^{(0)}$ and $\Gr_2^{(0)}$, the restriction of $\Hr$ to $\Gr_i^{(0)}$ is equal to $\Gr_i$, and every $\Hr$-orbit is a union of one $\Gr_1$-orbit and one $\Gr_2$-orbit.

Since the unit spaces $\Gr_i^{(0)}$ are compact, there exists a finite collection of compact open $\Hr$-bisections $\mathcal{T}=\{F_i\}_{i=1, \ldots, n}$ such that $\{\be(F_i)\}_{i=1, \ldots, n}$ is a partition of $\Gr_1^{(0)}$, and $\en(F_i)\subset\Gr^{(0)}_2$ for every $F_i\in\mathcal{T}$.

Let $A$ be a compact open $\Gr_1$-bisection. Define $\phi(1_A)$ as the matrix $(a_{ij})_{1\le i, j\le n}\in M_n(\Bbbk\Gr_2)$, where $a_{ij}=1_{F_iAF_j^{-1}}$. 
Extension of this map by linearity to $\Bbbk\Gr_1$ is defined by the condition that, for $f\in\Bbbk\Gr_1$, we have $\phi(f)=(f_{ij})_{1\le i, j\le n}$  for
\[f_{ij}(g)=\left\{\begin{array}{ll} f(F_i^{-1}gF_j) & \text{if $\be(g)\in\en(F_j)$ and $\en(g)\in\en(F_i)$;}\\
0 & \text{otherwise.}\end{array}\right.\]
If $f\ne 0$, then there exists $h\in\Gr_1$ such that $f(h)\ne 0$, and then $\phi(f)=(f_{ij})_{1\le i, j\le n}$ has a non-zero entry $f_{ij}$, where $i, j$ are such that $\en(h)\in\be(F_i)$ and $\be(h)\in\be(F_j)$.

 Let $B$ is a compact open $\Gr_1$-bisection, and $\phi(1_B)=(b_{ij})_{1\le i, j\le n}$. Then $\phi(A)\phi(B)=(c_{ij})_{1\le i, j\le n}$ for \[c_{ij}=\sum_{k=1}^na_{ik}b_{kj}=\sum_{k=1}^n 1_{F_iAF_k^{-1}}1_{F_kBF_j^{-1}}=1_{F_iA}\left(\sum_{k=1}^n 1_{F_k^{-1}F_k}\right)1_{BF_j^{-1}}=1_{F_iABF_j^{-1}}.\] It follows that $\phi$ is an isomorphic embedding of $\Bbbk\Gr_1$ into $M_n(\Bbbk\Gr_2)$.
 
Let $p\in M_n(\Bbbk\Gr_2)$ be the diagonal matrix $(p_{ij})_{1\le i, j\le n}$ with the entries $p_{ii}=1_{\en(F_i)}$.  It is an idempotent, and it follows directly from the definition of $\phi$ that $\phi(\Bbbk\Gr_1)\subset pM_n(\Bbbk\Gr_2)p$. On the other hand, for every $\Gr_2$-bisection $F$ such that $\be(F)\subset\en(F_j)$ and $\en(F)\subset\be(F_i)$ the matrix $\phi(1_{F_i^{-1}FF_j})$ has all entries equal to zero except for the entry $1_F$ in the row number $i$ and column number $j$. Consequently, $\phi\colon\Bbbk\Gr_1\arr pM_n(\Bbbk\Gr_2)p$ is an isomorphism of algebras.

It remains to show that $p$ is full. Let us denote, for $a\in\Bbbk\Gr_2$ and $1\le k, l\le n $, by $E_{kl}(a)$ the matrix $(a_{ij})_{1\le i, j\le n}$ given by 
\[a_{ij}=\left\{\begin{array}{ll} a & \text{if $i=k, j=l$;}\\
0 & \text{otherwise.}\end{array}\right.\]
We have then $E_{ij}(a)E_{kl}(b)=0$ if $j\ne k$, and $E_{ik}(a)E_{kj}(b)=E_{ij}(ab)$.

It is enough to show that for every compact open $\Gr_2$-bisection $F$ and every pair of indices $1\le k, l\le n$ we have $E_{kl}(1_F)\in M_n(\Bbbk\Gr_2)p M_n(\Bbbk\Gr_2)$.

Since intersections of $\Hr$-orbits with $\Gr_2^{(0)}$ are $\Gr_2$-orbits, the set $\bigcup_{i=1}^n\en(F_i)$ intersects every $\Gr_2$-orbit. Consequently, for every $g\in F$ there exists a compact open neighborhood $A$ of $g$, an index $1\le i\le n$, and  a $\Gr_2$-bisection $R_i$ such that $\en(A)=\be(R_i)$ and $\en(R_i)\subset\en(F_i)$. Then $A=R_i^{-1}\be(F_i)R_iA$, hence
\[E_{kl}(A)=E_{ki}(1_{R_i^{-1}})pE_{il}(1_{R_iA})\in M_n(\Bbbk\Gr_2)p M_n(\Bbbk\Gr_2).\]
Since $F$ is compact, we can find a partition of $F$ into compact open bisections $A_s$ satisfying the above conditions. Since $1_F$ is equal to the sum of the corresponding indicators $1_{A_s}$, we get that $E_{kl}(1_F)\in M_n(\Bbbk\Gr_2)p M_n(\Bbbk\Gr_2)$, which finishes the proof.
\end{proof}

\subsection{Finite generation}

The next theorem is proved in~\cite[Proposition~4.4]{nek:simplegrowth}.

\begin{theorem}
\label{th:fingenalgebra}
Let $\Gr$ be an ample groupoid. The algebra $\Bbbk\Gr$ is finitely generated if and only if $\Gr$ is expansive.
\end{theorem}

\begin{proof}
Suppose that $\Bbbk\Gr$ is finitely generated. Consider a finite generating set $\{f_i\}_{i\in I}$. We can write the generators $f_i$ as finite linear combinations of indicator functions of compact open bisections, and replace $f_i$ by the set of these indicator functions. Therefore, we may assume that the generating set of $\Bbbk\Gr$ is of the form $\{1_F : F\in\mathcal{S}\}$, where $\mathcal{S}$ is a finite set of compact open bisections. By increasing the set $\mathcal{S}$, we may assume that it is symmetric. 

Let us show that the inverse semigroup generated by $\mathcal{S}$ is a basis of topology of $\Gr$. Since $A=\{1_F : F\in\mathcal{S}\}$ is a generating set of $\Bbbk\Gr$, every element of $\Bbbk\Gr$ is a linear combination of products of elements of $A$. 

Suppose that the inverse semigroup $G$ generated by $\mathcal{S}$ is not a basis of topology. Then there exist two points $x, y\in\Gr^{(0)}$ such that they are not separated by any idempotent in $G$. Since $\Gr^{(0)}$ is Hausdorff, there exist compact open disjoint neighborhoods $V_1\ni x$ and $V_2\ni y$. If the function $1_{V_1}$ and $1_{V_2}$ belong to the algebra generated by $A$, then each of them is a linear combination of products of the elements of $A$, so there exist $V_i'\subset V_i$ such that $V_1', V_2'\in\bigcup_{n\ge 1}\mathcal{S}^n$, but this is a contradiction.

Suppose now that $\mathcal{S}$ is a set of bisections satisfying the conditions of Proposition~\ref{pr:expansive}. Let us show that $A=\{1_F : F\in\mathcal{S}\}$ is a generating set of $\Bbbk\Gr$. Let $\mathcal{A}$ be the algebra generated by $A$. Denote by $\mathcal{U}$ the set of compact open subsets of $\Gr^{(0)}$ such that their indicators belong to $\mathcal{A}$. By Proposition~\ref{pr:expansive}, the set $\mathcal{U}$ is a basis of topology of $\Gr^{(0)}$. By Equations~\eqref{eq:1operations}, the set $\mathcal{U}$ is closed under taking finite unions. Let $V$ be an arbitrary clopen subset of $\Gr^{(0)}$. Since $\mathcal{U}$ is a basis of topology, $V$ is a union of elements of $\mathcal{U}$. Since $V$ is compact, we can represent it as a finite union of elements of $\mathcal{U}$. Consequently, $1_V\in\mathcal{A}$. It follows that indicators of all compact open subsets of $\Gr^{(0)}$ belong to $\mathcal{A}$. It follows that if $F$ is a compact open bisection such that $1_F\in\mathcal{A}$, then for every compact open bisection $F'\subset F$ we have $1_{F'}\in\mathcal{A}$, since we have $1_{F'}=1_F1_{\be(F')}$.

Let $F\subset\Gr$ be an arbitrary compact open bisection. Since $\mathcal{S}$ satisfies the conditions of Proposition~\ref{pr:expansive}, for every $g\in F$ there exists a compact open bisection such that $g\in F_g\subset F$ and $1_{F_g}\in\mathcal{A}$. Since $F$ is compact, we can cover $F$ by a finite number of such sub-bisections. But then, shrinking these sub-bisections, we can assume that they are disjoint. It follows that $1_F$ is equal to a finite sum of elements of $\mathcal{A}$, so $1_F\in\mathcal{A}$, hence $\mathcal{A}=\Bbbk\Gr$.
 \end{proof}

\begin{example} 
Let $\F\subset\alb^{\Z}$ be an infinite shift-invariant closed set. Consider the groupoid $\mathfrak{F}$ of the action of $\Z$ on $\F$. As in Example~\ref{ex:itineraryexpansive}, the set $\{F_x\}_{x\in\alb}$ of restrictions of the shift to the subsets $U_x\subset\F$ of sequences equal to $x$ in the coordinate number 0, is an expansive generating set of the groupoid.
It follows that the set $\{1_{F_x}\}_{x\in\alb}$ generates $\Bbbk\mathfrak{F}$. 

Suppose that the orbit of a sequence $w=(x_n)_{n\in\Z}\in\F$ under the shift $\si\colon\F\arr\F$ is dense in $\F$. Let $M$ be the set of finitely supported functions from the orbit to $\Bbbk$, i.e., the vector space with the basis equal to the orbit $\{\si^n(w)\colon n\in\Z\}$. Then $M$ is a $\Bbbk\mathfrak{F}$-module with respect to convolution. Density of the orbit (and the fact that $\mathfrak{F}$ is Hausdorff) implies that the obtained representation of $\Bbbk\mathfrak{F}$ by linear operators on $M$ is faithful. If we use the ordered basis $\{\si^n(w)\colon n\in\Z\}$, and represent the operators on $M$ by matrices, then the generator $F_x$ is represented by the matrix $A_x=(a_{mn})_{m, n\in\Z}$ given by
\[a_{mn}=\left\{\begin{array}{ll} 1 & \text{if $n=m+1$ and $x_n=x$,}\\
0 & \text{otherwise.}\end{array}\right.\]
Consequently, $\Bbbk\mathfrak{F}$ is isomorphic to the $\Bbbk$-algebra generated by the matrices $A_x$ and $A^\top_x$ for all $x\in\alb$.
\end{example}

\subsection{Growth of algebras}
\label{ss:growthalgebras}

Let $\mathcal{A}$ be a unital algebra over a field $\Bbbk$, and suppose it is generated by a finite set $A$. Then $\mathcal{A}$ is an increasing union of the $\Bbbk$-spaces $V_n(A)$ equal to the span of $\bigcup_{n\ge 0}A^n$, where $A^0=\{1\}$. Denote $\gamma_A(n)=\dim_{\Bbbk}V_n(A)$. Note that the linear span of the set $V_n(A)V_m(A)$ is equal to $V_{n+m}(A)$, so $\gamma_A(n+m)\le\gamma_A(n)\gamma_A(m)$.

If $B$ is another generating set, then there exists $n_0$ such that $B\subset V_{n_0}(A)$. Then $V_n(B)\le V_{n_0n}(A)$, so $\gamma_B(n)\le\gamma_A(n_0n)$. Similarly, there exists $n_1$ such that $\gamma_A(n)\le\gamma_B(n_1n)$. 

We will say that two growth functions $f_1, f_2$ are \emph{equivalent}, written $f_1\sim f_2$, if there exists a constant $C>1$ such that  $f_1(n)\le f_2(Cn)$ and $f_2(n)\le f_1(Cn)$ for every $n$. By the above, the equivalence class of the function $\gamma_A(n)$ does not depend on the generating set $A$.

Let $\Gr$ be an \'etale groupoid, and let $S\subset\Gr$ be a compact open generating set of $\Gr$.
We define the \emph{growth function}  $\gamma_{x, S}(r)$ for $x\in\Gr^{(0)}$ as the number of elements $g\in\Gr$ such that $\be(g)=x$ and $g$ is equal to a product $s_1s_2\ldots s_n$ of length $n\le r$ of elements of $S\cup S^{-1}$.

In other words, $\gamma_{x, S}(r)$ is the number of points in the ball of radius $r$ of the Cayley graph $\G_x(\Gr, S)$ with center in $x$.

Note that if $\mathcal{S}$ is a finite cover of $S$ by $\Gr$-bisections, then $\gamma_{x, S}(r)\le (|\mathcal{S}|+1)^r$. In particular, the function
\[\overline\gamma_S(r)=\max_{x\in\Gr^{(0)}}\gamma_{x, S}(r)\]
is finite for every $r$.

Since the quasi-isometry class of a Cayley graph of a groupoid does not depend on the generating set (see Proposition~\ref{pr:quasiisometric}), we have the following.

\begin{proposition}
The growth rates (i.e., the equivalence classes for the relation $\sim$ defined above) of the functions $\gamma_{x, S}(r)$ and $\overline\gamma_S(r)$ do not depend on the choice of the compact generating set $S$ of the groupoid. 
\end{proposition}

Moreover, the rate of growth of $\gamma_{x, S}(r)$ is stable under passing to a Morita equivalent ample groupoids, see Proposition~\ref{pr:quasiisometric}. Consequently, it is a well defined notion for arbitrary groupoids that are Morita equivalent to ample groupoids.

Another natural function associated with Cayley graphs of groupoids is their \emph{complexity}, defined as follows.

\begin{definition}
\label{def:complexity}
Let $\mathcal{S}$ be a finite set of $\Gr$-bisections generating $\Gr$. The \emph{complexity} $\delta_{\mathcal{S}}(r)$ is the number of isomorphism classes of rooted labeled graphs spanned by balls of radius $r$ in the Cayley graphs $\G_x(\Gr, \mathcal{S})$. 

More explicitly, we define the ball $B_x(r)$ as the set of elements $g\in\Gr$ such that $\be(g)=x$ and $g\in F_1F_2\cdots F_n$ for $n\le r$ and $F_i\in\mathcal{S}\cup\mathcal{S}^{-1}$. A map $\beta\colon B_x(r)\arr B_y(r)$ is an isomorphism if $\beta(x)=\beta(y)$ and $\beta(Fg)=F\beta(g)$ for every $g\in B_x(r)$ and $F\in\mathcal{S}\cup\mathcal{S}^{-1}$. Existence of an isomorphism is an equivalence relation, and the number of its equivalence classes is, by definition, $\delta_{\mathcal{S}}(r)$.
\end{definition}

\begin{example}
In the setting of Example~\ref{ex:Itinerary}, we have the growth function $\gamma_{x, \mathcal{F}}(r)=2r+1$ and $\delta_{\mathcal{F}}(r)$ is equal to the number of sequences $i_1i_2\ldots i_{2r}$ such that there exists $x\in\X$ such that $f^k(x)\in U_{x_k}$ for every $k=1, \ldots, 2r$. Consider the set $\mathcal{I}$ of all \emph{itineraries} of points of $\X$ with respect to the partition $\{U_1, U_2, \ldots, U_n\}$, i.e., sequences $(a_i)_{i\in\Z}$ for which there exists $x\in\X$ such that $f^k(x)\in U_{a_k}$ for every $k\in\Z$. Then $\delta(r)$ is the number of words of length $2r$ that are subwords (segments) of elements of $\mathcal{I}$. This function is often called \emph{factor complexity} of the set $\mathcal{I}$.
\end{example}

The following is proved in~\cite[Theorem~1.1]{nek:gk}.

\begin{theorem}
\label{th:growthofalgebras}
Let $\Gr$ be an ample groupoid, and let $\mathcal{S}$ be a finite set of compact open $\Gr$-bisections. Let $V^n\subset\Bbbk\Gr$ be the linear span of the indicators of products $F_1F_2\cdots F_k$ of elements of $\mathcal{S}\cup\mathcal{S}^{-1}$ of length $k\le n$. Then
\[\dim V^n\le \overline\gamma_{\mathcal{S}}(n)\delta_{\mathcal{S}}(n).\]

In particular, if $\Gr$ is expansive, then the growth of the algebra $\Bbbk\Gr$ is bounded from above by $\gamma_{\mathcal{S}}(Cn)\delta_{\mathcal{S}}(Cn)$ for any expansive generating set $\mathcal{S}$ of $\Gr$ and some $C>1$ (depending on the generating set $\mathcal{S}$ and the generating set of the algebra).
\end{theorem}

Here $\overline\gamma_{\mathcal{S}}(n)=\overline\gamma_S(n)$, where $S$ is the union of the elements of $\mathcal{S}$.

\begin{example}
If, in the setting of Example~\ref{ex:Itinerary}, the generating set $\mathcal{F}$ is expansive, then the growth of the algebra $\Bbbk\Fr$ is, in fact, equivalent to $n\delta(n)$, where $\delta(n)$ is the factor complexity of the corresponding subshift, see~\cite{nek:gk}.
\end{example}

Theorem~\ref{th:growthofalgebras} together with Theorem~\ref{th:simplealgebra} is a source of many examples of simple finitely generated algebras over arbitrary fields with a prescribed growth function, see~\cite{nek:gk} and references therein.

\subsection{$C^*$-algebras}

Let $\Gr$ be an ample groupoid. A \emph{$*$-representation} of the inverse semigroup $\mathcal{B}(\Gr)$ on a Hilbert space $H$ is a map $\pi\colon \mathcal{B}(\Gr)\arr B(H)$, where $B(H)$ is the Banach algebra of bounded operators on $H$ such that $\pi(F)$ is a partial isometry for every $F\in\mathcal{B}(\Gr)$, and we have
\[\pi(F_1F_2)=\pi(F_1)\pi(F_2),\qquad \pi(F^{-1})=\pi(F)^*,\]
for all $F_1, F_2, F\in\mathcal{B}(\Gr)$.

Every $*$-representation $\pi$ of $\mathcal{B}(\Gr)$ on a Hilbert space $H$ extends by linearity to a $*$-representation of the algebra $\C\Gr$. 
A linear representation $\pi$ of a $*$-algebra $\mathcal{A}$ on a Hilbert space $H$ is called a \emph{$*$-representation} if $\pi(a^*)=\pi(a)^*$ for all $a\in\mathcal{A}$
Conversely, every $*$-representation of $\C\Gr$ is a linear extension of a $*$-representation of $\mathcal{B}(\Gr)$.

\begin{definition}
The \emph{full $C^*$-algebra} of $\Gr$ is the completion of $\C\Gr$ by the norm $\|a\|$ equal to the supremum of the operator norm $\|\pi(a)\|$ over all $*$-representations of the $*$-algebra $\C\Gr$.
\end{definition}

Note that, for any $*$-representation of $\C\Gr$, and any element $a=\sum_{k=1}^ma_k1_{F_k}$ we have $\|\pi(a)\|\le\sum_{k=1}^m|a_k|$.

In the general \'etale case, instead of the Steinberg algebra $\C\Gr$, one has to consider the convolution algebra generated by functions $f\colon \Gr\arr\C$ such that $f$ is equal to $0$ outside of an open bisection $F\subset\Gr$, is continuous on $F$ and is supported on a compact subset of $F$, see~\cite{connes:noncomg}. For a general theory of $C^*$-algebras associated with topological (not necessarlity \'etale) groupoids, see~\cite{renault:groupoids,williams:groupoidalgebras}.

Similarly to the case of groups, it is often easier to understand the \emph{regular representation}, and the associated \emph{reduces $C^*$-algebra}.

Choose $x\in\Gr^{(0)}$, and consider the Hilbert space $\ell^2(\Gr_x)$ of $\ell^2$-summable maps from $\Gr_x=\be^{-1}(x)$ to $\C$. For a bisection $F\subset\Gr$ and an element $\phi\in\ell^2(\Gr_x)$, we define $\lambda_x(F)(\phi)\in\ell^2(\Gr_x)$ by the condition
\[\lambda_x(1_F)(\phi)(g)=\phi(F^{-1}g).\]
Equivalently, $\lambda_x(a)$ for $a\in\C\Gr$ is defined by the condition that $\lambda_x(a)(f)=a*f$ for every $f\in\ell^2(\Gr_x)$ seen as a function $f\colon \Gr\arr\C$ supported on $\Gr_x$. (See the definition of convolution of an arbitrary function with a compactly supported one in~\ref{ss:Steingergalgebras}.)

\begin{definition}
The \emph{reduced $C^*$-algebra} of $\Gr$ is the completion of $\C\Gr$ with respect to the norm $\|a\|=\sup_{x\in\Gr^{(0)}}\|\lambda_x(a)\|$.
\end{definition}

The following analog of Theorem~\ref{th:idealbijection} and Theorem~\ref{th:simplealgebra} is proved in~\cite[Proposition~4.6]{renault:groupoids}.

\begin{theorem}
    If $\Gr$ is an essentially principal Hausdorff \'etale grouopid, then $U\mapsto C_r(\Gr|_U)$ is a bijection between open $\Gr$-invariant subsets of $\Gr^{(0)}$ and the ideals of $C_r(\Gr)$. In particular, $C_r(\Gr)$ is simple if and only if $\Gr$ is minimal.
\end{theorem}

For an important class of \emph{amenable groupoids} the algebras $C_r(\Gr)$ and $C^*(\Gr)$ coincide, i.e., the identity map on $\C\Gr$ induces isomorphism $C^*(\Gr)\arr C^*_r(\Gr)$.

Let us give one of many equivalent definitions of amenability. See~\cite{delaroche_renault} for the general theory of amenability for topological groupoids. The definition (\emph{Borel amenability}) is equivalent to a more classical \emph{topological amenability} (requiring the functions $g_n$ in the definition to be continuous) by the result of~\cite{renault:nonH}.

\begin{definition} 
\label{def:amenable}
Let $\Gr$ be an \'etale $\sigma$-compact groupoid. We say that it is \emph{amenable} if there exists a sequence $(f_n)_{n\ge 1}$ of non-negative Borel functions on $\Gr$ satisfying the following conditions.
\begin{enumerate}
\item $\sum_{g\in\Gr, \be(g)=x}f_n(g)\le 1$ for all $x\in\Gr^{(0)}$ and all $x\in\Gr^{(0)}$;
\item $\lim_{n\to\infty}\sum_{g\in\Gr, \be(g)=x}f_n(g)=1$ for every $x\in\Gr^{(0)}$;
\item $\lim_{n\to\infty}\sum_{g_1\in\Gr, \be(g_1)=\en(g)}|f_n(g_1g)-f_n(g_1)|=0$ for all $g\in\Gr$.
\end{enumerate}
\end{definition}

As an example, let us show that groupoids of sub-exponential growth are  amenable.

\begin{proposition}
\label{pr:subexponentialamenable}
Let $\Gr$ be an ample groupoid generated by a compact set $S$. Suppose that the growth of the Cayley graphs of $\Gr$ is sub-exponential, i.e., that $\lim_{n\to\infty}\gamma_{x, S}(n)^{1/n}=1$ for every $x\in\Gr^{(0)}$. Then $\Gr$ is amenable.
\end{proposition}

\begin{proof}
We are following the arguments of~\cite[Proposition~1]{kaim:leaves}.

We assume that $S$ is symmetric and contains the unit space of $\Gr$. Let  $\gamma_{x, S}(n)$ or just $\gamma_x(n)$ be the number of vertices in the ball of radius $r$ with center in $x$ in the Cayley graph $\G_x(\Gr, S)$.
Define $b_n\colon \Gr\arr\R$, for $n\ge 1$, by 
\[b_n(g)=\left\{\begin{array}{rl}
\gamma_x(n)^{-1} & \text{if $g\in S^n$,}\\
0 & \text{otherwise.}
\end{array}\right.\]
for $n\ge 1$. Note that the set of points $x\in\Gr^{(0)}$ with a given isomorphism class of a ball in $\G_x(\Gr, S)$ of radius $r$ with center in $x$ is Borel. Consequently, the functions $b_n$ are Borel. (They are not necessarily continuous, if $\Gr$ is not Hausdorff.)

Let $f_n(g)=\frac 1n\sum_{k=1}^nb_k(g)$. Let us show that $f_n$ satisfy the conditions of Definition~\ref{def:amenable}.

We have $\sum_{\be(g)=x} b_n(g)=1$, hence $\sum_{\be(g)=x} f_n(g)=1$, so the first two conditions of the definition are satisfied.

Let $g\in\Gr$ and let $l$ be the smallest natural number such that $g\in S^l$ (i.e., the length of $g$ with respect to the generating set $S$).
If $g_1\in S^{n-l}$, then $g_1g\in S^n$. Similarly, if $g_1g\in S^{n-l}$, then $g_1\in S^n$.  The ball of the Cayley graph of $\Gr$ of radius $n-l$ with center in $g$ is contained in the ball of radius $n$ with center in $\be(g)$, and is isomorphic to the ball of radius $n-l$ with center in $\en(g)$. Consequently, $\gamma_{\be(g)}(n)\ge\gamma_{\en(g)}(n-l)$.

Let us estimate $\sum_{g_1\in\Gr, \be(g_1)=\en(g)}|b_n(g_1g)-b_n(g_1)|$. 
We have
\[|b_n(g_1g)-b_n(g_1)|=\left\{\begin{array}{rl} 
|\gamma_{\be(g)}(n)^{-1}-\gamma_{\en(g)}(n)^{-1}| & \text{if $g_1g, g_1\in S^n$,}\\
\gamma_{\be(g)}(n)^{-1} & \text{if $g_1g\in S^n, g_1\notin S^n$,}\\
\gamma_{\en(g)}(n)^{-1} & \text{if $g_1g\notin S^n, g_1\in S^n$,}\\
0 & \text{otherwise.}
\end{array}\right.\]

We have $\sum_{g_1\in\Gr, \be(g_1)=\en(g)}(b_n(g_1g)-b_n(g_1))=0$. Consequently, if $\gamma_{\be(g)}(n)\le \gamma_{\en(g)}(n)$, then 
\begin{multline*}
\sum_{g_1\in\Gr, \be(g_1)=\en(g)}|b_n(g_1g)-b_n(g_1)|=\frac{2|\{g_1\in B_{\en(g)}(n) : g_1g\notin S^n\}|}{\gamma_{\en(g)}(n)}\le\\
2\frac{\gamma_{\en(g)}(n)-\gamma_{\en(g)}(n-l)}{\gamma_{\en(g)}(n)}.
\end{multline*}
Similarly, if $\gamma_{\be(g)}(n)\ge\gamma_{\en(g)}(n)$, then
\begin{multline*}
\sum_{g_1\in\Gr, \be(g_1)=\en(g)}|b_n(g_1g)-b_n(g_1)|=\frac{2|\{g_1\in B_{\be(g)}(n) : g_1g^{-1}\notin S^n\}|}{\gamma_{\be(g)}(n)}\le\\
2\frac{\gamma_{\be(g)}(n)-\gamma_{\be(g)}(n-l)}{\gamma_{\be(g)}(n)}.
\end{multline*}

Consequently,
\[
\sum_{g_1\in\Gr, \be(g_1)=\en(g)}|b_n(g_1g)-b_n(g_1)|\le 2\left(1-\frac{\gamma_{\be(g)}(n-l)}{\gamma_{\be(g)}(n)}\right)+2\left(1-\frac{\gamma_{\en(g)}(n-l)}{\gamma_{\en(g)}}\right).
\]

We have, as $n\to\infty$,  for every $x\in\Gr^{(0)}$:
\begin{multline*}
\frac{1}{n-l}\sum_{k=l+1}^n\left(1-\frac{\gamma_x(n-l)}{\gamma_x(n)}\right)=
1-\frac{1}{n-l}\sum_{k=l+1}^n\frac{\gamma_x(n-l)}{\gamma_x(n)}\le\\ 1-\left(\prod_{k=l+1}^n\frac{\gamma_x(n-l)}{\gamma_x(n)}\right)^{1/(n-l)}=1-\left(\frac{\prod_{k=1}^l\gamma_x(k)}{\prod_{k=n-l+1}^n\gamma_x(k)}\right)^{1/(n-l)}\le\\
1-C\left(\gamma_x(n)^{1/n}\right)^{\frac{nl}{l-n}}\to 0
\end{multline*}
for $C=\left(\prod_{k=1}^l\gamma_x(k)\right)^{\frac{1}{n-l}}$.

Consequently, $\lim_{n\to\infty}\sum_{g_1\in\Gr, \be(g_1)=\en(g)}|f_n(g_1g)-f_n(g_1)|=0$.
\end{proof}

\section{Full groups}

\subsection{Main definitions}

\begin{definition}
Let $\Gr$ be an ample groupoid with compact unit space. Its \emph{full group} $\mathsf{F}(\Gr)$ is the set of open bisections $F\subset\Gr$ such that $\be(F)=\en(F)=\Gr^{(0)}$. In other words, it is the group of units of the inverse semigroup $\mathcal{B}(\Gr)$.

If $\Gr^{(0)}$ is not compact, then we define $\mathsf{F}(\Gr)$ as the inductive limit of the groups $\mathsf{F}(\Gr|_U)$ over all compact open subsets $U\subset\Gr^{(0)}$. In other words, it is the group of all bisections of the form $F\cup(\Gr^{(0)}\setminus\be(F)$, where $F\in\mathcal{B}(\Gr)$ is such that $\be(F)=\en(F)$.
\end{definition}

If $\Gr$ has compact unit space, then every element $F\in\mathsf{F}(\Gr)$ can be seen as an element of $C_c^{(1)}(\Gr;\Z)$. Since $\delta_1(F)=\be(F)-\en(F)=0$, it is a cycle, so it defines an element $[F]$ of $H_1(\Gr)$. Since the element $F_2-F_1F_2+F_1$ is the boundary of $(F_1, F_2)\in C_c^{(2)}(\Gr;\Z)$, the map $I\colon F\mapsto [F]$ is a homomorphism from $\mathsf{F}(\Gr)$ to $H_1(\Gr)$, which we call the \emph{index map}.

If the unit space is not compact, then we also can define the index map,  defining $I(F)$ for $F\in\mathsf{F}(\Gr)$ as the class of $[F']$, where $F'$ is any bisection such that $F\setminus F'\subset\Gr^{(0)}$. Since the class $[U]$ for any compact open $U\subset\Gr^{(0)}$ is $0$ in $H_1(\Gr)$, the index map $I$ is well defined.

Suppose that a bisection $F\in\mathcal{B}(\Gr)$ is such that $\be(F)\cap\en(F)=\emptyset$. Then we can construct the element $\tau_F=F\cup F^{-1}\cup(\Gr^{(0)}\setminus(\be(F)\cup\en(F))$ of $\mathsf{F}(\Gr)$. In particular, this shows that the orbits of the full group coincide with the orbits of the groupoid.

The subgroup generated by such elements $\tau_F$ is called the \emph{symmetric full group}, denoted $\mathsf{S}(\Gr)$. It is a normal subgroup of $\mathsf{F}(\Gr)$.

The generators $\tau_F$ of the symmetric full group $\mathsf{S}(\Gr)$ belong to the kernel of the index map, hence $\mathsf{S}(\Gr)$ is contained in it.
In all known examples, the symmetric full group is equal to the kernel of the index map.

More generally, suppose that $(F_1, F_2, \ldots, F_n)$ is an ordered multisection (see~\ref{ss:homology}). Denote $F_{i, j}=F_{i+1}F_{i+2}\cdots F_j$ and $F_{j, i}=F_{i, j}^{-1}$ for $0\le i<j\le n$, $F_{i, i}=\be(F_i)$ for $i=1, 2, \ldots, n$, and $F_{0, 0}=\en(F_1)$. Then we have $F_{i, j}F_{j, k}=F_{i, k}$ for all $0\le i, j, k\le n$.

Suppose that $F_{i, i}$, for $0\le i\le n$, are pairwise disjoint, and let $U$ be the complement of their union in $\Gr^{(0)}$. Then for any permutation $\alpha$ of the set $\{0, 1, \ldots, n\}$, the set
\[\tau_\alpha=U\cup\bigcup_{i=0}^n F_{\alpha(i), i}\]
is an element of $\mathsf{F}(\Gr)$. Moreover, $\alpha\mapsto\tau_\alpha$ is a faithful representation of the symmetric group $\mathsf{S}_{n+1}$ in $\mathsf{F}(\Gr)$.
The subgroup generated by the union of the images of such representation is equal to the symmetric full group $\mathsf{S}(\Gr)$, defined above.

The \emph{alternating full group} $\mathsf{A}(\Gr)$ is the subgroup of $\mathsf{F}(\Gr)$ generated by the union of the images of the alternating groups under the described above representations.

The full groups are complete invariants of ample groupoids. Analogs of the following theorem (in various degrees of generality, but the proofs are basically applicable to the general setting) were proved in~\cite{rubin:reconstr,gior:full,matui:fullonesided,%
bezuglyiMedynets:fullgroup,medynets:aperiodic}. See also~\cite[Theorem~5.1.20]{nek:dyngroups}.

\begin{theorem}
\label{th:reconstruction}
Let $\Gr_1, \Gr_2$ be minimal effective ample groupoids. Then the following conditions are equivalent.
\begin{enumerate}
\item The groups $\mathsf{F}(\Gr_1)$ and $\mathsf{F}(\Gr_2)$ are isomorphic.
\item The groups $\mathsf{S}(\Gr_1)$ and $\mathsf{S}(\Gr_2)$ are isomorphic.
\item The groups $\mathsf{A}(\Gr_1)$ and $\mathsf{A}(\Gr_2)$ are isomorphic.
\item The groupoids $\Gr_1$ and $\Gr_2$ are isomorphic.
\end{enumerate}
\end{theorem}

\subsection{Simplicity and finite generation}

Many properties of the alternating full group are analogous to the properties of the Steinberg algebras. Let us describe, for example, the analogs of the results on ideals of the Steinberg algebras.

We have the following analog of Theorem~\ref{th:simplealgebra}. Note that we do not require that the groupoid $\Gr$ is Hausdorff.

\begin{theorem}
\label{th:simplicity}
Let $\Gr$ be a minimal effective ample groupoid. Then the group $\mathsf{A}(\Gr)$ is simple.
\end{theorem}

The converse of the theorem is also essentially true. 
Passing to the restriction of $\Gr$ to the union of $\Gr$-orbits of size $\ge 3$ does not change $\mathsf{A}(\Gr)$. Therefore, we may assume, without loss of any examples of full alternating groups, that all $\Gr$-orbits have size at least 3.

\begin{proposition}
Let $\Gr$ be an \'etale groupoid such that all of its orbits have size at least 3. Then $\mathsf{A}(\Gr)$ is simple if and only if $\Gr$ is minimal and effective.
\end{proposition}

\begin{proof}
Suppose that $\Gr$ is not minimal. Then there exists an open $\Gr$-invariant subset $U\subset\Gr^{(0)}$ such that $\emptyset\ne U\ne \Gr^{(0)}$. If $\mathsf{A}(\Gr)$ is simple, then the subgroup $\mathsf{A}(\Gr|_U)$ is either trivial or equal to the whole group $\mathsf{A}(\Gr)$. If $\mathsf{A}(\Gr|_U)$ is trivial, then all $\Gr$-orbits in $U$ have cardinality $\le 2$. Similarly, if $\mathsf{A}(\Gr|_U)$ is equal to $\mathsf{A}(\Gr)$, then every element of $\mathsf{A}(\Gr)$ acts identically on the complement of $U$, so all $\Gr$-orbits of the complement are of cardinality $\le 2$.

Suppose that $\Gr$ is not effective, and let $\Gr'$ be its effective quotient (i.e., the groupoid of germs of the inverse semigroup of local homeomorphisms of $\Gr^{(0)}$ defined by the $\Gr$-bisections). We have then a natural epimorphism $\mathsf{A}(\Gr)\arr\mathsf{A}(\Gr')$. There exists a compact open bisection $U\subset\Gr$ such that it is not contained in $\Gr^{(0)}$ but all its elements are isotropic (see Definition~\ref{def:effective}). 

Since we assume that every $\Gr$-orbit is of size at least 3, we can find two ordered multisections $(F_1, F_2, F_3)$ and $(F_1', F_2', F_3')$ such that $\be(F_i')=\be(F_i)$, $\en(F_i')=\en(F_i)$, and $F_1^{-1}F_1'$ is isotropic by not trivial. Then the quotient of the three-cycles defined by these multisections will be a non-trivial element of the kernel of the homomorphism $\mathsf{A}(\Gr)\arr\mathsf{A}(\Gr')$.
\end{proof}

We also have analogues of Theorem~\ref{th:fingenalgebra}.
The following theorem is proved in~\cite{nek:fullgr}, see also~\cite[Theorem~5.1.11]{nek:dyngroups}. It is a generalization of~\cite[Theorm~5.4]{matui:fullI}.

\begin{theorem}
Let $\Gr$ be an expansive ample groupoid such that all its orbits have at least 5 points. Then $\mathsf{A}(\Gr)$ is finitely generated.
\end{theorem}

Let us show that expansivity is a necessary condition. 

\begin{proposition}
If $\mathsf{A}(\Gr)$ is finitely generated, then $\Gr$ is expansive.
\end{proposition}

\begin{proof}
As it was noted above, removing all $\Gr$-orbits of size $\le 2$ from $\Gr^{(0)}$ does not change $\mathsf{A}(\Gr)$, so we may assume that all $\Gr$-orbits have at least $3$ points. 

Let us show that the action of $\mathsf{A}(\Gr)$ on $\Gr^{(0)}$ is expansive, if $\mathsf{A}(\Gr)$ is finitely generated. Suppose that $\mathsf{A}(\Gr)$ is finitely generated. Since every generator in the definition of $\mathsf{A}(\Gr)$ point-wise fixes complement of a compact subset of $\Gr^{(0)}$, the unit space $\Gr^{(0)}$ is compact. 
Let us introduce an arbitrary metric $d$ on $\Gr^{(0)}$.

Since each $\Gr$-orbit is of size $\ge 3$, for every $x\in\Gr^{(0)}$ there exists a compact open bisection $F\in\mathcal{B}(\Gr)$ such that $x\in\be(F)$ and $\be(F)\cap\en(F)=\emptyset$. Consequently, there exists a finite set $\mathcal{F}$ of compact open bisections such that the sets $\be(F)$, for $F\in\mathcal{F}$, cover $\Gr^{(0)}$, and $\be(F)\cap\en(F)=\emptyset$ for all $F\in\mathcal{F}$. There exists $\delta>0$ such that the distance between $\be(F)$ and $\en(F)$ is greater than $\delta$ for all $F\in\mathcal{F}$. Since all $\Gr$-orbits are assumed to be of size at least 3, this implies that for every $x\in\Gr^{(0)}$ there exists $g\in\mathsf{A}(\Gr)$ such that $d(x, g(x))>\delta$.

Let $x, y\in\Gr^{(0)}$ be such that $0<d(x, y)<\delta/2$. There exists $g\in\mathsf{A}(\Gr)$ such that $d(g(x), x)>\delta$. We have then $d(g(x), y)>\delta/2$.  If the orbit of $x$ has more than 3 points or if $y$ is not in the orbit of $x$, then there exists $h\in\mathsf{A}(\Gr)$ such that $h(x)=g(x)$ and $h(y)=y$. Then $d(h(x), h(y))>\delta/2$. If the orbit of $x$ is equal to $\{x, g(x), y\}$, then there exists $g\in\mathsf{A}(\Gr)$ such that $h(y)=x$ and $h(x)=g(x)$. Then $d(h(x), h(y))=d(g(x), x)>\delta$.

We see that for any $x, y\in\Gr^{(0)}$ there exists $g\in\mathsf{A}(\Gr)$ such that $d(g(x), g(y))>\delta/2$. Consequently, the action of $\mathsf{A}(\Gr)$ on $\Gr^{(0)}$ is expansive. Let $S$ be a finite generating set of $\mathsf{A}(\Gr)$. Consider a partition of $\Gr^{(0)}$ into compact open sets of diameters $<\delta/2$, and let $\mathcal{S}$ be the set of $\Gr$-bisections equal to restrictions of the elements of $S$ to the elements of the partition. Then the same arguments as in the proof of Theorem~\ref{pr:expansivegroupaction} show that the set $\mathcal{S}$ is an expansive generating set of $\Gr$.
\end{proof}

Combining the two theorems above, we get the following source of infinite finitely generated simple groups. For applications, see~\cite{kerrdrob:gamma,nek:burnside,nek:simplegrowth}.

\begin{theorem}
Let $\Gr$ be an expansive minimal effective ample groupoid. Then the alternating full group $\mathsf{A}(\Gr)$ is finitely generated and simple.
\end{theorem}

\subsection{Comparison property}

As we have seen above, properties of the alternating full group $\mathsf{A}(\Gr)$ are intimately related to dynamical properties of the groupoid $\Gr$. One the other hand, the definition of the full group $\mathsf{F}(\Gr)$ seems to be more natural, so it would be good to understand the differences between $\mathsf{A}(\Gr)$, $\mathsf{S}(\Gr)$, and $\mathsf{F}(\Gr)$.

In all known examples, the alternating full group $\mathsf{A}(\Gr)$ coincides with the commutator subgroup of $\mathsf{F}(\Gr)$, while the symmetric group $\mathsf{S}(\Gr)$ is equal to the kernel of the index map. 

It is not hard to see that the alternating full group $\mathsf{A}(\Gr)$ is equal to the commutator subgroup of the symmetric full group $\mathsf{S}(\Gr)$.

A partial description of the abelianization $\mathsf{S}(\Gr)/\mathsf{A}(\Gr)$ of the symmetric full group is given by the following theorem (due to H.~Matui for full groups of shifts~\cite[Section~4]{matui:fullI}, see the general statement in~\cite{nek:fullgr}). A more detailed result is contained in Theorem~\ref{th:Li} below).

For $F\in\mathcal{B}(\Gr)$ such that $\be(F)\cap\en(F)=\emptyset$, let $\tau_F$ be the transposition $F\cup F^{-1}\cup (\Gr^{(0)}\setminus(\be(F)\cup\en(F)))$. 

\begin{theorem}
\label{th:signaturemap}
The correspondence $[\be(F)]\mapsto [\tau_F]$, where $[\be(F)]\in H_0(\Gr;\Z/2\Z)$ and $[\tau_F]$ is the image of $\tau_f$ in $\mathsf{S}(\Gr)/\mathsf{A}(\Gr)$, extends to an epimorphism $H_0(\Gr;\Z/2\Z)\arr\mathsf{S}(\Gr)/\mathsf{A}(\Gr)$.
\end{theorem}

We will formulate below two (somewhat opposite) conditions due to H.~Matui on $\Gr$ ensuring that the alternating full group $\mathsf{A}(\Gr)$ coincides with the commutator subgroup of the full group $\mathsf{F}(\Gr)$. Both of them will imply the following general condition.

\begin{definition}
Let $\Gr$ be an ample groupoid, and denote by $M(\Gr)$ the set of all non-zero $\Gr$-invariant Radon measures on $\Gr^{(0)}$. We say that $\Gr$ has \emph{comparison property} if for any two non-empty compact open sets $U, V\subset\Gr^{(0)}$ such that $\mu(U)<\mu(V)$ for all $\mu\in M(\Gr)$ there exists a bisection $F\subset\Gr$ such that $\be(F)=U$ and $\en(F)\subset V$.
\end{definition}

We say that an ample groupoid $\mathfrak{K}$ is \emph{elementary} if it is compact, proper, and principal. Every elementary groupoid is the groupoid of germs of a finite inverse semigroup of homeomorphisms between compact open sets.

\begin{definition}
\label{def:almostfinite}
An ample groupoid $\Gr$ is \emph{almost finite} if for every compact set $C\subset\Gr$ and $\epsilon>0$ there exists an elementary sub-groupoid $\mathfrak{K}\subset\Gr$, with $\mathfrak{K}^{(0)}=\Gr^{(0)}$, such that
\[\frac{|C\mathfrak{K}x\setminus\mathfrak{K}x|}{|\mathfrak{K}x|}<\epsilon\]
for every $x\in\Gr^{(0)}$.
\end{definition}

It was shown in~\cite[Lemma~6.3]{matui:etale} that groupoids of free actions of $\Z^n$ on a Cantor set are almost finite.

For a proof of the following theorem, see~\cite{matui:etale} and~\cite[Theorem~5.5.4]{nek:dyngroups}.

\begin{theorem}
Suppose that $\Gr$ is almost finite. Then $M(\Gr)$ is non-empty and $\Gr$ has comparison property.
\end{theorem}

The following theorem is proved in~\cite[Theorem~4.7]{matui:fullonesided}, see also a proof in~\cite[Theorem~5.5.5]{nek:dyngroups}.

\begin{theorem}
\label{th:almostfinitealternating}
Let $\Gr$ be a minimal almost finite effective ample groupoid. Then the derived subgroup of the full group $\mathsf{F}(\Gr)$ coincides with the alternating full group $\mathsf{A}(\Gr)$. In particular, the derived subgroup is simple.
\end{theorem}

Another class of groupoids satisfying the comparison property and for which the derived subgroup of $\mathsf{F}(\Gr)$ coincides with $\mathsf{A}(\Gr)$ are \emph{purely infinite groupoids}.

\begin{definition}
An ample groupoid $\Gr$ is \emph{purely infinite} if for every compact open subset $A\subset\Gr^{(0)}$ there exist bisections $F_1, F_2\in\mathcal{B}(\Gr)$ such that $\be(F_1)=\be(F_2)=A$, $\en(F_1)\cap\en(F_2)=\emptyset$, and $\en(F_1)\cup\en(F_2)\subset A$.
\end{definition}

If $\Gr$ is purely infinite, then the space $M(\Gr)$ of invariant Borel measures is obviously empty. Consequently, $\Gr$ satisfies the comparison property.

We have the following result of~\cite[Theorem~4.16]{matui:fullonesided}, see also a proof in~\cite[Theorem~5.6.3]{nek:dyngroups}.

\begin{theorem}
\label{th:purelyinfinitealternating}
Let $\Gr$ be a minimal effective ample groupoid. Then the derived subgroup of $\mathsf{F}(\Gr)$ coincides with $\mathsf{A}(\Gr)$.
\end{theorem}

The following theorem is proved in~\cite[Corollary~6.14]{Li:fullgroups}.

\begin{theorem}
\label{th:Li}
Let $\Gr$ be a minimal ample groupoid such that $\Gr^{(0)}$ is not discrete. Assume that $\Gr$ has comparison property. Then there exists an exact sequence
\[H_2(\mathsf{A}(\Gr))\arr H_2(\Gr)\arr H_0(\Gr; \Z/2\Z)\arr H_1(\mathsf{F}(\Gr))\arr H_1(\Gr)\arr 0.\]
Moreover, the map $H_0(\Gr; \Z/2\Z)\arr H_1(\mathsf{F}(\Gr))=\mathsf{F}(\Gr)/[\mathsf{F}(\Gr), \mathsf{F}(\Gr)]$ is the composition of the map $H_0(\Gr; \Z/2\Z)\arr\mathsf{S}(\Gr)/\mathsf{A}(\Gr)$ from Theorem~\ref{th:signaturemap} with the homomorphism $\mathsf{S}(\Gr)/\mathsf{A}\arr\mathsf{F}(\Gr)/[\mathsf{F}(\Gr), \mathsf{F}(\Gr)]$ induced by the inclusions $\mathsf{S}(\Gr)\le\mathsf{F}(\Gr)$ and $\mathsf{A}(\Gr)\le[\mathsf{F}(\Gr), \mathsf{F}(\Gr)]$. The map $H_1(\mathsf{F}(\Gr))\arr H_1(\Gr)$ is induced by the index map $\mathsf{F}(\Gr)\arr H_1(\Gr)$.
\end{theorem}

Note that $H_1(\mathsf{F}(\Gr))$ is the abelianization of the full group. Consequently, in the conditions of Theorem~\ref{th:Li}, the abelianization $\mathsf{F}(\Gr)/\mathsf{A}(\Gr)$ of the full group is an extension of $H_1(\Gr)$ by a homomorphic image of $H_0(\Gr;\Z/2\Z)$.

\begin{corollary}
\label{th:finitegenerationoffull}
Suppose that an ample minimal expansive and effective groupoid $\Gr$
is almost finite or purely infinite. If the groups $H_1(\Gr)$ and $H_0(\Gr;\Z/2\Z)$ are finitely generated, then the full group $\mathsf{F}(\Gr)$ is also finitely generated.
\end{corollary}

\section{Self-similar groups}

\subsection{One-sided shift}
\label{sss:basicdefinitions}

Let $\alb$ be a finite alphabet, and let $\xo$ be the space of all sequence $x_1x_2\ldots$ of its letters with the direct product topology. We denote by $\xs$ the free monoid generated by $\alb$, i.e., the set of finite words over alphabet $\alb$. The \emph{(one-sided) shift} is the map $\si\colon\xo\arr\xo$ defined as
\[\si(x_1x_2\ldots)=x_2x_3\ldots.\]

It is a covering map, and its local inverses are the maps
\[S_x\colon \xo\arr\xo\colon w\mapsto xw\]
appending a letter $x\in\alb$ to the beginning of a sequence. It is a homeomorphism from $\xo$ to the clopen subset $x\xo$ of sequences starting with $x$. 

Consider the inverse semigroup $\mathcal{O}_\alb$ of local homeomorphisms of $\xo$ generated by the maps $S_x$ for all $x\in\alb$. We will write $S_{x_1}S_{x_2}\cdots S_{x_n}=S_{x_1x_2\ldots x_n}$. Then $S_v$ for $v\in\xs$ is the transformation $w\mapsto vw$.

The elements of $\mathcal{O}_\alb$ are \emph{prefix exchange transformation}
\[S_uS_v^{-1}\colon vw\mapsto uw,\]
replacing a prefix $v$ by a prefix $u$ for arbitrary $v, u\in\xs$.

The set of idempotents of $\mathcal{O}_\alb$ is in a natural bijection with the set of finite words $\xs$, where a word $v$ corresponds to the idempotent $S_vS_v^{-1}$. The order on the idempotents corresponds to the left divisibility $v\ge vu$ relation on the free monoid $\xs$. The space of ultrafilters is, therefore, naturally identified with the space $\xo$.

We denote by $\mathfrak{O}_\alb$ the groupoid of germs of the semigroup $\mathcal{O}_\alb$. It is equal to the groupoid generated by the germs of the one-sided shift $\si\colon\xo\arr\xo$.

\begin{definition}
Let $G$ be a group acting faithfully on $\xo$ by homeomorphisms. We say that it is \emph{self-similar} if for every $g\in G$ and $x\in\alb$ there exist $h\in G$ and $y\in\alb$ such that
\[gS_x=S_yh.\]
\end{definition}

If the action of $G$ is self-similar, then for every $v\in\xs$ and $g\in G$ there exist $u\in\xs$ and $h\in G$ such that $|u|=|v|$ and
\[gS_v=S_uh.\]
Then $v\mapsto u$ is an action of $G$ on $\xs$, and the element $h$ is uniquely determined by the condition $h=S_u^{-1}gS_v$. We will denote $u=g(v)$ and $h=g|_v$. 

Denote by $\mathcal{O}_{G, \alb}$ the inverse semigroup generated by the transformations $G$ and $S_x$ of $\xo$. Then every element of $\mathcal{O}_{G, \alb}$ can be uniquely written as 
\[S_ugS_v^{-1}\]
for $u, v\in\xs$ and $g\in G$. The set of idempotents of $\mathcal{O}_{G, \alb}$ coincides with the set of idempotents of its subsemigroup $\mathcal{O}_{\alb}$, hence is identified with $\xs$. The action of $G$ on $\xs$, described above, coincides with the action of $G\subset\mathcal{O}_{G, \alb}$ on the set of idempotents of $\mathcal{O}_{G, \alb}$ by conjugation.

The set of elements of $\mathcal{O}_{G, \alb}$ of the form $S_xg$ for $x\in\alb$ and $g\in G$ is invariant under both left and right multiplication by elements of $G$. Namely,
\[h_1\cdot\left(S_xg\right)\cdot h_2=S_{h_1(x)}h_1|_xgh_2.\]

The action on the right is free, and $S_{x_1}g_1, S_{x_2}g_2$ belong to the same orbit of the right action if and only if $x_1=x_2$. We generalize this in the following way.

\begin{definition}
Let $G$ be a group. A \emph{covering $G$-biset} is a set $\bim$ together with commuting left and right actions of $G$ on it such that the right action is free and has finitely many orbits.
\end{definition}

Suppose that $\bim_1$ and $\bim_2$ are covering $G$-bisets. Then $\bim_1\otimes\bim_2$ is defined as the set of orbits of the action $(x_1, x_2)\mapsto (x_1\cdot g, g^{-1}\cdot x_2)$ of $G$ on $\bim_1\times\bim_2$. The orbit of $(x_1, x_2)$ is denoted by $x_1\otimes x_2$. The set $\bim_1\otimes\bim_2$ is naturally a $G$-biset for the actions $g_1\cdot(x_1\otimes x_2)\cdot g_2=(g_1\cdot x_1)\otimes (x_2\cdot g_2)$. One can easily check that it is a covering $G$-biset and that we have a natural isomorphism between $(\bim_1\otimes\bim_2)\otimes\bim_3$ and $\bim_1\otimes(\bim_2\otimes\bim_3)$. In particular, bisets $\bim^{\otimes n}$ are naturally defined. We define $\bim^{\otimes 0}$ as  the group $G$ with the usual left and right actions on itself by multiplication.

Let $\bim^*$ be the disjoint union of the bisets $\bim^{\otimes n}$ for all $n\ge 0$. Then $\bim^*$ is a monoid with respect to the operation $xy=x\otimes y$. 

We can embedd the monoid $\bim^*$ into an inverse semigroup $\mathcal{O}_{\bim}$ by defining $x^{-1}y$, for $x, y\in\bim$, to be equal to the unique element $g\in G$ such that $y=x\cdot g$, if it exists (i.e., if $x$ and $y$ belong to the same right $G$-orbit). If such an element does not exist, then $x^{-1}y=0$.

For $x, y\in\bim^{\otimes n}$, the idempotents $xx^{-1}$ and $yy^{-1}$ coincide if and only if $x$ and $y$ belong to the same right $G$-orbit.

\begin{definition}
Let $\bim$ be a covering $G$-biset. A \emph{basis} of $\bim$ is a set $\alb\subset\bim$ intersecting every orbit of the right action exactly once.
\end{definition}

Let $\alb$ be a basis of $\bim$. Then every element of $\bim$ can be uniquely written as $x\cdot g$ for $x\in\alb$ and $g\in G$. 
Every element of the semigroup $\bim^*$ can be written uniquely as $v\cdot g$ for $v\in\xs$ and $g\in G$. Here $\xs$ is the sub-monoid generated by $\alb$ in $\bim$. One can show that it is free. 

Every element of $\mathcal{O}_{\bim}$ is written uniquely in the form $vgu^{-1}$ for $v, u\in\xs$ and $g\in G$.  In particular, every idempotent of $\mathcal{O}_{\bim}$ can be uniquely represented as $vv^{-1}$ for $v\in\xs$, so the lattice of idempotents in $\mathcal{O}_{\bim}$ coincides with the lattice of idempotents of the inverse subsemigroup  $\mathcal{O}_{\alb}$ generated by $\alb$.

For every $v\in\xs$ and $g\in G$ there exist unique $u\in\xs$ and $h\in G$ such that
\[g\cdot v=u\cdot h\]
in $\bim^*$. We denote $u=g(v)$ and $h=g|_v$, generalizing the notation for faithful self-similar group actions.  This action of $G$ on $\xs$ coincides with the action of $G$ on the lattice of idempotents of $\mathcal{O}_{\bim}$.

We denote by $\mathfrak{O}_{\bim}$ the groupoid of germs of the inverse semigroup $\mathcal{O}_{\bim}$. If the action of $G$ on $\xs$ is faithful, then the groupoid $\mathfrak{O}_{\bim}$ is effective. 

The action of $G$ on $\xs$ is defined by the recurrent rules of the form
\[g(xw)=yh(w)\]
for all $w\in\xs$, whenever $g\cdot x=y\cdot h$ in $\bim$.

We write these recurrent rules in a compact way as  the \emph{wreath recursion} $\phi\colon\mathsf{S}_\alb\ltimes G^\alb$, where $\mathsf{S}_\alb$ is the symmetric group of permutations of $\alb$ acting naturally on $G^\alb$, defined by the rule that $\phi(g)=\pi(g|_x)_{x\in\alb}$, where $\pi$ is the permutation of $\alb\subset\xs$ defined by $g$.

\begin{example}
\label{ex:addingmachine}
Consider the infinite cyclic group $G=\Z$. Let $\bim$ be the set of maps $\phi_m\colon\Z\arr\Z$ of the form $\phi_m(n)=2n+m$ for $m, n\in\Z$. The left and the right actions are the actions by post-compositions and pre-compositions, respectively. If $a$ is the generator of $G$, then we have
\[a\cdot\phi_m=\phi_{m+1},\qquad \phi_m\cdot a=\phi_{m+2}.\]
Both actions are free. The left one is transitive, while the right one has 2 orbits (even and odd $m$). Let us choose the basis $\alb=\{\phi_0, \phi_1\}$. Then the action of $G$ on $\xs$ is defined recursively by the equalities.
\[a(\phi_0w)=\phi_1w,\qquad a(\phi_1w)=\phi_0a(w),\]
since $a\cdot \phi_0=\phi_1$ and $a\cdot\phi_1=\phi_0\cdot a$. The elements $\phi_0$ and $\phi_1$ of $\bim$ are usually denoted by $0$ and $1$, respectively. Then the recurrent definition of the action of $a$ on a sequence $x_0x_1\ldots$ coincides with the rule of adding 1 to a binary integer $x_0+2x_1+2^2x_2+\cdots$.

If we denote the non-trivial element of $\mathsf{S}_\alb$ by $\pi$, then the wreath recursion is given by
\[\phi(a)=\pi(1, a),\]
where $1$ on the right-hand side is the identity element of $G$.

The described self-similar action of $\Z$ is called the \emph{(binary)  adding machine}. 
\end{example}

\subsection{Iterated monodromy}

An important class of self-similar groups comes from topological and conformal dynamics.

Let $\M$ be a path connected and compact topological space, and let $f\colon \M\arr\M$ be a covering map. For $t\in\M$, denote by $T_t$ the \emph{tree of preimages} with the set of vertices $\bigsqcup_{n\ge 0}f^{-n}(t)$, where a vertex $z\in f^{-n}(t)$ is connected by an edge to the vertex $f(z)\in f^{-(n-1)}(t)$.

Let $\gamma$ be a path in $\M$ starting in $t_1$ and ending in $t_2$. Define the isomorphism $S_\gamma\colon T_{t_1}\arr T_{t_2}$ by the condition that the image $S_\gamma(z)$ of a vertex $z\in f^{-n}(t_1)$ of $T_{t_1}$ is the end of the lift of the path $\gamma$ by $f^n$ starting in $z$.

Fix some basepoint $t\in\M$. For every vertex $z\in f^{-n}(t)$ of the tree $T_t$, the tree $T_z$ is naturally a subtree of $T_t$. Let $\mathcal{I}_{f, t}$ be the set of isomorphisms between such subtrees $T_z$ of $T_t$ equal to $S_\gamma$ for paths  $\gamma$ with endpoints in $\bigsqcup_{n\ge 0}f^{-n}(t)$.

Consider a composition $S_{\gamma_1}S_{\gamma_2}$ of elements of $\mathcal{I}_{f, t}$. Suppose that $\gamma_2$ is a path from $x_1$ to $y_1$, and $\gamma_1$ is a path from $x_2$ to $y_2$. Since the subtrees $T_z$ of $T_t$ are either disjoint or one is a subtree of the other, the composition $S_{\gamma_1}S_{\gamma_2}$ is non-empty if and only if $T_{y_2}$ and $T_{x_1}$ are not disjoint, i.e., either $f^k(x_1)=y_2$ or $f^k(y_2)=x_1$ for some $k\ge 0$. In the first case, let $\gamma_1'$ be the lift of $\gamma_2$ by $f^k$ starting in $x_1$. Then $S_{\gamma_1}S_{\gamma_2}=S_{\gamma_1'\gamma_2}$. In the second case, let $\gamma_2'$ be the lift of $\gamma_1$ by $f^k$ starting in $y_2$. Then $S_{\gamma_1}S_{\gamma_2}=S_{\gamma_1\gamma_2'}$. 

It follows that $\mathcal{I}_{f, t}$ is an inverse semigroup acting by  automorphisms between the rooted subtrees of the tree $T_t$. The set of idempotents of $\mathcal{I}_{f, t}$ is identified with the set of vertices of the tree $T_t$. The unit space of the associated groupoid of germs is identified with the boundary of $T_t$, i.e., with the inverse limit of the sets $f^{-n}(t)$.

The group of units of $\mathcal{I}_{f, t}$ is the subgroup of $\mathcal{I}_{f, t}$ of elements $S_\gamma$ such that $\gamma$ is a loop at $t$. It acts by automorphisms of the tree $T_t$, and is called the \emph{iterated monodromy group} of $f\colon \M\arr\M$, denoted $\img{f}$.

Let us describe a natural structure of a self-similar group on the iterated monodromy group.
Let $\bim$ be the set of elements $S_\gamma\in\mathcal{I}_{f, t}$ such that $\gamma$ starts in $t$ and ends in a vertex $z\in f^{-1}(t)$ of the first level of $T_t$. 

For a proof of the following proposition, see~\cite[Section~5.2]{nek:book} or~\cite[Section~3.1]{nek:fpresented}.

\begin{proposition}
The inverse semigroup $\mathcal{I}_{f, t}$ is generated by $\img{f}\cup\bim$. The set $\bim$ is invariant under the left and the right multiplications of $\img{f}$ and is a covering biset with respect to them. The inverse semigroup $\mathcal{I}_{f, t}$ is isomorphic to the inverse semigroup $\mathcal{O}_{\bim}$ generated by the $\img{f}$-biset $\bim$.
\end{proposition}

\begin{example}
Let $f\colon\R/\Z\arr\R/\Z$ be the double covering $x\mapsto 2x$ of the circle. The fundamental group of $\R/\Z$ is naturally isomorphic to $\Z$, where $n\in\Z$ corresponds to the loop in $\R/\Z$ equal to the image of the path from $0$ to $n$ in $\R$.

Choose the image of $0\in\R$ as the basepoint $t$. The preimage $f^{-1}(0)$ consists of $0$ and $1/2$. Consider the biset $\bim$ from Example~\ref{ex:addingmachine}. For an element $\phi(n)=2n+m$, consider the path $\ell_\phi$ from $0$ to $m/2$ in $\R$. It is checked directly that the map $\phi\mapsto\ell_\phi$ is an isomorphism of the biset $\bim$ with the biset $\bim_{f, t}$. It follows that the iterated monodromy group of $f$ is isomorphic, as a self-similar group generated by the binary adding machine.
\end{example}

\begin{example}
Let $f$ be a complex rational function seen as a map from the Riemann sphere $\widehat\C$ to itself. Let $C$ be the set of critical points of $f$, and let $P_f=\bigcup_{n\ge 1}f^n(C)$ be the \emph{post-critical set} of $f$. Suppose that $P_f$ is finite. Then the map $f\colon\widehat\C\setminus f^{-1}(P_f)\arr\widehat\C\setminus P_f$ is a covering of a punctured sphere by its subset. We can define the biset $\bim_{f, t}$ and the iterated monodromy group for this \emph{partial self-covering} in the same way as for a covering map. 

For example, the iterated monodromy group of $z^2-1$ is the self-similar group generated by the wreath recursion
\[a\mapsto \pi(1, b),\qquad b\mapsto (1, a),\]
where $\pi\in\mathsf{S}_2$ is the transposition, see~\cite[Subsection~5.2.2]{nek:book}.
\end{example}

\subsection{Contracting groups}

\begin{definition}
Let $(G, \bim)$ be a self-similar group, and let $\alb\subset\bim$ be a basis. We say that $(G, \bim)$ \emph{contracting} if there exists a finite subset $\nuke\subset G$ such that for every element $g\in G$  there exists $n$ such that $g|_v\in\nuke$ for all $v\in\xs$ of length $\ge n$.

The smallest set $\nuke\subset G$ satisfying the above condition is called the \emph{nucleus} of the self-similar group.
\end{definition}

A condition equivalent to the condition of the above definition is that every element of the groupoid $\mathfrak{O}_{\bim}$ is a germ of an element of $\mathcal{O}_{\bim}$ of the form $S_vgS_u^{-1}$ for $v, u\in\xs$ and $g\in\nuke$.

The property of a self-similar group to be contracting does not depend on the choice of the transversal $\alb$, but the nucleus does, see~\cite[Corollary~2.11.7]{nek:book}.

The nucleus is a \emph{self-similar subset} of $G$ in the sense that for every $g\in\nuke$ and every $x\in\alb$ we have $g|_x\in\nuke$. The \emph{Moore diagram} of the nucleus is the graph with the set of vertices $\nuke$, where for every $g\in\nuke$ and $x\in\alb$ we have an arrow from $g$ to $g|_x$ labeled by $(x, g(x))$.

The following theorem is proved in~\cite[Theorem~5.5.3]{nek:book} and~\cite[Proposition~5.1]{nek:fpresented}.

\begin{theorem}
\label{th:contractingimg}
Let $f\colon \M\arr\M$ be a covering map of a path-connected compact metric space, and suppose that it is \emph{expanding}, i.e., that there exist $\delta>0$ and $L>1$ such that for any two points $x, y\in\M$ such that $d(x, y)<\delta$ we have $d(f(x), f(y))\ge Ld(x, y)$.

Then the iterated monodromy group of $f$ is a contracting self-similar group.
\end{theorem}

The dynamical system $f\colon\M\arr\M$ can be reconstructed from its iterated monodromy group $(G, \bim)$, if it satisfies the conditions of  Theorem~\ref{th:contractingimg}, as the \emph{limit dynamical system} of $(G, \bim)$. It is defined in the following way.

Let $(G, \bim)$ be a contracting self-similar group. Denote by $\Omega$ the space of all left-infinite sequences $\ldots x_2x_1$ of elements of $\bim$ such that there exists a finite subset $A\subset\bim$ such that $x_n\in A$ for all $n$. In other words, $\Omega$ is the union of the sets $A^{-\omega}=\{\ldots x_2x_1 : x_n\in A\}$ over all finite subsets $A\subset\bim$. We introduce on $\Omega$ the \emph{inductive limit topology} defined by this union, where the sets $A^{-\omega}$ have direct product topology of discrete sets $A$. In other words, a subset $U\subset\Omega$ is open if and only if the set $U\cap A^{-\omega}$ is relatively open in $A^{-\omega}$ for every finite set $A\subset\bim$.

We interpret the elements of $\Omega$ as infinite products $\ldots x_2x_1$ of elements of the semigroup $B^*$. With this idea in mind, we define an equivalence relation on $\Omega$ identifying two sequences $\ldots x_2x_1$ and $\ldots y_2y_1$ if and only if there exists a sequence $g_n$ taking values in a finite subset of $G$ such that 
\[g_n\cdot x_n\ldots x_2x_1=y_n\ldots y_2y_1\]
in $\bim^*$ for all $n\ge 1$. 

The quotient of $\Omega$ by this equivalence relation is called the \emph{limit $G$-space} and is denoted by $\limg$. The natural right action of $G$ on $\Omega$ induces an action of $G$ on $\limg$.
This action is proper, so the groupoid of the action defines an orbispace structure on $\limg/G$. 

\subsection{Algebras associated with contracting groups}
\label{ss:algebrasgroups}

Let $(G, \bim)$ be a self-similar group, and let $\Bbbk$ be a field. Let $\Bbbk G$ be the group algebra, and let $\Bbbk\bim$ be the linear span of the biset $\bim$ (i.e., the vector space freely generated by $\bim$).

The left and the right actions of $G$ on $\bim$ extend by linearity to a structure of a \emph{$\Bbbk G$-bimodule} on $\Bbbk\bim$.

Since the right action of $G$ on $\bim$ is free, the right $\Bbbk G$-bimodule $\Bbbk\bim_{\Bbbk G}$ is free, i.e., is isomorphic to $(\Bbbk G)^d$ for $d=|\bim/G|$.  If $\alb\subset\bim$ is a basis of the right action, then it is naturally interpreted as the basis of the right module $\Bbbk\bim_{\Bbbk G}$.

For every $a\in\Bbbk G$, the map $x\mapsto a\cdot x$ on $\Bbbk\bim$ is an endomorphism of the right module, so we get a homomorphism
\[\phi\colon\Bbbk G\arr\mathop{\mathrm{End}}\Bbbk\bim_{\Bbbk G}\]
from the group algebra to the algebra of endomorphisms of the right module. 

Since the right module is isomorphic to the free module $(\Bbbk G)^d$, the algebra of endomorphism is isomorphic to the algebra of $d\times d$-matrices over the group algebra. After we identify $\mathop{\mathrm{End}}\Bbbk\bim_{\Bbbk G}$ with the algebra $M_d(\Bbbk G)$ of $d\times d$-matrices by choosing a basis $\alb\subset\bim$, we get the associated \emph{matrix recursion}
\[\phi\colon \Bbbk G\arr M_d(\Bbbk G).\]
It is the linear extension of the map
\[\phi(g)=\left(a_{xy}\right)_{x, y\in\alb}\]
given, for $g\in G$, by
\[a_{xy}=\left\{\begin{array}{rl} g|_y & \text{if $g(y)=x$;}\\
  0 & \text{otherwise.}\end{array}\right.\]
Note that the map $\phi\colon G\arr M_d(\Bbbk G)$ is just the natural representation of the wreath recursion $\phi\colon G\arr\mathsf{S}_d\ltimes G^\alb$ by matrices.
  
Our definition of tensor products of bisets agrees with the definition of the tensor product of bimodules, so that the bimodule $\Bbbk(\bim_1\otimes\bim_2)$ is naturally isomorphic to the bimodule $\Bbbk\bim_1\otimes_{\Bbbk G}\Bbbk\bim_2$. In particular, the bimodule $\Bbbk\bim^{\otimes n}$ is the $n$th tensor power of the bimodule $\Bbbk\bim$. The associated wreath recursion
\[\phi^n\colon \Bbbk G\arr M_{d^n\times d^n}(\Bbbk G)\]
is called the \emph{$n$th iteration} of the wreath recursion $\phi$.

\begin{example}
Consider the binary adding machine action of $\Z$ over the alphabet $\alb=\{0, 1\}$ generated by the transformation $a$ defined by the wreath recursion
\[a\mapsto\pi(1, a),\]
where $\pi$ is the transposition.

Then the associated matrix recursion is given by
\[\phi(a)=\left(\begin{array}{cc} 0 & a\\  1 & 0\end{array}\right).\]
The matrix recursion is iterated (if we order the basis $\alb^n$ lexicographically) by applying $\phi$ to the entries of the matrices:
\[\phi^2(a)=\left(\begin{array}{cc|cc} 0 & 0 & 0 & a\\ 0 & 0 & 1 & 0\\ \hline
1 & 0 & 0 & 0\\
0 & 1 & 0 & 0\end{array}\right)\]
and
\[\phi^3(a)=\left(\begin{array}{cc|cc|cc|cc} 
0 & 0 & 0 & 0 & 0 & 0 & 0 & a\\
0 & 0 & 0 & 0 & 0 & 0 & 1 & 0\\ \hline
0 & 0 & 0 & 0 & 1 & 0 & 0 & 0\\
0 & 0 & 0 & 0 & 0 & 1 & 0 & 0\\ \hline
1 & 0 & 0 & 0 & 0 & 0 & 0 & 0\\
0 & 1 & 0 & 0 & 0 & 0 & 0 & 0\\ \hline
0 & 0 & 1 & 0 & 0 & 0 & 0 & 0\\
0 & 0 & 0 & 1 & 0 & 0 & 0 & 0\end{array}\right).\]
\end{example}

More generally, for every pair $n, m\ge 0$, any endomorphism $\alpha$ of the right module $\Bbbk\bim^{\otimes n}_{\Bbbk G}$ induces an endomorphism of the right module $\Bbbk\bim^{\otimes (n+m)}$  by
\[v_1\otimes v_2\mapsto\alpha(v_1)\otimes v_2.\]
We get, therefore a homomorphism of $\Bbbk G$-algebras
\[\phi_{n+m, n}\colon M_{d^n}(\Bbbk G)\arr M_{d^{n+m}}(\Bbbk G).\]
We identify here the algebras of endomorphisms with the matrix algebras, by choosing the bases $\alb^n$ of the right modules $\Bbbk\bim^{\otimes n}$, thus identifying them with $(\Bbbk G)^{\alb^n}$.

We have then $\phi_{n_1, n_2}\circ\phi_{n_2, n_3}=\phi_{n_1, n_3}$ for all $n_1>n_2>n_3$. In particular, the $n$th iteration of the matrix recursion $\phi^n$, defined above, is equal to $\phi_{n, 0}$.

\begin{proposition}
Let $(G, \bim)$ be a faithful self-similar group, and let $\Gr$ be the groupoid of germs of its action on the space $\xo$ of infinite words. Then the Steinberg algebra $\Bbbk\Gr$ is isomorphic to the direct limit of the matrix algebras $M_{d^n}(\Bbbk G)$ with respect to the homomorphisms $\phi_{m, n}$, defined above.
\end{proposition}

The proof of the above proposition is the same as of~\cite[Proposition~3.8]{nek:cpalg}. Namely, the algebra $\Bbbk\Gr$ is isomorphic to the algebra given by the following presentation. The generators are $1_{S_ugS_v^{-1}}$ for $g\in G$ and $u, v\in\xs$ such that $|u|=|v|$. The defining relations are
\[1_{S_{u_1}g_1S_{v_1}^{-1}}1_{S_{u_2}g_2S_{v_2}^{-1}}=\left\{\begin{array}{rl}
1_{S_{u_1}g_1g_2S_{v_2}^{-1}} & \text{if $v_1=u_2$,}\\
0 & \text{otherwise,}\end{array}\right.\]
for $|u_1|=|v_1|=|u_2|=|v_2|$ and $g_1, g_2\in G$, and
\[1_{S_ugS_v^{-1}}=\sum_{x\in\alb}1_{S_{ug(x)}g|_xS_{vx}^{-1}}\]
for $|u|=|v|$ and $g\in G$. It is checked directly that the inductive limit is defined by the same presentation, where a generator $1_{S_ugS_v^{-1}}$ for $u, v\in\alb^n$ is identified with the matrix $(a_{w_1, w_2})_{w_1, w_2\in\alb^n}$, where $a_{w_1, w_2}=0$ for all $w_1, w_2\in\alb^n$ except for $a_{u, v}=g$.

Consider now the case $\Bbbk=\C$. We can introduce a $\C G$-valued inner product on the right $\C G$-bimodule $\C\bim$, i.e., a map $\langle\cdot |\cdot \rangle$ from $\C\bim\times\C\bim$ to $\C G$ satisfying \begin{gather*}
\langle v_1|v_2+v_3\rangle=\langle v_1|v_2\rangle+\langle v_1|v_3\rangle,\quad \langle v_1|v_2\rangle^*=\langle v_2|v_1\rangle,\\
\langle v_1|v_2\cdot a\rangle=\langle v_1|v_2\rangle\cdot a,\quad \langle v_1|a\cdot v_2\rangle=\langle a^*\cdot v_1|v_2\rangle,\end{gather*}
for all $v_1, v_2, v_3\in\C\bim$, $a\in\C G$. 

Namely, we set, for $x, y\in\bim$,
\[\langle x| y\rangle=\left\{\begin{array}{rl} g & \text{if $x\cdot g=y$ if such $g\in G$ exists,}\\
0 & \text{otherwise,}\end{array}\right.\]
and then extend it to $\C\bim\times\C\bim$ by
\[\left\langle\sum_i \alpha_ix_i|\sum_j\beta_iy_i\right\rangle=\sum_{i, j}\overline{\alpha_i}\beta_j\langle x_i|y_j\rangle\]
for $\alpha_i, \beta_j\in\C$ and $x_i, y_j\in\bim$. Note that the definition of $\langle x|y\rangle$ coincides with the definition of $x^{-1}y$ in the inverse semigroup $\mathcal{O}_\bim$, see~\ref{sss:basicdefinitions}.

If we choose a basis $\alb$, then the definition of the inner product becomes
\[\langle x\cdot g|y\cdot h\rangle=\left\{\begin{array}{rl} \overline\alpha\beta g^{-1}h & \text{if $x=y$,}\\
0 & \text{otherwise,}\end{array}\right.\]
for $g, h\in G$ and $x, y\in\alb$.  Note that the basis $\alb$ is ``orthonormal'' with respect to the defined inner product, i.e.,  $\langle x|y\rangle=0$ if $x\ne y$ and $\langle x|x\rangle=1$ for all $x\in\alb$.

The defined inner product naturally agrees with the tensor product of bimodules in the sense that
\[\langle v_1\otimes v_2|u_1\otimes u_2\rangle=\langle v_2|\langle v_1|u_1\rangle\cdot u_2\rangle.\]

For $x, y\in\C\bim$, define the ``rank one'' operator $\theta_{x, y}$ on $\C\bim$ by
\[\theta_{x, y}(z)=x\langle y|z\rangle.\]
It is an endomorphism of the right module $\C\bim_{\C G}$.

\begin{definition} Consider a self-similar group $(G, \bim)$. 
Let $\phi\colon\C G\arr\mathop{\mathrm{End}}(\C \bim_{\C G})$ be the associated matrix recursion. The \emph{Cuntz-Pimsner algebra} of the bimodule $\C\bim$ is the universal $C^*$-algebra generated by the $*$-algebra $\C G$ and operators $S_x$ for every $x\in\C\bim$ satisfying the following relations:
\begin{enumerate}
\item $S_{\alpha x+\beta y}=\alpha S_x+\beta S_y$ for all $\alpha, \beta\in\C$, $x, y\in\C\bim$.
\item $S_x\cdot a=S_{x\cdot a}$ and $a\cdot S_x=S_{a\cdot x}$ for all $a\in\C G$ and $x\in\C\bim$.
\item $S_x^*S_y=\langle x|y\rangle$ for all $x, y\in\C\bim$.
\item If $\phi(a)=\sum_i\alpha_i\theta_{x_i, y_i}$ for $a\in\C G$, $\alpha_i\in\C$, and $x_i, y_i\in\C\bim$, then $a=\sum_i\alpha_iS_{x_i}S_{y_i}^*$.
\end{enumerate}
\end{definition}

One can show that the relation are equivalent (by linearity) to the following presentation.

\begin{proposition}
\label{pr:CuntzPimsner}
Let $(G, \bim)$ be a self-similar group, and let $\alb$ be a basis of the biset $\bim$. 
Then the Cuntz-Pimnser algebra of the bimodule $\C\bim$ is the unversal $C^*$-algebra generated by operators $S_x$ for $x\in\alb$ and $u_g$ for $g\in G$ subject to the following defining relations:
\begin{enumerate}
\item $u_{g_1g_2}=u_{g_1}u_{g_2}$ and $u_g^*=u_{g^{-1}}$ for all $g, g_1, g_2\in G$.
\item Cuntz relations: $S_x^*S_x=1$ for all $x\in\alb$, and $1=\sum_{x\in\alb}S_xS_x^*$.
\item $g=\sum_{x\in\alb}S_{g(x)}u_{g|_x}S_x^*$ for all $g\in G$.
\end{enumerate}
\end{proposition}

The following theorem is proved in~\cite{nek:cpalg}.

\begin{theorem}
Let $(G, \bim)$ be a faithful self-similar group. Then the Cuntz-Pimsner algebra of the bimodule $\C\bim$ is naturally isomorphic to the full $C^*$-algebra $C^*(\mathfrak{O}_\bim)$ of the groupoid of germs $\mathfrak{O}_\bim$ generated by the action of $G$ and the action of the one-side shift on $\xo$.
\end{theorem}

We will discuss the Steinberg algebras $\Bbbk\mathfrak{O}_\bim$ in a greater generality of self-similar inverse semigroups. In particular, we will show (see Theorem~\ref{th:convolutionfinpres}) that the Steinberg algebra $\Bbbk\mathfrak{O}_\bim$ is finitely presented for contracting self-similar groups. Since $C^*(\mathfrak{O}_\bim)$ is the universal enveloping $C^*$-algebra of the $*$-algebra $\C\mathfrak{O}_\bim$, it will imply that the Cuntz-Pimnser algebra is also given by a finite presentation (in the category of $C^*$-algebras).

\section{Self-similar inverse semigroups}

As we have seen above, self-similarity of a group action is naturally defined using the inverse semigroup $\mathcal{O}_\alb$ of local self-similarities of the full one-sided shift $\xo$. It is reasonable therefore to expect that self-similar groups are examples of a more general notion of self-similar inverse semigroups.

The notion of a self-similar inverse semigroup was introduced in~\cite{bgn}, where some concrete examples were given. The notion of a \emph{contracting self-similar inverse semigroup} and its \emph{limit solenoid} was introduced in~\cite{nek:smale}. It was shown there that they appear naturally as holonomy semigroups of \emph{Smale spaces}. The connection between Smale spaces and contracting inverse semigroups is a direct generalization of the connection between expanding covering maps and iterated monodromy groups.

J.~Walton and M.F.~Whittaker studied in~\cite{WaltonWhittaker:tilings} the \emph{tiling semigroups} as defined by \cite{kellendonk:noncom,kellendonk:coinvar} and described explicitly their self-similarity structure in the case of self-similar tilings.

We present there an overview of these results, illustrating them by several examples.

\subsection{Shifts of finite type}

\begin{definition}
\label{def:expansive}
Let $f\colon\X\arr\X$ be a continuous map, where $\X$ is compact. We say that $f$ is \emph{expansive} if there exists a neighborhood of the diagonal $U\subset\X\times\X$ such that if $(f^n(x), f^n(y))\in U$ for all $n\ge 0$, then $x=y$.
\end{definition}

\begin{definition}
Let $\alb$ be a finite alphabet. 
A subset $\mathcal{F}\subset\xo$ is called a \emph{subshift} if it is closed and $\si(\mathcal{F})\subseteq\mathcal{F}$, where $\si\colon\xo\arr\xo$ is the one-sided shift (see~\ref{sss:basicdefinitions}).

For $v\in\xs$, we denote by ${}_v\F$ the set $v\xo\cap\F$ of elements of $\F$ having the word $v$ as its beginning.
\end{definition}

We have the following well known result (for proofs see, for example~\cite{reddy:liftingexpansive}, \cite[Proposition~2.8]{curnpap:symb}, and~\cite[Theorem~1.2.10]{nek:dyngroups}).

\begin{proposition}
\label{pr:openisfinitetype}
Let $f\colon \X\arr\X$ be a continuous map such that $\X$ is a compact totally disconnected space. It is expansive if and only if $(f, \X)$ is topologically conjugate to a subshift.
\end{proposition}

Since the complement of a subshift is an open subset, it can be written as a union of cylindrical sets $v\xo$. The set $v\xo$ is contained in the complement of a subshift $\mathcal{F}\subset\xo$ if and only if $v$ does not appear as a subword $x_ix_{i+1}\ldots x_{i+k}$ in any element $x_0x_1\ldots\in\mathcal{F}$.

Conversely, if $P\subset\xs$ is any set of finite words, then the set of infinite sequences $x_0x_1\ldots$ such that $x_ix_{i+1}\ldots x_{i+k}\notin P$ for all $i, k\ge 0$ is a (possibly empty) subshift. We call $P$ the \emph{set of prohibited words defining} $\mathcal{F}$.

\begin{definition}
A subshift $\mathcal{F}\subset\xo$ is a \emph{shift of finite type} if it is defined by a finite set $P\subset\xs$ of prohibited words.

We say that $\mathcal{F}$ is a \emph{topological Markov chain} if it is defined by a set of prohibited words $P$ contained in $\alb^2$.
\end{definition}

It is natural to define a topological Markov chain $\F$ by the complement of $P$ in $\alb^2$, i.e., by the set of \emph{allowed transitions} $T\subset\alb^2$. A sequence $x_1x_2\ldots$ belongs to the subshift defined by $T$ if and only if $x_nx_{n+1}\in T$ for every $n$.

The \emph{width $k$ block-code} is the homeomorphic embedding of $\xo$ into $(\alb^k)^\omega$ defined by
\[x_1x_2x_3\ldots\mapsto (x_1x_2\ldots x_k)(x_2x_3\ldots x_{k+1})(x_3x_4\ldots x_{k+2})\ldots.\]
We will usually write the elements $x_1x_2\ldots x_k\in\alb^k$ for the block-code as $(x_1)_{x_2\ldots x_k}$, so that the map is
\[x_1x_2\ldots\mapsto (x_1)_{x_2\ldots x_k}(x_2)_{x_3\ldots x_{k+1}}(x_3)_{x_4\ldots x_{k+2}}\ldots.\]

It is easy to see that the block code commutes with the shifts $\si\colon \xo\arr\xo$ and $\si\colon (\alb^k)^\omega\arr(\alb^k)^\omega$ and that the image of $\xo$ under the block code is a topological Markov chain in $(\alb^k)^\omega$ defined by the set of allowed transitions of the form $(x_1)_{x_2\ldots x_k}(x_2)_{x_3\ldots x_{k+1}}$.

It is also easy to see that if $\mathcal{F}\subset\xo$ is a subshift of finite type and $k$ is larger than the maximal length words in a finite set of prohibited words defining $\mathcal{F}$, then the image of $\mathcal{F}$ under the width $k$ block code is a topological Markov chain. In particular, up to topological conjugacy, the class of shifts of finite type and the class of topological Markov chains coincide.

Let $\Delta$ be a finite directed graph. We denote by $V$ and $E$ the sets of vertices and edges of $\Delta$, and by $\be(e), \en(e)$ the initial and the final vertices of an edge $e$, respectively. The \emph{edge shift} of $\Delta$ is the space of all infinite directed paths $e_1e_2\ldots\in E^\omega$. A sequence $e_1e_2\ldots$ is a path in $\Delta$ if and only if $\en(e_n)=\be(e_{n+1})$ for all $n$. Therefore, the edge shift is a topological Markov chain.

Conversely, if $\F\subset\xo$ is a topological Markov chain defined by a set of transitions $T\subset\alb^2$, then its image under the with $2$ block code is the edge shift of the graph $\Delta$ with $V=\alb$, $E=T$, and $\be(xy)=x$, $\en(xy)=y$.

Consequently, every shift of finite type is topologically conjugate to an edge shift of a finite graph.

\begin{lemma}
\label{lem:eventuallyonto}
Let $\si\colon \F\arr\F$ be a shift of finite type. Then there exists $n$ such that $\si^{n+1}(\F)=\si^n(\F)$, and $\si\colon \si^n(\F)\arr\si^n(\F)$ is a shift of finite type.
\end{lemma}

\begin{proof}
We may assume that $\F\subset\xo$ is a topological Markov chain defined as the edge shift of a finite graph $\G$. 

The set $\si^n(\F)$ is the set of the directed paths in $\G$ that start in vertices which are endpoints of finite directed paths of length $n$.

Let $V_n$ be the set of vertices that are endpoints of paths of length $n$. 
We obviously have, $V_{n+1}\subseteq V_n$. Since $\alb$ is finite, the sequence $V_n$ is eventually constant, hence there exists $n$ such that $V_{n+1}=V_n$, which implies $\si^{n+1}(\F)=\si^n(\F)$. The subshift $\si^n(\F)$ is the topological Markov chain defined by the graph spanned in $\G$ by the set of vertices $V_n$, hence it is a topological Markov shift.
 \end{proof}

The following proposition is a combination of Proposition~\ref{pr:openisfinitetype} and~\cite[Theorem~1]{parry:symbolic}. 

\begin{proposition}
\label{pr:expansiveopensft}
A continuous map $f\colon \X\arr\X$, where $\X$ is a compact totally disconnected space, is topologically conjugate to a one-sided shift of finite type if and only if $f$ is expansive and open. 
\end{proposition}

\begin{proof}
If $\si\colon \F\arr\F$ is a topological Markov chain for $\mathcal{F}\subset\xo$, then for every $x\in\alb$ the set $\si({}_x\F)$ is equal to the union of the sets ${}_y\F$ for all $y$ such that $xy$ is allowed in $\mathcal{F}$, hence $\si({}_x\F)$ is open in $\F$. 

The map $\si\colon {}_x\F\arr \si({}_x\F)$ is a homeomorphism. Consequently, for every open set $U\subset{}_x\F$, the set $\si(U)$ is an open subset of $\si({}_x\F)$, hence an open subset of $\F$. It follows that $\si\colon \F\arr\F$ is open.

It remains to show that if $\X\subset\xo$ is a subshift and the shift map $\si\colon \X\arr\X$ is open, then $\X$ is a subshift of finite type. For every $x$, the map $\si\colon {}_x\X\arr\si({}_x\X)$ is a homeomorphism (its inverse is the map $w\mapsto xw$). 

Since $\si$ is open, and ${}_x\X$ is compact and open, the set $\si({}_x\X)$ is compact and open. Therefore, there exists $m$ such that for every $x\in\alb$ the set $\si({}_x\X)$ is equal to a union $\bigcup_{v\in W_x}{}_v\X$ for some sets $W_x\subset\alb^m$, where we assume that ${}_v\X$ is non-empty for every $v\in W_x$. (If the set ${}_x\X$ is empty, then we assume that $W_x$ is empty.) 

Let $T=\bigcup_{x\in\alb}xW_x$, and let $\X_T$ be the set of sequences $x_1x_2\ldots \in\xo$ such that $x_ix_{i+1}\ldots x_{i+m}\in T$ for every $i$. Then $\X_T$ is a shift of finite type containing $\X$. 

Let us show that $\X_T\subset\X$, i.e., that $\X_T=\X$.
By the choice of $T$, for every $x\in\alb$ and every $a_1a_2\ldots a_m\in W_x$ there exists a sequence $x_1x_2\ldots\in\xo$ such that $a_1a_2\ldots a_mx_1x_2\ldots \in\X$, and for every such a sequence, the sequence $xa_1a_2\ldots a_mx_1x_2\ldots$ also belongs to $\X$. Arguing by induction, we conclude that for every finite word $u\in\xs$ such that every subword of $u$ belongs to $T$, the set ${}_u\X$ is non-empty. 

Consequently, for every $y_1y_2\ldots\in\X_T$, the sets ${}_{y_1y_2\ldots y_n}\X$ form a decreasing sequence of non-empty subsets of $\X$. By compactness, their intersection is non-empty, hence $y_1y_2\ldots$ belongs to $\X$. 
\end{proof}

\begin{definition}
Let $\si\colon\F\arr\F$ be a continuous map.
A \emph{Markov partition} of $\si\colon\F\arr\F$ is a partition $\mathcal{P}$ of $\F$ into a finite number of disjoint clopen subsets such that for every $A\in\mathcal{P}$ the map $\si\colon A\arr\si(A)$ is a homeomorphism and $\si(A)$ is equal to a union of elements of $\mathcal{P}$.
\end{definition}

Every topological Markov subshift $\mathcal{F}$ has a natural Markov partition into the sets $x\mathcal{F}$ for the letters $x$ of the alphabet. 

Conversely, if $\mathcal{P}=\{U_x\}_{x\in\alb}$ is a Markov partition, then we can encode a point $w\in\F$ by its \emph{itinerary} $I(w)=x_0x_1\ldots$, where $x_n$ is defined by the condition $\si^n(w)\in U_{x_n}$. Then $I(\F)\subset\xo$ is a topological Markov chain. If $\si$ is expansive and the elements of $\mathcal{P}$ are sufficiently small, so that the set $U=\bigcup_{x\in\alb}U_x\times U_x$ satisfies the conditions of Definition~\ref{def:expansive}, then the map $I$ is a topological conjugacy between $\F$ and the topological Markov chain $I(\F)$.

Thus, if $\si\colon\F\arr\F$ is an expansive open map on a totally disconnected compact space $\F$, then choosing an encoding of $\F$ as a topological Markov shift is equivalent to a choice of a Markov partition.

\subsection{Cuntz-Krieger algebras}
\label{ss:CuntzKrieger}

Let $\si\colon \F\arr\F$ be a subshift of finite type. Since $\si$ is a local homeomorphisms, we can consider the ample groupoid $\Fr$ generated by the germs of $\si$.

More explicitly, $\Fr$ is the groupoid of germs of the inverse semigroup generated by the homeomorphisms $\si\colon A\arr B$, where $A$ and $B$ are clopen subsets of $\mathcal{F}$ such that $\si\colon A\arr B$ is a homeomorphism.

By the results of the previous subsection, after choosing a Markov partition, we can identify $\F$ with an edge shift on a finite graph $\G$ with the set of edges $\alb$, so that $\F\subset\xo$ is the set of all directed paths sequences $x_1x_2\ldots\in\xo$ in $\G$. We denote the source and the range of an arrow $x\in\alb$ be $\be(x)$ and $\en(x)$, respectively. 

The groupoid $\Fr$ is the groupoid of germs of the inverse semigroup generated by the homeomorphisms $\si\colon {}_x\F\xo\arr \bigcup_{y\in\alb, \be(y)=\en(x)}{}_y\F$. The inverse of this homeomorphism is the map $S_x:w\mapsto xw$ defined on the set of paths $w\in\F$ starting in the vertex $\en(x)$ of $\G$.

We will use the multi-index notation $S_{x_1x_2\ldots x_n}=S_{x_1}S_{x_2}\cdots S_{x_n}$, and denote by $I_v=S_vS_v^{-1}$ idempotent corresponding to the set ${}_v\F$ of paths  starting with $v$. Similarly, if $\alpha$ is a vertex of $\G$, then we denote by ${}_\alpha\F$ the set of paths in $\G$ starting in $\alpha$. We have ${}_\alpha\F=\bigcup_{\be(x)=\alpha}{}_x\F$.

Then the maps $S_v$ are defined for all finite paths $v$ in $\G$.
The inverse semigroup generated by the transformations $S_v$ is equal to the set of maps of the form $S_vS_u^{-1}$, where $v$ and $u$ are finite directed paths in $\G$ ending in the same vertex. The groupoid $\Fr$ is the groupoid of germs of this inverse semigroup.

The elements of the groupoid $\Fr$ are germs of the form $g=(\si^n, w_1)^{-1}(\si^m, w_2)$ for $w_1, w_2\in\F$ such that $\si^m(w_2)=\si^n(w_1)$. If $x_1x_2\ldots x_m$ and $y_1y_2\ldots y_n$ are the prefixes of length $m$ and $n$ of $w_2$ and $w_1$, respectively, then $S_{y_1y_2\ldots y_n}S_{x_1x_2\ldots x_m}^{-1}$ is a compact open bisection containing $g$.

We will identify $S_x$ with the corresponding element $1_{S_x}$ of the algebra $\C\Fr$, and use the addition symbol instead of disjoint union of local homeomorphisms (bisections).

The $C^*$-algebras of ample groupoids generated by subshifts of finite type were defined and studied by J.~Cuntz and W.~Krieger in~\cite{cuntzkrie}, and are known as \emph{Cuntz-Krieger algebras}.

They are defined by the following presentation (both in the category of algebras and in the category of $C^*$-algebras).

Its generators are $I_\alpha$, $S_x$ (also $S_x^{-1}$ in the case of $\Bbbk\Fr$) labeled by vertices $\alpha\in\mathsf{V}$ edges $x\in\alb$ of $\G$. We set $S_x^*=S_x^{-1}$ in the case of $*$-algebras. The defining relations are:
\[I_{\be(x)}=S_x^{-1}S_x,\qquad I_\alpha=\sum_{\be(x)=\alpha}S_xS_x^{-1}.\]
and
\[\sum_{\alpha\in\mathsf{V}}I_\alpha=1,\qquad I_{\alpha_1}I_{\alpha_2}=0, \qquad I_\alpha^2=I_\alpha\]
for all vertices $\alpha_1\ne \alpha_2$ and $\alpha$.

The fact that $\Bbbk\Fr$ is isomorphic to the algebra defined by the above presentation will follow from a more general Theorem~\ref{th:convolutionfinpres} below.

\subsection{Self-similar inverse semigroups}

\begin{definition}
\label{def:ssset}
Let $\F\subset\xo$ be a one-sided topological Markov chain. A set $\mathcal{A}$ of homeomorphisms between clopen subsets of $\F$ is said to be \emph{self-similar} if the following condition holds.

For every $F\in\mathcal{A}$ and every $x\in\alb$ the composition $FS_x$ is either empty or is equal to a union $S_{x_1}F_1\cup S_{x_2}F_2\cup\ldots\cup S_{x_k}F_k$ such that $F_i\in\mathcal{A}$, the letters $x_i$ are pairwise different, and $\en(F_i)\subset\be(S_{x_i})$ for every $i$.

We say that an inverse semigroup of homeomorphisms between clopen subsets of $\F$ is \emph{self-similar} if it is a self-similar set.
\end{definition}

In the conditions of the definition, the ranges of the transformations $S_{x_i}F_i$ are pairwise disjoint, since they are contained in $x_i\xo$. Consequently, the domains $\be(S_{x_i}F_i)=\be(F_i)$ are also pairwise disjoint (as their union is a map). We also have $\be(F_i)\subset\be(S_x)$.

\begin{proposition}
\label{pr:ssset}
A set $\mathcal{A}$ of homeomorphisms between clopen subsets of $\F$ is self-similar if and only if for all $F\in\mathcal{A}$, $x, y\in\alb$ the homeomorphism $S_y^{-1}FS_x$ is either empty or belongs to $\mathcal{A}$.
\end{proposition}

\begin{proof}
It follows directly from Definition~\ref{def:ssset} and the remark after it that, in the conditions of the definition, we have $F_i=S_{x_i}^{-1}FS_x$, and $S_y^{-1}FS_x$ is non-empty if and only if $y=x_i$ for some $i$. 

Conversely, suppose that $\mathcal{A}$ is such that for all every $F\in\mathcal{A}$, $x, y\in\alb$, the transformation $S_y^{-1}FS_x$ is either empty or belongs to $\mathcal{A}$.

Since the identity map $\F\arr\F$ is equal to the disjoint union of the maps $S_xS_x^{-1}$ for $x\in\alb$, every $F\in\mathcal{A}$ is equal to the disjoint union of the maps $S_y(S_y^{-1}FS_x)S_x^{-1}$.  Denote $F_{y, x}=S_y^{-1}FS_x$. It is either empty or an element of $\mathcal{A}$.

We have $F_{y, x}S_x^{-1}S_x=S_y^{-1}FS_xS_x^{-1}S_x=S_y^{-1}FS_x=F_{y, x}$. Similarly, $S_y^{-1}S_yF_{y, x}=F_{y, x}$, hence $\en(F_{y, x})=\be(S_y)$.

For every $x\in\alb$, we have $FS_x=\bigcup_{z, y\in\alb}S_yF_{y, z}S_z^{-1}S_x=\bigcup_{y\in\alb} S_yF_{y, x}S_x^{-1}S_x=\bigcup_{y\in\alb}S_yF_{y, x}$. Consequently, $\mathcal{A}$ is self-similar.
\end{proof}

Let $\mathcal{A}$ be a self-similar set, i.e., a set of partial homeomorphisms of $\F$ satisfying the conditions of Definition~\ref{def:ssset} or Proposition~\ref{pr:ssset}.  We will represent it by the \emph{Moore diagram} $\G_{\mathcal{A}}$, defined in the following way.

It is a directed graph with the set of vertices $\mathcal{A}\setminus\{\emptyset\}$. There is an arrow from a vertex $F_1$ to a vertex $F_2$ labeled by $x|y$ for $x, y\in\alb$ if and only if $F_2=S_y^{-1}F_1S_x$.

If $x_1x_2\ldots, y_1y_2\ldots\in\F$ and $F_0\in\mathcal{A}$ are such that $F_0(x_1x_2\ldots)=y_1y_2\ldots$, then, for every $n\ge 1$ we have 
\[(S_{y_1y_2\ldots y_n}^{-1}F_0S_{x_1x_2\ldots x_n})(x_{n+1}x_{n+2}\ldots)=y_{n+1}y_{n+2}\ldots.\] It follows that there is a directed path $e_1e_2\ldots$ in the Moore diagram $\G_{\mathcal{A}}$, were $e_n$ is the arrow from $F_n=S_{y_1y_2\ldots y_{n-1}}^{-1}F_0S_{x_1x_2\ldots x_{n-1}}$ to $F_{n+1}=S_{y_n}^{-1}F_nS_{x_n}=S_{y_1y_2\ldots y_n}^{-1}F_0S_{x_1x_2\ldots x_n}$ labeled by $x_n|y_n$ for every $n$.

Conversely, suppose that there exists a directed path $e_1e_2\ldots$ in the Moore diagram $\G_{\mathcal{A}}$ in $F_0\in\mathcal{A}$, let $x_i|y_i$ be the label of $e_i$, and let $F_n$ be the end of $e_n$. Then for every $n$ and every $a_1a_2\ldots\in\be(F_n)$ we have $F_0(x_1x_2\ldots x_na_1a_2\ldots)=y_1y_2\ldots y_nF_{n+1}(a_1a_2\ldots)$. By continuity of $F$ and compactness of $\be(F_0)$, this implies that $F_0(x_1x_2\ldots)=y_1y_2\ldots$.

We generalize this in the following definition.

\begin{definition}
An \emph{$\omega$-deterministic finite automaton} is a directed graph $\Gamma$ with arrows labeled by pairs $x|y$ of letters of an alphabet $\alb$ such that for every vertex $v$ of $\Gamma$ and every sequence $x_1x_2\ldots\in\xo$ there exists at most one directed path starting in $v$ with arrows labeled by $(x_1|y_1)(x_2|y_2)\ldots$. We say then that $v$ \emph{accepts} $x_1x_2\ldots$ and \emph{transforms} it into $y_1y_2\ldots$. The map $x_1x_2\ldots\mapsto y_1y_2\ldots$ is the \emph{transformation defined by $v$}. The vertices of $\Gamma$ are called the \emph{states} of the automaton.
\end{definition}

As we have seen above, if $\mathcal{A}$ is a self-similar set of local homeomorphisms of $\F$, then the Moore diagram of $\mathcal{A}$ is an $\omega$-deterministic automaton with the set of states $\mathcal{A}$ such that the transformation defined by a state $F\in\mathcal{A}$ of the automaton coincides with the local homeomorphism $F$. 

Let $x_1x_2\ldots\in\F$, $F\in\mathcal{A}$, and let $x_1x_2\ldots\in\be(F)$. Then there exists a unique path $e_1e_2\ldots$ in the Moore diagram of $\mathcal{A}$ starting in $F$ and such that $e_n$ is labeled by $x_n|y_n$ for some $y_n\in\alb$. Moreover, then $y_1y_2\ldots=F(x_1x_2\ldots)$.

In the conditions of Definition~\ref{def:ssset}, the domains of the elements of $\mathcal{A}$ are clopen, hence there exists $k_0$ such that they are equal to unions of the sets of the form ${}_v\F$ for $v\in\alb^{k_0}$. Then the prefix of length $k_0+1$ of a sequence $w\in\F$ uniquely determines the first letter of $F(w)$, and hence the arrow $e_1$. It follows that, for every $n$, the arrow $e_n$ is uniquely determined by the end of $e_{n-1}$ and by the word $x_nx_{n+1}\ldots x_{n+k_0}$. Consequently, the length $n$ prefix of $F(x_1x_2\ldots)$ is uniquely determined by $F$ and $x_1x_2\ldots x_{n+k_0}$. In other words, we can make the automaton $\mathcal{A}$ deterministic if we allow it too ``look up'' $k_0$ symbols ``in the future.''

\subsection{Contracting inverse semigroups}

Definition~\ref{def:ssset} seems to be very restrictive. We will show now that sets satisfying it appear in natural situations.

Let, as before, $\si\colon \F\arr\F$ be an expansive open map, where $\F$ is a compact totally disconnected space.
Let $\Gr$ be an effective \'etale groupoid with the unit space $\F$, i.e., the groupoid of germs of an inverse semigroup of local homeomorphisms of $\F$. We will identify here open $\Gr$-bisections with the associated local homeomorphisms, since $\Gr$ is effective.

Let $g\in\Gr$, and let $F$ be an open $\Gr$-bisection containing $g$, seen as a homeomorphism $\be(F)\arr\en(F)$. Since $\si\colon \F\arr\F$ is a local homeomorphism, the germ $\si\cdot g\cdot \si^{-1}$ is well defined. Let $\be(g)=x_1x_2\ldots\in\be(F)$ and $\en(g)=y_1y_2\ldots=F(x_1x_2\ldots)$. Replacing $F$ by smaller neighborhood of $g$, we may assume that $\be(F)\subset x_1\xo$ and $\en(F)\subset y_1\xo$. Then $\si\cdot g\cdot\si^{-1}$ is the germ of the homeomorphism $S_{y_1}^{-1}FS_{x_1}$ at the point $x_2x_3\ldots$.

We will denote $\varsigma(g)=\si\cdot g\cdot\si^{-1}$.

\begin{definition}
An effective \'etale groupoid $\Gr$ with the unit space $\F$ is called \emph{shift-invariant} if $\varsigma(\Gr)\subset\Gr$. We say that it is \emph{completely shift-invariant} if it is shift-invariant and for all $x, y\in\alb$ and $g\in\Gr$ such that $x\be(g), y\en(g)\in\F$, the germ $S_ygS_x^{-1}$ belongs to $\Gr$.
\end{definition}

If $\Gr$ is shift-invariant, then we can make it completely shift invariant by adding to it all germs of the form $S_vgS_u^{-1}$ for all finite words $v, u$ such that $|v|=|u|$ and $u\be(g), v\en(g)\in\F$. We call the obtained completely shift-invariant groupoid the \emph{shift-completion of $\Gr$}.

It is checked directly that $\varsigma$ is a functor. If $\Gr$ is completely shift-invariant, then the functor $\varsigma\colon\Gr\arr\Gr$ defines an equivalence of grouopids  (see Proposition~\ref{prop:equivalencefunctor}). 

For every compact open bisection $F\subset\Gr$ such that $\be(F)\subset x\xo$ and $\en(F)\subset y\xo$, we have
\[\varsigma(F)=S_y^{-1}FS_x.\]
All germs of the local homeomorphism $S_y^{-1}FS_x$ belong to $\Gr$, i.e., $S_y^{-1}FS_x$ is a compact open $\Gr$-bisection. In particular, $\varsigma\colon \Gr\arr\Gr$ is an open continuous map.

\begin{theorem}
\label{th:contracting}
Let $\Gr$ be a shift-invariant groupoid with the space of units $\F\subset\xo$. The following conditions are equivalent.
\begin{enumerate}
\item There exists a compact set $C\subset\Gr$ such that for every $g\in\Gr$ there exists $n$ such that $\varsigma^n(g)\in C$.
\item There exists a compact open subset $N\subset\Gr$ satisfying condition (1) and such that $\varsigma(N)=N$. 
\item There exists a finite set $\nuke\subset\mathcal{B}(\Gr)$ such that for every compact open $\Gr$-bisection $F$ there exists $n$ such that for every $v_1, v_2\in\alb^m$ such that $m\ge n$, the partial homeomorphism $S_{v_1}^{-1}FS_{v_2}$ belongs to $\nuke$ (we include the empty map into $\nuke$).
\end{enumerate}
Moreover, the set $N$ in condition (2) is unique, and there exists the smallest set $\nuke$ satisfying condition (3).
\end{theorem}

The smallest set $\nuke$ satisfying condition (3) of Theorem~\ref{th:contracting} is called the \emph{nucleus} of the groupoid $\Gr$. Note that it is a self-similar set of local homeomorphisms (see Proposition~\ref{pr:ssset}).

\begin{definition}
We say that a shift-invariant groupoid $\Gr$ is \emph{contracting} if it satisfies the conditions of Theorem~\ref{th:contracting}. 

A \emph{contracting inverse semigroup} is a self-similar inverse semigroup $\mathcal{G}$ such that its groupoid of germs $\Gr$ is a shift-invariant contracting groupoid, and $\mathcal{G}$ contains the nucleus of $\Gr$.
\end{definition}

\begin{example}
If $G$ is a self-similar group acting on $\xo$, then the groupoid of germs $\Gr$ of the action is shift-invariant. The action of the functor $\varsigma$ is given by
\[\varsigma[g, x_1x_2\ldots]=[g|_{x_1}, x_2x_3\ldots].\]

If $G$ is contracting with the nucleus $\nuke$, then $\nuke\cup\{\emptyset\}$, seen as a set of $\Gr$-bisections satisfies condition (3) of Theorem~\ref{th:contracting}. 

Note that the smallest set $C$ satisfying condition (1) is often smaller than the union of the elements of $\nuke$. For example, if $G$ is the binary adding machine acting on $\{0, 1\}^\omega$ (Example~\ref{ex:addingmachine}), then $C$ is equal to the union of $\xo$ with the set of two germs: $[a, 111\ldots]$ and $[a^{-1}, 000\ldots]$.
\end{example}

\begin{proof}[Proof of Theorem~\ref{th:contracting}]
The implications (3)$\Longrightarrow$(2)$\Longrightarrow$(1) are trivial. It remains to prove the converse implications.

Assuming (1), let us prove the following.

\begin{lemma}
There exists a compact open set $N_1\subset\Gr$ satisfying the condition (1) and such that $\varsigma(N_1)\subseteq N_1$.
\end{lemma}

\begin{proof}
Let $C$ be a compact set satisfying condition (1), and let $N_0\subset\Gr$ be an arbitrary compact open neighborhood of $C$. 

For every $g\in N_0$ there exists $n\ge 1$ such that $\varsigma^n(g)\in C$. By continuity of $\varsigma$ and since $N_0$ is a neighborhood of $C$, there exists an open bisection $U$ such that $\varsigma^n(U)\subset N_0$. As $N_0$ is compact, it follows that we can find a finite cover of $N_0$ by open sets $U_i$ such that $\varsigma^{n_i}(U)\subset N_0$ for some $n_i\ge 1$. Let $m$ be the maximal value of $n_i$. Then $\varsigma^m(U_i)\subset\varsigma^{m-n_i}(N_0)$ and $0\le m-n_0\le m-1$ for every $i$. 

It follows that $\varsigma^m(N_0)\subset\bigcup_{k=0}^{m-1}\varsigma^k(N_0)$.
Let $N_1=\bigcup_{k=0}^{m-1}\varsigma^k(N_0)$. It is a compact open subset of $\Gr$, since the map $\varsigma\colon \Gr\arr\Gr$ is open and continuous. We have \[\varsigma(N_1)=\varsigma^m(N_0)\cup\bigcup_{k=0}^{m-2}\varsigma^{k+1}(N_0)
\subset\bigcup_{k=0}^{m-1}\varsigma^k(N_0)=N_1,\]
which finishes the proof of the lemma.
\end{proof}

Let $N_1$ be a set satisfying the conditions of the lemma.
Since $N_1$ is compact and open, we can represent it as a finite union $\bigcup_{i=1}^k F_i$ of compact open $\Gr$-bisections (we do not assume that $\Gr$ is Hausdorff, so we can not make the union disjoint in general). Let $\mathcal{C}=\bigcup_{i=1}^k F_i\times\{i\}$ be the formal disjoint union of the sets $F_i$. It is a compact totally disconnected topological space. Since $\varsigma(N_1)\subset N_1$, for every $g\in F_i$ there exists a compact open neighborhood $V\subset F_i$ of $g$ such that $\varsigma(V)\subset F_j$ for some $j$. It follows that we can find a local homeomorphism $s\colon \mathcal{C}\arr\mathcal{C}$ such that for every $(g, i)\in F_i\times\{i\}$ we have $s(g, i)=(\varsigma(g), j)$ for some $j$. The map $s\colon \mathcal{C}\arr\mathcal{C}$ is expansive and open, since $\be\colon F_i\arr\be(F_i)$ are homeomorphisms and $\be(\varsigma(g))=\si(\be(g))$. 

Therefore, by Proposition~\ref{pr:openisfinitetype}, the map $s\colon \mathcal{C}\arr\mathcal{C}$ is topologically conjugate to a one sided shift of finite type. By Lemma~\ref{lem:eventuallyonto}, there exists $m$ such that $s^{m+1}(\mathcal{C})=s^m(\mathcal{C})$. We have then $\varsigma^{m+1}(N_1)=\varsigma^m(N_1)$. 

Then $N=\varsigma^m(N_1)$ will be an open compact set satisfying $\varsigma(N)=N$. Since $N_1$ satisfied the condition (1) of our theorem, so does $N$. This proves the implication (1)$\Longrightarrow$(2).

Let us prove the implication (2)$\Longrightarrow$(3). 
Let $\mathcal{C}$ be a formal disjoint union $\bigcup_{i=1}^k F_i\times\{i\}$ of $\Gr$-bisections $F_i\subset\Gr$ constructed above. Let $s\colon \mathcal{C}\arr\mathcal{C}$ be the corresponding lift of $\varsigma$ to $\mathcal{C}$. We have shown that after passing to a compact open subset $s^m(\mathcal{C})$, i.e., after shrinking the bisections $F_i$, we may assume that $s$ is onto, and that $s\colon \mathcal{C}\arr\mathcal{C}$ is topologically conjugate to a one-sided shift of finite type. Let $P\colon \mathcal{C}\arr N$ be the surjection $(g, i)\mapsto g$. We have $\varsigma\circ P=P\circ s$.

Since $s\colon \mathcal{C}\arr\mathcal{C}$ is a shift of finite type, we can find a Markov partition $\{A_i\}_{i=1}^l$ of $s\colon \mathcal{C}\arr\mathcal{C}$. We may assume that the elements of the partition are sufficiently small, so that for every $A_i$ there exists $j\in\{1, \ldots, k\}$ such that $A_i\subset F_j\times\{j\}$, and $\be(P(A_i))\subset x_i\xo$ and $\en(P(A_i))\subset y_i\xo$ for some $x_i, y_i\in\alb$.   Then $\varsigma(P(A_i))=S_{y_i}^{-1}P(A_i)S_{x_i}$.

Then for every compact open subset $V\subset\mathcal{C}$ the set $s^n(V)$ is a union of the elements of the Markov partition $\{A_i\}_{i=1}^l$, for all $n$ large enough.

Consequently, for every compact open bisection $F\subset\Gr$ there exists $n_0$ such that for any $u, v\in\alb^n$, for $n\ge n_0$, the set $S_u^{-1}FS_v$ is a (possibly empty) union of sets of the form $P(A_i)$.
Since the number of subsets of the set $\{A_i\}_{i=1}^l$ is finite, there exists a finite set $\nuke$ of $\Gr$-bisections satisfying condition (3).

It remains to prove the uniqueness statements of the theorem.
Let us prove that the set $N$ satisfying condition (2) is unique. Suppose that $N_1, N_2$ are two such sets.

For every $g\in N_2$ there exists $n$ such that $\varsigma^n(g)\in N_1$. By continuity, and since $N_1$ is open, there exists a neighborhood $U$ of $g$ such that $\varsigma^n(U)\subset N_1$. Since $\varsigma(N_1)=N_1$, we get that $\varsigma^m(U)\subset N_1$ for all $m\ge n$.

Since $N_2$ is compact, there exists a finite cover of $N_2$ by open susets $U$ such that $\varsigma^n(U)\subset N_1$ for all $n$ big enough. It follows that there exists $n$ such that $\varsigma^n(N_2)\subset N_1$. But $\varsigma^n(N_2)=N_2$, hence $N_2\subset N_1$. By symmetry, $N_1=N_2$.

Let us show that there exists a unique smallest set $\nuke$ satisfying condition (3). If $\nuke_0$ satisfies condition (3), then for every $n$ the set $\nuke_n=\{S_v^{-1}FS_u : F\in\nuke_0, v, u\in\alb^n\}$ also satisfies condition (3), and we have $\nuke_{n+1}\subseteq\nuke_n$ for every $n$. Since $\nuke_0$ is finite, the sequence $\nuke_n$ is eventually constant.  Let $\nuke_\infty$ be the limit of the sequence $\nuke_n$.  Then $\nuke_\infty$ satisfies condition (3).

If $\nuke'$ is another set satisfying condition (3), then there exists $n$ such that the set $\nuke_n'$, defined for $\nuke'$ in the same way as $\nuke_n$ was defined for $\nuke$, will be also a non-decreasing sequence of subsets of $\nuke'$. By condition (3) for $\nuke$, we will have $\nuke'_{n_1}\subset\nuke$ for some $n_1$, hence $\nuke'_n\subset\nuke_{n-n_1}$ for all $n\ge n_1$. Consequently, $\nuke'_n\subset\nuke_\infty$ for all $n$ large enough, hence $\nuke'_\infty=\nuke_\infty$, and $\nuke_\infty$ is the unique smallest set satisfying condition (3).
\end{proof}

The nucleus depends on the choice of the encoding (i.e., of the Markov partition) of the shift $\sigma\colon \F\arr\F$. 

If we apply a width $k$ block code to $\F$, then the new bisections $\wt{S}_{(x_1)_{x_2\ldots x_k}}$ are expressed in terms of the old bisections $S_x$ by the formula
\[\wt{S}_{(x_1)_{x_2\ldots x_k}}=S_{x_1}I_{x_2\ldots x_k}.\]
Their non-zero products are
\[\wt{S}_{(x_1)_{x_2\ldots x_k}}\wt{S}_{(x_2)_{x_3\ldots x_{k+1}}}\cdots\wt{S}_{(x_m)_{x_{m+2}\ldots x_{m-1+k}}}=S_{x_1x_2\ldots x_m}I_{x_{m+1}\ldots x_{m+k-1}},\]
where $I_v=S_vS_v^{-1}$ is the range of $S_v$.

It follows that the new nucleus is the set of all bisections of the form $I_uFI_v$ for the elements $F$ of the old nucleus and finite words $u, v\in\alb^{k-1}$.

As a corollary we get the following characterization of Hausdorff contracting groupoids.

\begin{proposition}
\label{prop:Hausdorffdisjoint}
A contracting self-similar groupoid $\Gr$ is Hausdorff if and only if for some Markov partition of the shift $\sigma\colon \Gr^{(0)}\arr\Gr^{(0)}$ the associated nucleus consists of pairwise disjoint bisections.
\end{proposition}

\subsection{Hyperbolic homeomorphisms}

We describe here a connection between a class of dynamical systems and contracting self-similar inverse semigroups, following~\cite{nek:smale}.

\subsubsection{Contracting groupoids from solenoids} 
A homeomorphism $f\colon \X\arr\X$ of a compact space is called \emph{expansive} if the action of $\Z$ generated by it is expansive, i.e., if there exists $\delta>0$ such that $d(f^n(x), f^n(y))<\delta$ for all $n\in\Z$ implies that $x=y$ (compare with \ref{def:expansive}).

We say that $x, y\in\X$ are \emph{stably equivalent} if $\lim_{n\to+\infty}d(f^n(x), f^n(y))=0$. We say that they are \emph{unstably equivalent} if $\lim_{n\to-\infty}d(f^n(x), f^n(y))=0$. Since these two relations are equivalence, they can be seen as groupoids. Let us introduce topologies on them.

Denote by $W_{s, \epsilon}$ the set of pairs $(x, y)\in\X^2$ such that $d(f^n(x), f^n(y))<\epsilon$ for all $n\ge 0$. Similarly, define $W_{u, \epsilon}=\{(x, y)\in\X^2 : d(f^n(x), f^n(y))<\epsilon, \forall n\le 0\}$. 

If $\epsilon$ is small enough, then the stable and unstable equivalence relations are equal to $\bigcup_{n<0}f^n(W_{s, \epsilon})$ and $\bigcup_{n>0}f^n(W_{u, \epsilon})$, respectively (where $f$ acts on $\X^2$ diagonally). Both unions are increasing. Let us introduce on the stable and unstable equivalence groupoids the topology of the inductive limit of the relative topologies on $f^n(W_{s, \epsilon})\subset\X^2$ and $f^n(W_{u, \epsilon})\subset\X^2$. 

It is also natural to consider the groupoids generated by the stable (or unstable) equivalence groupoid and the action of the homeomorphism $f$, which is defined in the following way.

\begin{definition}
Let $f\colon \X\arr\X$ be an expansive homeomorphism. The \emph{stable (resp.\ unstable) Ruelle groupoid} is the direct product $\Gr\times\Z$, where $\Gr$ is the stable (resp.\ unstable) equivalence groupoid of $f$, with the source and range maps 
\[\be((x, y), n)=y,\qquad\en((x, y), n)=f^n(x),\]
and multiplication
\[((x_1, y_1), n_1)((x_2, y_2), n_2)=((f^{-n_2}(x_1), y_2), n_1+n_2),\]
where the product is defined if $y_1=f^{n_2}(x_2)$.
\end{definition}

The map $\nu((x, y), n)=n$ is a continuous cocycle (i.e., a functor) from the Ruelle groupoids to $\Z$.

We will denote $W_{s, \epsilon}(x)=\{y\in\X : (x, y)\in W_{s, \epsilon}\}$ and $W_{u, \epsilon}(x)=\{y\in\X : (x, y)\in W_{u, \epsilon}\}$.

Suppose that $z_1, z_2\in W_{s, \epsilon}(x)\cap W_{u, \epsilon}(y)$ for some $z_1, z_2, x, y\in\X$. Then $d(f^n(z_1), f_n(x))<\epsilon$ and $d(f^n(z_2), f_n(x))<\epsilon$ for all $n\ge 0$. Consequently, $d(f^n(z_1), f^n(z_2))<2\epsilon$ for all $n\ge 0$. Similarly, $d(f^n(z_1), f^n(z_2))<2\epsilon$ for all $n\le 0$. Consequently, the intersection $W_{s, \epsilon}(x)\cap W_{u, \epsilon}(y)$ has at most one point for all $x, y\in\X$, if $\epsilon$ is small enough.

\begin{definition}
\label{def:hyperbolichomeo}
We say that a homeomorphism $f\colon \X\arr\X$ of a compact space $\X$ is \emph{hyperbolic} (is a \emph{Ruelle-Smale system}) if it is expansive and for every $\delta>0$ there exists $\epsilon>0$ such that the set $W_{s, \delta}(x)\cap W_{s, \delta}(y)$ is non-empty for all $x, y\in\X$ such that $d(x, y)<\epsilon$.
\end{definition}

\begin{example}
A homeomorphism of a compact totally disconnected space is expansive if and only if it is topologically conjugate to a subshift, i.e., to the action of the shift $\sigma\colon \alb^{\Z}\arr\alb^{\Z}$ on a closed shift-invariant set  $\F\subset\alb^{\Z}$. The sets $W_{s, \epsilon}$ are sets of pairs of sequences $(x_n)_{n\in\Z}$ and $(y_n)_{n\in\Z}$ such that $x_n=y_n$ for all $n\ge -n_0$, where $n_0>0$ depends on $\epsilon$ (the smaller is $\epsilon$ the larger is $n_0$). Similarly, $W_{u, \epsilon}$ is the set of pairs $(x_n)_{n\in \Z}, (y_n)_{n\in\Z}$ such that $x_n=y_n$ for all $n\le n_0$. Consequently, a subshift is a Ruelle-Smale system if and only if for there exists $n_0$ such that if $(x_n)_{n\in\Z}$ and $(y_n)_{n\in\Z}$ are such that $x_n=y_n$ for all $-n_0\le n\le n_0$, then the sequence $(z_n)_{n\in\Z}$ defined by $z_n=x_n$ for $n\le 0$ and $z_n=y_n$ for $n\ge 0$ belongs to $\F$. One can show that this is equivalent to the condition that $\F$ is a subshift of finite type. Thus, a homeomorphism of a compact totally disconnected space is hyperbolic if and only if it is topologically conjugate to a two-sided subshift of finite type.
\end{example}

\begin{example}
Let $A$ be an $n\times n$ matrix with integer entries and such that $\det A=\pm 1$. Then it induces a homeomorphism of the torus $\R^n/\Z^n$. Suppose that $A$ does not have eigenvalues $\lambda$ on the unit circle. Then $\R^n$ decomposes into a direct sum $V_1\oplus V_2$ such that $A$ is contracting on $V_1$ and expanding on $V_2$. In other words, the homeomorphism of the torus is \emph{Anosov}. If $\epsilon>0$ is small enough, then $W_{s, \epsilon}(x)$ is equal to the image in the torus of the intersection of $V_1+x$ with a neighborhood of $x$. Similarly, $W_{u, \epsilon}(x)$ is the intersection of $V_2+x$ with a neighborhood of $x$. Consequently, if $x$ and $y$ are close enough, then $W_{s, \epsilon}(x)\cap W_{u, \epsilon}(y)$ contains the image of the point $V_1+x\cap V_2+y\in\R^n$ in the torus. Hence, every Anosov automorphism of a torus is a Ruelle-Smale system. In fact, it is well known that any Anosov diffeomorphism is a Ruelle-Smale system.
\end{example}

\begin{example}
If $f\colon\mathcal{M}\arr\mathcal{M}$ is an expanding covering map (i.e., satisfying the conditions of Theorem~\ref{th:contractingimg}), then the homeomorhism induced by $f$ on the inverse limit of the sequence
\[\mathcal{M}\stackrel{f}{\longleftarrow}\mathcal{M}\stackrel{f}{\longleftarrow}\mathcal{M}\stackrel{f}{\longleftarrow}\cdots\]
is a Ruelle-Smale system, see~\cite[Theorem~1.4.35]{nek:dyngroups}.
\end{example}

\begin{definition}
\label{def:rectangle}
Let $f\colon \X\arr\X$ be a Ruelle-Smale system. Let $\delta, \epsilon>0$ satisfy the conditions of Definition~\ref{def:hyperbolichomeo}. An open subset $R\subset\X$ together with a homeomorphism $\phi\colon A\times B\arr R$ with a direct product of topological spaces is called a \emph{sub-rectangle} of $\X$ (or just a \emph{rectangle}) if the following conditions hold:
\begin{enumerate}
\item for every $x\in A$ the set $\phi(\{x\}\times B)$ is contained in one stable equivalence class;
\item for every $y\in B$ the set $\phi(A\times\{y\})$ is contained in one unstable equivalence class;
\item for every $\phi(x, y)\in R$ there exists a neighborhood $W$ of $\phi(x, y)$ such that $W_{s, \epsilon}(\phi(x, y))\cap W=\phi(\{x\}\times B)\cap W$
and $W_{u, \epsilon}(\phi(x, y))\cap W=\phi(A\times\{y\})\cap W$.
\end{enumerate}

If $R$ and $\phi\colon A\times B\arr R$ is a sub-rectangle of $\X$, then for every point $x=\phi(a, b)\in R$, we denote $W_{s, R}(x)=\phi(\{a\}\times B)$ and $W_{u, R}(x)=\phi(A\times\{b\})$, and call them the \emph{stable} and \emph{unstable plaques} of $x$ in $R$, respectively.
\end{definition} 

It follows from the definition of a Ruelle-Smale system that every point has a rectangular neighborhood (i.e., an open neighborhood together with a direct product decomposition satisfying the conditions Definition~\ref{def:rectangle}).

Instead of specifying the direct product decomposition $\phi\colon A\times B\arr R$ of a rectangle, it is convenient to consider the map $(x, y)\mapsto [x, y]$ on the rectangle $R$ given by
\[[\phi(a_1, b_1), \phi(a_2, b_2)]=\phi(a_1, b_2).\]
It satisfies the properties $[x, x]=x$, $[[x, y], z]=[x, z]$, and $[x, [y, z]]=[x, z]$.
Conversely, given any continuous map $R\times R\arr R$ satisfying these three properties, we can take $A=[R, x]$, $B=[x, R]$ for some fixed $x\in R$, and define the direct product decomposition $[R, x]\times [x, R]\arr R$ by the map $(y_1, y_2)\mapsto [y_1, y_2]$. 
We have then $W_{s, R}(x)=[x, R]$ and $W_{u, R}(x)=[R, x]$.

The \emph{canonical homeomorphism} between plaques $W_{s, R}(x)=[x, R]$ and $W_{s, R}(y)=[y, R]$ comes from the identification of the plaques with the corresponding factor of the direct product decomposition of $R$, and is equal to the map $[y, \cdot]\colon W_{s, R}(x)\arr W_{s, R}(y)$. The canonical homeomorphism of the unstable plaques is equal to $[\cdot, y]\colon W_{u, R}(x)\arr W_{u, R}(y)$.

Therefore, a \emph{local product decomposition} of $\X$ can be defined as a continuous map $[\cdot, \cdot]$ defined on a neighborhood of the diagonal and satisfying the conditions:
\[[x, x]=x,\quad [x, [y, z]]=[[x, y], z]=[x, z],\]
for all $x, y, z$ for which the corresponding values are defined. Two functions $[\cdot, \cdot]_1$ and $[\cdot, \cdot]_2$ define the same local product decomposition if they agree on a neighborhood of the diagonal.

Thus, for every Ruelle-Smale system $f\colon \X\arr\X$, we have a canonical local product decomposition given by the condition $\{[x, y]\}=W_{s, \epsilon}(x)\cap W_{u, \epsilon}(y)$ for a small positive $\epsilon$.

The homeomorphism $f$ preserves the local product structure, i.e., maps sufficiently small sub-rectangles to sub-rectangles,  and maps stable (resp.\ unstable) plaques to stable (resp.\ unstable) plaques.

\begin{proposition}
Let $f\colon \X\arr\X$ be a Ruelle-Smale system. Choose a cover $\{R_i\}_{i\in I}$ of $\X$ by open rectangles. Let $\mathcal{S}=\bigsqcup_{i\in I} W_{s, R_i}(x_i)$ be a formal disjoint union of the stable plaques of the rectangles of the cover (where $x_i\in R_i$ are arbitrary). Then the restriction of the unstable equivalence groupoid to $\mathcal{S}$ is a principal \'etale groupoid. The restriction of the unstable Ruelle groupoid to $\mathcal{S}$ is an effective \'etale groupoid.
\end{proposition}

\begin{proof}
If $R$ with a direct product decomposition $\phi\colon A\times B\arr R$ is a subrectangle of $\X$, then $f(R)$ with the direct product decomposition $f\circ\phi$ is also a subrectangle.

It follows two points $x_1, x_2\in\X$  are stably equivalent if and only if there exists a subrectangle $R$ such that $x_1$ and $x_2$ belong to the same stable plaque of $R$. Moreover, by the definition of the topology on the stable equivalence groupoid, a neighborhood of the pair $(x_1, x_2)$ is the set 
of all pairs of points of $R$ contained in one stable plaque of $R$. The set of such neighborhoods form a basis of the stable equivalence groupoid.

Let $\{R_i\}_{i\in I}$ be a cover of $\X$ by open rectangles. Choose a point $x_i\in R_i$, and consider the formal disjoint union $\bigsqcup_{i\in I} W_{s, R_i}(x_i)$ of their stable plaques. Let us restrict the unstable equivalence groupoid to it. The restriction is a groupoid equvialaent to the unstable equivalence groupoid. Let $y_1\in W_{s, R_{i_1}}(x_{i_1})$ and $y_2\in W_{s, R_{i_2}}(x_{i_2})$ be unstably equivalent points. Then a neighborhood of $(y_1, y_2)$ in the unstable equivalence groupoid is defined, as described in the previous paragraph by a rectangle $R$ such that $y_1, y_2$ belong to one unstable plaque $W_{u, R}(y_1)=W_{u, R}(y_2)$. The intersection of $R$ with $R_{i_1}$ contains a neighborhood $R_{y_1}$ of $y_1$ that is a sub-rectangle of $R$ and $R_{i_1}$ (i.e., $R_{y_1}\subset R\cap R_{i_1}$ and the direct product decomposition of $R_{y_1}$ agrees with both direct product decompositions  of $R$ and $R_{i_1}$). Similarly, there is a neighborhood $R_{y_2}$ of $y_2$ that is a sub-rectangle of both $R$ and $R_{i_2}$. Moreover, by shrinking $R$ if necessary, we may assume that the stable plaques of $y_1$ and $y_2$ in $R$ coincide with their stable plaques in $R_{i_1}$ and $R_{i_2}$, respectively. Then a neibhborhood of $(y_1, y_2)$ in the unstable equivalence groupoid is the set of pairs of points $(z_1, z_2)$, where $z_i\in R_{y_i}$, and $W_{u, R}(z_1)=W_{u, R}(z_2)$. The intersection of this neighborhood with the restriction of the unstable equivalence groupoid to $W_{s, R_{i_1}}(x_1)\sqcup W_{s, R_{i_2}}(x_2)$ is the set of pairs $(z_1, z_2)\in W_{s, R}(y_1)\times W_{s, R}(y_2)$ such that $W_{u, R}(z_1)=W_{u, R}(z_2)$. This set is a bisection, since the relation $z_1\mapsto z_2$ is the canonical homeomorphism between the stable plaques of $R$. 

The proof of the statement about the Ruelle groupoid is similar.
\end{proof}

\begin{definition}
A Ruelle-Smale system $f\colon \X\arr\X$ is called a \emph{solenoid} if the unstable plaques of sub-rectangles are totally disconnected.
\end{definition}

Suppose that $f\colon \X\arr\X$ is a solenoid. Let $\{R_i\}_{i\in I}$ be a finite cover of $\X$ by open sub-rectangles such that the unstable plaques of each $R_i$ are compact and totally disconnected. Let $\mathfrak{S}$ be the restriction of the stable equivalence relation groupoid to the disjoint union $\mathcal{U}=\bigsqcup_{i\in I}W_{u, R_i}(x_i)$ of the unstable plaques of $R_i$. Let $\mathfrak{R}_s$ be the restriction of the stable Ruelle groupoid to $\mathcal{U}$. Let $\nu\colon \mathfrak{R}_s\arr\Z$ be the associated cocycle. We have then $\mathfrak{S}=\nu^{-1}(0)$.

Since $\mathcal{U}$ is compact and totally disconnected, we can find a finite collection of $\mathfrak{R}_s$-bisections $F$ whose domains $\be(F)$ form a partition of $\mathcal{U}$ and such that $\nu(F)=\{1\}$. Let $\sigma\colon \mathcal{U}\arr\mathcal{U}$ be their union (seen as maps $\be(g)\mapsto\en(g)$ for $g\in F$).

The following characterization of the stable, unstable, and Ruelle groupoids on solenoids was proved (using a slightly different language of Markov partitions) in~\cite{nek:smale}. We give here the main ideas of the proof.

\begin{proposition}
The map $\sigma\colon \mathcal{U}\arr\mathcal{U}$ is conjugate to a one-sided subshift of finite type. The groupoid $\mathfrak{S}$ is $\sigma$-invariant and $\sigma$-contracting. The groupoid $\mathfrak{R}_s$ coincides (as the groupoid of germs of local homeomorphisms of the space $\mathcal{U}$) to the groupoid $\langle\mathfrak{S}, \sigma\rangle$ generated by $\mathfrak{S}$ and the germs of $\sigma$.
\end{proposition}

\begin{proof}
The map $\sigma$ is local homeomorphism, hence it is an open map. It is expansive, since $f$ expands the unstable plaques of rectangles. Consequently, $\sigma$ is conjugate to a subshift of finite type by Proposition~\ref{pr:expansiveopensft}. 
The groupoid $\mathfrak{S}$ is $\sigma$-invariant, since $f$ preserves the stable equivalence classes and maps rectangles to rectangles.
The fact that $\mathfrak{S}$ is contracting follows from the fact that $f$ contracts the stable plaques of the rectangles, so for every germ $g\in\mathfrak{S}$ there exists $n$ such that $\varsigma^n(g)$ is the germ of the homeomorphism between the unstable plaques of two rectangles induced by their overlap. The set of such elements of $\mathfrak{S}$ is compact, since we assume that the set of rectangles is finite.
\end{proof}

\subsubsection{Solenoids from contracting groupoids}
One can go in the opposite direction, and associate a hyperbolic dynamical system, called \emph{limit solenoid}, with any contracting principal groupoid. In the non-principal case the dynamical system will belong to a more general class.

Let $\F\subset\xo$ be a topological Markov chain defined as the edge-shift of a finite digraph $\G$. Let $\Gr$ be a contracting shift-invariant contracting groupoid with the space of units $\Gr^{(0)}=\F$.

Let $\F_\omega=\{\ldots x_{-1}x_0x_1\ldots\in\alb^{\Z} : x_ix_{i+1}\in T\;\forall i\}$ be the \emph{two-sided} Markov shift defined by the same set of transitions as $\F$, i.e., the subshift of all bi-infinite directed paths in $\G$. It is naturally homeomorphic to the \emph{natural extension} of $\si\colon \F\arr\F$, i.e., to the inverse limit of the sequence \[\F\stackrel{\si}{\longleftarrow}\F\stackrel{\si}{\longleftarrow}\F\cdots.\]

\begin{definition}
\label{def:equivalence}
We say that $\ldots x_{-1}x_0x_1\ldots, \ldots y_{-1}y_0y_1\ldots\in\F_\omega$ are $\Gr$-equivalent if there exists a compact set $C\subset\Gr$ and a sequence $g_n\in C$ such that $\be(g_n)=x_nx_{n+1}\ldots$ and $\en(g_n)=y_ny_{n+1}\ldots$ for all $n\in\Z$.
\end{definition}

The quotient of $\F_\omega$ by this equivalence relation is called the \emph{limit solenoid} of $\Gr$. Note that the equivalence relation is invariant under the shift of $\F_\omega$. It follows that the shift induces a homeomorphism of the limit solenoid. We call the obtained homeomorphism the \emph{limit dynamical dynamical system} of $\Gr$.

If $(g_n)$ is a sequence satisfying the condition of Definition~\ref{def:equivalence} for a pair of sequences $(x_n)_{n\in\Z}$ and $(y_n)_{n\in\Z}$. Then, for every $k\ge 0$, the sequence $h_n=\varsigma^k(g_{n-k})$ also satisfies the condition for the same pair of elements of $\F_\omega$. 
Consequently, if the elements are equivalent, then the sequence $g_n$ implementing the equivalence can be chosen to be contained in the union of the elements of the nucleus of $\Gr$.

This implies that two sequences $(x_n)_{n\in\Z}$ and $(y_n)_{n\in\Z}$ are equivalent if and only if there exists a bi-infinite directed path $(e_n)_{n\in\Z}$ in the Moore diagram of the nucleus such that $e_n$ is labeled by $x_n|y_n$.

Let $\Gr$ be a contracting completely shift-invariant groupoid. Let $w\in\Gr^{(0)}$.
It follows from the definition that the set of sequences $(x_n)_{n\in\Z}$ such that $x_0x_1\ldots$ belongs to the $\Gr$-orbit of $w$, is invariant under the equivalence relation defining the solenoid (i.e., is a union of equivalence classes). We call the image of this subset  of the limit solenoid the \emph{leaf} defined by $w$ and denote it $\mathcal{L}_w$.

The \emph{intrinsic topology} on the leaf is the inductive limit topology for the decomposition of $\mathcal{L}_w$ into the images $T_u$ of the sets of sequences $(x_n)_{n\in\Z}$ such that $x_0x_1\ldots=u$, over all elements $u$ of the $\Gr$-orbit of $w$. In other words, we declare a subset of $\mathcal{L}_w$ open if and only if its intersection with every set $T_u$ is relatively open in $T_u$. 

\subsubsection{Limit solenoid of a principal contracting groupoid}

\begin{lemma}
\label{lem:principalnucleus}
Suppose that $\Gr$ is principal. Then there exists $m$ such that any element $g\in\Gr$ contained in an element of the nucleus of $\Gr$ is uniquely determined by the pair of prefixes of length $m$ of $\be(g), \en(g)$, respectively.

In other words, all labels of the arrows in the Moore diagram of the nucleus of $\Gr$ become pairwise different, after we pass to the width $m$ sliding block coding of $\Gr^{(0)}$.
\end{lemma}

\begin{proof}
Suppose that it is not true. Then for every $m$ there exists two different elements $g_1, g_2$ of the nucleus such that the beginnings of length $m$ of $\be(g_1)$ and $\be(g_2)$ coincide, and the beginnings of length $m$ of $\en(g_1)$ and $\en(g_2)$ coincide. Then $h_m=g_1^{-1}g_2$ has the property that it is isotropic, not a unit, and the sequences $\be(h_m)$ and $\en(h_m)$ have a common prefix of length $m$. Using compactness, we can pass to the limit of a subsequence. Since in the Hausdorff case the set of  units is open, the limit of a subsequence of $h_m$ is not a unit, but this contradicts the condition that $\Gr$ is principal.
\end{proof}

\begin{proposition}
If $\Gr$ is a principal completely shift-invariant contracting groupoid, then its limit dynamical system $f\colon\X\arr\X$ is hyperbolic (i.e., is a Ruelle-Smale system).

The groupoid $\Gr$ is equivalent to the stable equivalence groupoid of $f$.
\end{proposition}

\begin{proof} The proposition was proved (using slightly different, but equivalent, definition of the stable equivalence groupoid) in~\cite{nek:smale}. It also follows from the duality theorem for hyperbolic groupoids, see~\cite{nek:hyperbolic}. 

We describe here explicitly the construction of the local direct product structure on the limit solenoid and of the isomorphism between $\Gr$ and a restriction of the stable equivalence groupoid to a transversal inducing the equivalence of groupoids.

Let $\F\subset\xo$ be the shift of finite type on which $\Gr$ acts, and let $\F_\omega$ be the corresponding two-sided shift of finite type. 
Let $m$ be as in Lemma~\ref{lem:principalnucleus}. Suppose that points $\xi_1, \xi_2$ of the limit solenoid can be represented by sequences $\ldots x_{-2}x_{-1}x_0x_1x_2\ldots$ and $\ldots y_{-2}y_{-1}y_0y_1y_2\ldots$ such that $x_1x_2\ldots x_m=y_1y_2\ldots y_m$. Define then 
$[\xi_1, \xi_2]$ as the point represented by $\ldots x_{-2}x_{-1}x_0y_1y_2\ldots$. Let us show that it is well defined.

Suppose that we can represent the points $\xi_1, \xi_2$ by another pair of sequences $\ldots x_{-2}'x_{-1}'x_0'x_1'x_2'\ldots$ and $\ldots y_{-2}'y_{-1}'y_0'y_1'y_2'\ldots$ such that $x_1'x_2'\ldots x_m'=y_1'y_2'\ldots y_m'$. Consider the sequence $g_n$ implementing the equivalence from $\ldots x_{-2}'x_{-1}'x_0'x_1'x_2'\ldots$ to $\ldots x_{-2}x_{-1}x_0x_1x_2\ldots$, and let $h_n\in\Gr$ be the sequence implementing the equivalence from $\ldots y_{-2}'y_{-1}'y_0'y_1'y_2'\ldots$ to $\ldots y_{-2}y_{-1}y_0y_1y_2\ldots$. By the condition on $m$, we have $g_n=h_n$ for all $n\ge 0$. Therefore, we can use the sequence $h_n$ to show that the sequences $\ldots x_{-2}x_{-1}x_0y_1y_2\ldots$ and $\ldots x_{-2}'x_{-1}'x_0'y_1'y_2'\ldots$  are equivalent. We get a local product structure on the limit solenoid. 

The fact that the shift induces a Ruelle-Smale system on the limit solenoid is proved using the same arguments (using a symbolic log-scale) as \cite[Theorem~4.5.31]{nek:dyngroups} for contracting self-similar groups and their limit dynamical systems. 

Let us show that $\Gr$ is equivalent to the stable equivalence groupoid.
For every word $v$ of length $m$ that appears as a subword of an element of $\F_\omega$, consider the set $R_v$ of the points of the limit solenoid that can be represented as sequences $\ldots x_{-2}x_{-1}x_0x_1x_2\ldots$ such that $x_1x_2\ldots x_m=v$. By the above, it is a rectangle with the unstable plaque naturally homeomorphic to ${}_v\F$. Consider the restriction of the stable equivalence groupoid to the disjoint union of the unstable plaques of the rectangles $R_v$. Let $g$ be an element of the stable equivalence groupoid such that $\be(g)\in R_{v_1}$ and $\en(g)\in R_{v_2}$. Then $\be(g)$ and $\en(g)$ are represented by sequences $(x_n)_{n\in\Z}$ and $(y_n)_{n\in\Z}$ such that $x_1x_2\ldots x_m=v_1$ and $y_1y_2\ldots y_m=v_2$. Since $\be(g)$ and $\en(g)$ are stably equivalent, there exists $k$ such that $f^k(\be(g))$ and $f^k(\en(g))$ belong to the same stable plaque of a rectangle $R_v$ for some $v\in\alb^m$, i.e., can be represented by sequences $(a_n)_{n\in\Z}$ and $(b_n)_{n\in\Z}$ such that $a_1a_2\ldots =b_1b_2\ldots$ (and $a_1a_2\ldots a_m=b_1b_2\ldots b_m=v$). Consequently, there exist elements $g_1, g_2\in\Gr$ such that $\be(g_1)=x_{k+1}x_{k+2}\ldots$, $\be(g_2)=y_{k+1}y_{k+2}\ldots$, and $\en(g_1)=\en(g_2)=a_1a_2\ldots$. We get the element $h=S_{y_1y_2\ldots y_k}g_2^{-1}g_1S_{x_1x_2\ldots x_k}^{-1}$ of $\Gr$ such that $\be(h)=y_1y_2\ldots$ and $\en(h)=x_1x_2\ldots$. One can show that the map $g\mapsto h$ is a well defined isomorphism between the restriction of the stable equivalence groupoid and the groupoid $\Gr$.
\end{proof}

\subsubsection{Limit solenoids in the Hausdorff case}

In the case when the groupoid $\Gr$ is not principal, but still Hausdorff, we can show that the limit space is naturally an \emph{orbispace} Ruelle-Smale dynamical system.

Denote by $\F^-$ the space of left-infinite directed paths $\ldots x_2x_1$ in $\G$, i.e., infinite sequences such that $\be(x_n)=\en(x_{n+1})$. 

Consider the subspace of the direct product $\xmo\times\Gr$ consisting of sequences $\ldots x_2x_1\cdot g$ such that $\ldots x_2x_1\en(g)\in\F_\omega$. An equivalent condition is $\en(g)\in \be(S_{x_1})$. 

Let us declare two sequences $\ldots x_2x_1\cdot g$ and $\ldots y_2y_1\cdot h$ equivalent if the sequence
\[(S_{x_n\ldots x_1}g)(S_{y_n\ldots y_1}h)^{-1}\]
belongs to a compact subset of $\Gr$. Same as in the case of the equivalence on $\F_\omega$, we may assume that $(S_{x_n\ldots x_1}g)(S_{y_n\ldots y_1}h)^{-1}$ belong to the elements of the nucleus of $\Gr$.
Let $\mathcal{X}_{\Gr}$ be the space of equivalence classes for the defined equivalence relation. Informally, it is the space of infinite products $\ldots S_{x_2}S_{x_1}g$.

The groupoid $\Gr$ acts naturally on the space $\mathcal{X}_{\Gr}$ from the right 
\[\ldots x_2x_1\cdot g\cdot h=\ldots x_2x_1\cdot gh\]
over the anchor
\[\ldots x_2x_1\cdot g\mapsto \be(g)\colon \mathcal{X}_{\Gr}\arr\F.\]
Informally, it is the natural action of $\Gr$ on the space of infinite products $\ldots S_{x_2}S_{x_1}\cdot g$ by the right multiplication.

The action groupoid consists of pairs $(\ldots x_2x_1\cdot g, h)$ such that $\en(h)=\be(g)$. Multiplication is defined by
\[(\ldots x_2x_1\cdot g, h_1)(\ldots x_2x_1\cdot gh_1, h_2)=
(\ldots x_2x_1\cdot g, h_1h_2)\]
with
\[\be(\ldots x_2x_1\cdot g, h)=\be(h),\qquad \en(\ldots x_2x_1\cdot g, h)=\be(g).\]

One can show that this action and hence the action groupoid are proper, so they define an orbispace. Two sequences $\ldots x_2x_1\cdot g$ and $\ldots y_2y_1\cdot h$ represent points of the same orbit of the action if and only if the sequences $\ldots x_2x_1\en(g)$ and $\ldots y_2y_1\en(h)$ represent the same points of the limit solenoid. We get therefore an orbispace structure on the limit solenoid defined by the action groupoid.

Let us define a local direct product structure on $\X_{\Gr}$. By Proposition~\ref{prop:Hausdorffdisjoint}, we may assume that the elements of the nucleus $\nuke$ of $\Gr$ are  pairwise disjoint.

 Let $F$ be a compact open $\Gr$-bisection such that $\be(F)$ and $\en(F)$ are contained in sets of paths starting in vertices $v_1$ and $v_2$ of $\G$, respectively. Consider the subset of $\X_{\Gr}$ consisting of points that can be represented as sequences $\ldots x_2x_1\cdot g$ for $g\in F$. For any two such points, define $[\ldots x_2x_1\cdot g_1, \ldots y_2y_1\cdot g_2]=\ldots x_2x_1\cdot g_2$. Let us show that this is well defined. Suppose that $\ldots x_2'x_1'\cdot g_1'$ and $\ldots y_2'y_1'\cdot g_2'$ are sequences representing the same points, where $g_i'\in F$. Then there exist sequences $F_n$ and $G_n$ of elements of the nucleus such that $S_{x_n'}^{-1}F_nS_{x_n}=F_{n-1}$, $F_0g_1=g_1'$, and $S_{y_n'}G_nS_{y_n}=G_{n-1}$, $G_0g_2=g_2'$. Since $\be(g_1)=g_1'$ and $g_1, g_1'\in F$, we must have $g_1=g_1'$, hence the germ of $F_0$ at $\en(g_1)$ is a unit, which implies that $F_0$ is an idempotent. But this implies that $\ldots x_2x_1\cdot g_2$ and $\ldots x_2'x_1'\cdot g_2'$ are equivalent.

We get a local product structure on $\X_{\Gr}$. It is obviously preserved under the action of $\Gr$. The shift also preserves it, so the limit solenoid can be considered as a Ruelle-Smale orbispace dynamical system. 

The groupoids generated by $\Gr$ and the shift is an example of a \emph{hyperbolic groupoid}, see~\cite{nek:hyperbolic}. As one of the consequences of the general theory of hyperbolic groupoids, we have the following description of the leaves of the limit solenoid.

\begin{theorem}
Let $\Gr$ be a contracting shift-invariant groupoid acting on a shift of finite type $\si\colon \F\arr\F$. Let $\Hr$ be the groupoid generated by $\Gr$ and germs of the shift $\si$. Then the groupoid $\Hr$ is compactly generated. For every compact generating set $S$ of $\Hr$ the Cayley graph $\G_x(\Hr, S)$ is Gromov hyperbolic. The Gromov boundary of $\G_w(\Hr, S)$ is naturally homeomorphic to the one-point compactification of the leaf $\mathcal{L}_w$, where the additional point is the limit of the sequence of germs $[\sigma^n, w]$. 
\end{theorem}

\section{Examples of contracting inverse semigroups}

\subsection{Golden mean rotation}
\label{s:rotation}

Many aspects of contracting self-similar inverse semigroups are well illustrated by one of the classical examples of minimal $\Z$-actions: the rotation $x\mapsto x+\frac{1+\sqrt{5}}2$ of the circle $\R/\Z$.

\subsubsection{The rotation}

Let $\varphi=\frac{1+\sqrt{5}}2$ be the golden mean, i.e., the positive root of the equation $x^2-x-1=0$. Consider the rotation $R\colon x\mapsto x+\varphi$ of the circle $\R/\Z$ by $\varphi$ (it is equal to the rotation $x\mapsto x+\varphi^{-1}$, since $\varphi^{-1}=\varphi-1$). Consider the orbit $\{R^n(0) : n\ge\Z\}$ of $0$. It is equal to the set  $\Z[\varphi]/\Z=(\varphi\Z+\Z)/\Z$. 

Let us
double the point $0\in\R/\Z$, hence replace $\R/\Z$ by the interval $[0, 1]$, and then replace each point $x=n\varphi+m\in(0, 1)$ by two points $x+0$ and $x-0$ with the natural order. Let $\F$ be the obtained space with the order topology.
We will denote intervals $[x+0, y-0]$, for $x, y\in\Z[\varphi]$, just by $[x, y]$. They are clopen in $\F$, and the set of such intervals forms a basis of topology on $\F$. It follows that $\F$ is a compact metrizable totally disconnected space without isolated points, hence it is homeomorphic to the Cantor set.

The rotation $R$ naturally lifts to a homeomorphism of $\F$, which we will also denote by $R$.

\begin{proposition}
\label{pr:rotationexpansive}
The action of $\Z$ generated by $R$ on $\F$ is expansive. Namely, for any two different points $x, y\in\F$ there exists $n\in\Z$ such that $R^n(x)$ and $R^n(y)$ belong to different elements of the clopen partition $[0, \varphi^{-1}]\cup [\varphi^{-1}, 1]$ of $\F$.
\end{proposition}

Note that the rotation of the circle is not expansive, since it is an isometry.

\begin{proof}
Recall that $\varphi^{-1}=\varphi\pmod{1}$.
Since the orbit of $0$ is dense in $\R/\Z$ (by Kroneker theorem~\cite{kronecker:approximation}), there are points of the orbit $R^n(0)$ in both arcs of $\R/\Z$ with the endpoints $x$ and $y$. Consequently, there exists $n$ such that $R^n(0)$ and $R^{n+1}(0)$ are in different arcs. Then $0$ and $R(0)=\varphi$ are in different arc with the endpoints $R^n(x)$ and $R^n(y)$, i.e., $R^n(x)$ and $R^n(y)$ are in different arcs into which $0$ and $\varphi$ partition the circle.
\end{proof}

The set $\Z[\varphi]$ is invariant under the map $x\mapsto \varphi x$, hence it induces local homeomorphisms on $\F$. Let $\Gr$ be the groupoid generated by germs of the local homeomorphisms of $\F$ equal to restrictions of the maps $x\mapsto \varphi^k x+l\varphi+m$ for $k, l, m\in\Z$. Let $\mathfrak{R}$ be the sub-groupoid of germs of transformations $x\mapsto x+l\varphi+m$, i.e., the groupoid of germs generated by $R$.

Consider the following clopen subsets of $\F$
\[I_0=[0, \varphi^{-1}],\qquad I_1=[\varphi^{-1}, 1],\]
and define $\si\colon \F\arr\F$ equal to $x\mapsto \varphi x$ on $I_0$ and to $x\mapsto \varphi x-1$ on $I_1$.  

We have $\si(I_0)=\F$ and $\si(I_1)=I_0$, so $\{I_0, I_1\}$ is a Markov partition for $\si$. It follows that $\si\colon \F\arr\F$ is topologically conjugate to the Markov subshift of $\{0, 1\}^\omega$ defined by prohibited subword $11$.
We will denote by $\phi$ the homeomorphism from this subshift to $\F$.

Inverses of the restriction of $\si$ to the elements of the Markov partition are the maps $S_0\colon [0, 1]\arr [0, \varphi^{-1}]$ and $S_1\colon [0, \varphi^{-1}]\arr [\varphi^{-1}, 1]$ given by
\[S_0(x)=\varphi^{-1}x, \qquad S_1(x)=\varphi^{-1}x+\varphi^{-1}.\]

It follows from the description of the maps $S_0$ and $S_1$ that the homeomorphism $\phi$ identifying the Markov subshift $\{x_1x_2\ldots\in\{0, 1\}^\omega : x_nx_{n+1}\ne 11\}$ with $\F$ is given by 
\[\phi(x_1x_2\ldots)=\sum_{k=1}^\infty x_k\varphi^{-k}.\]
We have $\phi(101010\ldots)=\frac{\varphi^{-1}}{1-\varphi^{-2}}=\frac{\varphi}{\varphi^2-1}=1$. Consequently, $\phi(010101\ldots)=\varphi^{-1}$. It follows that the endpoints of $[0, 1]$ are $\phi(000\ldots)$ and $\phi(101010\ldots)$, and the endpoints of $[0, \varphi^{-1}]$ are $\phi(000\ldots)$ and $\phi(010101\ldots)$.

Denote by $I_{x_1x_2\ldots x_n}$ the range of $S_{x_1x_2\ldots x_n}=S_{x_1}S_{x_2}\cdots S_{x_n}$, where $x_1x_2\ldots x_n$ is a prefix of an element of $\F$. By induction, we get
\[I_{x_1x_2\ldots x_n}=[\phi(x_1x_2\ldots x_n000\ldots), \phi(x_1x_2\ldots x_n101010\ldots)]\] if $x_n=0$, and
\[I_{x_1x_2\ldots x_n}=[\phi(x_1x_2\ldots x_n000\ldots), \phi(x_1x_2\ldots x_n010101\ldots)]\]
if $x_n=1$.

In particular, 
\begin{align*}
I_{00} &=[0, \phi(001010\ldots)=[0, \varphi^{-2}],\\
I_{01} &=[\phi(01000\ldots), \phi(010101\ldots)]=[\varphi^{-2}, \varphi^{-1}],\\
I_{10} &=[\phi(10000\ldots), \phi(101010\ldots)]=[\varphi^{-1}, 1].
\end{align*}

The rotation $R$ acts on $[0, 1]$ as an \emph{interval exchange transformation} switching the intervals $[0, \varphi^{-2}]=I_{00}$ and $[\varphi^{-2}, 1]=I_{01}\cup I_{10}$ by parallel translations to the intervals $[\varphi^{-1}, 1]$ and $[0, \varphi^{-1}]$, respectively.

Let us denote by $R_0$ and $R_1$ the restrictions of $R$ to these intervals. We have
\begin{align*}
R_0\colon x\mapsto x+\varphi^{-1}\colon [0, \varphi^{-2}]\arr [\varphi^{-1}, 1],\\
R_1\colon x\mapsto x+\varphi^{-1}-1\colon [\varphi^{-2}, 1]\arr [0, \varphi^{-1}].
\end{align*}

By Proposition~\ref{pr:rotationexpansive}, $\{R_0, R_1\}$ is an expansive generating set of the groupoid $\mathfrak{R}$, i.e., the inverse semigroup generated by them is a basis of topology for $\mathfrak{R}$.

The following proposition is proved by direct computation of the transformations $S_x^{-1}R_iS_y$.

\begin{proposition}
The groupoid $\mathfrak{R}$ is shift invariant. The bisections $R_0$ and $R_1$ satisfy the following relations
\begin{equation}
\label{eq:R0R1}
R_0=S_1I_0S_0^{-1},\qquad R_1=S_0R_0^{-1}S_0^{-1}+S_0R_1^{-1}S_1^{-1},
\end{equation}
and
\begin{equation}
\label{eq:R0R1star}
R_0^{-1}=S_0I_0S_1^{-1},\qquad R_1^{-1}=S_0R_0S_0^{-1}+S_1R_1S_0^{-1},
\end{equation}
and
\begin{equation}
\label{eq:R0R1I}
I_0=S_01S_0^{-1},\qquad 1=S_01S_0^{-1}+S_1I_0S_1^{-1}.
\end{equation}

The nucleus of $\mathfrak{R}$ is the set $\nuke=\{\emptyset, 1, I_0, R_0, R_1, R_0^{-1}, R_1^{-1}\}$.
\end{proposition}

(We use the plus sign for the operation of disjoint union of maps.)

We get the automaton shown on Figure~\ref{fig:goldenautomaton}.

\begin{figure}
\centering
\includegraphics{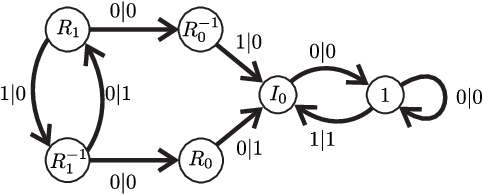}
\caption{Nucleus of the golden mean rotation}
\label{fig:goldenautomaton}
\end{figure}

\begin{proof}
The equalities~\eqref{eq:R0R1}--\eqref{eq:R0R1I} are checked by direct computation of the corresponding maps.

The statement about the nucleus follows from the following relations
\begin{gather*}
R_1^{-1}R_1=I_{01}+I_{10}=S_{01}I_0S_{01}^{-1}, R_1R_1^{-1}=I_0,\\
 R_0^{-1}R_0=I_{00}=S_0I_0S_0^{-1}, R_0R_0^{-1}=I_{10}=S_1I_0S_1^{-1}.
 \end{gather*}

We also have 
\[R_0I_0=R_0, I_0R_0=0, R_1I_0=S_0R_0^{-1}S_0^{-1}, I_0R_1=R_1.\]
and
\[R_1^2=S_{00}R_1S_{10}^{-1},\qquad R_0^2=0.\]

These equalities show that for every product $F_1F_2$ of two elements of $\nuke=\{\emptyset, 1, I_0, R_0, R_1, R_0^{-1}, R_1^{-1}\}$ and every $u, v\in\alb^2$, we have $S_u^{-1}F_1F_2 S_v\in\nuke$. This shows that $\nuke$ is the nucleus of the groupoid generated by $R$. 

\end{proof}

The subshift $\F$ is not an edge shift of a graph, since $S_0^{-1}S_0=1$ and $S_1^{-1}S_1=I_0$ are not disjoint. In order to realize $\F$ as an edge shift, we have to apply the width 2 block code to it.

We will get then the subshift over the alphabet $0_0, 0_1, 1_0$ with allowed transitions $0_00_0, 0_00_1, 0_11_0, 1_00_0, 1_00_0$. The corresponding graph has two vertices $0, 1$ and three arrows $x_y$ starting in $x$ and ending in $y$. We have $S_{x_y}=S_xS_yS_y^{-1}$.

The nucleus of the groupoid $\mathfrak{R}$ for this encoding will consist of three restrictions of the rotation:
\[A=I_1RI_0,\qquad B=I_0RI_0,\qquad C=I_0RI_1\]
and their inverses, where $I_0=S_0S_0^{-1}$ and $I_1=S_1S_1^{-1}$. We have $R_0=A$ and $R_1=B+C$.

The recursion is
\[A=S_{1_0}I_0S_{0_0}^{-1},\qquad B=S_{0_0}A^{-1}S_{0_1}^{-1},\qquad
C=S_{0_0}B^{-1}S_{1_0}^{-1}+S_{0_1}C^{-1}S_{1_0}^{-1},\]
where $I_0=S_0S_0^{-1}$, as before. 
The idempotent elements $I_0$ and $I_1$ satisfy
\[I_0=S_{0_0}I_0S_{0_0}^{-1}+S_{0_1}I_1S_{0_1}^{-1},\qquad I_1=S_{1_0}I_0S_{1_0}^{-1}.\]

\subsubsection{Matrix recursion}

Matrix recursion for self-similar groups (see Subsection~\ref{ss:algebrasgroups}) are naturally generalized to self-similar inverse semigroups. Let us illustrate this using the golden mean rotation example.

We are using here the edge-shift encoding of $\F$ as paths in the graph $\G$ with the set of vertices $\{0, 1\}$ and the set of edges $\{0_0, 0_1, 1_0\}$, as above.

Consider the vector space $C(\F, \C)$ of continuous functions on $\F$. It is naturally decomposed into the direct sum $V_0\oplus V_1$ of functions supported on paths starting in the vertices $0$ and $1$ of $\G$, respectively.

The rotation is written, with respect to this decomposition, as a matrix $\left(\begin{array}{cc} B & C\\ A & 0\end{array}\right)$, and the identity operator as $\left(\begin{array}{cc} I_0 & 0\\ 0 & I_1\end{array}\right)$ with respect to this decomposition.

The shift induces isomorphisms from $V_0$ to $V_0\oplus V_1$ and from $V_1$ to $V_0$.

If $F$ a local homeomorphism such that $\be(F)\subset {}_i\F$ and $\en(F)\subset {}_j\F$, then it induces  a linear operator from $V_i$ to $V_j$, and the above isomorphisms induced by the shift will give a representation of the linear operator as a matrix of size $2\times 2, 1\times 2, 2\times 1$, or $1\times 1$.

Namely, depending on the values of $i$ and $j$ we write the matrix recursion as
\[F\mapsto\left(\begin{array}{cc} S_{0_0}^{-1}FS_{0_0} & S_{0_0}^{-1}F S_{0_1}\\
S_{0_1}^{-1}FS_{0_0} & S_{0_1}^{-1}FS_{0_1}\end{array}\right)\]
if $i=j=0$, 
as
\[F\mapsto\left(\begin{array}{cc}S_{1_0}^{-1}FS_{0_0} & S_{1_0}^{-1}FS_{0_1}\end{array}\right)\]
if $i=0$ and $j=1$,
as
\[F\mapsto\left(\begin{array}{c} S_{0_0}^{-1}FS_{1_0}\\ S_{0_1}^{-1}FS_{1_0}\end{array}\right)\]
if $i=1$ and $j=0$,
and as
\[F\mapsto\left(S_{1_0}^{-1}FS_{1_0}\right)\]
if $i=j=1$.

Applying this formulas to the elements of the nucleus of $\mathfrak{R}$, we get
\[I_0\mapsto\left(\begin{array}{cc} I_0 & 0\\ 0 & I_1\end{array}\right),\quad
I_1\mapsto I_0,\quad A\mapsto\left(I_0\quad 0\right),\quad B\mapsto\left(\begin{array}{cc}0 & A^{-1}\\ 0 & 0\end{array}\right),\quad C\mapsto\left(\begin{array}{c}B^{-1}\\ C^{-1}\end{array}\right),\]
so that the matrix of the rotation is transformed as follows
\[\left(\begin{array}{c|c} B & C\\ \hline A & 0\end{array}\right)\mapsto \left(\begin{array}{cc|c}
0 & A^* & B^*\\
0 & 0 & C^*\\ \hline
I_0 & 0 & 0\end{array}\right),
 \]
 where the entries are written in the order coming from the decomposition $V=V_{00}\oplus V_{01}\oplus V_{10}$.
 
 Applying the procedure second time, we get the decomposition $V=V_{000}\oplus V_{001}\oplus V_{010}\oplus V_{100}\oplus V_{101}$ and the matrix
 \[\left(\begin{array}{cc|c|cc}
 0 & 0 & I_0 & 0 & 0\\
 0 & 0 & 0 & A & 0\\ \hline
 0 & 0 & 0 & B & C\\ \hline
 I_0 & 0 & 0 & 0 & 0\\
 0 & I_1 & 0 & 0 & 0
 \end{array}\right)\]

Third time:
\[\left(\begin{array}{cc|c|cc|cc|c}
0 & 0 & 0 & I_0 & 0 & 0 & 0 & 0\\
0 & 0 & 0 & 0 & I_1 & 0 & 0 & 0\\ \hline
0 & 0 & 0 & 0 &  0 &  I_0 & 0 & 0\\ \hline
0 & 0 & 0 & 0 & 0 & 0 & A^* & B^*\\
0 & 0 & 0 & 0 & 0 & 0 & 0 & C^*\\ \hline
I_0 & 0 & 0 & 0 & 0 & 0 & 0 & 0 \\
0 & I_1 & 0 & 0 & 0 & 0 & 0 & 0\\ \hline
0 & 0 & I_0 & 0 & 0 & 0 & 0 & 0
\end{array}\right).\]

We see that the matrix of the rotation is given, with respect to the decomposition of $V$ into the direct sum $\bigoplus_{w\in\alb_n}V_w$ with the lexicographic ordering of the summands, by the matrix 
\[\left(\begin{array}{cccc|cccc|cc}
0           & 0           & \ldots & 0            & I_{x_1} & 0              & \ldots  & 0 & 0 & 0\\
0           & 0           & \ldots & 0            & 0            & I_{x_2}   & \ldots  & 0 & 0 & 0\\
\vdots & \vdots & \ddots & \vdots & \vdots  & \vdots     & \ddots &\vdots & \vdots & \vdots\\
0          & 0             & \ldots & 0 & 0 & 0 & \ldots & I_{x_{u_{n+1}-2}} & 0 & 0\\ \hline
0 & 0 & \ldots & 0 & 0 & 0 & \ldots & 0 & A & 0\\
0 & 0 & \ldots & 0 & 0 & 0 & \ldots & 0 & B & C\\ \hline
I_{x_1} & 0 & \ldots & 0 & 0 & 0 & \ldots & 0 & 0 & 0\\
0 & I_{x_2} & \ldots & 0 & 0 & 0 & \ldots & 0 & 0 & 0 \\
\vdots & \vdots & \ddots & \vdots & 0 & 0 & \ldots & 0 & 0 & 0\\
0 & 0 & \ldots & I_{x_{u_n}} & 0 & 0 & \ldots & 0 & 0 & 0\end{array}\right),\]
for odd $n$, where $u_n$ is the Fibonacci sequence $u_1=1, u_2=1, u_3=2, \ldots$, and $D_m$ is the diagonal matrix $\mathrm{diag}(I_{x_1}, I_{x_2}, \ldots, I_{x_m})$, where $x_1x_2\ldots=01001010\ldots$ is the fixed point of the substitution $0\mapsto 01$, $1\mapsto 0$. For even $n$ one has to replace the matrix $\left(\begin{array}{cc}A & 0\\ B & C\end{array}\right)$  by the matrix $\left(\begin{array}{cc} A^* & B^*\\ 0 & C^*\end{array}\right)$. 

Note that the lexicographic order on sequences $a_1a_2\ldots$ agrees with the natural order of the associated numbers $\phi(a_1a_2\ldots)=\sum_{k=1}^\infty a_k\varphi^{-k}\in [0, 1]$, which agrees with the above description by a matrix of the rotation $x\mapsto x+\varphi^{-1}$ (seen as an interval exchange transformation).

\subsubsection{Cayley graphs and the Fibonacci substitution}
\label{sss:fibonacci}
The groupoid $\mathfrak{R}$ is expansively generated by the set $\mathcal{R}=\{R_0, R_1\}$.
The corresponding Cayley graph $\G_x(\mathfrak{R}, \mathcal{R})$ is a chain with arrows labeled by $R_i$ for $i=0, 1$, see Example~\ref{ex:Itinerary}. By Proposition~\ref{pr:rotationexpansive}, this rooted labeled graph uniquely determines $x$.

Let us see how the shift $\sigma\colon \F\arr\F$ transforms the orbital graphs. According to the recurrent formulas for $R_0$ and $R_1$, depending on the first two letters of $x=x_1x_2w$, the arrow starting in $x$ has one of the following forms:
\[(00w)\stackrel{R_0}{\arr}(10w),\qquad (01w)\stackrel{R_1}{\arr}(0R_0^{-1}(1w)),\qquad (10w)\stackrel{R_1}{\arr}(0R_1^{-1}(w)).\]

We see that the shift transforms an edge $x\stackrel{R_1}{\arr}y$ to $\sigma(x)\stackrel{R_0^{-1}}{\arr}\sigma(y)$ and a pair of consecutive edges $x\stackrel{R_0}{\arr}y\stackrel{R_1}{\arr}z$ to $\sigma(x)\stackrel{R_1^{-1}}{\arr}\sigma(z)$.

\begin{definition}
\label{def:fibonaccisubst}
The \emph{Fibonacci substitutional shift} is defined as the set of all two sided infinite sequences $(x_n)_{n\in\Z}$ over the alphabet $\{0, 1\}$ such that every subword $x_nx_{n+1}\ldots x_{n+k}$ for $n\in\Z$, $k\ge 0$, is a subwords of the word $\tau^n(1)$ for some $n$, where $\tau$ is the endomorphism of the free monoid generated by the symbols $0, 1$ given by $\tau(0)=1$, $\tau(1)=10$.
\end{definition}

\begin{proposition}
\label{pr:graphsrotation}
Let $\ldots R_{i_{-1}}R_{i_0}R_{i_1}\ldots$ be the sequence of labels of the arrows in the Cayley graph $\G_{\sigma(x)}(\mathfrak{R}, \mathcal{S})$. Then the Cayley graph $\G_x(\mathfrak{R}, \mathcal{R})$ is obtained by replacing in this sequence every letter $R_1$ by $R_1R_0$, every letter $R_0$ by $R_1$, and then writing the result in the opposite direction.

A sequence $\ldots R_{i_{-1}}R_{i_0}R_{i_1}\ldots$ is a sequence of labels of arrows of some Cayley graph $\G_x(\mathfrak{R}, \mathcal{R})$ if and only if $\ldots i_{-1}i_0i_1\ldots$ belongs to the Fibonacci substitutional shift.
\end{proposition}

\begin{proof}
We will denote $\G_w=\G_w(\mathfrak{R}, \mathcal{R})$.

We say that a word $v=x_1x_2\ldots x_n\in\{0, 1\}^*$ is \emph{read on a Cayley graph} if there exists a point $x\in\F$ such that the Cayley graph $\G_x$ contains a chain $w_1\stackrel{R_{x_1}}{\arr}w_2\stackrel{R_{x_2}}{\arr}\ldots \stackrel{R_{x_n}}{\arr} w_{n+1}$. 
A word $x_1x_2\ldots x_n$ is read on a Cayley graph if and only if the composition $R_{x_n}\cdots R_{x_2}R_{x_1}$ is non-empty.

It follows from compactness of $\F$ that an infinite sequence is read on a Cayley graph if and only if all its finite subwords are read on orbital graphs.

As we have already seen above, the statement of the first paragraph follows directly from the recurrent definitions of $R_i$ and $R_i^{-1}$. Namely, the Cayley graph $\G_{0w}$ is obtained from the Cayley graph $\G_w$ by replacing every edge $w_1\stackrel{R_0}{\arr} w_2$ by $0w_1\stackrel{R_1}{\longleftarrow}0w_2$ and every edge $w_1\stackrel{R_1}{\arr} w_2$ by $0w_1\stackrel{R_1}{\longleftarrow} 1w_2\stackrel{R_0}{\longleftarrow} 0w_2$. Applying this fact twice, we get that the Cayley graph $\G_{00w}$ is obtained from the Cayely graph $\G_w$ by replacing every edge of the form $w_1\stackrel{R_0}{\arr} w_2$ by the chain 
\[00w_1\stackrel{R_0}{\arr}10w_1\stackrel{R_1}{\arr}00w_2\]
and every edge $w_1\stackrel{R_1}{\arr} w_2$ by the chain
\[00w_1\stackrel{R_0}{\arr} 10w_1\stackrel{R_1}{\arr} 01w_2\stackrel{R_1}{\arr}00w_2.\]

Consider the associated symbolic substitution $\tau_1(0)=01, \tau_1(1)=011$. It follows from the above, that if we cut the graph $\G_w$ into segments between consecutive appearances of vertices of the form $0^{2n}v$, then the words read on the obtained segments of the Cayley graph are $\tau_1^n(0)$ and $\tau_1^n(1)$.

Consequently, for any segment of a Cayley graph that does not contain the vertex $000\ldots\in\F$, the word read on it is a subword of $\tau_1^n(0)$ for some $n$. Moreover, all subwords of the words $\tau_1^n(0)$ and $\tau_1^n(1)$ are read on orbital graphs.

The vertex $000\ldots\in\F$ has one out-going arrow labeled by $R_0$ and an incoming arrow labeled by $R_1$. Consequently, the two-sided infinite sequence read on the Cayley graph $\G_{0^\omega}$ is obtained by concatenating the right-infinite limit of the sequence $\tau_1^n(0)$ with the left-infinite limit of the sequence $\tau_1^n(1)$. Namely, $\tau_1^n(0)$ is a prefix of $\tau_1^{n+1}(0)$, so we get the inductive limit of the words $\tau_1^n(0)$ with respect to the inclusion of $\tau_1^n(0)$ into $\tau_1^{n+1}(0)$ as a prefix. Similarly, $\tau_1^n(1)$ is a suffix of $\tau_1^{n+1}(1)$, so we get a left-infinite lift.

Consequently, any finite word read on the Cayley graph $\G_{0^\omega}$ is contained as a subword of $\tau_1^n(10)$ for some $n$. But $10$ is a subword of $\tau_1^2(0)$. Consequently,  any finite segment of any Cayley graph $\G_w$ is a subword of $\tau_1^n(0)$ (equivalently, of $\tau_1^n(1)$) for some $n$. In other words, the set of bi-infinite sequence read on the Cayley graphs is equal to the subshift generated by the substitution $\tau_1$.

It remains to show that the substitutions $\tau_1$ and $\tau$ generate the same subshift. Let us prove the following lemma.

\begin{lemma}
\label{lem:tau1tau}
Let us write $\tau_1^n(0)=0v_n$ for a word $v_n$. Then for every finite word $w\in\{0, 1\}^*$ the word $\tau_1^n(w)$ is of the form $w'v_n$ with $\tau^{2n}(w)=v_nw'$.
\end{lemma}

\begin{proof}
It is enough to prove that $\tau_1^n(1)$ is of the form $u_nv_n$ for some $u_n$, so that $\tau^{2n}(0)=v_n0$ and $\tau^{2n}(1)=v_nu_n$.

Let us prove the statement by induction on $n$. It is true for $n=1$, since $\tau_1(0)=01$, $\tau_1(1)=011$, $\tau^2(0)=10$, $\tau^2(1)=101$, so we have  $v_1=1$ and $u_1=01$. 

Suppose that the statement is true for $n\ge 1$. Then, for every finite word $w$, we have $\tau_1^n(w)=w'v_n$ for some $w'$ such that $\tau^{2n}(w)=v_nw'$.

We have $\tau_1^{n+1}(0)=\tau_1(0v_n)=01\tau_1(v_n)$, $\tau_1^{n+1}(1)=\tau_1(u_n)\tau_1(v_n)$. 

Let us define $v_{n+1}=1\tau_1(v_n)$. By the inductive hypothesis, $\tau_1(v_n)=w'1$ for some $w'$, and $\tau^2(v_n)=1w'$. We have then $v_{n+1}=1w'1$. 

Since the last letter of $\tau_1(u_n)$ is $1$, the word $1\tau_1(v_n)$ is a suffix of $\tau_1^{n+1}(1)$. We define then $u_{n+1}$  by the condition $\tau_1(u_n)=u_{n+1}1$, so that $\tau_1^{n+1}(1)=\tau_1(u_nv_n)=u_{n+1}v_{n+1}$ and $\tau_1^{n+1}(0)=0v_{n+1}$. We have $\tau^2(u_n)=1u_{n+1}$, by the inductive hypothesis.

Then $\tau^{2(n+1)}(0)=\tau^2(v_n0)=1w'10=v_{n+1}0$ and $\tau^{2(n+1)}(1)=\tau^2(v_nu_n)=1w'1u_{n+1}=v_{n+1}u_{n+1}$, which finishes the proof.
\end{proof}

It is easy to see now that the last lemma implies that the substitutions $\tau_1$ and $\tau^2$ generate the same subshift.
\end{proof}

We have the following description of the graph $\G_{0^\omega}(\mathfrak{R}, \mathcal{R})$.

\begin{proposition}
Consider the Cayley graph $\G_{w_0}(\mathfrak{R}, \mathcal{R})$ for $w_0=000\ldots$:
\[\ldots\stackrel{R_{x_{-2}}}{\arr} w_{-2} \stackrel{R_{x_{-1}}}{\arr}w_{-1}\stackrel{R_1}{\arr}w_0\stackrel{R_0}{\arr}w_1\stackrel{R_{x_1}}{\arr}w_2\stackrel{R_{x_2}}{\arr}w_3\cdots.\]
The sequences $x_1x_2\ldots$ and $x_{-1}x_{-2}\ldots$ are both equal to the left-infinite limit $1011010110110\ldots$ of the sequence $\tau^n(1)$, where $\tau(0)=1, \tau(1)=10$.
\end{proposition}

\begin{proof}
We have seen in the proof of Proposition~\ref{pr:graphsrotation} that $0x_1x_2\ldots$ is the left-infinite limit of the sequence $\tau_1^n(0)$, which by Lemma~\ref{lem:tau1tau} implies that $x_1x_2\ldots$ is the limit of $\tau^{2n}(1)$. 

According to the statement of the first paragraph of Proposition~\ref{pr:graphsrotation}, the sequence $\ldots x_{-2}x_{-1}1.0x_1x_2\ldots$ is equal to the sequence $\tau(\ldots x_{-2}x_{-1})10.1\tau(x_1x_2\ldots)$ written in the opposite direction. But $\tau(x_1x_2\ldots)=x_1x_2\ldots$, hence $x_{-1}x_{-2}\ldots=x_1x_2\ldots$.
\end{proof}

We will need later a description of palindromic elements of the Fibonacci substitutional subshift. 

Consider a point $x\in\F$ represented by $t\in [0, 1]$, or by $t+0$ or $t-0$ for $t\in [0, 1]$.  Let $-x\in\F$ be the corresponding point $-t, -t-0$ or $-t+0$, respectively. If the labels of the Cayley graph $\G_x(\mathfrak{R}, \{R_0, R_1\})$  are $\ldots R_{i_{-1}}R_{i_0}R_{i_1}\ldots$, where $R_{i_0}$ is the label of the edge from $x$ to $R(x)$, then the labels of the Cayley graph $\G_{-x}(\mathfrak{R}, \{R_0, R_1\})$ are $\ldots R_{i_1}R_{i_0}R_{i_{-1}}\ldots$, where $R_{i_0}$ is the label of the arrow from $R(-x)$ to $-x$.

Consequently, if the sequence of the labels is palindromic, then either $R(x)=-x$ (if the arrow $R_{i_0}$ is the center of symmetry) or $x=-x$ (if the vertex $x$ is the center of symmetry), or $x$ belongs to the orbit of one of such points (if the center of symmetry is further away from $x$). If $R(x)=-x$, then $x+\varphi=-x+n$ for some $n\in\Z$, hence $x=(n-\varphi)/2$, i.e., $x=\varphi^{-2}/2$ or $x=(1+\varphi^{-2})/2$. If $x=-x$, then $t=-t$, so $t=0$ or $t=1/2$. The case $t=0$ corresponds to $x=0-0$ or $x=0+0$, but then $-x=0+0$ and $-x=0-0$, respectively, so $x\ne -x$. 

Therefore, there are only three palindromic Cayley graphs $\G_x(\mathfrak{R}, \{R_0, R_1\})$: for $x=1/2$, $x=\varphi^{-1}/2$, and $x=(1+\varphi^{-1})/2$. They are precisely the midpoints of $[0, 1]$ and the domains of $R_0$ and $R_1$, respectively. Let us describe the corresponding elements of the Fibonacci subshift.

The words $\tau^n(0)$ satisfy the recurrent relation $\tau^n(0)=\tau^{n-1}(0)\tau^{n-2}(0)$ for all $n\ge 2$. The last two letters of $\tau^n(0)$ are $01$ for odd $n>1$ and $10$ for even $n>0$. Let $w_n$ be the word obtained from $\tau^n(0)$ by deleting the last two symbols. We have then, for even $n$:
\[w_n10=\tau^n(0)=\tau^{n-1}(0)\tau^{n-2}(0)=\tau^{n-2}(0)\tau^{n-3}(0)\tau^{n-2}(0)=w_{n-2}10w_{n-3}01w_{n-2}10,\]
hence $w_n=w_{n-2}10w_{n-3}01w_{n-2}$.
Similarly, for odd $n$, we have:
\[w_n01=\tau^n(0)=\tau^{n-2}(0)\tau^{n-3}(0)\tau^{n-2}(0)=w_{n-2}01w_{n-3}10w_{n-2}01,\]
hence $w_n=w_{n-2}01w_{n-3}10w_{n-2}$.
Since $w_2=\emptyset$ and $w_3=1$, we get that $w_n$ are palindromes for all $n$. Moreover, we see that each of the sequences $w_{3n}$, $w_{3n+1}$, and $w_{3n+2}$ when centered in the middle converge to three bi-infinite palindromic elements of the shift. The limit of $w_{3n}$ is symmetric with respect to a middle symbol $1$:
\[\ldots 010110110\;1\; 011011010\ldots,\]
the limit of $w_{3n+1}$ is symmetric with respect to a middle symbol $0$:
\[\ldots 101101011\;0\;110101101   \ldots,\]
and the limit of the sequence $w_{3n+2}$ is symmetric with respect to the space between two consecutive symbols:
\[\ldots 110110101\;101011011\ldots.\]
These three sequences are the labels of the Cayley graphs $\G_x(\mathfrak{R}, \{R_0, R_1\})$ for $x=(1+\varphi^{-2})/2$, $\varphi^{-2}/2$, and $1/2$, respectively.

\subsubsection{Limit solenoid} Let us describe the limit dynamical system of the shift-invariant groupoid $\mathfrak{R}$.

Take the squares $K_0=[0, \varphi^{-2}]\times[0, \varphi^{-2}]$ and $K_1=[\varphi^{-2}, 1]\times [0, \varphi^{-1}]$ in $\F\times\R$. We identify the union of their sides $[0, \varphi^{-2}]\times \{0\}$ and $[\varphi^{-2} 1]\times\{0\}]$ with the corresponding sub-intervals of $\F$.  Let us take the quotient of the union $[0, \varphi^{-2}]\times[0, \varphi^{-2}]$ and $[\varphi^{-2}, 1]\times [0, \varphi^{-1}]$ by the identifications $(x, \varphi^{-2})\sim (R_0(x), 0)$ for $x\in [0, \varphi^{-2}]$ and $(x, \varphi^{-1})\sim (R_1(x), 0)$ for $x\in [\varphi^{-2}, 1]$, see Figure~\ref{fig:goldentorus}. 

\begin{figure}
\centering
\includegraphics{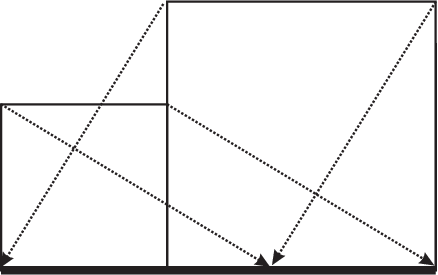}
\caption{Mapping torus of the rotation}
\label{fig:goldentorus}
\end{figure}

We get the \emph{mapping torus} of the rotation $R\colon \F\arr\F$. The vertical flow $(x, y)\mapsto (x, y+t)$ inside the mapping torus will intersect the horizontal segment $\F\times\{0\}$ in the orbit $x, R(x), R^2(x), \ldots$ of the rotation. 

The identifications 
\begin{align*}
(x, \varphi^{-2})&\mapsto (R_0(x), 0)=(x+\varphi^{-1}, 0),\\
(x, \varphi^{-1})&\mapsto (R_1(x), 0)=(x+\varphi^{-1}-1, 0)
\end{align*}
are induced by the translations of $\R^2$ by the vectors $(\varphi^{-1}, -\varphi^{-2})$ and $(-\varphi^{-2}, -\varphi^{-1})$. Let $G$ be the lattice generated by these translations. It is easy to check that the images of the squares $[0, \varphi^{-2}]\times [0, \varphi^{-2}]$ and $[\varphi^{-2}, 1]\times [0, \varphi^{-1}]$ under the action of the elements of $G$ tile $\R^2$, see Figure~\ref{fig:anosovtiling}.

\begin{figure}
\centering
\includegraphics{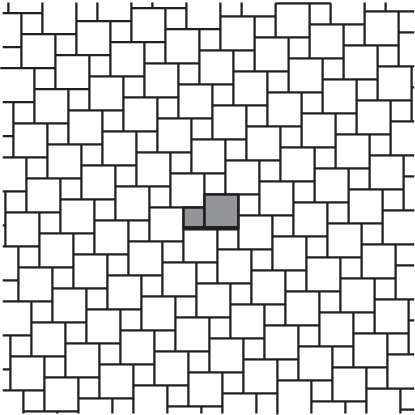}
\caption{Tiling}
\label{fig:anosovtiling}
\end{figure}

It follows that the mapping torus of $R\colon \F\arr\F$ can be defined in the following way. 

\begin{proposition}
Let $\mathcal{R}$ be the real line $\R$ in which we replace every $x\in\Z[\varphi]$ by two copies $x-0$ and $x+0$ with the natural order topology.

Let $G<\R^2$ be group generated by the translations $(x, y)\mapsto (x-\varphi^{-1}, y+\varphi^{-2})$ and $(x, y)\mapsto (x+\varphi^{-2}, y+\varphi^{-1})$ of $\mathcal{R}\times\R$.

Then the mapping torus of $F\colon \F\arr\F$ is naturally homeomorphic to $(\mathcal{R}\times\R)/G$. 

Let $K_0=[0, \varphi^{-2}]\times[0, \varphi^{-2}]$ and $K_1=[\varphi^{-2}, 1]\times [0, \varphi^{-1}]$, seen as subsets of $\mathcal{R}\times\R$. Then the sets $g(K_i)$ for $g\in G$ and $i\in\{0, 1\}$ tile $\mathcal{R}\times\R$. For every $x\in \F$, the intersections of the line $\{x_0\}\times\R$ with the sets $g(K_i)$ are intervals $\{x_0\}\times [a_n, a_{n+1}]$, where $a_0=0$, and
\[a_{n+1}=\left\{\begin{array}{rl}
a_n+\varphi^{-2} & \text{if $R^n(x_0)\in [0, \varphi^{-2}]$,}\\
a_n+\varphi^{-1} & \text{if $R^n(x_0)\in [\varphi^{-2}, 1]$.}
\end{array}
\right.
\]
\end{proposition}

In other words, the tilings of $\R$ equal to intersections of the vertical lines $\{x_0\}\times\R$ with the tiling by the sets $g(K_i)$ are geometric realizations of Cayley graphs $\G_x(\mathfrak{R}, \mathcal{R})$, where the edges corresponding to $R_0$ and $R_1$ are realized as intervals (tiles) of lengths $\phi^{-2}$ and $\phi^{-1}$, respectively.

Consider the map $S(x, y)=(\varphi x, -\varphi^{-1}y)$ on $\R^2$.
It acts on the generators $g_1=(\varphi^{-2}, \varphi^{-1})$ and $g_2=(\varphi^{-1}, -\varphi^{-2})$ of the lattice $G$ by
\[S(g_1)=(\varphi^{-1}, -\varphi^{-2})=g_2,\quad S(g_2)=(1, \varphi^{-3})=g_1+g_2.\]

Consequently, $S\colon \R^2/G\arr\R^2/G$ is conjugate to the Anosov diffeomorphism of $\R^2/\Z^2$ defined by the matrix $\left(\begin{array}{cc}0 & 1 \\ 1 & 1\end{array}\right)$. The conjugation is defined by the map $g_1\mapsto \left(\begin{array}{c}1\\ 0\end{array}\right)$ and $g_2\mapsto\left(\begin{array}{c}0\\ 1\end{array}\right)$, i.e., it is the linear map with the matrix $\left(\begin{array}{rr} \varphi^{-2} & \varphi^{-1}\\ \varphi^{-1} & -\varphi^{-2}\end{array}\right)^{-1}=\frac{\varphi}{\sqrt{5}}\left(\begin{array}{rr} 1 & \varphi \\ \varphi & -1\end{array}\right)$.

We will also denote by $S$ the homeomorphism defined by the same formula on $(\mathcal{R}\times\R)/G$. It is a hyperbolic homeomorphism. The local direct product decomposition is preserved by $S$, and $S$ is expanding the first coordinate (acting as the shift) and is contracting the second coordinate (by coefficient $\varphi^{-1}$).

As a corollary, we get.

\begin{proposition}
The limit solenoid of the groupoid generated by $R_0$ and $R_1$ is topologically conjugate to the homeomorphism induced by the Anosov diffeomorphism $\left(\begin{array}{cc} 0 & 1\\ 1 & 1\end{array}\right)\colon \R^2/\Z^2\arr\R^2/\Z^2$ on the space obtained on the torus by doubling the stable manifold passing through $\left(\begin{array}{c}0\\ 0\end{array}\right)$.
\end{proposition}

The result of doubling the stable manifold passing through a fixed point of an Anosov diffeomorphism is sometimes called the \emph{DA attractor}.

The map $S$ transforms $K_0$ and $K_1$ into the rectangles $[0, \varphi^{-1}]\times [-\varphi^{-3}, 0]$ and $[\varphi^{-1}, \varphi]\times [-\varphi^{-2}, 0]$. It follows that the images of $K_0$ and $K_1$ in $(\mathcal{R}\times\R)/G$ form a Markov partition for the homeomorphism $S$.
The self-similarity of the Cayley graphs of $\mathfrak{R}$ and of the action of $R$ on $\F$ are induced by the map $S$ and by this Markov partition.

Let $S_0=S^{-1}$ and let $S_1(\mathbf{x})=S^{-1}(\mathbf{x})+g_2$.
Then projections of $S_0$ and $S_1$ onto the first coordinate are the maps
\[S_0(x)=\varphi^{-1}x,\qquad S_1(x)=\varphi^{-1}x+\varphi^{-1},\]
which coincide with the maps $S_0$ and $S_1$ on $\F$, where the domain of $S_0$ is the union of the sides $[0, \varphi^{-2}]$ and $[\varphi^{-2}, 1]$ of $K_0$ and $K_1$, and the domain of $S_1$ is image $[0, \varphi^{-1}]$.

Let $L_0$ and $L_1$ be the vertical sides $[0, \varphi^{-2}]$ and $[0, \varphi^{-1}]$ of $K_0$ and $K_1$, respectively.
Then 
\begin{align*}
S^{-1}(L_0)&=[-\varphi^{-1}, 0]=-L_1,\\
S^{-1}(L_1) &=[-1, 0]=[-1, -1+\varphi^{-1}]\cup [-\varphi^{-2}, 0]= (-L_1+1)\cup -L_0,
\end{align*}
which agrees with the recursion for the Cayley graphs $\G_x(\mathfrak{R}, \mathcal{R})$ described in Proposition~\ref{pr:graphsrotation}.

\subsection{Penrose tiling}
\label{ss:penrose}

We will describe here a two-dimensional analog of the golden mean rotation, using two equivalent definitions. In the first definition, the semigroup acts by isometries of triangles in the Euclidean plane (cut into a Cantor set, like in the example with the rotation). In the second definition, it is the semigroup of ``pattern equivariant'' partial permutations of the tiles of a Penrose tiling (which can be seen as two-dimensional generalizations of the Fibonacci substitutional subshift).

\subsubsection{Golden triangle}

A natural two-dimensional generalization of the golden ratio for segments is the relation between ``golden triangles.''

Consider two isosceles triangles formed by diagonals and sides of a regular pentagon, see Figure~\ref{fig:2triangles}. The acute triangle with angles $72^\circ, 72^\circ, 36^\circ$ is often called the ``golden triangle,'' while the obtuse one (with the angles $36^\circ, 36^\circ, 108^\circ$) is called the ``golden gnomon" (since it is obtained by removing from the golden triangle a similar triangle).

Both of them are naturally subdivided into similar triangles: one acute and one obtuse, as it is shown on Figure~\ref{fig:2triangles}. 

\begin{figure}
\centering
\includegraphics{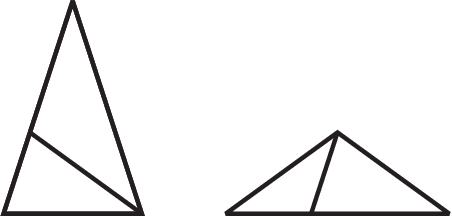}
\caption{Golden triangle and gnomon}
\label{fig:2triangles}
\end{figure}

Let us start from an acute golden triangle, and apply the described subdivisions successively in arbitrary orientation and an arbitrary number of times. We will call the smaller triangles obtained on some stage of this process \emph{sub-triangles}. Each time when we make a subdivision, we double points along the cut, similarly to what we did with the circle in the rotation example, so that all  sub-triangles are considered to be clopen. Let $\mathcal{T}$ be the obtained Cantor set. We will call it \emph{golden Cantor triangle}.  The sub-triangles will correspond to clopen subsets of $\mathcal{T}$ forming a basis of topology. When necessary, the original (uncut) golden triangle $T$ will be called \emph{Euclidean}.

Let us identify the golden Cantor triangle $\mathcal{T}$ with a topological Markov shift, by defining the branches $S_x\colon \mathcal{T}\arr\mathcal{T}\colon w\mapsto xw$ of the inverse of the shift.

Let us subdivide the golden Cantor triangle into four sub-triangles marked by symbols $0, 1, 2$, as it is shown on the left-hand side part of Figure~\ref{fig:dualdecomp}.

Let $S_0\colon \mathcal{T}\arr\mathcal{T}$ be map induced on $\mathcal{T}$ by the \emph{orientation-preserving} similarity from the golden Cantor triangle to the sub-triangle marked by $0$. 

Let $S_1\colon \mathcal{T}\arr\mathcal{T}$ the map induced by \emph{orientation-reversing} similarity from the golden Cantor triangle to the sub-triangle marked by $1$. Let $S_2$ be partial homeomorphism of $\mathcal{T}$ induced by the \emph{orientation-preserving} similarity from the golden gnomon equal to the union of the sub-triangles  marked by $1$ and $2$ to the sub-triangle marked by  2. 

We mark on Figure~\ref{fig:dualdecomp} the orientation of the maps by circular directions, where the original orientation is assumed to be positive (counter-clockwise). All maps $S_x$ are induced by similarities with the coefficient $\varphi^{-1}=\frac{\sqrt{5}-1}2$.

The ranges of compositions $S_{xy}=S_xS_y$ and their orientations are shown on the right-hand side part of Figure~\ref{fig:dualdecomp}. We mark a sub-triangle by $xy$ if it is the range of $S_{xy}$. The convention describing orientation of the maps $S_{xy}$ is the same as for the maps $S_x$.

\begin{figure}
\centering
\includegraphics{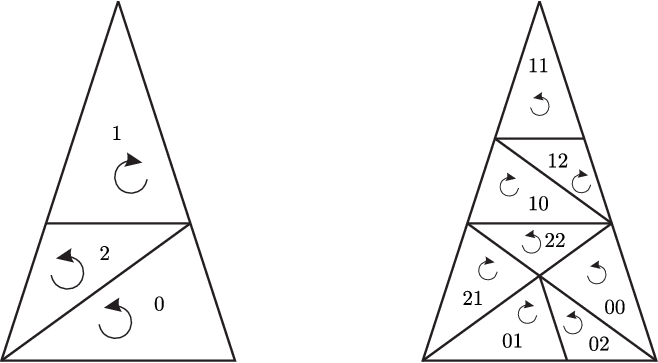}
\caption{T\"ubingen triangle}
\label{fig:dualdecomp}
\end{figure}

Ranges of compositions of length $n$ of the transformations $S_x$ subdivide the golden Cantor triangle $\mathcal{T}$ into sub-triangles, which we will call \emph{level $n$ sub-triangles}.  If we rescale this sub-division  by applying a homothety with coefficient $\tau^n$, then we get a patch of a tiling of the plane. Inductive limits of such patches are called T\"ubingen triangle tilings. See a part of it on Figure~\ref{fig:tubingenc}. See the paper~\cite{bksz:tuebingen} for properties of this tiling.

\begin{figure}
\centering
\includegraphics[width=5in]{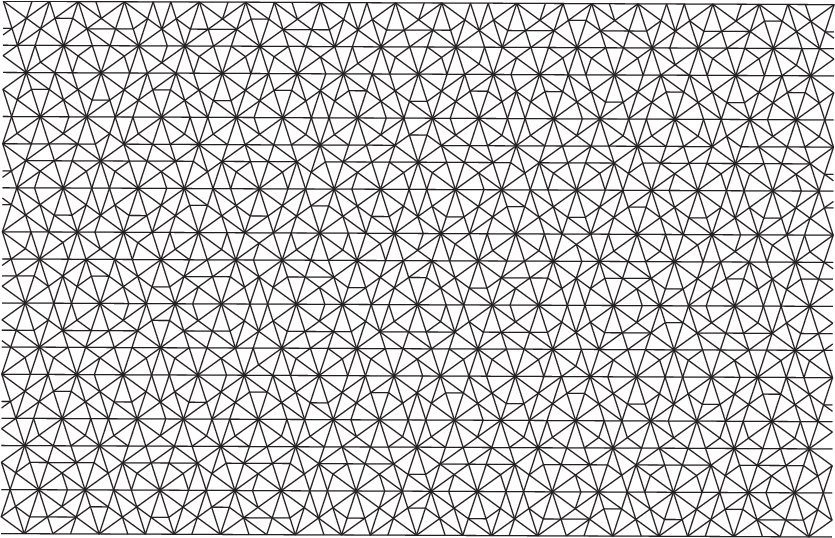}
\caption{T\"ubingen tiling}
\label{fig:tubingenc}
\end{figure}

We see from Figure~\ref{fig:dualdecomp} that for both ways of cutting the golden Cantor triangle $\mathcal{T}$ into two sub-triangles, the sub-triangles are equal to unions of level 2 sub-tiles. Similarly, both for ways of cutting the gnomon, the sub-triangles are unions of level 2 sub-tiles. It follows that every clopen sub-triangle of $\mathcal{T}$ is a finite union of tiles of some level.

We get therefore a homeomorphism of $\mathcal{T}$ with the Markov subshift of sequences $x_1x_2\ldots\in\{0, 1, 2\}^\omega$ such that $x_nx_{n+1}\ne 20$ for all $n$. We will identify the golden Cantor triangle with the subshift $\mathcal{T}$.

It follows from the above that any isometry between sub-triangles of the golden triangle induces a homeomorphism between clopen subsets of the golden Cantor triangle $\mathcal{T}$. Let us denote the groupoid of germs of such homeomorphisms by $\mathfrak{T}$. It is an effective \'etale groupoid with the space of units $\mathfrak{T}^{(0)}=\mathcal{T}$.

Consider the $\mathfrak{T}$-bisections induced by the following isometries of sub-triangles.
We denote by $I_v$ the sub-triangle (sub-tile) equal to the range of $S_v$.

Let $A_0$ be the orientation reversing isometry $I_0\arr I_1$, let $A_1$ be its inverse, and let $A_2$ by the reflection of the sub-triangle $I_2$ about its axis of symmetry. The bisections $A_0, A_1, A_2$ are pairwise disjoint. We will denote by $A$ their union. Then $A_k$ is the restriction of $A$ onto $I_k$

Let $B$ be the reflection of the sub-triangle $I_0$, and let $C$ be the reflection of the gnomon $I_1\cup I_2$. 
Let $D$ be the reflection of the whole golden triangle $\mathcal{T}$ about its axis of symmetry.

\begin{figure}
\centering
\includegraphics{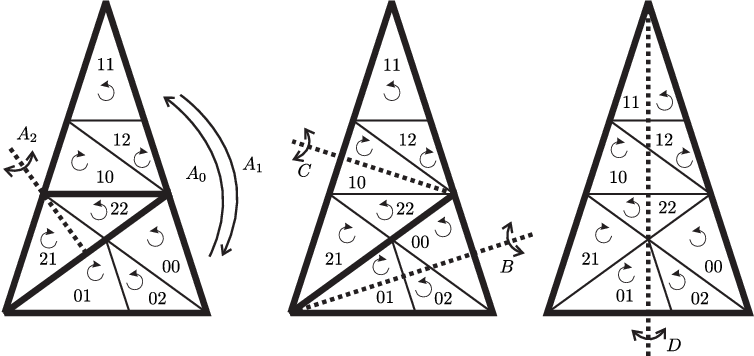}
\caption{Transformations $A_i, B, C, D$}
\label{fig:SML}
\end{figure}

Figure~\ref{fig:SML} shows how the transformations $A_i, B, C, D$ act on the second level decomposition of $\mathcal{T}$. We see from it and from the definitions of the maps $S_0, S_1, S_2$ that the transformations satisfy the following recurrent relations:
\begin{gather*}
A_0(0w)=1w,\qquad A_1(1w)=0w,\qquad A_2(2w)=2C(w),\\
B(0w)=0D(w),\qquad C(10w)=10D(w)=1B(0w),\\
C(11w)=21w,\qquad C(21w)=11w,\qquad C(12w)=22w,\qquad C(22w)=12w,\\
D(00w)=21w=2A_0(0w),\qquad D(21w)=00w=0A_1(1w),\\
D(01w)=0C(1w),\qquad D(02w)=0C(2w),\\
D(1w)=1D(w),\qquad D(22w)=22C(w)=2A_2(2w).
\end{gather*}
Consequently, $\{A_0, A_1, A_2, B, C, D, I, 1\}$, where $I=I_1+I_2$, is a self-similar set. Its Moore diagram is shown on Figure~\ref{fig:penroseautom}.

\begin{figure}
\centering
\includegraphics{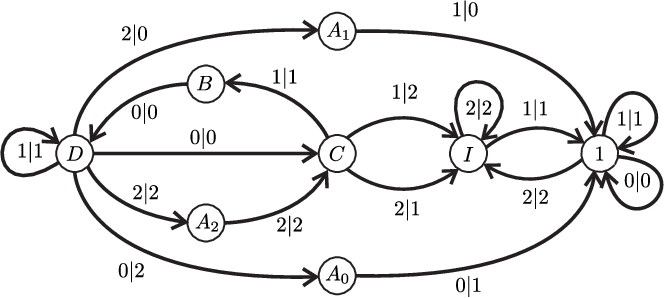}
\caption{The automaton generating the Penrose tiling}
\label{fig:penroseautom}
\end{figure}

\subsubsection{Penrose tilings}

The basic tiles of a Penrose tiling are golden triangle and golden gnomon cut from the same regular pentagon. The vertices of the tiles are marked by white and black labels, and one of the sides is marked by an arrow, as it is shown on Figure~\ref{fig:match}. A tiling of the plane by such marked triangles is a \emph{Penrose tiling} if common vertices of the triangles are marked by the same color and sides marked by arrows are adjacent to sides marked by arrows pointing in the same direction. If two tiles touch, then their intersection is either a single vertex or is a common side. See a patch of a Penrose tiling on Figure~\ref{fig:ptilingc}. This version of Penrose tilings is often called Robinson triangles, see~\cite{dandrea:penrose}. 

\begin{figure}
\centering
\includegraphics{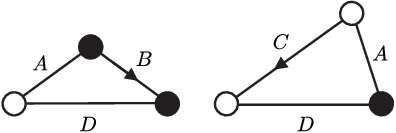}
\caption{Matching rules}
\label{fig:match}
\end{figure}

\begin{figure}
\centering
\includegraphics{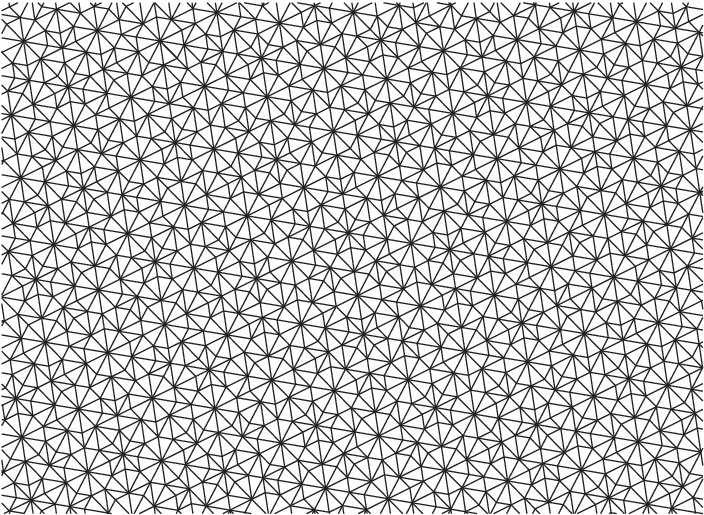}
\caption{Robinson triangle tiling}
\label{fig:ptilingc}
\end{figure}

One can show that we can group the tiles of every Penrose tiling in a unique way into groups as it is shown of Figure~\ref{fig:inflation}, so that the groups will form again a Penrose tiling by similar triangles (with the similarity coefficient $\varphi=(1+\sqrt{5})/2$). We call the new tiling the \emph{inflation} of the original tiling. The inverse procedure of subdividing tiles in order to get a Penrose tiling by smaller tiles is called \emph{deflation}.

\begin{figure}
\centering
\includegraphics{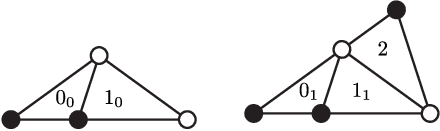}
\caption{Inflation}
\label{fig:inflation}
\end{figure}

Consider the set of all pairs $(\mathcal{P}, P)$, where $\mathcal{P}$ is a Penrose tiling, and $P$ is a tile of $\mathcal{P}$. We call $P$ \emph{marked tile} or the \emph{root} of the pair. We consider $(\mathcal{P}, P)$ up to a root-preserving isometry (alternatively, one can consider only orientation-preserving isometries or parallel translations). The set of all marked tilings comes with a natural topology in which a neighborhood of $(\mathcal{P}, P)$ is determined by a finite union $\mathcal{C}\ni P$ of $\mathcal{P}$ (called a \emph{patch}). A marked tiling tiling $(\mathcal{P}', P')$ belongs to the neighborhood defined by the pair $(\mathcal{C}, P)$ if and only if there exists an isometry from $\mathcal{C}$ to a union of tiles of $\mathcal{P}'$ mapping tiles to tiles and $P$ to $P'$. 

If $(\mathcal{P}, P)$ is a marked tiling, then we denote by $\si(\mathcal{P}, P)=(\si(\mathcal{P}, \si(P))$ the rescaled inflated tiling $(\mathcal{P}', P')/\varphi$, where the root $P'$ is the tile of $\mathcal{P}'$ containing $P$.

The \emph{itinerary} of $(\mathcal{P}, P)$ is the sequence $x_0x_1\ldots$ where $x_n$ is the label of root of $\si^n(\mathcal{P}, P)$ in the inflated tiling $\si^{n+1}(\mathcal{P}, P)$ according to the labeling shown on Figure~\ref{fig:inflation}. Unless it is stated otherwise, we will omit the indices in $0_0, 0_1, 1_0, 1_1$, since they are uniquely determined by the next symbol of the itinerary: if the next symbol is $0_*$, then the index is $0$, otherwise it is $1$.

It follows from the inflation rule that the itinerary of a tile of a Penrose tilings belongs to Markov shift $\mathcal{T}=\{x_1x_2\ldots\in\{0, 1, 2\}^\omega : x_nx_{n+1}\ne 20\}$.

Conversely, every sequence $w=x_0x_1\ldots$ belonging to $\mathcal{T}$ defines a sequence of subdivided triangles $P_0, P_1, P_2, \ldots$ such that $P_n$ is equal to the part labeled by $x_n$ of the inflated triangle $P_{n+1}$ according to the subdivision shown on Figure~\ref{fig:inflation}. The direct limit of the tiled triangles $P_n$ is either a Penrose tiling with root $P_0$, or tiling of a half-plane, or tiling of an angle of size $2\pi/5$. The latter case happens only for the itineraries cofinal to $101010\ldots$ and $010101\ldots$. Partial tilings in all such cases are uniquely extended to Penrose tilings of the plane by reflections over their boundaries.

It follows that two tilings $(\mathcal{P}_1, P_1)$ and $(\mathcal{P}_2, P_2)$ 
have the same itinerary if and only if they are isomorphic as marked Penrose tilings.
Moreover, the defined correspondence between marked Penrose tilings and their itineraries is a homeomorphism between the space of marked Penrose tilings and the subshift $\mathcal{T}$. We will identify, therefore, the space of marked Penrose tilings with $\mathcal{T}$.

The definition of the graph shift from~\ref{sss:graphshift} is naturally applicable to Penrose tilings, and gives us a groupoid and an inverse semigroup originally defined by J.~Kellendonk, see~\cite{kellendonk:noncom,kellendonk:coinvar}.

Let $\mathfrak{P}$, similarly to~\ref{sss:graphshift}, be the topological space of isomorphism classes of ordered triples $(\mathcal{P}, P_2, P_1)$, where $\mathcal{P}$ is a Penrose tiling and $T_1, T_2$ are its tiles. The topology is defined, similarly to the space $\mathcal{T}$ using isomorphism classes of finite patches of the tiling (same as for the graph shift). Namely, a neighborhood of an element $(\mathcal{P}, P_2, P_1)$ is defined by a triple $(\mathcal{C}, P_2, P_1)$, where $\mathcal{C}$ is a finite patch of $\mathcal{P}$ containing $P_1$ and $P_2$. A triple $(\mathcal{P}', P_2', P_1')\in\mathfrak{P}$
belongs to the neighborhood $(\mathcal{C}, P_2, P_1)$ if there exists an isometry from $\mathcal{C}$ to a patch of $\mathcal{P}'$ mapping $P_i$ to $P_i'$ for $i=1, 2$. 

The space $\mathfrak{P}$ is naturally an \'etale groupoid (called the \emph{tiling groupoid}), where the product is defined by \[(\mathcal{P}, P_3, P_2)(\mathcal{P}, P_2, P_1)=(\mathcal{P}, P_3, P_1). \]
One has to check that the product is well defined. Since symmetries of tiles do not preserve the marking of their vertices, automorphisms of tilings do not have fixed tiles. Consequently, for every representative $(\mathcal{P}, P_2, P_1)$ of $g_1$ the representative of the form $(\mathcal{P}, P_3, P_2)$ of $g_2$ is unique. Hence, the product is well defined, and we have a structure of a groupoid on $\mathfrak{P}$. It is easy to see that this groupoid is \'etale and Hausdorff.

Suppose that $g\in\mathfrak{P}$ is isotropic but not a unit. Then $g$ is the isomorphism class of a marking $(\mathcal{P}, P_1, P_2)$ of a tiling such that $P_2$ is the image of $P_1$ under an automorphism of the tiling $\mathcal{P}$. Consider an arbitrarily small clopen bisection $U$ containing $g$ and defined by a finite patch of $\mathcal{P}$ containing $P_1$ and $P_2$. Since Penrose tilings are repetitive, this patch will appear in every Penrose tiling, in particular in a non-symmetric one. Then the corresponding point of $U$ will not be isotropic. This shows that $\mathfrak{P}$ is effective, i.e., is the groupoid of germs of the set of partial homeomorphisms of $\mathfrak{P}^{(0)}$.

Every basic neighborhood $(\mathcal{C}, P_2, P_1)$, where $\mathcal{C}$ is a finite patch of a Penrose tiling is a $\mathfrak{P}$-bisection. The set of such bisections is an inverse semigroup called the \emph{tiling semigroup}. Let $(\mathcal{C}_2, P_{2, 2}, P_{1, 2})$ and $(\mathcal{C}_1, P_{2, 1}, P_{1, 1})$ be two triples representing bisections $F_2$ and $F_1$, respectively. Find a Penrose tiling $\mathcal{P}$ and isomorphic embeddings $\psi_i\colon \mathcal{C}_i\arr\mathcal{P}$ such that $\psi_1(P_{2, 1})=\psi_2(P_{1, 2})$. If such a tiling and embeddings do not exist, then the product $F_2F_1$ is empty. If it does exist, then $F_2F_1$ is defined by 
$(\psi_2(\mathcal{C}_2)\cup\psi_1(\mathcal{C}_1), \psi_2(P_{2, 2}), \psi_1(P_{1, 1}))$.

Let $\mathcal{K}$ be the set of $\mathfrak{P}$-bisections defined by triples $(\mathcal{C}, P_2, P_1)$, where the patch $\mathcal{C}$ has connected interior, i.e., for any two tiles in $\mathcal{C}$ there is a sequence of tiles in $\mathcal{C}$ from one of them to the other, such that every tile has a common side with the next tile in the sequence. The above description of the multiplication of bisections implies that $\mathcal{K}$ is an inverse semigroup. We call it the \emph{Kellendonk semigroup} of the Penrose tiling.

If $F$ is one of the letters $A, B, C, D$, and $x_0x_1\ldots$ is the itinerary of a Penrose tiling $(\mathcal{P}, P)$, then let $F(x_0x_1\ldots)$ be the itinerary of the tile Penrose tiling $(\mathcal{P}, P')$, where $P'$ is the tile of $\mathcal{P}$ sharing with $P$ the side labeled by $F$, according to Figure~\ref{fig:match}. We say then that $P$ and $P'$ are \emph{$F$-adjacent} and that $P'$ is the \emph{$F$-neighbor} of $P$. Then each $F$ is a homeomorphism between clopen subsets of the space $\mathcal{T}$. (The transformations $A$ and $D$ are everywhere defined.)

The transformations $B$ and $C$ are elements of the Kellendonk semigroup $\mathcal{K}$. The transformations $A$ and $D$ are disjoint unions of elements of the Kellendonk semigroup $A=A_{10}\cup A_{01}\cup A_{11}$ and $D=D_{00}\cup D_{01}\cup D_{10}\cup D_{11}$, where $F_{ij}$ is the part of $F$ mapping tiles of type $j$ to tiles of type $i$, where  type is $0$ for obtuse and $1$ for acute tiles. We rename $A_0=A_{10}, A_1=A_{01}, A_2=A_{11}$. See Figure~\ref{fig:kellendonkgen}, where the corresponding elements of the Kellendonk semigroup are shown as the corresponding patches, where the source and range tiles are marked by $1$ and $2$, respectively.

\begin{figure}
\includegraphics{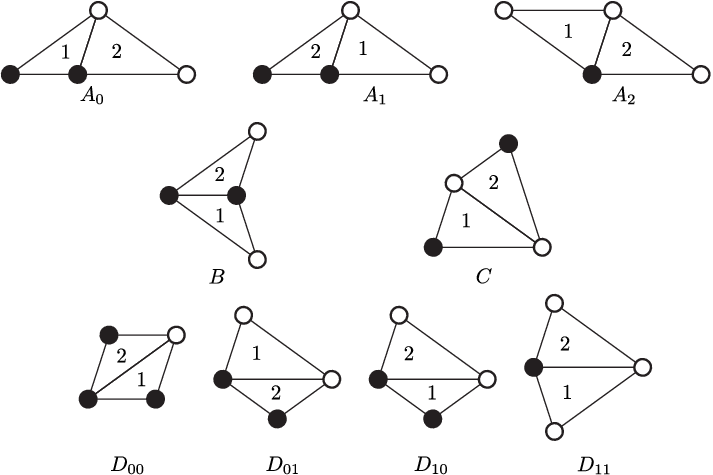}
\caption{Generators of the Kellendonk semigroup}
\label{fig:kellendonkgen}
\end{figure}

\begin{figure}
\includegraphics{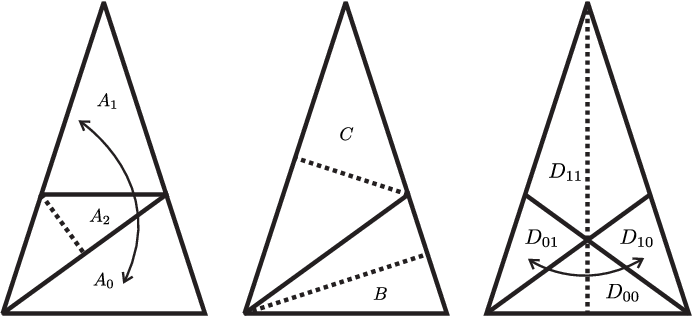}
\caption{Action of the generators on the triangle}
\label{fig:SMLsplit}
\end{figure}

\begin{proposition}
\label{pr:Kellgenerators}
The Kellendonk semigroup $\mathcal{K}$ is generated by the set  \[\{A_0, A_1, A_2, B, C, D_{00}, D_{01}, D_{10}, D_{11}\}.\]
\end{proposition}

\begin{proof}
For any triple $(\mathcal{C}, P_2, P_1)$, where $\mathcal{C}$ is a patch with connected interior containing $P_1$ and $P_2$, we can find a path of tiles starting in $P_1$, ending in $P_2$ such that each tile is adjacent to the next tile, and the set of tiles visited by the path is equal to $\mathcal{C}$. Then the path will define a product of the generators equal to the bisection represented by the triple.
\end{proof}

Let us relate now the transformations $A, B, C, D$ with the previously defined transformations of the golden Cantor triangle.

\begin{proposition}
The identifications of the space of Penrose tilings and the golden Cantor triangle with the subshift $\{x_1x_2\ldots : x_nx_{n+1}\ne 20\}\subset\{0, 1, 2\}^\omega$ conjugates the transformations $A, B, C, D$ of the space of Penrose tilings with the namesake transformations of the golden Cantor triangle.
\end{proposition}

\begin{proof} We have to show that the transformations $A, B, C, D$ of the space of Penrose tilings satisfy the same recursion as the transformations of the golden Cantor triangle.

Figure~\ref{fig:matching} shows all possible pairs of adjacent  inflated tiles.  Note that two obtuse triangles can not be $A$-adjacent, since then they will overlap with their $B$-adjacent tiles.

\begin{figure}
\centering
\includegraphics{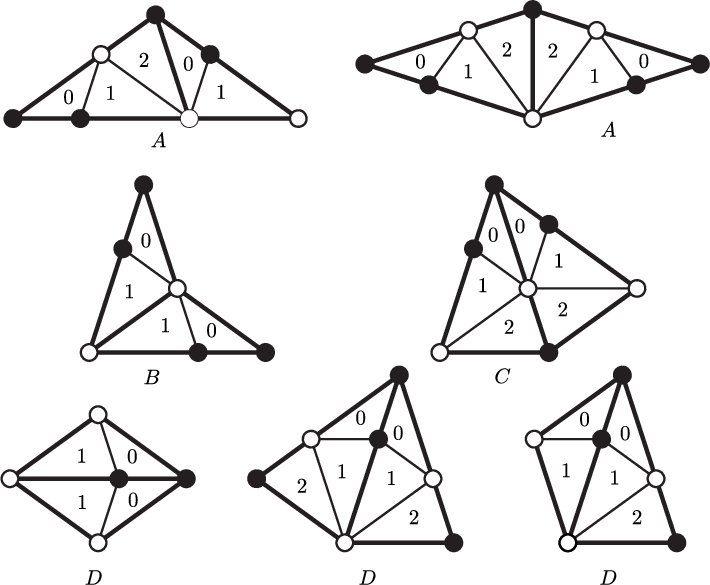}
\caption{Adjacency of inflated tiles}
\label{fig:matching}
\end{figure}

We see that two (smaller) $A$-adjacent tiles (the side labeled by $A$ is the shorter one with endpoints of different color) only if they are either labeled by $0$ and $1$ and belong to one inflated tile, or they are labeled by $2$ and belong to inflated $C$-adjacent tiles. Consequently,
\[A(0w)=1w,\qquad A(1w)=0w,\qquad A(2w)=2C(w).\]

Two tiles are $B$-adjacent (the side labeled by $B$ is the shorter one with two black vertices) only when they are labeled both by $0$. There are three instances of this on Figure~\ref{fig:matching}, and in all of them the tiles belong  to inflated $D$-adjacent tiles. Consequently,
\[B(0w)=0D(w).\]

Two tiles are $C$-adjacent (the side labeled by $C$ is the longer one with two white vertices) either when they are labeled by $1$ and $2$ and are contained in one acute inflated tile, or they are both labeled by $1$ and are inside two obtuse $B$-adjacent inflated tiles. Consequently,
\[C(11w)=21w,\quad C(12w)=22w,\quad C(21w)=11w,\quad C(22w)=12w,\] and \[ C(10w)=1B(0w).\]

The side labeled by $D$ is the long side with vertices of different color. There are six such $D$-adjacent pairs of tiles on Figure~\ref{fig:matching}:
\begin{enumerate}
\item labeled by $0$ and $2$ inside $A$-adjacent obtuse and acute inflated tiles, respectively, which gives us $D(00w)=2A_0(0w)$ and $D(21w)=0A_1(1w)$;
\item labeled by $2$ and $2$ inside $A$-adjacent inflated tiles, which gives us $D(2w)=2A_2(w)$;
\item labeled by $0$ and $0$ inside $C$-adjacent acute inflated tiles, which gives us $D(0xw)=0C(xw)$ for $x=1, 2$;
\item three cases of tiles labeled by $1$ and $1$ (two obtuse, two acute, or one acute and one obtuse) inside $D$-adjacent tiles, which gives $D(1w)=1D(w)$.
\end{enumerate}

We see that the recursion for $A, B, C, D$ is the same as for the eponymous transformations of the golden Cantor triangle.
\end{proof}

The action of the generators $A_0, A_1, B, C, D_{00}, D_{01}, D_{10}, D_{11}$ on the golden triangle are shown again on Figure~\ref{fig:SMLsplit}. The letters mark the domains of the corresponding transformations. The dashed lines are axes of the corresponding reflections.

\begin{corollary}
The groupoid $\mathfrak{P}$ is contained in the groupoid $\mathfrak{T}$.
\end{corollary}

Recall that $\mathfrak{T}$ is the groupoid of germs of local homeomorphisms of the golden Cantor triangle induced by the isometries between its sub-triangles.
We will see later that, in fact, $\mathfrak{P}$ and $\mathfrak{T}$ coincide.

Let us write the recursion defining the generators of the Kellendonk semigroup using matrix notation.  We use the encoding of $\mathcal{T}$ as the edge shift of the graph $\G$ with two vertices $v_0, v_1$ corresponding  to the obtuse and acute tiles, respectively, and the set of arrows $\{0_0, 0_1, 1_0, 1_1, 2\}$ corresponding to the inflation rule as it is shown on Figure~\ref{fig:inflation}. 

In terms of the golden Cantor triangle $\mathcal{T}$, the corresponding maps $S_{0_*}$ and $S_{1_*}$ and $S_2$ are the restrictions of the original maps $S_0, S_1$ to acute $I_0$ and obtuse $I_1\cup I_2$ sub-triangles of $T$.  See the ranges of these maps on Figure~\ref{fig:Sgraph}.

\begin{figure}
\centering
 \includegraphics{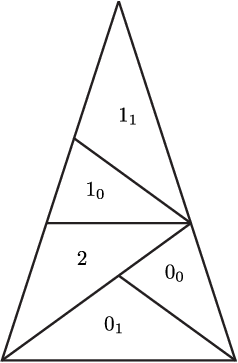}
\caption{The ranges of maps $S_x$}
\label{fig:Sgraph}
\end{figure}

Let us denote $P_0$ and $P_1$ the idempotents corresponding to the vertices $v_0$ and $v_1$, respectively. They correspond to the obtuse and the acute tiles, respectively. We denote all empty bisections inside the matrices by $0$.

It will be convenient to use the notation $A_0=A_{10}$, $A_1=A_{01}$, and $A_2=A_{11}$ consistent with the notation for the parts of $D$.

Then the matrix recursion for the generators will be written in one of the following four ways:
\[F\mapsto\left(\begin{array}{cc} 
S_{0_0}^{-1}FS_{0_0} & S_{0_0}^{-1}FS_{0_1}\\
S_{0_1}^{-1}FS_{0_0} & S_{0_1}^{-1}FS_{0_1}
\end{array}\right),\quad
F\mapsto \left(\begin{array}{ccc}
S_{1_0}^{-1}FS_{1_0} & S_{1_0}^{-1}FS_{1_1} & S_{1_0}^{-1}FS_2\\
S_{1_1}^{-1}FS_{1_0} & S_{1_1}^{-1}FS_{1_1} & S_{1_1}^{-1}FS_2\\
S_{2}^{-1}FS_{1_0} & S_{2}^{-1}FS_{1_1} & S_{2}^{-1}FS_2
\end{array}\right),
\]
\[
F\mapsto\left(\begin{array}{cc} 
S_{1_0}^{-1}FS_{0_0} & S_{1_0}^{-1}FS_{0_1}\\
S_{1_1}^{-1}FS_{0_0} & S_{1_1}^{-1}FS_{0_1}\\
S_2^{-1}FS_{0_0} & S_2^{-1}FS_{0_1}
\end{array}\right),\quad
F\mapsto \left(\begin{array}{ccc}
S_{0_0}^{-1}FS_{1_0} & S_{0_0}^{-1}FS_{1_1} & S_{0_0}^{-1}FS_2\\
S_{0_1}^{-1}FS_{1_0} & S_{0_1}^{-1}FS_{1_1} & S_{0_1}^{-1}FS_2
\end{array}\right),
\]
depending on the cases whether $\be(F)\subset P_0$ or $\be(F)\subset P_1$ and whether $\en(F)\subset P_0$ or $\en(F)\subset P_1$.

We have
\[P_0\mapsto\left(\begin{array}{cc}
P_0 & 0\\ 0 & P_1\end{array}\right),\qquad P_1\mapsto\left(\begin{array}{ccc}
P_0 & 0 & 0\\
0 & P_1 & 0\\
0 & 0 & P_1\end{array}\right)\]
for the idempotents.

The recursion for $B$ and $C$ is
\[
B\mapsto\left(\begin{array}{cc} D_{00} & D_{01}\\ D_{10} & D_{11}\end{array}\right),\qquad
C\mapsto\left(\begin{array}{ccc} B & 0 & 0\\
0 & 0 & P_1\\ 0 & P_1 & 0\end{array}\right).
\]

The parts of $A$ satisfy
\[A_{10}\mapsto\left(\begin{array}{cc} P_0 & 0\\ 0 & P_1\\ 0 & 0\end{array}\right),\quad
A_{01}\mapsto\left(\begin{array}{ccc} P_0 & 0 & 0\\ 0 & P_1 & 0\end{array}\right),
\quad
A_{11}\mapsto\left(\begin{array}{ccc} 0 & 0 & 0\\ 0 & 0 & 0\\ 0 & 0 & C\end{array}\right).\]

For the parts $D_{ij}$ of $D$, we have
\[D_{00}\mapsto\left(\begin{array}{cc} 0 & 0\\ 0 & C\end{array}\right),\qquad
D_{11}\mapsto\left(\begin{array}{ccc} D_{00} & D_{01} & 0\\ D_{10} & D_{11} &
0\\ 0 & 0 & A_{11}\end{array}\right),\]
and
\[D_{01}\mapsto\left(\begin{array}{ccc} 0 & 0 & A_{01}\\ 0 & 0 & 0\end{array}\right),\qquad
D_{10}\mapsto\left(\begin{array}{cc} 0 & 0\\ 0 & 0\\ A_{10} & 0\end{array}\right).\]

\begin{example}
\label{ex:wormidentities}
Let us show, using the matrix recursion, that $D_{11}A_{10}D_{00}A_{01}D_{11}=D_{11}CA_{11}CD_{11}=A_{10}BD_{00}BA_{01}$.
We have
\begin{multline*}D_{11}A_{10}D_{00}A_{01}D_{11}\mapsto\\
\left(\begin{array}{ccc} D_{00} & D_{01} & 0\\ D_{10} & D_{11} &
0\\ 0 & 0 & A_{11}\end{array}\right)\left(\begin{array}{cc} P_0 & 0\\ 0 & P_1\\ 0 & 0\end{array}\right)\left(\begin{array}{cc} 0 & 0\\ 0 & C\end{array}\right)\cdot\\ \left(\begin{array}{ccc} P_0 & 0 & 0\\ 0 & P_1 & 0\end{array}\right)\left(\begin{array}{ccc} D_{00} & D_{01} & 0\\ D_{10} & D_{11} &
0\\ 0 & 0 & A_{11}\end{array}\right)=\\
\left(\begin{array}{ccc}
D_{01}CD_{10} & D_{01}CD_{11} & 0\\
D_{11}CD_{10} & D_{11}CD_{11} & 0\\
0 & 0 & 0\end{array}\right),
\end{multline*}
\begin{multline*}
D_{11}CA_{11}CD_{11}\mapsto\\
\left(\begin{array}{ccc} D_{00} & D_{01} & 0\\ D_{10} & D_{11} &
0\\ 0 & 0 & A_{11}\end{array}\right)\left(\begin{array}{ccc} B & 0 & 0\\
0 & 0 & P_1\\ 0 & P_1 & 0\end{array}\right)
\left(\begin{array}{ccc} 0 & 0 & 0\\ 0 & 0 & 0\\ 0 & 0 & C\end{array}\right)\cdot\\
\left(\begin{array}{ccc} B & 0 & 0\\
0 & 0 & P_1\\ 0 & P_1 & 0\end{array}\right)
\left(\begin{array}{ccc} D_{00} & D_{01} & 0\\ D_{10} & D_{11} &
0\\ 0 & 0 & A_{11}\end{array}\right)=\\
\left(\begin{array}{ccc}
D_{01}CD_{10} & D_{01}CD_{11} & 0\\
D_{11}CD_{10} & D_{11}CD_{11} & 0\\
0 & 0 & 0\end{array}\right),
\end{multline*}
and
\begin{multline*}
A_{10}BD_{00}BA_{01}\mapsto\\
\left(\begin{array}{cc} P_0 & 0\\ 0 & P_1\\ 0 & 0\end{array}\right)
\left(\begin{array}{cc} D_{00} & D_{01}\\ D_{10} & D_{11}\end{array}\right)
\left(\begin{array}{cc} 0 & 0\\ 0 & C\end{array}\right)
\left(\begin{array}{cc} D_{00} & D_{01}\\ D_{10} & D_{11}\end{array}\right)
\left(\begin{array}{ccc} P_0 & 0 & 0\\ 0 & P_1 & 0\end{array}\right)=\\
\left(\begin{array}{ccc}
D_{01}CD_{10} & D_{01}CD_{11} & 0\\
D_{11}CD_{10} & D_{11}CD_{11} & 0\\
0 & 0 & 0\end{array}\right),
\end{multline*}
which implies the equalities.
\end{example}

\begin{proposition}
The groupoid $\mathfrak{P}$ is contracting with the nucleus equal to the set of bisections defined by triples of the form $(\{P_1, P_2\}, P_1, P_2)$, where the tiles $P_1$ and $P_2$ have a non-empty intersection. 
\end{proposition}

\begin{proof}
We say that a patch $\mathcal{C}$ \emph{enforces} a larger patch $\mathcal{C}'$ if every Penrose tiling containing $\mathcal{C}$ contains $\mathcal{C}'\supset\mathcal{C}$ at the same place. In particular, a triple $(\mathcal{C}, P_2, P_1)$ defines the same $\mathfrak{P}$-bisection as $(\widehat{\mathcal{C}}, P_2, P_1)$.

It follows from the recurrent rules defining the transformations $A, B, C, D$, that the patch equal to a twice inflated tile enforces the patch consisting of it and of all (not inflated) tiles adjacent to it. Consequently, an $m$ times inflated tile $\si^m(P)$ enforces its union $\widehat{\si^m(P)}$ with of all $m-2$ times inflated tiles adjacent to it. 

Let $(\mathcal{C}, P_1, P_2)$ be a marked patch defining a bisection $F$. Suppose that $m$ is large enough, so that if $\mathcal{C}$ intersects an inflated tile $\si^m(P)$, then it is contained in $\widehat{\si^m(P)}$.

If $u, v$ are paths of length $m$ in the graph $\G$, then $S_u^{-1}FS_v$, if not empty, is obtained from $(\mathcal{C}, P_1, P_2)$ by finding a patch isomorphic to $\mathcal{C}$ in a set $\widehat{\si^m(P)}$ so that $P_1$ and $P_2$ have addresses $u$ and $v$, respectively, for the $m$th inflation, and then passing to the triple $(\si^m(\mathcal{C}), \si^m(P_1), \si^m(P_2))$.

If $m$ is large enough, then the intersection of all tiles of $\si^m(\mathcal{C})$ will be non-empty. Moreover, the patch $\mathcal{C}$ is enforced by the patch $\si^m(P_1)\cup\si^m(P_2)$ (in the original tiling). Then the triple $(\si^m(\mathcal{C}), \si^m(P_1), \si^m(P_2))$ defines the same bisection as $(\{\si^m(P_1), \si^m(P_2)\}, \si^m(P_1), \si^m(P_2))$, which finishes the proof.
\end{proof}

There are seven patches of all tiles containing a vertex of a Penrose tiling, see Figure~\ref{fig:neighborhoods}. They are sometimes called (left to right and top to bottom): wheel (or sun), ice cream, big kite, star, cat, shield, diamond (see~\cite{dandrea:penrose}).

\begin{figure}
\centering
\includegraphics{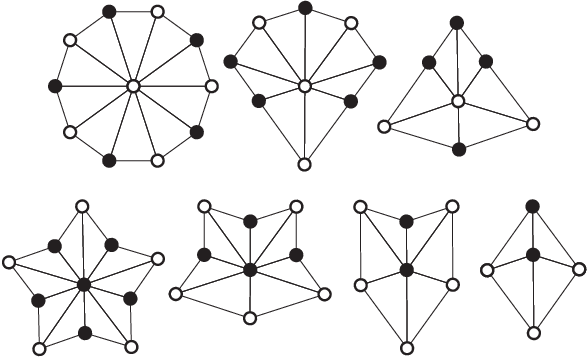}
\caption{Neighborhoods of vertices}
\label{fig:neighborhoods}
\end{figure}

By inspecting all pairs of tiles with non-empty intersection, one can check that every patch $\{P_1, P_2\}$ such that $P_1$ and $P_2$ have a common vertex, enforce a patch with connected interior containing $P_1$ and $P_2$. Consequently, all elements of the nucleus belong to the Kellendonk semigroup $\mathcal{K}$.

All patches $\{P_1, P_2\}$ together with some tiles in the enforced patch are shown on Figure~\ref{fig:nucleus}. Each of them corresponds to one or two elements of the nucleus (depending whether the corresponding bisection is an involution or not).

\begin{figure}
\centering
\includegraphics{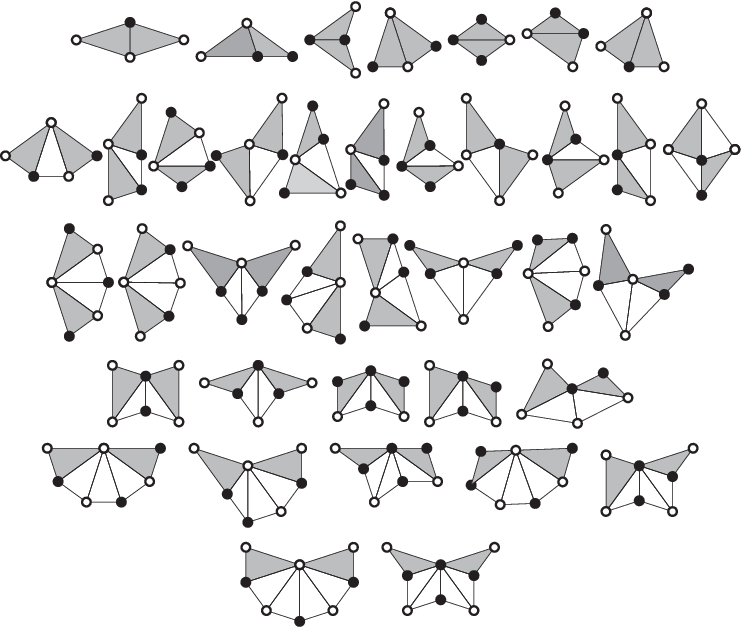}
\caption{Patches of two tiles}
\label{fig:nucleus}
\end{figure}

By looking at the deflation of the patches around vertices on Figure~\ref{fig:neighborhoods}, we see that all vertices of the second iteration of inflation are centers of a wheel or a star patch. Namely, if we apply the deflation to a patch shown on Figure~\ref{fig:neighborhoods}, then the central vertex of a wheel patch becomes the central vertex of a star, the centers of a star, a cat, and a shield become the center of a wheel, the center of an ice cream becomes a center of a cat, the center of a diamond becomes the center of an ice cream, and the center of a big kite becomes the center of a shield.

Consequently, if $F$ is an element of the nucleus defined by a pair of tiles not in the wheel or the star patch, then $S_v^{-1}FS_u$ for $|u|\ge 2$ is either empty, or an idempotent, or is defined by two tiles with a common side.

It follows that the Moore diagram of the nucleus has three strong connected components: the set $\{CD_{11}, D_{11}C, CD_{11}C, D_{11}CD_{11},  (CD_{11})^2, (D_{11}C)^2, D_{11}(CD_{11})^2\}\cup\{ BD_{00}, D_{00}B, BD_{00}B, D_{00}BD_{00}, (D_{00}B)^2, (BD_{00})^2, B(D_{00}B)^2\}$ defined by the pairs of tiles having a single common vertex in the wheel and the star patch, the set $\{D_{00}, D_{11}, A_{11}, B, C\}$ defined by pairs of tiles with a common side, and the set $\{P_0, P_1\}$ of idempotents. All the other elements of the Moore diagram of the nucleus belong to paths connecting a vertex in one of these sets to a vertex in a subsequent set.

\subsubsection{Pattern-equivariant structures}

The golden Cantor triangle $\mathcal{T}$, seen as the space of marked Penrose tiling is a convenient tool in studying \emph{pattern-equivariant} (i.e., determined by isomorphism classes of patches of the tiling)  structures of the tiling.

Clopen subsets of the Cantor triangle correspond to isomorphism classes of patches of Penrose tilings.

For example, the wheel and the star patches correspond to the pentagons in the golden Cantor triangle shown on Figure~\ref{fig:sun}. The corresponding elements of the nucleus act as the dihedral groups on these pentagons (some elements have larger domains than the pentagons). The centers of these pentagons correspond to the Penrose tilings with dihedral group of symmetries. 

\begin{figure}
\centering
\includegraphics{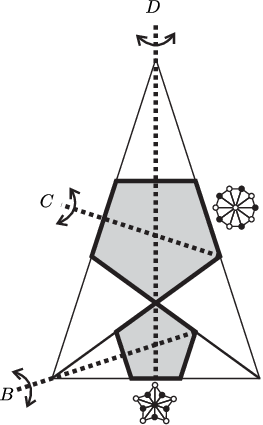}
\caption{Wheel and Star patches}
\label{fig:sun}
\end{figure}

Pattern-equivariant  ways of grouping tiles correspond to choosing compact open sub-groupoid of $\mathfrak{P}$ and grouping the tiles according to the orbits of the sub-groupoid.

There are two classical ways of grouping the Robinson triangles into larger tiles. 

The ``kites-and-darts'' tilings corresponds to the sub-groupoid $B\cup C\cup\mathcal{T}$ generated by $B$ and $C$.  It is obtained by grouping every pair of $B$-adjacent tiles into ``kites'' and every pair of $C$-adjacent tiles into ``darts.'' See Figure~\ref{fig:kitesdartsBW}, where the tiling and the corresponding partition of the golden triangle is shown. 

\begin{figure}
\centering
\includegraphics{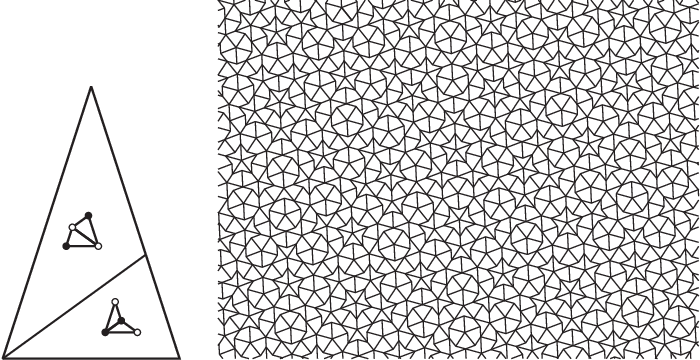}
\caption{Kites-and-darts tiling}
\label{fig:kitesdartsBW}
\end{figure}

The ``rhombus'' tiling corresponds to the sub-groupoid generated by $A$, $D_{11}$ and $B$, see Figure~\ref{fig:rhombusBW}. It groups every pair of $A_2$-adjacent tiles into one type of rhombuses, and every pair of $A_0$ or $A_1$-adjacent tiles into triangles, which are then grouped into rhombuses consisting of four tiles: two $D$-adjacent acute triangles and two $B$-adjacent obtuse triangles.

\begin{figure}
\centering
\includegraphics{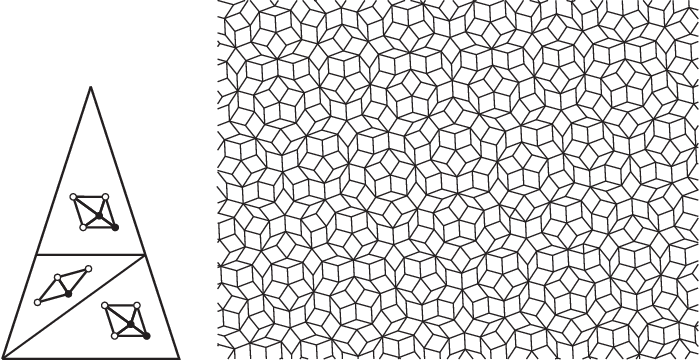}
\caption{Rhombus tiling}
\label{fig:rhombusBW}
\end{figure}

In both cases the sub-groupoid consists of all congruences between the parts into which the golden triangle on the left-hand side of the figure is partitioned.

As another example, consider the group generated by the transformations $A$ and $D$. Both of the transformations induce congruences between the pieces of the partition shown on the left-hand side part of Figure~\ref{fig:ADBW}. It is checked directly that the group induces all possible congruences between them. Consequently, the group acts as the dihedral groups $D_4$, $D_7$, and $D_{10}$ on the unions of the obtuse triangles, acute triangles, and pentagons, respectively. It follows that $\langle A, D\rangle$ is isomorphic to $D_{140}$.

\begin{figure}
\centering
\includegraphics{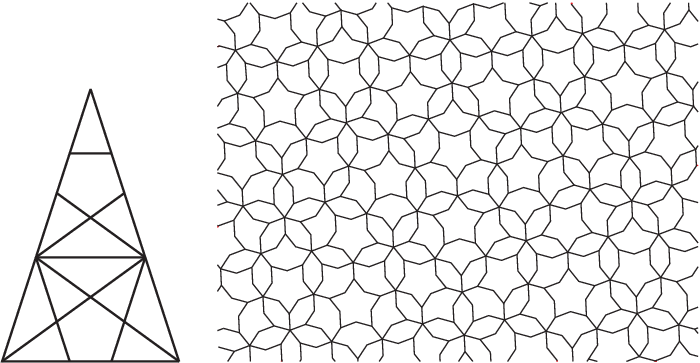}
\caption{Group generated by $A$ and $D$}
\label{fig:ADBW}
\end{figure}

\subsubsection{Cut-and-project interpretation}

Let us show how the groupoids $\mathfrak{T}$ and $\mathfrak{P}$ are related to the ``cut-and-project'' interpretations of the Penrose and T\"ubingen tilings given in~\cite{bruijn:pen1,bruijn:pen2} and~\cite{bksz:tuebingen}. We give here a self-contained description explaining the relation between the groupoids $\mathfrak{T}$ acting on the golden Cantor triangle and the groupoid $\mathfrak{P}$ acting on the space of Penrose tilings.

Consider the Euclidean space $\R^5$ with the orthonormal basis $\mathbf{e}_0, \mathbf{e}_1, \ldots, \mathbf{e}_4$ and the rotation operator $R$ permuting the basis vectors cyclically $R(\mathbf{e}_i)=\mathbf{e}_{i+1}$, where indices are taken modulo 5. 

The complexification of $R$ has five eigenvectors $\mathbf{v}_k=\sum_{j=0}^4\zeta^{kj}\mathbf{e}_j$, where $\zeta=e^{\frac{2\pi i}{5}}$.
The real space $\R^5$ is decomposed therefore into three orthogonal $R$-invariant subspaces: the span of $\mathbf{v}_0$ and the real parts of $\langle \mathbf{v}_1, \mathbf{v}_4\rangle$ and $\langle\mathbf{v}_2, \mathbf{v}_3\rangle$, which we will denote by $V_1$ and $V_2$, respectively.

The operator $R$ acts on $V_1$ and $V_2$ as rotations by the angles $\pm\frac{2\pi}5$ and $\pm\frac{4\pi}{5}$ (where the sign depends on the choice of their identifications with $\R^2$). The orthogonal projections of the basis vectors $\mathbf{v}_i$ onto $V_k$ have lengths $\sqrt{2/5}$. 

The direct sum $V_1\oplus V_2$ is the orthogonal complement of $\sum_{i=0}^4\mathbf{e}_i$, hence it is equal to the set of vectors $\sum_{i=0}^4x_i\mathbf{e}_i$ satisfying $\sum_{i=0}^4x_i=0$.

Consequently, the stabilizer $H=(V_1\oplus V_2)\cap \Z^5$ of the subspace $V_1\oplus V_2$ in $\Z^5$ is the subgroup of vectors $\sum_{i=0}^4a_k\mathbf{e}_k\in\Z^5$ satisfying $\sum_{i=0}^4a_k=0$. 
This subgroup is isomorphic to $\Z^4$ and is freely generated, for example, by $\mathbf{e}_1-\mathbf{e}_0, \mathbf{e}_1-\mathbf{e}_0, \mathbf{e}_3-\mathbf{e}_0, \mathbf{e}_4-\mathbf{e}_0$. The intersection $H\cap(V_1\oplus V_2)$ is the classical $A_4$ lattice. The quotient $\Z^5/H$ is isomorphic to $\Z/5\Z$, where the canonical epimorphism is $\sum_{j=0}^4a_j\mathbf{e}_j\mapsto \sum_{j=0}^4 a_j\pmod{5}$. 

The torus $(V_1\oplus V_2)/H$ has volume $\sqrt{5}$, which can be seen, for example, by considering the Gram determinant of the vectors $\mathbf{e}_k-\mathbf{e}_0$ for $k=1, 2, 3, 4$.

Consider the operators
\begin{align*}
M_0&\colon \mathbf{e}_0\mapsto\mathbf{e}_0,\quad\mathbf{e}_1\leftrightarrow\mathbf{e}_4,\quad\mathbf{e}_2\leftrightarrow\mathbf{e}_3,\\
M_1&\colon \mathbf{e}_1\mapsto\mathbf{e}_1,\quad\mathbf{e}_0\leftrightarrow\mathbf{e}_2,\quad\mathbf{e}_3\leftrightarrow\mathbf{e}_4,\\
M_2&\colon \mathbf{e}_2\mapsto\mathbf{e}_2,\quad\mathbf{e}_0\leftrightarrow\mathbf{e}_4,\quad\mathbf{e}_1\leftrightarrow\mathbf{e}_3,\\
M_3&\colon \mathbf{e}_3\mapsto\mathbf{e}_3,\quad\mathbf{e}_0\leftrightarrow\mathbf{e}_1,\quad\mathbf{e}_2\leftrightarrow\mathbf{e}_4,\\
M_4&\colon \mathbf{e}_4\mapsto\mathbf{e}_4,\quad\mathbf{e}_1\leftrightarrow\mathbf{e}_2,\quad\mathbf{e}_0\leftrightarrow\mathbf{e}_3.
\end{align*}
We have $R=M_{i+3}M_i$. The reflections $M_i$ act on the planes $V_k$ as the symmetries of the pentagon fixing the projection of $\mathbf{e}_i$. 
See Figure~\ref{fig:pentagons}, where the projections of the vectors $\mathbf{e}_i$ onto the planes $V_k$ and the action of the reflections $M_i$ are shown. It follows that powers of $R$ and the operators $M_i$ form a dihedral group $\mathcal{D}_5$ of order $10$.

\begin{figure}
\centering
\includegraphics{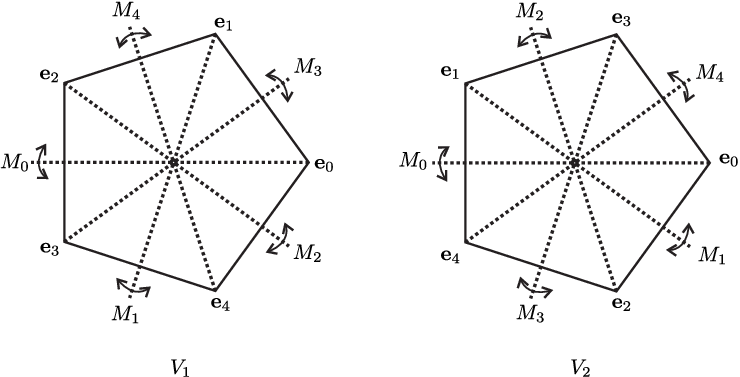}
\caption{Planes $V_1$ and $V_2$}
\label{fig:pentagons}
\end{figure}

Let $G$ be the group of all affine transformations of $V_1\oplus V_2$ of the form
\[\mathbf{x}\mapsto \pm L(\mathbf{x})+h\]
for $L\in\mathcal{D}_5$ and $h\in H$. It is isomorphic to the semi-direct product $\Z^4\rtimes D_{10}$. Each of the planes $V_i$ is $G$-invariant.

Since $H$ is an index 20 subgroup of $G$, the Euclidean orbifold $G\backslash(V_1\oplus V_2)$ has volume $\frac{\sqrt{5}}{20}$.

The subspaces $V_1$ and $V_2$ are invariant with respect to the operator $1+R$. The eigenvalues of $1+R$ on them are $1+\zeta, 1+\zeta^4$ and $1+\zeta^2, 1+\zeta^3$, respectively. We have $|1+\zeta|=|1+\zeta^4|=\varphi$ and $|1+\zeta^2|=|1+\zeta^3|=\varphi^{-1}$.

The group $G$ is invariant under the conjugation by $1+R$. Consequently, $1+R$ induces an Anosov diffeomorphism of the orbifold $G\backslash(V_1\oplus V_2)$. The images of the planes parallel to $V_1$ and $V_2$ are the unstable and stable orbifolds of this Anosov diffeomorphism, respectively.

Let $\mathcal{L}$ be the set of lines in $V_2$ parallel to the vectors $\mathbf{e}_i-\mathbf{e}_j$ and passing through points of $H$. The set $\mathcal{L}$ is then $G$-invariant.

Let $\ell_1, \ell_2\in\mathcal{L}$ be lines parallel to $\mathbf{e}_1-\mathbf{e}_0$ and $\mathbf{e}_2-\mathbf{e}_0$, respectively, both passing through a point $\mathbf{v}\in H$. Then $H$ is decomposed into the direct sum of subgroups \[\langle \mathbf{e}_1-\mathbf{e_0}, \mathbf{e}_2-\mathbf{e}_4\rangle\oplus\langle\mathbf{e}_2-\mathbf{e}_0, \mathbf{e}_3-\mathbf{e}_4\rangle,\]
where the first and the second summands preserve the lines $\ell_1$ and $\ell_2$, respectively. It follows that the intersection point $\mathbf{u}$ of any two lines parallel to $\ell_1$ and $\ell_2$ belonging to $\mathcal{L}$ is the image of $\mathbf{v}$ under the action of $H$, hence $\mathbf{u}\in H$. The same argument works for every pair of intersecting lines in $\mathcal{L}$.

We get the following description of the group $G$ of the action of $G$ on $V_2$.

\begin{proposition}
\label{pr:trianglesaction}
The action of the group $G$ on $V_2$ coincides with the set of affine maps inducing isometries between triangles formed by lines belonging to $\mathcal{L}$. 
\end{proposition}

Let $T$ be the Euclidean golden triangle in $V_2$ with the vertices $\mathbf{0}$, $\mathbf{e}_1-\mathbf{e}_0$ and $\mathbf{e}_4-\mathbf{e}_0$. Let us identify the golden Cantor triangle $\mathcal{T}$ and the groupoid $\mathfrak{T}$ with the ones constructed from the golden triangle $T$.

\begin{proposition}
The boundaries of the sub-triangles of $T$ are contained in the lines from the set $\mathcal{L}$. Their vertices belong to $H$. Moreover, for every line $\ell\in\mathcal{L}$ there exists $n$ such that the intersection of $\ell$ with $T$ is contained in the union of boundaries of the level $n$ sub-triangles.

The set of vertices of the sub-triangles of $T$ is equal to $T\cap H$.
\end{proposition}

\begin{proof}
The facts that boundaries of the sub-triangles are contained in the lines from $\mathcal{L}$ and that the vertices belong to $H$ follow by induction on the level of sub-triangles from the fact that intersections of lines in $\mathcal{L}$ belong to $H$.

It follows from the description of the second level sub-triangles (see Figure~\ref{fig:dualdecomp}) that for every vertex of a level $n$ sub-triangle $\Delta$ and every line in $\mathcal{L}$ passing through it, the intersection of the line with $\Delta$ is contained in the union of the boundaries  of the level $n+1$ sub-triangles.

The ratio of the lengths of parallel vectors $\mathbf{e}_1-\mathbf{e}_0$ and $\mathbf{e}_2-\mathbf{e}_4$ is the golden mean $\varphi$. Consequently, the intersection $\ell\cap H$ of any line parallel to them is equal to the set of points of the form $\mathbf{v}+(a+b\varphi)(\mathbf{e}_1-\mathbf{e}_0)$ for $a, b\in\Z$ and a fixed $\mathbf{v}\in\ell\cap H$.

Every side of a sub-triangle is sub-divided in proportions $1:\varphi$ and $\varphi:1$ by some later subdivision. It follows then by induction that intersection of any side of a sub-triangle with $H$ is equal to the intersection of the side with the set of vertices of all sub-triangles.

One can check, by inspecting patches of the T\"ubingen tiling, that boundary of every level $n$ sub-triangle can be extended indefinitely to the boundary of $T$ inside the union of the boundaries of the level $n+2$ sub-triangles. Together with the statement about the vertices of triangles proved in the previous paragraph, this finishes the proof.
\end{proof}

Consequently, the golden Cantor triangle $\mathcal{T}$ can be defined as the result of splitting the triangle $T$ into a Cantor set along all lines belonging to    $\mathcal{L}$. 
Proposition~\ref{pr:trianglesaction} implies that we have a well defined groupoid of germs $\Gr$ of the action of $G$ on the golden Cantor triangle equal to the groupoid of germs of local homeomorphisms of $\mathcal{T}$ induced by the isometries between the sub-triangles of the Cantor triangle formed by the lines belonging to the set $\mathcal{L}$. We have the inclusions $\mathfrak{P}\subseteq\mathfrak{T}\subseteq\Gr$ of groupoids of germs with the same unit space $\mathcal{T}$. We will see later that all three groupoids are equal.

Let $P_0$ and $P_1$ be the triangles in $V_1$ with the sets of vertices equal to the projections of $\{-\mathbf{e}_0, -\mathbf{e}_0-\mathbf{e}_3, \mathbf{e}_1+\mathbf{e}_4\}$ and $\{-\mathbf{e}_0, \mathbf{e}_1, \mathbf{e}_1+\mathbf{e}_4\}$ onto $V_1$, respectively. 
See the left-hand side of Figure~\ref{fig:mpartition} for these triangles. Since they and their images will be tiles of Penrose tilings, we have marked their vertices according to the matching rules. Note that the colors of the vertices $\sum_{i=0}^4a_i\mathbf{e}_i$ are determined by their image $\sum_{i=0}^4 a_i\pmod{5}$ in $\Z^5/H=\Z/5Z$: black vertices are mapped to $\pm 2$, white ones to $\pm 1$.

Let $T_0$ and $T_1$ be the acute and obtuse sub-triangles of $T\subset V_2$ with the sets of vertices equal to the projections of $\{\mathbf{e}_1-\mathbf{e}_0, \mathbf{e}_4-\mathbf{e}_0, \mathbf{e}_1-\mathbf{e}_3\}$ and $\{\mathbf{0}, \mathbf{e}_1-\mathbf{e}_0, \mathbf{e}_1-\mathbf{e}_3\}$ onto $V_2$, respectively. See the right-hand side of Figure~\ref{fig:mpartition} and compare it to Figure~\ref{fig:dualdecomp}.

Using the fact that projections of $\mathbf{e}_i$ onto the planes $V_k$ have lengths $\sqrt{2/5}$, we can find that the areas of the triangles $P_0, P_1, T_0, T_1$ are equal to $\frac{\varphi^{-1}\sqrt{4\varphi+3}}{10}$, $\frac{\sqrt{4\varphi+3}}{10}$, $\frac{(3\varphi-4)\sqrt{4\varphi+3}}{10}$, $\frac{(7-4\varphi)\sqrt{4\varphi+3}}{10}$, respectively.

\begin{figure}
\centering
\includegraphics{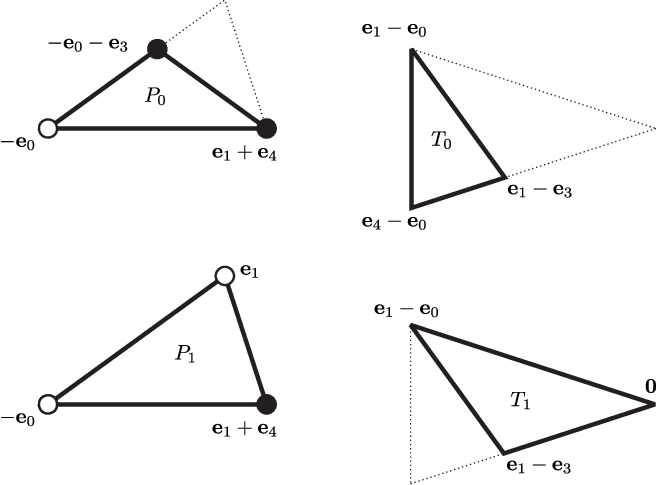}
\caption{Triangles $P_i$ and $T_i$}
\label{fig:mpartition}
\end{figure}

Consider the subsets $K_0=P_0\oplus T_0$ and $K_1=P_1\oplus T_1$ of $V_1\oplus V_2$. Their volumes are $\frac{\phi^{-1}}{20}$ and $\frac{\phi}{20}$, respectively. Consequently, their union has volume $\frac{\sqrt{5}}{20}$ (they are disjoint, since so are $T_0$ and $T_1$), which is exactly equal to the volume of the orbifold $G\backslash(V_1\oplus V_2)$.

Recall that we identify $T=T_0\cup T_1$ with the triangle used in the definitions of the golden Cantor triangle and the transformations $A_0, A_1, A_2, B, C, D$ of $T$.

The following proposition is checked directly by computing the action of the respective operators on the vertices of the triangles, see Figure~\ref{fig:V1adjacency} for the tiles of the Penrose tiling adjacent to $P_0$ and $P_1$.

\begin{proposition}
\label{pr:LFaction}
Consider the following affine transformations of $V_1\oplus V_2$:
\begin{gather*}
L_{A_0}(\mathbf{x})=-M_1(\mathbf{x})+\mathbf{e}_1-\mathbf{e}_2,\quad
L_{A_1}(\mathbf{x})=-M_1(\mathbf{x})-\mathbf{e}_0+\mathbf{e}_1,\\
L_{A_2}(\mathbf{x})=M_4(\mathbf{x})+\mathbf{e}_1-\mathbf{e}_2\\
L_B(\mathbf{x})=M_2(\mathbf{x})-\mathbf{e}_0+\mathbf{e}_1-\mathbf{e}_3+\mathbf{e}_4,\quad L_C(\mathbf{x})=M_3(\mathbf{x})+\mathbf{e}_1-\mathbf{e}_0,\\
L_D(\mathbf{x})=M_0(\mathbf{x}).
\end{gather*}

Let $L_F$ be one of them.
Let $P_0'$ or $P_1'$ denote the obtuse or the acute $F$-neighbors, respectively, of a tile $P_i\in\{P_0, P_1\}$ in $V_1$. Then:
\begin{itemize}
\item the action of $L_F$ on $V_1$ induces the unique marking-preserving affine map from $P_k'$ to $P_k$;
\item the action of $L_F$ on $V_2$ induces the action of $F$ on its domain: if $\mathbf{v}\in T_i$ and $L_F(\mathbf{v})$ belongs to $T_0\cup T_1$, then $L_F(\mathbf{v})=F(\mathbf{v})$.
\end{itemize}
\end{proposition}

\begin{figure}
\centering
\includegraphics{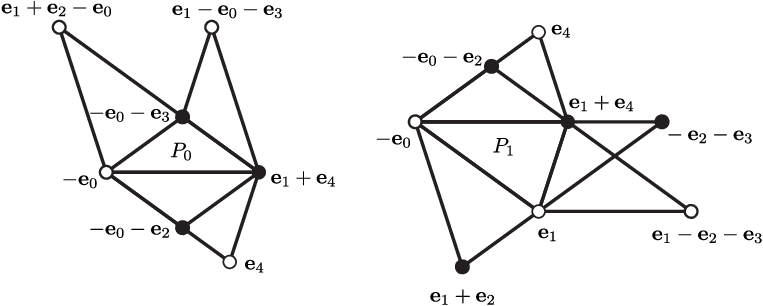}
\caption{Tiles adjacent to $P_0$ and $P_1$}
\label{fig:V1adjacency}
\end{figure}

We say that a point $\mathbf{v}\in V_2$ is \emph{regular} if it does not belong to any line in the set $\mathcal{L}$. A point in the interior of $T$ is regular if and only if it has a unique preimage in the golden Cantor triangle. If an internal point of $T$ is non-regular but does not belong to $H$, then it has two preimages. If it belongs to $H$, then it has 10 preimages.

\begin{theorem}
The sets $g(K_i)$ for $g\in G$, $i=0, 1$, cover $V_1\oplus V_2$ and have pairwise disjoint interiors.

The Penrose tiling parametrized by a regular point $\mathbf{v}\in T\subset V_2$ is the intersection of the plane $V_1+\mathbf{v}$ with the tesselation of $V_1\oplus V_2$ by the sets $g(K_i)$. 

A non-regular point $\mathbf{v}\in T$, will be split into several points of the golden Cantor triangle. Then the corresponding tilings are the limits of the intersections of the planes $V_1+\mathbf{u}_n$ with the sets $g(K_i)$ when $\mathbf{u}_n$ is a sequence of regular points converging to the corresponding point of the Cantor triangle.

The groupoid $\Gr$ is equal to its sub-groupoids $\mathfrak{P}$ and $\mathfrak{T}$.
\end{theorem}

\begin{proof}
Consider a point $w\in\mathcal{T}$ in the golden Cantor triangle. We identify $\mathcal{T}$ with the result of cutting $T\subset V_2$ along the lines of $\mathcal{L}$ and with the space of marked Penrose tilings, as described above.

Let $\mathbf{v}$ be the image of $w$ in $T\subset V_2$. Consider the plane $V_1+\mathbf{v}$ in $V_1\oplus V_2$ and tile it by the Penrose tiling corresponding to $w$, so that the marked tile is $P_k+\mathbf{v}$ for the appropriate $i\in\{0, 1\}$. Then the marked tile is equal to the intersection of $K_i$ with $V_1+\mathbf{v}$.

Suppose that $P$ is another tile of this tiling of $V_1+\mathbf{v}$. Let $X_0=P_i, X_1, X_2, \ldots, X_n=P$ be a sequence of tiles such that $X_{k+1}$ is a tiling adjacent to $X_k$. By the above proposition, we will have $L_F(X_1+\mathbf{v})=P_l+L_F(\mathbf{v})$ for the appropriate $l\in\{0, 1\}$ (depending on the shape of $X_1$), and $L_F(\mathbf{v})$ will correspond to the point $F(\xi)$ of the Cantor triangle. In particular, the tile $X_1+\mathbf{v}$ is the intersection of $V_1+\mathbf{v}$ with $L_F^{-1}(K_l)$.

Since the image of a Penrose tiling under an isometry will satisfy the same matching rules, it is again a Penrose tiling. Consequently, $L_F$ will move the tiling of $V_1+\mathbf{v}$ to the corresponding tiling $L_F(V_1)+L_F(\mathbf{v})$. Moving this way repeatedly from one tile to its neighbor in the tiling of $V_1+\mathbf{v}$, we will conclude that every tile of the tiling is equal to the  intersections of the plane $V_1+\mathbf{v}$ with a set $g(K_j)$ for some $g\in G$ and $j\in\{0, 1\}$.

Since $T$ intersects every orbit of the action of $G$ on $V_2$, the sets $g(K_j)$ cover $V_1\oplus V_2$. Since the volume of $K_0\cup K_1$ is equal to the volume of $G\backslash(V_1\oplus V_2)$, the sets $g(K_j)$ have disjoint interiors, and cover the space $V_1\oplus V_2$ without overlaps.

Consequently, the Penrose tiling of the plane $V_1+\mathbf{v}$, if $\mathbf{v}$ is regular, coincides with the intersection of plane with the tiling of the space by the sets $g(K_j)$. The statement in the non-regular case follows then from the continuity of the dependence of a tiling on its parameter in the Cantor triangle.

It remains to show that $\mathfrak{P}=\Gr$. By Propositions~\ref{pr:Kellgenerators} and~\ref{pr:LFaction}, $\mathfrak{P}$ is the groupoid generated by the germs of action of the affine maps $L_F$ on the triangle $T$, so it is enough to show that for any $g\in G$ and $\mathbf{v}\in T$ such that $g(\mathbf{v})\in T$, the germ of $g$ in $\mathbf{v}$ is a product of germs of the transformations $L_F$ in $T$. By the proven above, the map $g$ transforms the tiling of $V_1+\mathbf{v}$ to the tiling of $V_1+g(\mathbf{v})$.
Consider a sequence $X_0, X_1, \ldots, X_m$ of tiles in $V_1+\mathbf{v}$ such that $X_{i+1}$ is andjacent to $X_i$,  $X_0\in\{P_0+\mathbf{v}, P_1+\mathbf{v}\}$ is the root of the tiling of $V_1+\mathbf{v}$, and $X_m\in\{g^{-1}(P_0)+\mathbf{v}, g^{-1}(P_1)+\mathbf{v}\}$ is the image of the root of the tiling $V_1+g(\mathbf{v})$ under the action of $g^{-1}$. Let $k_0, k_1, \ldots, k_m\in\{0, 1\}$ for $k_i=0$ or $k_i=1$ if $X_i$ is obtuse or $X_i$ is acute, respectively.

Then Proposition~\ref{pr:LFaction} implies that there exists a sequence of the affine maps $L_{F_i}$ and a sequence $\mathbf{v}_i\in T$ such that $\mathbf{v}=\mathbf{v}_0$, $L_{F_i}(\mathbf{v}_i)=\mathbf{v}_{i+1}$, and $L_{F_i}^{-1}(P_{k_{i+1}})+\mathbf{v}_i$ is the $F_i$-neighbor of $P_{k_i}+\mathbf{v}_i$. Then $L_{F_i}$ maps the tiling $V_1+\mathbf{v}_i$ to the tiling $V_1+\mathbf{v}_{i+1}$. Since the isomorphism from the tiling $V_1+\mathbf{v}$ to the tiling $V_1+g(\mathbf{v})$ is unique, we get the equality 
\[[g, \mathbf{v}]= [L_{F_{m-1}}, \mathbf{v}_{m-1}]  \cdots[L_{F_1}, \mathbf{v}_1][L_{F_0}, \mathbf{v}_0]\]
of germs, thus proving the inclusion $\Gr\subseteq\mathfrak{P}$.
\end{proof}

Let us define the following transformations of $V_1\oplus V_2$
\begin{align*}
L_0(\mathbf{x})&=-R^4(1+R)(\mathbf{x})-\mathbf{e}_0+\mathbf{e}_1,\\
L_1(\mathbf{x})&=M_4(1+R)(\mathbf{x}),\\
L_2(\mathbf{x})&=R^3(1+R)(\mathbf{x})-\mathbf{e}_0+\mathbf{e}_1.
\end{align*}

Note that they are all of the form $g((1+R)(\mathbf{x}))$ for $g\in G$, hence all of them act on the orbifold $G\backslash(V_1\oplus V_2)$ as $1+R$.

\begin{figure}
\centering
\includegraphics{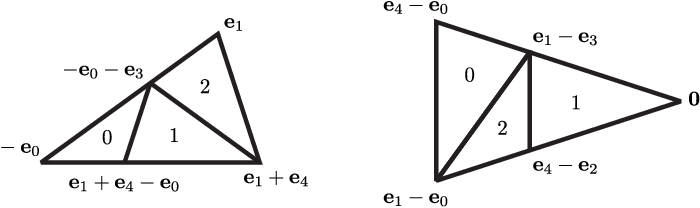}
\caption{Subdivisions of Penrose and T\"ubingen tiles}
\label{fig:penrosesubdiv}
\end{figure}

The following proposition is also proved just by computing the action of the maps on the vertices of the triangles. See the subdivisions of the Penrose and T\"ubingen triangles $P_i$ and $T_i$ in the planes $V_1$ and $V_2$ on Figure~\ref{fig:penrosesubdiv}.

\begin{proposition}
For every $i=0, 1, 2$, the transformation $L_i$, restricted to $V_1$, induces on the first level subtiles of $P_i$ the  expanding map $S_i^{-1}$.

Restricted to $V_2$, it induces the contracting map $S_i$ from $T$ to the corresponding sub-triangle of the T\"ubingent triangle.
\end{proposition}

This gives the following characterization of the Penrose and T\"ubingen tilings.

\begin{theorem}
The sets $P_0\oplus T_0$ and $P_1\oplus T_1$ form a Markov partition of the Anosov diffeomorphism $1+R$
of the orbifold $G\backslash(V_1\oplus V_2)$.

When lifted to the universal covering, the intersections of the unstable, (respectively stable) manifolds with this Markov partition form the Penrose (respectively T\"ubingen) tilings. 
\end{theorem}

\subsubsection{Conway worms}

Denote by $L$ the lateral side of the golden Euclidean triangle $T\subset V_2$ with vertices $\mathbf{0}$ and $\mathbf{e}_1-\mathbf{e}_0$. The transformation $C$ acts on it as a symmetry about its center.

Let $\F_L$ be the subset of the golden Cantor triangle $\mathcal{T}$ consisting of points $p$ such that all clopen sub-triangles of $\mathcal{T}$ containing $p$ have a side contained in $L$. The set $\F_L$ is contained in the full preimage of $L$ in $\mathcal{T}$. If a point $p$ of $\mathcal{T}$ is a preimage of a point of $L$ but does not belong to $\F_L$, then it is a vertex of a sub-triangle of $\mathcal{T}$ with no sides contained in $L$. All vertices of sub-triangles of $\mathcal{T}$ belong to one $\mathfrak{P}$-orbit. A point $p$ of $L$ which is a vertex of a sub-triangle (except when it is one of the endpoints of $L$) has five preimages in $\mathcal{T}$. Two of them belong to $\F_L$, the remaining three do not. See the left-hand side of Figure~\ref{fig:bowties} where the side $L$ and two (shaded) sub-triangles $G_0$ and $G_1$ with sides contained in $L$ are shown. It also shows three sub-triangles touching $L$ but with no sides contained in it.

\begin{figure}
\centering
\includegraphics{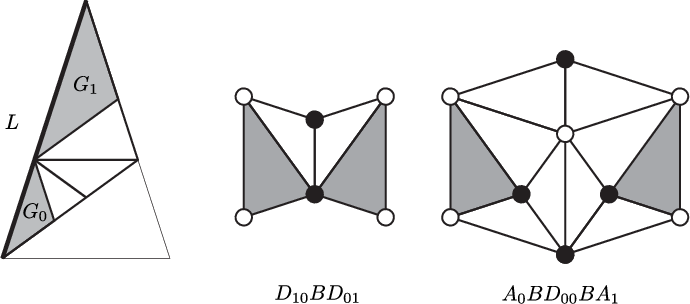}
\caption{Rotation of the side of the golden triangle}
\label{fig:bowties}
\end{figure}

Consider the transformation $C$ and the symmetries of the sub-triangles $G_0$ and $G_1$. The latter two symmetries can be written as $D_{10}BD_{01}$ and $A_0BD_{00}BA_1$. The right-hand side of Figure~\ref{fig:bowties} shows the corresponding elements of the Kellendonk semigroups, given by the patches of the Penrose tiling, where grey triangles are the initial and final tiles of the patch (they are symmetric, so it does not matter which is which). The product $A_0BD_{00}BA_1$ is just one of the paths inside the larger patch, the rest of it is enforced by it. Other paths inside the patch, for example $D_{11}CA_2CD_{11}$ and $D_{11}A_0D_{00}A_1D_{11}$, define the same elements of the tiling semigroup, see Example~\ref{ex:wormidentities}, hence each of the paths enforces the whole patch.

The transformations $R_0=CD_{10}BD_{01}$ and $R_1=CA_0BD_{00}BA_1$ act on the segment $\F_L$ in the same way as the restrictions $R_0$ and $R_1$ of the golden mean rotations, as described in Subsection~\ref{s:rotation} (i.e., are topologically conjugate to them).

The Cayley graphs of the groupoid of germs generated by the bisections $R_0$ and $R_1$ on $\F$ (and hence on $\F_L$) are described in Proposition~\ref{pr:graphsrotation}. They are bi-infinite chain labeled by $R_0$ and $R_1$, so that the corresponding sequence of the indices in the labeling is an element of the Fibonacci substitutional subshift.

If we take a point on the side $\F_L$ of the golden Cantor triangle, then the corresponding marked Penrose tiling will contain a realization of the corresponding orbital graph. Namely, if the point has orbital graph with consecutive edge labeled by $\ldots R_{x_{-1}}R_{x_0}R_{x_1}\ldots$, then the corresponding marked tile of the Penrose tiling will be contained in a bi-infinite chain obtained by pasting together the patches shown on the right-hand side of Figure~\ref{fig:bowties} along their vertical sides in the order corresponding to the order of the labels $R_{x_i}$, see Figure~\ref{fig:worms}. These chains are called \emph{Conway worms} (see~\cite{dandrea:penrose}). 

\begin{figure}
\centering
\includegraphics{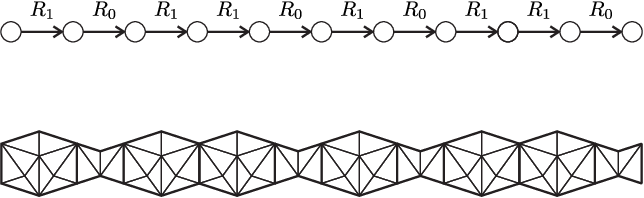}
\caption{Conway worms}
\label{fig:worms}
\end{figure}

Since the orbital graph of $x$ uniquely determines $x$, the Conway worm with a marked tile uniquely determines the corresponding point of $\F_L$. Moreover, if a point of the Euclidean golden triangle does not belong to $L$, then it does not belong to the domain of the rotation $R^n=(R_0+R_1)^n$ for some positive and some negative $n$, as the point will eventually ``fall into the crack'' between the domains $G_0$ and $G_1$ of the symmetries. It follows that such points have Conway worms of finite length only. 

Let $x\in\mathcal{T}$ be one of the three points that do not belong to $\F_L$ but whose image in the  Euclidean 
golden triangle is the common vertex of the triangles $G_0$ and $G_1$. Then $x$ does not belong to the domain of $R$, but it belongs to the domains the maps $R^n$ for all negative $n$. It follows that the corresponding tilings contain Conway worms that are infinite in one direction only. Since all such points belong to one orbit and belong to axes of symmetries of sub-triangles of $\mathcal{T}$, the corresponding tiling will have exactly six infinite in one direction Conway worms and two Conway worms infinite in both directions, corresponding to the vertices of $G_0$ and $G_1$ mapped to the same vertex of the Euclidean triangle.
This Penrose tiling is called \emph{Cartwheel tiling}. See Figure~\ref{fig:cartwheelSmall}, where its central part is shown. The line of symmetry and eight Conway worms (two bi-infinite and six one-sided infinite) are highlighted.

The symmetry group of the Cartwheel tiling is of order 2. However, if we remove all eight Conway worms and eight central tiles isolated by them, then the remaining infinite patch will have group of symmetries isomorphic to the dihedral group $D_{10}$ of order 20. There will be 10 ways of putting the removed tiles back, corresponding to 10 preimages in the golden Cantor triangle of an internal point of the Euclidean triangle.

The points of the golden Cantor triangle $\mathcal{T}$ that belong to sides of sub-triangles but are not vertices of sub-triangle will have a unique Conway worm infinite in both directions (and, of course, infinitely many finite Conway worms). Since the boundary of an infinite in both directions Conway worm is symmetric, if we remove it from the Penrose tiling, the remaining part will be symmetric with respect to the worm, and there will be two ways to put the worm back into the tiling. These two ways correspond to two preimages of a point of Euclidean triangle in the Cantor triangle.

\begin{figure}
\centering
\includegraphics{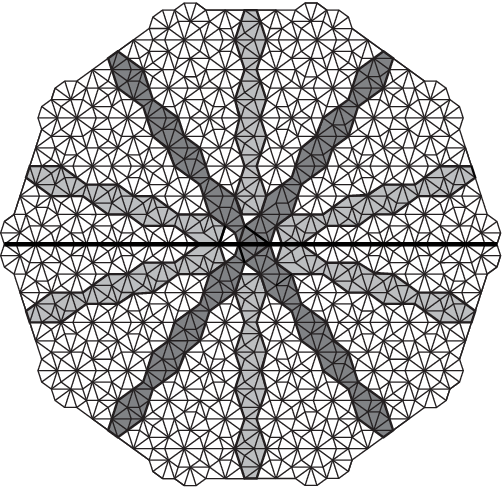}
\caption{Cartwheel tiling}
\label{fig:cartwheelSmall}
\end{figure}

\subsection{A simple group of intermediate growth}
\label{ss:simpleintermediate}

Consider again the subshift $\F\subset\{0, 1\}^\omega$ defined by the prohibited word 11. 
Let us define transformations $b_0, c_0, d_0$ of $\F$ by:
\begin{alignat*}{3}
b_0(00w)&=10w,&\qquad b_0(10w)&=00w,&\qquad b_0(010w)&=010b_0(w),\\
c_0(00w)&=10w,&\qquad c_0(10w)&=00w,&\qquad c_0(010w)&=010d_0(w),\\
d_0(00w)&=00w,&\qquad d_0(10w)&=10w,&\qquad d_0(010w)&=010b_0(w).
\end{alignat*}

Let us denote 
\[c_2=S_0^{-1}b_0S_0,\qquad d_2=S_0^{-1}c_0S_0,\qquad b_2=S_0^{-1}d_0S_0.\]
Their domains and ranges are all equal to $I_1=S_1S_1^{-1}$, and the only non-empty sections are
\[b_1=S_1^{-1}b_2S_1,\qquad c_1=S_1^{-1}c_2S_1,\qquad d_1=S_1^{-1}d_2S_1.\]
Domains and ranges of these homeomorphisms are all equal to $I_0$ and the only non-trivial sections are
\[b_0=S_0^{-1}b_1S_0,\qquad c_0=S_0^{-1}c_1S_0,\qquad d_0=S_0^{-1}d_1S_0.\]
In other words, we get the automaton shown in Figure~\ref{fig:simplegroup}.

\begin{figure}
\centering
\includegraphics{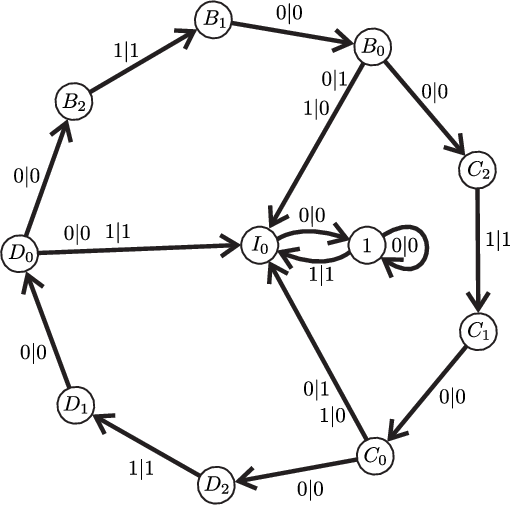}
\caption{A simple group of intermediate growth}
\label{fig:simplegroup}
\end{figure}

The inverse semigroup generated by the defined transformations is self-simiar and contracting.
It was shown in~\cite{nek:burnside} that the full group of its groupoid of germs is torsion and has intermediate growth. The groupoid is expansive and minimal, hence the alternating full group of the groupoid is finitely generated and simple. It is equal to the derived subgroup of the full group and has finite index. Consequently, the alternating full group is a finitely generated simple torsion group of intermediate growth.

\section{Properties of self-similar inverse semigroups}

\subsection{Growth and almost finiteness}

Let us show that the Cayley graphs of contracting shift-invariant groupoids have polynomial growth.

\begin{theorem}
\label{th:polynomialgrowth}
Let $\Gr$ be a contracting shift-invariant groupoid. There exists $\alpha>1$ such that for every compact set $S\subset\Gr$ there exists $C_S>0$ such that 
\[|S^rx|\le C_Sr^\alpha\]
for all $r\ge 1$ and $x\in\Gr^{(0)}$. 
\end{theorem}

\begin{proof}
Let $\nuke\subset\mathcal{B}(\Gr)$ be the nucleus of $\Gr$ defined with respect to some encoding of $\Gr^{(0)}$ as a Markov subshift $\F\subset\alb^{\omega}$.

There exists $n_0$ such that $S_v^{-1}FS_u\in\nuke$ for all $F\in\nuke^2$ and all words $v, u\in\alb^{n_0}$ for which $S_v$ and $S_u$ are defined.

Let $N$ be the union of the elements of $\nuke$. Let us first prove the statement of the proposition for $S=N$. 

Consider a product $s_rs_{r-1}\ldots s_1$ of elements of $N$. We have then $\varsigma^{n_0}(s_is_{i+1})\in N$ for every $i$. Consequently,
$\varsigma^{n_0}$ maps, for every $w\in\F$, the set $N^rw$ to $N^{\lceil\frac r2\rceil}\sigma^{n_0}(w)$. The map
\[\varsigma^{n_0}\colon N^rw\arr N^{\lceil\frac r2\rceil}\sigma^{n_0}(w)\]
is at most $\delta(n_0)$-to-one, where $\delta(n_0)$ is the number of subwords of length $n_0$ of the elements of the subshift $\F$. 

Denote $\overline\gamma(r)=\max_{w\in\F}|N^rw|$. Then the arguments of the previous paragraph imply that
\[\overline\gamma(r)\le\delta(n_0)\overline\gamma\left(\left\lceil\frac r2\right\rceil\right),\]
hence
\[\overline\gamma(r)\le\delta(n_0)^{\lceil\log_2r\rceil}\overline\gamma(2)\le |\nuke|^2\delta(n_0)r^{\log_2\delta(n_0)}.\]

Let $S$ be an arbitrary compact subset of $\Gr$. There exists $n_S$ such that $\varsigma^{n_S}(S)\subset N$. The map $\varsigma^{n_S}$ is an at most $\delta(n_S)$-to-one map from $S^rw$ to $N^r\sigma^{n_S}(w)$ for all $w\in\F$. Consequently,
\[|S^rw|\le\delta(n_S)\overline\gamma(r),\]
which finishes the proof.
\end{proof}

Using Proposition~\ref{pr:subexponentialamenable}, we get the following.

\begin{corollary}
Every contracting shift-invariant groupoid is amenable. In particular, their full and reduced $C^*$-algebras coincide.
\end{corollary}

\begin{proposition}
\label{pr:polcomplexity}
Let $\Gr$ be a contracting shift-invariant groupoid generated by a finite set $\mathcal{S}\subset\mathcal{B}(\Gr)$. Then the complexity function $\delta(n)$ (see Definition~\ref{def:complexity}) is bounded from above by a polynomial of $n$.
\end{proposition}

\begin{proof}
Fix an encoding of $\Gr^{(0)}$ as a Markovian subshift $\F\subset\xo$, and let $\nuke$ be the corresponding nucleus. By the argument in the proof of Theorem~\ref{th:polynomialgrowth}, there exists $C>1$ such that if $v, u\in\alb^m$ for $m\ge C_1\log_2 R$ and $F$ is a product of at most $R$ elements of $\mathcal{S}\cup\mathcal{S}^{-1}$, then $S_u^{-1}FS_v\in\nuke$. 

The ball $B_x(R)$ of radius $R$ with center in $x$ in the graph $\G_x(\Gr, \mathcal{S})$ is determined by a finite set of conditions of one of the following type:
\begin{enumerate}
\item $x\notin\be(F)$;
\item $F_1x=F_2x$;
\item $F_1x\ne F_2x$,
\end{enumerate}
where $F, F_1, F_2$ are products of at most $R$ elements of $\mathcal{S}\cup\mathcal{S}^{-1}$.

\begin{lemma}
There exists $k_1, C>1$ such that for every product $F=F_1F_2\ldots F_n$ of elements of $\mathcal{S}\cup\mathcal{S}^{-1}$, $w=x_1x_2\ldots\in\F$, and $k\ge C\log_2n+C$ the first $k$ letters of $F(w)=y_1y_2\ldots$ (or the fact $F(w)$ is not defined) and the bisection $S_{y_1\ldots y_k}^{-1}FS_{x_1x_2\ldots x_k}$ depend only on the first $k+k_1$ letters of $w$.
\end{lemma}

\begin{proof}
Let $k_1$ be such that $\be(F)$ for every $F\in\nuke$ is a union of cylindrical sets ${}_v\F$ for $v\in\alb^{k_1}$.

By the same arguments as in the proof of Theorem~\ref{def:complexity}, there exist a constant $C>0$ such that for every product $F=S_1S_2\ldots S_n$ of length $n$ of elements of $\mathcal{S}\cup\mathcal{S}^{-1}$, for every $k\ge C\log_2(n)+C$, and every $u, v\in\alb^k$, the bisections $S_u^{-1}FS_v$ belong to $\nuke$.

Then 
\[F=\bigcup_i S_{u_i}F_iS_{v_i}^{-1}\]
for some collection $F_i\in\nuke$ and $u_i, v_i\in\alb^k$. 
Note that if $v_{i_1}=v_{i_2}$, then $\be(F_{i_1})$ and $\be(F_{i_2})$ are disjoint. Therefore, for every word $v$ of length $k+k_1$ there is at most one index $i$ such that $v=v_iv'$ for some $v'\in\xs$ and ${}_{v'}\F\subset\be(F_i)$. Then for very infinite continuation $vw$ of $v$, the sequence $F(vw)$ starts with $u_i$ and $S_{u_i}^{-1}FS_{v_i}=F_i\in\nuke$.
\end{proof}

It follows from the lemma that there exist constants $C_1, C_2$ such that each of the conditions (1)--(3) describing the ball $B_w(R)$ depends only on the beginning of length $k\le C_1\log_2R+C_2$ of $w$ and the isomorphism class of the ball of radius $1$ with center in $\varsigma^k(w)$ in the graph $\G_{\varsigma^k(w)}(\Gr, \nuke)$. Since there are only finitely many such isomorphism classes, the number of the isomorphism classes of balls of radius $R$ in $\G_w(\Gr, \mathcal{S})$ is at most $C_3|\alb|^{C_1\log_2 R+C_2}$, which is a polynomial estimate in $R$.
\end{proof}

We get the following corollary of Theorem~\ref{th:polynomialgrowth}, Proposition~\ref{pr:polcomplexity} and Theorem~\ref{th:growthofalgebras}.

\begin{corollary}
Let $\Gr$ be a contracting groupoid. Then there exists $\alpha>1$ such that every finitely generated sub-algebra of $\Bbbk\Gr$ has growth bounded from above by $Cn^\alpha$.
\end{corollary}

We say that the edge shift defined by a finite directed graph $\G$ is \emph{primitive} if there exists $m$ such that for every pair of vertices $\alpha, \beta$ of $\G$ there exists a directed path of length $m$ from $\alpha$ to $\beta$.

\begin{theorem}
\label{th:contractingalmostfinite}
Let $\si\colon \F\arr\F$ be a  primitive shift of finite type, and let $\Gr$ be a contracting $\si$-invariant groupoid. Then $\Gr$ is almost finite.
\end{theorem}

\begin{proof}
Consider an encoding of $\F$ as the edge shift for a directed graph $\G$ with set of vertices $\mathsf{V}$ and set of edges $\alb$.  
Let $\nuke$ be the nucleus defined by the encoding. 

\begin{lemma}
\label{lem:regularnbhd}
There exists a non-empty clopen set $U\subset\F$ satisfying the following conditions. 
\begin{enumerate}
\item For every $F\in\nuke$ we have either $U\cap\be(F)=\emptyset$ or $U\subset\be(F)$.
\item If $F_1, F_2\in\nuke$ are such that $U\subset\be(F_1)\cap\be(F_2)$, then either $F_1U=F_2U$ or $\en(F_1U)\cap\en(F_2U)=\emptyset$.
\end{enumerate}
\end{lemma}

Note that since $\nuke$ contains an idempotent $I$ such that $U\subset I$, the conclusion of the lemma implies that for every $F\in\nuke$ such that $U\subset\be(F)$ we have either $FU=U$ or $\en(FU)\cap U=\emptyset$.

\begin{proof}
Pick a point $w\in\F$. Since the elements of $\nuke$ are compact open bisections, there exists a clopen neighborhood $U_0$ of $w$ such that for every $F\in\nuke$ we have $U_0\subset\be(F)$ or $U_0\cap\be(F)=\emptyset$. Consider a pair of elements $F_1, F_2\in\nuke$ such that  $U_0\subset\be(F_1)\cap\be(F_2)$. Let $D_{F_1, F_2}\subset U_0$ be the set of points $u\in U_0$ such that $\en(F_1u)\ne\en(F_2u)$ or $F_1u=F_2u$.  The set $D_{F_1, F_2}$ is open. Since $\Gr$ is effective, if $u\in U_0\setminus D_{F_1, F_2}$, then $F_1u\ne F_2u$, i.e., $(F_1u)^{-1}F_2u$ is not a unit, hence for every neighborhood $U_1$ of $u$ there exists $u'\in U_1$ such that $\en(F_1u')\ne\en(F_2u')$ (see Definition~\ref{def:effective}). Consequently, the set $D_{F_1, F_2}$ is an open dense subset of $U_0$. Since $\nuke$ is a finite set, the intersection of the sets $D_{F_1, F_2}$ for all pairs $F_1, F_2$ is an open dense set. Let $w'$ be point in the intersection. 

Then for every $F_1, F_2\in\nuke$ we have either $F_1w'=F_2w'$ or $\en(F_1w')\ne\en(F_2w')$. In the first case, since $F_i$ are open bisections, there is a neighborhood $U'$ of $w'$ such that $F_1U'=F_2U'$. In the second case, there exists a neighborhood $U'$ of $w'$ such that $\en(F_1w')\cap\en(F_2w')=\emptyset$. Take the intersection $U$ of all such neighborhoods of $w'$ for all pairs $F_1, F_2\in\nuke$. Then $U$ will satisfy the conditions of the lemma.
\end{proof}

Let $U\subset\F$ satisfy the conditions of the lemma. There exists a natural number $k$ be such that the set $U$ and all sets of the form $\en(FU)$ for $F\in\nuke$ are unions of the cylindrical sets ${}_v\F$ for $v\in\alb^k$.

Let $\mathfrak{O}_{\Gr}$ be the groupoid generated by $\Gr$ and the shift $\sigma$. Consider the set $\mathcal{S}_0$ of non-empty $\mathfrak{O}_{\Gr}$-bisections $S_v$ for all $v\in\alb^k$ such that ${}_v\F\cap\bigcup_{F\in\nuke}\en(FU)=\emptyset$. Let $\mathcal{S}_1$ be the set of all non-empty bisections $S_v$ for $v\in\alb^k$ such that ${}_v\F\subset U$. Let $\mathcal{S}_2$ be the set of all non-empty bisections of the form $FS_v$ for $v\in\alb^k$ such that ${}_v\F\subset U$ and $U\cap\en(FU)=\emptyset$. Let $\mathcal{S}=\mathcal{S}_0\cup\mathcal{S}_1\cup\mathcal{S}_2$. 

\begin{lemma}
\label{lem:Scomposition}
Ranges of the elements of $\mathcal{S}$ are pairwise disjoint and form a partition of $\F$. 
For every pair $S_1, S_2\in\mathcal{S}$ we have either $\en(S_2)\cap\be(S_1)=\emptyset$ or $\en(S_2)\subset\be(S_1)$.
\end{lemma}

\begin{proof}
By the conditions of Lemma~\ref{lem:regularnbhd} and the definition of the set $\mathcal{S}$, if $S_1\in\mathcal{S}_0$ and $S_2\in\mathcal{S}_1\cup\mathcal{S}_2$, then $\en(S_1)$ and $\en(S_2)$ are disjoint. Similarly, if $S_1\in\mathcal{S}_1$ and $S_2\in\mathcal{S}_2$, then $\en(S_1)\subset U$ and $\en(S_2)$ are disjoint.

If $S_1=F_1S_{v_1}$ and $S_2=F_2F_{v_2}$ are elements of $\mathcal{S}_2$, then, by the conditions of Lemma~\ref{lem:regularnbhd}, the bisections $F_1U$ and $F_2U$ either have disjoint ranges or are equal. In the first case $\en(S_1)\cap\en(S_2)=\emptyset$. In the second case, if $v_1\ne v_2$, then the ranges of $S_{v_1}$ and $S_{v_2}$ are disjoint subsets of $U$. If $v_1=v_2$, then $S_1=F_1US_{v_1}=F_2US_{v_2}=S_2$. 

This proves that the ranges of the elements of $\mathcal{S}$ are pairwise disjoint. The fact that they cover $\F$ follows directly from the definition of the set $\mathcal{S}$.

The domain of $S_v$ for $v\in\alb^k$ is equal to the set ${}_\alpha\F$ of paths in $\G$ starting in the final vertex $\alpha\in\mathsf{V}$ of the path $v$. Since for every non-empty $FS_v\in\mathcal{S}_2$ we have $\en(S_v)={}_v\F\subset U\subset\be(F)$, the domain of $FS_v$ is the same as the domain of $S_v$. Hence, the set of domains of the elements of $\mathcal{S}$ are the same as the set of domains of the nonempty bisections  $S_v$ for $v\in\alb^k$. Since the range of every element of $\nuke$ is a subset of a set ${}_\alpha\F$ for $\alpha\in\mathsf{V}$, the range of $FS_v$, as a subset of the range of $F$, is also contained in a set ${}_\alpha\F$ for $\alpha\in\mathsf{F}$.
\end{proof}

Let $\tilde\G$ be the graph with the set of vertices equal to the set of vertices $\mathsf{V}$ of $\G$ and the set of edges equal to the set of paths of length $k$ in $\G$ with the natural source and range maps. In other words, $\tilde\G$ is the graph with the adjacency matrix equal to the $k$th power of the adjacency matrix of $\G$. Since we assume that $\G$ is primitive, if $k$ is large enough, all entries of this matrix are positive, i.e., for any ordered pair of vertices $\alpha, \beta\in\mathsf{V}$ there exists an edge $v\in\alb^k$ such that $\be(v)=\alpha$ and $\en(v)=\beta$.

Let ${\tilde\G}'$ be the graph with the same set of vertices $\mathsf{V}$ and set of edges in a bijection with $\mathcal{S}$, where an element $S\in\mathcal{S}$ corresponds to an arrow starting in $\alpha$ such that $\be(S)\subset{}_\alpha\F$ and ending in $\beta$ such that $\en(S)\subset{}_\beta\F$.

Let $\alpha_0\in\mathsf{V}$ be such that $U\subset{}_{\alpha_0}\F$. It exists, since $U$ is a subset of the source of an element of $\nuke$.

If $S=S_v\in\mathcal{S}_0\cup\mathcal{S}_1$, then the corresponding edge of ${\tilde\G}'$ connects the same vertices as the edge $v$ does in $\tilde\G$. If $S=FS_v\in\mathcal{S}_0\cup\mathcal{S}_1$, then the corresponding edge starts in $\alpha_0$. Consequently, the sets of arrows starting in the vertices different from $\alpha_0$ in the graphs $\tilde\G$ and ${\tilde\G}'$ are the same. We may assume that $U$ is strictly smaller than ${}_{\alpha_0}\F$, so there will be also arrows in ${\tilde\G}'$ starting in $\alpha_0$ and ending in the same vertices, different from $\alpha_0$, as  in $\tilde\G$. It follows that ${\tilde\G}'$ is strongly connected.

A product $S_1S_2\cdots S_m$ of elements of $\mathcal{S}$ is non-empty if and only if it corresponds to a path in the graph ${\tilde\G}'$. It follows from Lemma~\ref{lem:Scomposition} that for any two different paths in ${\tilde\G}'$ the corresponding products have disjoint ranges.

A product $(S_1S_2\cdots S_m)(S_1'S_2'\cdots S_m')^{-1}$ is non-empty if and only if the paths corresponding to the products $S_1S_2\cdots S_m$ and $S_1'S_2'\cdots S_m'$ exist in ${\tilde\G}'$ and start in the same vertex. Since the domain and the range of the product is equal to the ranges of the products $S_1'S_2'\cdots S_m'$ and $S_1S_2\cdots S_m$, the product uniquely determines the corresponding pair of paths.

Let $\mathfrak{H}_m$ be the subgroupoid of $\Gr$ equal to the union of bisections of the form $S_1S_2^{-1}$ for $S_1, S_2\in\mathcal{S}^m$.
It follows from the above that there exist $C_1, C_2>0$ such that for every $w\in\F$ we have
\[C_1\lambda^m\le |\mathfrak{H}_mw|\le C_2\lambda^m,\]
where $\lambda$ is the spectral radius (the Perron eigenvalue) of the adjacency matrix of ${\tilde\G}'$.

Let $N$ be the union of the elements of $\nuke$. We assume that $k$ is large enough so that $\varsigma^k(N^3)\subset N$.

Let $g\in N$ and $S_1, S_2\in\mathcal{S}$ are such that $gS_1$ is defined and $\en(g)\in\en(S_2)$. Then $S_i=F_iS_{v_i}$ for some $v_i\in\alb^k$ and $F\in\nuke$ (recall that idempotents in $\nuke$ are the subsets of $\F$ associated with vertices of $\G$, so every $S_{v_i}=IS_{v_i}$ for some idempotent $I\in\nuke$). Then $S_2^{-1}gS_1=S_{v_2}^{-1}F_2^{-1}gF_1S_{v_1}=\varsigma^k(F_2^{-1}gF_1)\in N$.
Consequently, $gS_1=S_2h$ for some $h\in N$.

Suppose now that $S_1\in\mathcal{S}_1$. Let $F\in\nuke$ be such that $g\in F$. Then $U\subset\be(F)$ and $gS_1\in FS_1\in\mathcal{S}_2$. Hence, 
$g S_1=S_2$ for some $S_2\in\mathcal{S}$.

We see that for every $g\in  N$ and $S_1\in\mathcal{S}$ the product $gS_1$ is of the form $S_2h$ for $h\in N$ and $S_2\in\mathcal{S}$, and if $S_1\in\mathcal{S}_1$, then $h$ is a unit.

Consider a product $S_1S_2\ldots S_m$ of elements of $\Sigma$. By the above, if at least one of the factors $S_i$ belongs to $\mathcal{S}_1$, then $gS_1S_2\ldots S_m$ is an element of a product of $m$ elements of $\mathcal{S}$. Consequently, for any such a product, we have $N\cdot S_1S_2\ldots S_m(S_1'S_2'\ldots S_m')^{-1}\subset\Hr_m$ for every $S_1'S_2'\ldots S_m'$.

By~\cite[Theorem~4.4.7]{marcus} the number of paths of length $m$ in ${\tilde\G}'$ not containing an edge corresponding to elements of $\mathcal{S}_1$ is less than $C_2\lambda_1^m$ for some $C_2>0$ and $\lambda_1<\lambda$.

It follows that
\[\lim_{m\to\infty}\frac{|N\Hr_mw\setminus\Hr_mw|}{|\Hr_mw|}\le\lim_{m\to\infty}\frac{|\nuke|C_2\lambda_1^m}{C_1\lambda^m}=0\]
for all $w\in\F$, 
which verifies the condition of Definition~\ref{def:almostfinite} for the set $N=\bigcup_{F\in\nuke}F$.

Let now $K$ be an arbitrary compact subset of $\Gr$. There exists $n_0$ such that $\varsigma^k(K)\subset N$. We can take then the elementary subgroupoids $\Hr_{n_0, m}$ equal to the union of bisections of the form $S_{v_1}S_1(S_{v_2}S_2)^{-1}$ for $v_1, v_2\in\alb^{n_0}$ and $S_1, S_2\in\mathcal{S}^m$. We will have then
\[\frac{K\cdot\Hr_{n_0, m}w\setminus K\cdot\Hr_{n_0, m}w}{K\cdot\Hr_{n_0, m}w}\le C_K\frac{N\cdot\Hr_mw\setminus K\cdot\Hr_mw}{N\cdot\Hr_mw}\]
for some constant $C_K$, hence the condition of Definition~\ref{def:almostfinite} holds also for the set $K$.
\end{proof}

\subsection{Algebras}

Let $\Gr$ be a completely shift-invariant groupoid with the unit space $\Gr^{(0)}$ identified with a shift of finite type $\si\colon\F\arr\F$. Denote by $\mathfrak{O}_{\Gr}$ the subgroupoid of germs generated by $\Gr$ and the shift $\si$. 

The algebra $\Bbbk\mathfrak{O}_{\Gr}$ is generated by the elements $S_x, S_x^{-1}$ satisfying the Cuntz-Krieger relations for the subshift $\F=\Gr^{(0)}$ (see~\ref{ss:CuntzKrieger}) and by the indicators of the elements $F\in\nuke$. We will simplify the notation by writing $F$ instead of $1_F$ for $F\in\nuke$.

As before, we are using the notation $S_{x_1x_2\ldots x_n}=S_{x_1}S_{x_2}\cdots S_{x_n}$ and $S_{x_1x_2\ldots x_n}^{-1}=S_{x_n}^{-1}\cdots S_{x_2}^{-1}S_{x_1}^{-1}$.

\begin{theorem}
\label{th:convolutionfinpres}
The algebra $\Bbbk\mathfrak{O}_{\Gr}$ is isomorphic to the $\Bbbk$-algebra defined by the presentation with generators $S_x, S_x^{-1}$, for $x\in\alb$, and $F\in\nuke$ with the following defining relations (holding in $\Bbbk\Gr$).

Cuntz-Krieger relations:
\begin{equation}
\label{eq:CK1}
S_x^{-1}S_x=\sum_{xy\in T}S_yS_y^{-1},\qquad \sum_{y\in\alb}S_yS_y^{-1}=1,\qquad S_a^{-1}S_b=0 
\end{equation}
for all $x, a, b\in\alb$, $a\ne b$.

Description of the nucleus:
\begin{equation}
\label{eq:nuc1}
F=\sum_{x, y\in\alb} S_yF_{x, y}S_x^{-1},
\end{equation}
for $F_{x, y}=S_y^{-1}FS_x\in\nuke$.

Decompositions of pairwise products of elements of the nucleus:
\begin{equation}
\label{eq:nuc2}
F_1F_2=\sum_{v, u\in\alb^{n_0}}S_uF_{v, u}S_v^{-1},
\end{equation}
for all $F_1, F_2\in\nuke$, for a fixed $n_0\ge 1$, and for some $F_{v, u}=S_u^{-1}F_1F_2S_v\in\nuke$.
\end{theorem}

\begin{proof}
It is checked directly that the algebra $\Bbbk\mathfrak{O}_{\Gr}$ satisfies the relations listed in the theorem.

Let $\mathcal{A}$ be the algebra given by the presentation. 

The Cuntz-Krieger relations imply $S_x=\sum_{y\in\alb}S_yS_y^{-1}S_x=S_xS_x^{-1}S_x$ and $S_x^{-1}=\sum_{y\in\alb}S_x^{-1}S_yS_y^{-1}=S_x^{-1}S_xS_x^{-1}$ in $\mathcal{A}$.

If $F_1=S_y^{-1}F_2S_x$ for $F_1, F_2\in\nuke$ holds in $\mathfrak{O}_{\Gr}$, then $F_1S_x^{-1}S_x=F_1$, since $S_xS_x^{-1}S_x=S_x$. 

Let us show that an equality of the form $FS_x^{-1}S_x=F$ for $F\in\nuke$, if it holds in $\Bbbk\mathfrak{O}_{\Gr}$, will follow from~\eqref{eq:nuc1} and the Cuntz-Krieger relations.  We can rewrite $S_x^{-1}S_x$ as $\sum_{xy\in T}S_yS_y^{-1}$, and then write $FS_x^{-1}S_x=\sum_{a, b\in\alb, xy\in T}S_bF_{a, b}S_a^{-1}S_yS_y^{-1}$.  It follows from the Cuntz-Krieger relations that $S_a^{-1}S_yS_y^{-1}$ is equal to $S_y^{-1}$ if $a=y$ and to $0$ otherwise. Consequently, $FS_x^{-1}S_x=\sum_{b\in\alb, xy\in T}S_bF_{y, b}S_y^{-1}$ in $\mathcal{A}$. The right-hand side of this equality is the same as the right-hand side of the equality $\sum_{x, y\in\alb}S_xF_{y, x}S_y^{-1}$ (up to ignoring zero summands and order of summands), if $FS_x^{-1}S_x=F$ in $\Bbbk\Gr$. Consequently, $FS_x^{-1}S_x=F$ in $\mathcal{A}$.

It follows that the elements $F_{x, y}$ in~\eqref{eq:nuc1} satisfy $F_{x, y}S_x^{-1}S_x=F_{x, y}$ in $\mathcal{A}$. Consequently, using $S_a^{-1}S_x=0$ for $a\ne x$, we get the following relations in $\mathcal{A}$:
\begin{equation}
\label{eq:Sright}
FS_x=\sum_{y\in\alb}S_yF_{x, y}.
\end{equation}

Similar arguments show that
\begin{equation}
\label{eq:Ssleft}
S_y^{-1}F=\sum_{x\in\alb} F_{x, y}S_x^{-1}
\end{equation}
holds in $\mathcal{A}$.

Consequently, using the Cuntz-Krieger relations and relations~\eqref{eq:Sright}, \eqref{eq:Ssleft}, we can push in every product of generators of $\mathcal{A}$ all generators $S_x$ to the left, and all generators $S_x^{-1}$ to the right, and thus show that every product of generators is equal to a sum of products of the form $S_uF_1F_2\cdots F_nS_v^{-1}$ for some $F_i\in\nuke$ and $u, v\in\xs$.

Using repeatedly the relation~\eqref{eq:nuc2}, we can rewrite any product $S_uF_1F_2\cdots F_nS_v^{-1}$ as a sum of products of the form $S_{w_2}FS_{w_1}^{-1}$ for $w_1, w_2\in\xs$ and $F\in\nuke$.

It follows that any equality of products of the generators that hold in $\mathfrak{O}_{\Gr}$ also hold in $\mathcal{A}$.

The map $\nu:\mathfrak{O}_{\Gr}\arr\Z$ equal to $|v|-|u|$ on $S_vFS_u^{-1}$ for $v, u\in\xs$ and $F\subset\Gr$ is a continuous cocycle. The algebra $\Bbbk\mathfrak{O}_{\Gr}$ is a direct sum of subspaces $A_n$, $n\in\Z$, of functions supported on the fibers $\nu^{-1}(n)$. The subspaces satisfy $A_{n_1}A_{n_2}\subset A_{n_1+n_2}$, i.e., define a $\Z$-grading on $\Bbbk\mathfrak{O}_{\Gr}$.

Let $f$ be an element of $\mathcal{A}$ mapped to the zero element of $\Bbbk\mathfrak{O}_{\Gr}$ by the canonical epimorphism $\pi\colon \mathcal{A}\arr\Bbbk\mathfrak{O}_{\Gr}$. 
By the arguments above, we can write $f$ as a linear combination $\sum_{i=1}^m \alpha_iS_{u_i}F_iS_{v_i}^{-1}$ for some $u_i, v_i\in\xs$ and $F_i\in\nuke$.
If we split the sum into parts according to the value of $|u_i|-|v_i|$, then the image of each of the parts in $\Bbbk\mathfrak{O}_{\Gr}$ will be homogeneous with respect to the grading, hence mapped to $0$ by $\pi$. Consequently, it is enough to consider the case when $|u_i|=|v_i|$ for every $i$.

Using relations~\eqref{eq:nuc1}, we may assume that $u_i, v_i\in\alb^n$ for some fixed $n$. Then $f=\sum_{u, v\in\alb^n}S_uf_{u, v}S_v^{-1}$ for some $f_{u, v}$ belonging to the linear span of $\nuke$. For any $F_1, F_2\in\nuke$ $u_1, u_2, v_1, v_2\in\xs$ the bisections $S_{u_1}F_1S_{v_1}^{-1}$ and $S_{u_2}F_2S_{v_2}^{-1}$ have disjoint sources or disjoint ranges, unless $u_1=u_2$ and $v_1=v_2$. It follows that $\pi(f)=0$ if and only if $\pi(f_{u, v})=0$ for all $u, v\in\alb^n$. 

\begin{lemma}
\label{lem:linearrelnucleus}
Any linear relation
\begin{equation}
\label{eq:linnuc}
\sum_{F\in\nuke}\alpha_FF=0
\end{equation}
between the elements of $\nuke$ in $\Bbbk\Gr$ follows from the defining relations of the algebra $\mathcal{A}$.
\end{lemma}

\begin{proof}
Suppose that $\alpha_F\in\Bbbk$ are such that~\eqref{eq:linnuc} holds.

Let $x_1x_2\ldots\in\F$. Since the elements of $\nuke$ are compact open, and $\be\colon\mathfrak{O}_{\Gr}\arr\F$ is an open continuous map, for every subset $A\subset\nuke$ the set $\be\left(\bigcap_{F\in A}F\right)\subset\F$ is compact and open. Consequently, there exists $n$ such that for every every  subset $A\subset\nuke$, the set $\be\left(\bigcap_{F\in A}\be(F)\right)$ is either disjoint with $x_1x_2\ldots x_n\xo\cap\F$ or contains it. 

By compactness of $\F$, we can find then $n$ such that for every $v\in\alb^n$ and every subset $A\subset\F$ the set $\be\left(\bigcap_{F\in A}F\right)$ is either disjoint with $v\xo\cap\F$ or contains it. 

Consider the element $\sum_{F\in\nuke}\alpha_FF S_vS_v^{-1}$ of $\mathcal{A}$. If the bisection $FS_vS_v^{-1}$ is empty, then $FS_vS_v^{-1}=0$ in $\mathcal{A}$. Otherwise, by the choice of $v$, we have $v\xo\cap\F\subset\be(F)$, hence $\be(FS_vS_v^{-1})=v\xo\cap\F$. Choose $x_1x_2\ldots\in v\xo\cap\F$, and let $g_1, g_2, \ldots, g_k$ be all germs of elements of $\nuke$ in $x_1x_2\ldots$. For every $F\in\nuke$ such that $v\xo\cap\F\subset\be(F)$ there will exist exactly one germ $g_i$ such that $g_i\in F$. Let us split the sum 
\[f=\sum_{F\in\nuke}\alpha_FF S_vS_v^{-1}
=\sum_{F\in\nuke, v\xo\cap\F\subset\be(F)}\alpha_FF S_vS_v^{-1}=\sum_{i=1}^k
\sum_{F\in\nuke, g_i\in F}\alpha_FF S_vS_v^{-1}.\] 
By the choice of $v$ and $g_i$, the set $\be\left(\bigcap_{F\in\nuke, g_i\in F} F\right)$ contains $v\xo\cap\F$, hence $FS_vS_v^{-1}=\bigcap_{F\in\nuke, g_i\in F}F=H_i$ for all $F\in\nuke$ such that $g_i\in F$, and the equalities  $FS_vS_v^{-1}=H_i$ hold in $\mathcal{A}$. Consequently, we have 
\[\sum_{F\in\nuke, g_i\in F}\alpha_FF S_vS_v^{-1}=\left(\sum_{F\in\nuke, g_i\in F}\alpha_F\right)H_i=0,\]
since the value of the $f$ as a function on $\Gr$ at $g_i$ is equal to $\sum_{F\in\nuke, g_i\in F}\alpha_F$, so it is equal to 0, since $f=0$ in $\Bbbk\mathfrak{O}_{\Gr}$. 

We see that $f=0$ in $\mathcal{A}$. Since $\sum_{F\in\nuke}\alpha_FF=\sum_{F\in\nuke}\alpha_FF\sum_{v\in\alb^n}S_vS_v^{-1}$ in $\mathcal{A}$, we conclude that the relation~\eqref{eq:linnuc} holds in $\mathcal{A}$.
\end{proof}
This finishes the proof of the theorem.
\end{proof}

As an example, consider the groupoid $\mathfrak{R}$ associated with the golden mean rotation described in Section~\ref{s:rotation}.

\begin{proposition}
The algebra $\Bbbk\mathfrak{O}_{\mathfrak{R}}$ is isomorphic to the $\Bbbk$-algebra defined by the presentation with generators
\[S_0, S_1, S_0^{-1}, S_1^{-1}, R_1, R_1^{-1}\] and relations
\begin{gather*}
S_1^{-1}S_1=S_0S_0^{-1},\quad S_0^{-1}S_0=S_0S_0^{-1}+S_1S_1^{-1}=1,\quad S_0^{-1}S_1=S_1^{-1}S_0=0,\\
R_1=S_{00}S_{01}^{-1}+S_0R_1^{-1}S_1^{-1},\\
R_1^{-1}=S_{01}S_{00}^{-1}+S_1R_1S_0^{-1},\\
R_1^{-1}R_1=S_{01}S_{01}^{-1}+S_1S_1^{-1}.
\end{gather*}
\end{proposition}

\begin{proof}
According to Theorem~\ref{th:convolutionfinpres}, it is enough to check that the decompositions of the products of the elements of the nucleus $\{1, I_0, R_1, R_1^{-1}, R_0, R_0^{-1}\}$ follow from the listed relations.

The equality $R_1^{-1}R_1=S_{01}S_{01}^{-1}+S_1S_1^{-1}$ is one of the defining relations.

We have $R_1R_1^{-1}=(S_{00}S_{01}^{-1}+S_0R_1^{-1}S_1^{-1})(S_{01}S_{00}^{-1}+S_1R_1S_0^{-1})=S_{00}S_{01}^{-1}S_{01}S_{00}^{-1}+S_0R_1^{-1}S_1^{-1}S_1R_1S_0^{-1}=S_{000}S_{000}^{-1}+S_0R_1^{-1}S_0S_0^{-1}R_1S_0^{-1}=S_{000}S_{000}^{-1}+S_0R_1^{-1}R_1S_0^{-1}=S_{000}S_{000}^{-1}+S_0(S_{01}S_{01}^{-1}+S_1S_1^{-1})S_0^{-1}=S_{000}S_{000}^{-1}+S_{001}S_{001}^{-1}+S_{01}S_{01}^{-1}=S_{00}S_{00}^{-1}+S_{01}S_{01}^{-1}=S_0S_0^{-1}$.

We also have
$R_1R_1=(S_{00}S_{01}^{-1}+S_0R_1^{-1}S_1^{-1})(S_{00}S_{01}+S_0R_1^{-1}S_1^{-1})=S_{00}S_{01}^{-1}S_{00}S_{01}^{-1} +S_{00}S_{01}^{-1}S_0R_1^{-1}S_1^{-1}=S_{00}S_1^{-1}S_0S_{01}^{-1} +S_{00}S_1^{-1}R_1^{-1}S_1^{-1}=S_{00}R_1S_{10}^{-1}$.

Similar computations show that $R_1^{-1}R_1^{-1}=S_{10}R_1^{-1}S_{00}^{-1}$.

Decompositions of the remaining products of elements of the nucleus are straightforward.
\end{proof}

The proof of Theorem~\ref{th:convolutionfinpres} works also for the quotient of $\Bbbk\mathfrak{O}_{\Gr}$ by the essential ideal, except for Lemma~\ref{lem:linearrelnucleus}, which is not true in general for the quotient. Consequently, we get the following description of the simple quotient of $\Bbbk\mathfrak{O}_{\Gr}$.

\begin{proposition}
The essential ideal of $\Bbbk\mathfrak{O}_{\Gr}$ is generated by the intersection of the essential ideal with the span $\Bbbk\nuke$ of the nucleus. In particular, the quotient of $\Bbbk\mathfrak{O}_{\Gr}$ is isomorphic to the algebra given by the presentation described in Theorem~\ref{th:convolutionfinpres} with addition of a finite number of elements of $\Bbbk\nuke$.
\end{proposition}

\subsection{Full groups}

Let $\Gr$ be a shift-invariant contracting groupoid acting on a shift of finite type $\si\colon \F\arr\F$. Let $\mathfrak{O}_{\Gr}$ be the groupoid generated by $\Gr$ and the shift.

By Theorem~\ref{th:contractingalmostfinite}, the groupoid $\Gr$ is almost finite (if the associated subshift is primitive), hence, by Theorem~\ref{th:almostfinitealternating}, the derived subgroup of the full group $\mathsf{F}(\Gr)$ is equal to the alternating full group $\mathsf{A}(\Gr)$. The groupoid $\mathfrak{O}_{\Gr}$ is purely infinite, so the same is true for its full group, by Theorem~\ref{th:purelyinfinitealternating}.

The full groups of $\mathfrak{O}_{\Gr}$ for the case of self-similar groups was studied in~\cite{nek:fullgr}, where it was proved that $\mathsf{F}(\mathfrak{O}_{\Gr})$ is finitely presented. Before, several cases of such full groups were studied by K.~Roever (see~\cite{roever:comm,roever}).

The following theorem is a particular case of a more general result of~\cite{BelkBleakMatucciZaremsky:hyperbolic}.

\begin{theorem}
The full group $\mathsf{F}(\mathfrak{O}_{\Gr})$ is finitely presented.
\end{theorem}

Note that it follows from Theorem~\ref{th:reconstruction}  that the group $\mathsf{F}(\mathfrak{O}_{\Gr})$ is a complete invariant of the groupoid $\Hr$. 

Full groups of shift-invariant groupoids, on the other hand, are rarely finitely presented. They are often sources of groups with interesting finiteness properties like amenability or intermediate growth (see, for example~\ref{ss:simpleintermediate}).

\subsection{Homology of contracting groupoids}

\subsubsection{Multi-dimensional nucleus and homology}

Let $\Gr$ be a contracting completely shift-invariant groupoid with the space of units $\F$ equal to the edge-shift of a graph $\G$ with the set of vertices $\mathsf{V}$ and set of edges $\alb$. 

We denote by $\alb_n\subset\alb^n$ the set of all finite paths of length $n$ in $\G$, i.e., the set of all subwords of length $n$ of elements of $\F$.

We denote by $I_\alpha$ the indicator of the set of paths starting in a vertex $\alpha\in V$, and by $I_w$, for $w\in\alb_m$, the indicator of the set ${}_w\F$ of paths beginning with $w$. We have $S_x^{-1}S_x=I_{\en(x)}$ and $S_xS_x^{-1}=I_x$. 

The map $\varsigma\colon\Gr\arr\Gr$ is a functor of groupoids, so it induces maps $\varsigma\colon \Gr^{(n)}\arr\Gr^{(n)}$. (We also set $\varsigma=\si$ seen as a map $\Gr^{(0)}\arr\Gr^{(0)}$.)

It follows from the definitions that $\varsigma$ commutes with the maps $d_i\colon \Gr^{(n)}\arr\Gr^{(n-1)}$.

Recall that an ordered multisection is defined by a $(F_1, F_2, \ldots, F_n)$ of compact open $\Gr$-bisections such that $\be(F_i)=\en(F_{i+1})$ for every $i=1, \ldots, n-1$, and is identified with the subset of $\Gr^{(n)}$ consisting of sequences $(g_1, g_2, \ldots, g_n)\in\Gr^{(n)}$ such that $g_i\in F_i$. 

Let $(F_1, F_2, \ldots, F_n)$ be an ordered multisection, and let $(v_0, v_1, \ldots, v_n)$ be a sequence of elements of $\alb^k$ for $k\ge 1$.
Denote then by $(F_1, F_2, \ldots, F_n)_{(v_0, v_1, \ldots, v_n)}$ the subset of $\Gr^{(n)}$ consisting of the elements $(g_1, g_2, \ldots, g_n)\in (F_1, F_2, \ldots, F_n)$ such that
\[\en(g_1)\in I_{v_0}, \be(g_1)\in I_{v_1}, \be(g_2)\in I_{v_2},\ldots, \be(g_n)\in I_{v_n}.\]
Then
\[\varsigma^k\colon (F_1, F_2, \ldots, F_n)_{(v_0, v_1, \ldots, v_n)}\arr\varsigma^k\left((F_1, F_2, \ldots, F_n)_{(v_0, v_1, \ldots, v_n)}\right)\]
is a homeomorphism between two ordered multisections.

Let us denote, by analogy with self-similar groups,
\[\varsigma^k\left((F_1, F_2, \ldots, F_n)_{(v_0, v_1, \ldots, v_n)}\right)=
(F_1, F_2, \ldots, F_n)|_{(v_0, v_1, \ldots, v_n)}.\]

In the case $n=1$, we have
\[F|_{(v_0, v_1)}=S_{v_0}^{-1}FS_{v_1}.\]

The same arguments as in the proof of Theorem~\ref{th:contracting} imply the following.

\begin{proposition}
There exists a set $\mathcal{N}^{(n)}$ of ordered multisections of length $n$ such that for every ordered multisection $(F_1, F_2, \ldots, F_n)$ there exists $m_0$ such that for every sequence $v_0, v_1, v_2, \ldots, v_d\in\alb_m$, for $m\ge m_0$, we have
\[(F_1, F_2, \ldots, F_n)|_{(v_0, v_1, \ldots, v_n)}\in\mathcal{N}^{(n)}\cup\{(\emptyset, \emptyset, \ldots, \emptyset)\}.\]
The smallest set $\mathcal{N}^{(n)}$ satisfying this condition is unique.
\end{proposition}

We call the unique smallest set $\nuke^{(n)}$ satisfying the condition of the proposition the \emph{$n$-dimensional nucleus} of $\Gr$.

The $0$-dimensional nucleus $\mathcal{N}^{(0)}$ is equal to the set of idempotent bisections $I_\alpha$ for $\alpha\in\mathsf{V}$. The $1$-dimensional nucleus $\mathcal{N}^{(1)}$ is the usual nucleus of $\Gr$.

The maps $d_i\colon \Gr^{(n)}\arr\Gr^{(n-1)}$ map ordered multisections to ordered multisections. Namely,
\[d_0(F_1, F_2, \ldots, F_n)=(F_2, F_3, \ldots, F_n),\quad d_n(F_1, F_2, \ldots, F_n)=(F_1, F_2, \ldots, F_{n-1})\]
and
\[d_i(F_1, F_2, \ldots, F_n)=(F_1, \ldots, F_iF_{i+1}, \ldots, F_n).\]

It is checked directly that 
\begin{equation}
\label{eq:disections}
d_i(U|_{(v_0, v_1, \ldots, v_n)})=(d_i(U))|_{(v_0, \ldots, \hat{v_i}, \ldots, v_n)},
\end{equation}
where $(v_0, \ldots, \hat{v_i}, \ldots, v_n)$ is the sequence obtained from $(v_0, v_1, \ldots, v_n)$ by removing the entry $v_i$.

It follows from~\eqref{eq:disections} that for every $U=(F_1, F_2, \ldots, F_n)\in\nuke^{(n)}$, the multisections $d_i(U)$ are subsets of elements of $\nuke^{(n-1)}$. In particular, every product $F_iF_{i+1}\cdots F_{i+j}$ is a subset of an element of the nucleus $\nuke^{(1)}$.

Let $A$ be an abelian group. 
Let us denote, for $a\in A$ and a multisection $U\subset\Gr^{(n)}$, by $a_U$ the function equal to $a$ on $U$ and to $0$ everywhere else. Recall that $C_c(\Gr^{(n)}, A)$ is the group of functions $\Gr^{(n)}\arr A$ generated by the elements of the form $a_U$. We have $a_\emptyset=0$.

Let $A\nuke^{(n)}$ be the subgroup of $C_c(\Gr^{(n)}; A)$ generated by $a_U$ for $U\in\nuke^{(n)}$.

Denote by $\varsigma_*\colon C_c(\Gr^{(n)}; A)\arr C_c(\Gr^{(n)}; A)$ the map given by
\[\varsigma_*(f)(g_1, \ldots, g_n)=
\sum_{(h_1, \ldots, h_n)\in\Gr^{(n)}, \varsigma(h_1, \ldots, h_n)=(g_1,  \ldots, g_n)} f(h_1, \ldots, h_n).\]
It can be also defined uniquely as the homomorphism satisfying 
\[\varsigma_*(a_U)=\sum_{(x_0, x_1, \ldots, x_n)\in\alb^n}a_{U|_{(x_0, x_1,  \ldots, x_n)}}\]
for all multisections $U\subset\Gr^{(n)}$.
We have $\varsigma_*(A\nuke^{(n)})\le A\nuke^{(n)}$.

Denote by $\mathcal{D}^{(n)}$ the direct limit of  the sequence
\[A\nuke^{(n)}\stackrel{\varsigma_*}{\arr}A\nuke^{(n)}\stackrel{\varsigma_*}{\arr}\cdots.\]

It follows from the definition of the maps $d_i$ and the fact that $\varsigma$ is a functor of groupoids that the diagram
\[
\begin{array}{ccc}
C_c(\Gr^{(n)}; A) & \stackrel{\delta_n}{\arr} & C_c(\Gr^{(n-1)}; A) \\
\mapdown{\varsigma_*} & & \mapdown{\varsigma_*}\\
C_c(\Gr^{(n)}; A) & \stackrel{\delta_n}{\arr} & C_c(\Gr^{(n-1)}; A),
\end{array}
\]
is commutative. Consequently, $\delta_n$ induces a homomorphism $\delta_n\colon \mathcal{D}^{(n)}\arr\mathcal{D}^{(n-1)}$.

Let us write the sequence in the definition of $\mathcal{D}^n$ as
\[A_0\nuke^{(n)}\stackrel{\varsigma_*}{\arr}
A_1\nuke^{(n)}\stackrel{\varsigma_*}{\arr}
A_2\nuke^{(n)}\stackrel{\varsigma_*}{\arr}\cdots.\]

For every $f\in C_c(\Gr^{(n)}; A)$ there exists $m$ such that $\varsigma_*^m(f)\in A\nuke^{(n)}$. Denote by $J(a)\in\mathcal{D}^{(n)}$ the element represented by $\varsigma_*^m(f)\in A_m\nuke^{(n)}$. We get a well defined homomorphism $J\colon C_c^(\Gr^{(n)}; A)\arr\mathcal{D}^{(n)}$. We call it the \emph{natural homomorphism}.

\begin{theorem}
\label{th:homologydimension}
 The natural homomorphism $J\colon C_c(\Gr^{(n)}; A)\arr\mathcal{D}^{(n)}$ induces an isomorphism of $H_n(\Gr; A)$ with the homology groups of the chain complex
\[0\stackrel{\delta_0}{\longleftarrow}\mathcal{D}^{(0)}
\stackrel{\delta_1}{\longleftarrow}\mathcal{D}^{(1)}
\stackrel{\delta_2}{\longleftarrow}\mathcal{D}^{(2)}\stackrel{\delta_3}{\longleftarrow}\cdots.\]
\end{theorem}

\begin{proof}
Let us choose for every vertex $\alpha$ in the graph $\G$ defining $\F$, a letter (i.e., an arrow) $x\in\alb$ such that $\be(S_x)=I_\alpha$, i.e., such that the end the corresponding edge is equal to $\alpha$. Let $\mathsf{T}$ be the set of chosen letters. Denote by $\mathsf{T}_m$ the set of all allowed concatenations of length $m$ of the elements of $\mathsf{T}$, i.e., the set of paths of length $m$ in $\G$ consisting only of edges from $\mathsf{T}$. Then for every vertex $\alpha$ of $\G$ there is a unique path $v\in\mathsf{T}_m$ ending in $\alpha$.

Let $(F_1, F_2, \ldots, F_n)$ be a multisection and let $x_0, x_1, \ldots, x_n\in\mathsf{T}$ be such that $\en(F_1)\subset\be(S_{x_0})$, and $\be(F_i)\subset\be(S_{x_i})$ for all $i=1, 2, \ldots, n$. Note that such a sequence $(x_0, \ldots, x_n)$ exists for every element of $\nuke^{(n)}$. 

Define then
\[\tau(F_1, F_2, \ldots, F_n)=(S_{x_0}F_1S_{x_1}^{-1}, S_{x_1}F_2S_{x_2}^{-1}, \ldots, S_{x_{n-1}}F_nS_{x_n}^{-1}).\]
Let $\tau_*$ be the homomorphism of $C_c(\Gr^{(n)}; A)$ induced by the condition $\tau_*(a_U)=a_{\tau(U)}$. 

We have then $\varsigma_*\circ \tau_*=Id$. It is also easy to check that $\tau_*$ is a chain map.

\begin{lemma}
\label{lem:tausigma}
The map $\tau_*\circ\varsigma_*$ is chain homotopic to the identity map.
\end{lemma}

\begin{proof}
Let $U$ be a multisection
\[(S_{y_0}F_1S_{y_1}^{-1}, S_{y_1}F_2S_{y_2}^{-1}, \ldots, S_{y_{n-1}}F_nS_{y_n}^{-1}), \]
where $y_i\in\alb$, and $\en(F_1)\subset\be(S_{y_0})$, $\be(F_i)\subset\be(S_{y_i})$ for all $i=1, 2, \ldots, n$. (Every multisection can be split into a disjoint union of multisections satisfying this condition.)
Then $U'=\tau\circ\varsigma(U)$ is the multisection
\[(S_{x_0}F_1S_{x_1}^{-1}, S_{x_1}F_2S_{x_2}^{-1}, \ldots, S_{x_{n-1}}F_nS_{x_n}^{-1}), \]
where $x_i\in\mathsf{T}$ are such that $\be(S_{x_i})=\be(S_{y_i})$ for every $i$. 

Let us define a chain homotopy $D_n\colon C_c(\Gr^{(n)}; A)\arr C_c(\Gr^{(n+1)}; A)$ 
by the condition that $D_n(a_U)$ is equal to $a_{U_0}-a_{U_1}+\cdots+(-1)^na_{U_n}$, where
\begin{align*}
U_0&=(S_{x_0}S_{y_0}^{-1}, S_{y_0}F_1S_{y_1}^{-1}, S_{y_1}F_2S_{y_2}^{-1}, \ldots, S_{y_{n-1}}F_nS_{y_n}^{-1}),\\
U_1&=(S_{x_0}F_1S_{x_1}^{-1}, S_{x_1}S_{y_1}^{-1}, S_{y_1}F_2S_{y_2}^{-1}, \ldots, S_{y_{n-1}}F_nS_{y_n}^{-1}),\\
U_2&=(S_{x_0}F_1S_{x_1}^{-1}, S_{x_1}F_2S_{x_2}^{-1}, S_{x_2}S_{y_2}^{-1}, \ldots, S_{y_{n-1}}F_nS_{y_n}^{-1}),\\
\vdots & \\
U_n&=(S_{x_0}F_1S_{x_1}^{-1}, S_{x_1}F_2S_{x_2}^{-1}, \ldots, S_{x_{n-1}}F_nS_{x_n}^{-1}, S_{x_n}S_{y_n}^{-1}).
\end{align*}

It is checked directly that $\delta_n\circ D_n(a_U)+D_n\circ\delta_{n+1}(a_U)=a_U-a_{U'}$.
\end{proof}

Let us show that the homomorphism $J_*$ induced by the natural homomorphism $C_c(\Gr^{(n)}; A)\arr\mathcal{D}^{(n)}$ is an isomorphism, by constructing the inverse homomorphism.

Suppose that we have an element $f\in\mathcal{D}^{(n)}$ represented by an element $f_m\in A_m$. Define $J^{-1}([f])$ as the class of $[\tau_*^m(f_m)]\in H_n(\Gr; A)$. We have to show first that it is well defined. The same element of $\mathcal{D}^{(n)}$ is represented by $f_{m+1}=\varsigma(f_m)\in A_{m+1}$. Then $[\tau_*^{m+1}(f_{m+1})]=[\tau^m_*(\tau_*\varsigma_*(f_m))]=[\tau_*^m(f_m)]$, by Lemma~\ref{lem:tausigma}, which shows that the map $J^{-1}$ is well defined.

Let us show that $J^*$ and $J_*$ are inverse to each other. If an element $f$ of $\mathcal{D}^{(n)}$ is represented by $f_m\in A_m\nuke^{(n)}$, then $\varsigma_*^m(\tau_*^m(f_m))=f_m\in A\nuke^{(n)}$, hence $J(\tau_*^m(f_m))=f$. Consequently, $J_*\circ J^*$ is the identity. Conversely, suppose that $f\in C_c(\Gr^{(n)}; A)$ is such that $\varsigma_*^m(f)\in A\nuke^{(n)}$. Then $J_*(f)$ is represented by $\varsigma_*^m(f)\in A_m\nuke^{(n)}$, hence $J^*(J_*(f))$ is represented by $\tau_*^m\circ\varsigma_*^m(f)$. But then Lemma~\ref{lem:tausigma} implies that $[\tau_*^m\circ\varsigma_*^m(f)]=[f]$. 
\end{proof}

\subsubsection{Group case}

Let $G$ be a contracting self-similar group acting faithfully on $\xo$, and let $\Gr$ be the associated groupoid of germs (including the germs of all finitary automorphisms of $\xs$). 

Every sequences $(g_1, g_2, \ldots, g_n)\in G^n$, can be seen as an ordered multisection. Let $(x_0, x_1,  \ldots, x_n)\in\alb^{n+1}$. Then $(g_1, g_2, \ldots, g_n)_{(x_0, x_1, \ldots, x_n)}$ is non-empty if and only if $g_n(x_n)=x_{n-1}$, $g_{n-1}(x_{n-1})=x_{n-2}$, \ldots, $g_1(x_1)=x_0$. If it is non-empty, then
\[(g_1, g_2, \ldots, g_n)|_{(x_0, x_1, \ldots, x_n)}=(g_1|_{x_1}, g_2|_{x_2}, \ldots, g_n|_{x_n}).\]

Consequently, $\varsigma_*$ acts by the formula 
\[\varsigma_*(g_1, g_2, \ldots, g_n)=\sum_{x\in\alb}(g_1|_{g_2\cdots g_n(x)}, g_2|_{g_3\cdots g_n(x)}, \ldots, g_n|_x).\]

The $n$-dimensional nucleus (i.e., the subset of $G^n$ to which iterations of $\varsigma_*$ converge) was studied in~\cite{nek:models}, and used  to construct simplicial models of the limit space of a contracting group. It is shown there that 
\[\til\cdot h_0\cap\til\cdot h_1\cap\ldots\cap\til\cdot h_n\ne\emptyset\]
where $\til$ is the image of $\xmo$ in the limit $G$-space $\limg$,
if and only if $(h_0h_1^{-1}, h_1h_2^{-1}, \ldots, h_{n-1}h_n^{-1})$ belongs to the $n$-dimensional nucleus. In particular the $n$-dimensional nucleus describes the nerve of the cover of $\limg$ by the sets $\til\cdot g$. 

The map $\varsigma^*$ induces an endomorphism of the group homology $H_n(G; \Z)$, which we will also denote by $\varsigma^*$. If $G$ is transitive on the first level of the tree $\xs$, then this map is equal to the composition of the transfer (restriction) map from $H_n(G; \Z)$ to $H_n(G_x; \Z)$ with the map induced by the homomorphism $g\mapsto g|_x$ from $G_x$ to $G$. Here $G_x$ is the stabilizer of a letter $x\in\alb$.

Since homology is continuous with respect to inductive limits, Theorem~\ref{th:homologydimension} implies the following description of the homology groups $H_n(\Gr; \Z)$.

\begin{theorem}
Let $(G, \bim)$ be a contracting self-similar group, and let $\Gr$ be the associated groupoid of germs. Then the homology group $H_n(\Gr; \Z)$ is isomorphic to the inductive limit of the homology groups $H_n(G; \Z)$ with respect to the iterations of the map $\varsigma_*\colon H_n(G; \Z)\arr H_n(G; \Z)$.
\end{theorem}

In particular,  $H_0(\Gr)$ is isomorphic to $\Z[1/d]$ for $d=|\alb|$. For every clopen subset $U\subset\xo$, the corresponding image $[U]$ is equal to the Lebesgue measure of $U$. 

Since $H_1(G; \Z)$ is isomorphic to the abelianization $G/[G, G]$, we get the following description of $H_1(\Gr)$.

\begin{proposition}
Let $(G, \bim)$ be a contracting self-similar group. Let $\Gr$ be the groupoid of germs of the faithful quotient. Then $H_1(\Gr)$ is isomorphic to the direct limit of the groups $G/[G, G]$ for iterations of the map $\varsigma\colon G/[G, G]\arr G/[G, G]$ defined by
\[\varsigma([g])=\sum_{x\in\alb}\left[g|_x\right].\]
\end{proposition}

Moreover, we do not need to know the faithful quotient of $G$. Theorem~\ref{th:homologydimension} implies the following description of $H_1(\Gr)$ (see also Theorem~\ref{th:indexmap} below).

\begin{proposition}
Let $\nuke$ be the nucleus of a contracting group $(G, \bim)$. Let $H$ be the abelian group given by the presentation consisting of the generators $[g]$ for $g\in\nuke$ and relations $[g_1]+[g_2]+[g_3]=0$ for every triple $g_i\in\nuke$ such that $g_1g_2g_3=1$ in $G$. Let $\varsigma\colon H\arr H$ be the homomorphism defined by \[\varsigma([g])=\sum_{x\in\alb}[g|_x].\]
Then $H_1(\Gr)$ is isomorphic to the direct limit of the iterations of $\varsigma$.
\end{proposition}

More on homology of the groupoids associated with self-similar groups (including the purely infinite groupoid generated by the group and the shift), see the paper~\cite{millersteinber}.

\subsubsection{Examples of computation of $H_0(\Gr)$ and $H_1(\Gr)$}

Let $\Gr$ be a shift-invariant contracting groupoid acting on a shift of finite type $\F$. Let $\Fr$ be the groupoid generated by the germs of the shift $\sigma\colon \F\arr\F$.

For an element $a$ of $C_c(\Gr; \Z)$, we will denote by $[a]$ the element of $H_1(\Gr)$ equal to the class of $J(a)\in\mathcal{D}^{(1)}$. We have then the following description of $H_1(\Gr)$ and the map $a\mapsto [a]$.

\begin{theorem}
\label{th:indexmap}
Suppose that $\delta_1\colon \mathcal{D}^{(1)}\arr\mathcal{D}^{(0)}$ is zero. Let $H$ be an abelian group, and let $\Psi\colon C_c(\Gr; \Z)\arr H$ be an (additive) group homomorphism satisfying the following properties.
\begin{enumerate}
\item If $F_1, F_2$ are bisections such that $\be(F_1)=\en(F_2)$, then $\Psi(F_1F_2)=\Psi(F_1)+\Psi(F_2)$. (In particular, if  $F$ is an idempotent, then $\Psi(F)=0$.)
\item If $\varsigma_*(a)=0$, then $\Psi(a)=0$.
\end{enumerate}
Then there exists a unique homomorphism $\psi\colon H_1(\Gr)\arr H$ such that $\Psi(a)=\psi([a])$.
\end{theorem}

\begin{proof}
It is enough to show that $\psi([a])=\Psi(a)$ is a well defined homomorphism $\psi\colon H_1(\Gr)\arr H$.

Suppose that $[a_1]=[a_2]$. Then $J(a_1-a_2)$ belongs to the image of $\delta_2$. Consequently, there exists $m$ such that $\varsigma_*^m(a_1-a_2)$ belongs to the subgroup of $C_c(\Gr^{(n)}; \Z)$ generated by elements of the form $F_2-F_1F_2+F_1$, where $F_1, F_2$ are bisections such that $\be(F_1)=\en(F_2)$. But then $\Psi(F_2-F_1F_2+F_1)=0$, so we have $\Psi(\varsigma_*^m(a_1-a_2))=0$, which implies, by property (2), that $\Psi(a_1-a_2)=0$, i.e., that $\Psi(a_1)=\Psi(a_2)$.
\end{proof}

We say that a bisection $F\subset\Gr$ is \emph{finitary} if it belongs to $\Fr$. If $\Psi$ satisfies the conditions of Theorem~\ref{th:indexmap}, then $\Psi(F)=0$ for any finitary bisection $F$. Consequently, $\Psi(F)$ depends only on the intersection of $F$ with $\Gr\setminus\Fr$.

Let us show at first how to use Theorem~\ref{th:indexmap} to compute $H_1(\mathfrak{R})$ for the  groupoid of germs $\mathfrak{R}$ of the golden mean rotation from~\ref{s:rotation}.

\begin{proposition}
We have $H_0(\mathfrak{R})=\Z^2$ and $H_1(\mathfrak{R})=\Z$.
\end{proposition}

\begin{proof}
We use the encoding of the subshift as an edge-shift using the alphabet $\{0_0, 0_1, 1_0\}$.
Recall that the graph defining the subshift $\F$ has two vertices $0, 1$ and three edges: one loop $0_0$ at $0$, the arrow $0_1$ path from $0$ to $1$, and the arrow $1_0$ from $1$ to $0$.

The 0-dimensional nucleus consists of the idempotents $I_0$ and $I_1$ corresponding to the sets of paths starting in the corresponding vertices. 

Since $\sigma(I_0)=I_0+I_1$ and $\sigma(I_1)=I_0$, the map $\varsigma_*\colon \Z\nuke^{(0)}\arr\Z\nuke^{(0)}$ is given by the matrix $\left(\begin{array}{cc} 1 & 1\\ 1 & 0\end{array}\right)$. Since this map is an automorphism of $\Z^2$, we get $\mathcal{D}^{(0)}\cong\Z^2$ freely generated by the images of $I_0$ and $I_1$.

It is convenient to interpret $\mathcal{D}^{(0)}$ as the group $\Z[\varphi]$, where the image of $I_0$ is $\varphi^{-1}=\varphi-1$ and $I_1$ is $\varphi^{-2}=2-\varphi$. Then the map $\varsigma_*\colon \mathcal{D}^{(0)}\arr\mathcal{D}^{(0)}$ maps $\varphi^{-1}$ to $1$ and $\varphi^{-2}$ to $\varphi^{-1}$, i.e., it is multiplication by $\varphi$. Then the image of an idempotent $I_v$ in $\mathcal{D}^{(0)}$ is equal to the measure of the corresponding subset of $[0, 1]$.

Since the rotation preserves the measure, every bisection represents a cycle, i.e., $\delta_1\colon \mathcal{D}^{(1)}\arr\mathcal{D}^{(0)}$ is zero.  Consequently, we can compute $H_1(\Gr; \Z)$ using Theorem~\ref{th:indexmap}. 

We see from the recursion defining $R_1, R_0$ that $R_0$ is finitary, so $[R_0]=0$. We have $[\varsigma_*(R_1)]=[R_1^*]=-[R_1]$. Since the nucleus of $\Gr$ consists of restrictions of $R$ and $R^{-1}$ to clopen sets, for every $F\in C_c(\Gr^{(1)}; \Z)$ there exists $m$ such that $\varsigma^m_*(F)$ is supported on $R\cup R^{-1}\cup\Gr^{(0)}$.  Consequently, $H_1(\Gr; \Z)$ is generated by $[R_1]=[R]$. In order to prove that it is isomorphic to $\Z$, it is enough to construct an epimorphism $\Psi\colon C_c(\Gr; \Z)\arr\Z$ satisfying the conditions of Theorem~\ref{th:indexmap}.

All $\Gr$-orbits are equal to $\Fr$-orbits (i.e., cofinality classes in $\F$), except for the orbit of $101010\ldots$, which is a union of two $\Fr$-orbits. This follows from the fact that the only instances when an element of the nucleus changes a cofinality class is $R_1(101010\ldots)=010101\ldots$ and $R_1^{-1}(010101\ldots)=101010\ldots$. Define, for a bisection $F\subset\Gr$, the number $\Psi(F)\in\Z$ as the number of points moved by $F$ from the $\Fr$-orbit of $101010\ldots$ to the $\Fr$ orbit of $010101\ldots$ minus the number of points moved in the opposite direction. It is easy to see that $\Psi$ satisfies the first condition of Theorem~\ref{th:indexmap} and that $\Psi(\varsigma(F))=-\Psi(F)$, so it also satisfies the second condition. We also have $\Psi(R_1)=1$, so $\Psi$ is surjective.

This proves that $H_1(\Gr)\cong\Z$ and gives an explicit description of $[F]$ as the net flow between two $\Fr$-orbits.
\end{proof}

\subsubsection{Penrose tiling}

Let $\mathfrak{P}$ be the groupoid associated with the Penrose tiling, see~\ref{ss:penrose}.

The homology $H_n(\mathfrak{P}; A)$ has a natural interpretation as the \emph{pattern equivariant homology}, studied, for example, in~\cite{walton:pehomology}.
Namely, the sets $\mathfrak{P}^{(n)}$ are naturally identified with the spaces of isomorphism classes of ordered tuples $(\mathcal{P}, P_0, P_1, \ldots, P_n)$, where $\mathcal{P}$ is a Penrose tiling, and $P_i$ are its tiles. The topology is defined in the same way as for $\mathfrak{P}^{(0)}=\F$ and $\mathfrak{P}^{(1)}=\mathfrak{P}$. The isomorphism in our case is any isometry between the tilings preserving the marking of tiles. Alternatively, one may allow only parallel translations or only orientation-preserving isometries, as it is done in~\cite{walton:pehomology}. (The paper~\cite{walton:pehomology} uses vertices of the tiles instead of tiles, but this does not change the homology, since it amounts to passing to a Morita equivalent groupoid.)
The groups of chains $C_c(\mathfrak{P}^{(n)}; A)$ and the boundary homomorphisms $\delta_n$ defined in the natural way agree with the definitions for the groupoid homology.

\begin{theorem}
\label{th:Penrosehomology}
We have $H_0(\mathfrak{P})\cong \Z^2$ generated by the classes of tiles $P_0$ and $P_1$, and $H_1(\mathfrak{P})\cong (\Z/2\Z)^3$ generated by the classes of $B, C$, and $D$. 
\end{theorem}

\begin{proof}
The nucleus of $\mathfrak{P}$ contains two idempotents $I_0$ and $I_1$ corresponding to two types of tiles $P_0$, $P_1$ of the Penrose tiling. The  inflation rule and the corresponding matrix recursion imply that we have the following relations in $\mathcal{D}^{(0)}$:
\[[\varsigma_*(I_0)]=[I_0]+[I_1], \qquad [\varsigma_*(I_1)]=[I_0]+2[I_1].\]
Since the corresponding matrix $\left(\begin{array}{cc} 1 & 1\\ 1 & 2\end{array}\right)$ has determinant 1, the dimension group $\mathcal{D}^{(0)}$ is isomorphic to $\Z^2$, and can be interpreted as $\Z[\varphi]$, where $[I_0]=\varphi^{-2}$ and $[I_1]=\varphi^{-1}$, and the action of $\varsigma$ is identified with multiplication by $\varphi^2$. Then, for every clopen set $U\subset\mathfrak{P}^{(0)}$ we have $[U]$ equal to the measure of the corresponding subset of the golden triangle $T$.

Since the elements of $\mathfrak{P}=\mathfrak{T}$ act on $\F$ by measure-preserving transformations, $\delta_1$ is zero, so we have $H_0(\mathfrak{P}; \Z)\cong\Z^2\cong\Z[\varphi]$, and we can use Theorem~\ref{th:indexmap} to compute $H_1(\mathfrak{P}; \Z)$.

The elements of the nucleus corresponding to pairs of tiles having a common side are the generators $A_i, B, C, D_{ij}$ of the Kellendok semigroup. The matrix recursion for them implies that the elements $A_0, A_1, D_{10}, D_{01}$ are finitary, hence their images in $H_1(\mathfrak{P})$ are zero. We also have $\varsigma_*(A_2)=\varsigma_*(D_{00})=C$, so $[A_2]=[D_{00}]$. 

The recursion for the transformations $B, C, D$ implies therefore
\begin{align*}
[\varsigma_*(B)]&=[D_{00}]+[D_{11}],\\
[\varsigma_*(C)]&=[B],\\
[\varsigma_*(D_{00})]&=[C],\\
[\varsigma_*(D_{11})]&=[D_{11}].
\end{align*}

The restriction of $C$ to the triangles $I_{11}, I_{21}, I_{12}, I_{22}$ is finitary (see Figure~\ref{fig:SML}). It follows that $[C]=[CI_{10}]$. Since the bisection $S_{11}S_{10}^{-1}$ conjugates the action of $C$ on $I_{10}$ to the action of $D$ on $I_{11}$, we have $[DI_{11}]=[C]$.
Compositions of the restrictions of $D_{11}$ and $C$ to the pentagon $I_{12}\cup I_{10}\cup I_{22}$ is a rotation of the pentagon (corresponding to the wheel patch). Since the bisections $D_{11}$ and $C$ are involutions, and the rotation has order 5, the image of the rotation in $H_1(\mathfrak{P})$ is 0. It follows that $[D_{11}]+[C]=[DI_{11}]=[C]$, hence $[D_{11}]=0$.

Consequently,
\begin{align*}
[\varsigma_*(B)]&=[D_{00}],\\
[\varsigma_*(C)]&=[B],\\
[\varsigma_*(D_{00})]&=[C],
\end{align*}
so $\varsigma_*$ preserves the subgroup of $H_1(\mathfrak{P})$ generated by $[B], [C], [D_{00}]$. We also know that the nucleus is contained in the Kellendonk semigroup, i.e., is generated by $A_i, B, C, D_{ij}$. Consequently, $H_1(\mathfrak{P})$ is generated by $[B], [C]$, and $[D_{00}]$, so $H_1(\mathfrak{P})$ is a quotient of $(\Z/2\Z)^3$.

Let us construct a map $\Psi\colon C_c(\mathfrak{P}; \Z)\arr (\Z/2\Z)^3$ satisfying the conditions of Theorem~\ref{th:indexmap}.

Consider $g\in\mathfrak{P}$. We have $g\in\Fr$ if and only if $\varsigma^n(g)$ is a unit for some $n$. Suppose that $g\notin\Fr$. Then there exists $n$ such that $\varsigma^n(g)$ is a germ of an element of the nucleus. It follows from the structure of the strongly connected components of the Moore diagram of the nucleus that then either for all $n$ big enough the elements $\varsigma^n(g)$ are germs of the dihedral groups acting on the wheel and the star pentagons in $T$, or they belong to one of the bisections $D_{00}, D_{11}, B, C, A_2$. It follows either $\varsigma^n(g)$ is an isotropic germ of a rotation of a pentagon for all $n$ large enough, or it is an isotropic germ of a reflection of a triangle or a pentagon for all $n$ large enough. Consequently, $\be(g)$ is a preimage under $\sigma^n$ of a center of a pentagon formed by sub-triangles of $T$, or $\be(g)$ belongs to a preimage under $\sigma^n$ of an axis of symmetry of a sub-triangle of $T$.

Let $F$ be a $\mathfrak{P}$-bisection. Let $\Sigma(F)\subset\F$ be the set of points $w$ such that the germ $[F, w]$ does not belong to $\Fr$. It follows from the description of the nucleus of $\mathfrak{P}$ and from the discussion above that the set $\Sigma(F)$ is a union of axes of symmetries of some sub-triangles of $T$ and of a finite set of points at which $F$ modulo $\Fr$ acts as a rotation by $\frac{2k\pi}{5}$. The images of the neighborhoods of the germs of the latter type in $H_1(\mathfrak{P}; \Z)$ are equal to $0$ (since they are compositions of bisections of order 2). 

Let us use the terminology of the action of the affine group $G$ on the plane $V_2\supset T$. For each of the axes of symmetries of sub-triangles, one of the endpoints is a vertex of a sub-triangle of $T$, and hence belongs to the abelian group $H<G$. The other endpoint belongs to $\frac 12H\setminus H$. Consequently, we can represent $\Sigma(F)$ as a union of segments with endpoints in $\frac 12H$. Since conjugation by $G$ preserves $H$ and $\frac 12H$, the group $G$ permutes the cosets $\frac 12H/H$. There are four orbits: the coset $H$, the orbit of $(\mathbf{e}_1+\mathbf{e}_2-\mathbf{e}_3-\mathbf{e}_4)/2$, the orbit of $(\mathbf{e}_1-\mathbf{e}_0)/2$, and the orbit of $(\mathbf{e}_2-\mathbf{e}_0)/2$. One can check that the operator $1+R$ permutes the latter three orbits cyclically.

Let $F$ be a $\mathfrak{P}$-bisection. Represent $\Sigma(F)$ as a union of segments with endpoints in $\frac 12H$.
Define $\Psi(F)$ as the quadruple $(a_0, a_1, a_2, a_3)\in (\Z/2\Z)^4$ of parities of the number of obtained endpoints in each of the $G$-orbits of the cosets $(\frac 12H)/H$.  Since the total number of endpoints is even, we have $a_0+a_1+a_2+a_3=0$, so the quadruple $(a_0, a_1, a_2, a_3)$ is uniquely represented by $(a_1, a_2, a_3)$, where we assume that $a_0$ is the parity of the number of endpoints in $H$. Then $\Psi(F)$ takes values in $(\Z/2\Z)^3$. Note that the value of $\Psi(F)$ does not depend on the way we partition $\Sigma(F)$ into the axes. 

Since $1+R$ cyclically permutes the non-integer $G$ orbits of $\frac 12H/H$, the triple $\Psi(\varsigma_*(F))$ is a cyclic permutation of the triple $\Psi(F)$. Consequently, $\Psi$ satisfies condition (2) of Theorem~\ref{th:indexmap}.

In order to check condition (1), let us investigate how the axes of symmetry intersect with the lines in $\mathcal{L}$ and with each other.

Consider the axis of symmetry of the triangle $T$ with vertices $\mathbf{0}$, $\mathbf{e}_4-\mathbf{e}_0$ and $\mathbf{e}_1-\mathbf{e}_0$. Let $\ell$ be any line from the set $\mathcal{L}$. If it is not perpendicular to the axis, then the line $\ell$ intersects the axis in the same point as it intersects the image of $\ell$ under the symmetry of $T$. Then the intersection of the line $\ell$ with the axis will be an intersection point of two elements of $\mathcal{L}$, hence an element of $H$. If $\ell$ is perpendicular to the axis, then the intersection point will be the middle of the intersection of $\ell$ with the sides of $T$, hence it will belong to $\frac 12H$. Since $\mathcal{L}$, $\frac 12H$, and $H$ are $G$-invariant, this implies that intersection of any axis of a sub-triangle of $T$ with a line from $\mathcal{L}$ belongs to $\frac 12 H$.

If two axis of sub-triangles intersect, then the composition of the corresponding symmetries is a rotation of a pentagon formed by lines in $\mathcal{L}$. The midpoints of the sides of this pentagon belong to one $G$-orbit in $\frac 12H$.

It follows that condition (1) is satisfied for $\Psi$.

We check directly that the endpoints of the axes of symmetries defined by $B$, $C$, $D$ belong to three different $G$-orbits, so that $\Psi(B), \Psi(C), \Psi(D)$ are the free generators of $(\Z/2\Z)^3$. 
\end{proof}

Another interpretation of the index map, described in the proof of Theorem~\ref{th:Penrosehomology}, comes from the relation of the sides of the sub-triangles of the golden Cantor triangle $\mathcal{T}$ with the Conway worms.

If $x$ is the intersection of an axis of symmetry of a sub-triangle of $\mathcal{T}$ with its side, then the corresponding Penrose tiling has an axis of symmetry and an infinite Conway worm. As we have seen, there are three orbits of such points $x$. This also follows from the fact that Conway worms are described by sequences belonging to the Fibonacci substitutional subshift, which has exactly three palindromic sequences (see the end of~\ref{sss:fibonacci}). See Figure~\ref{fig:threepoints} where parts of these three tilings and the corresponding midpoints of sides of sub-triangles are shown.

\begin{figure}
\centering
\includegraphics{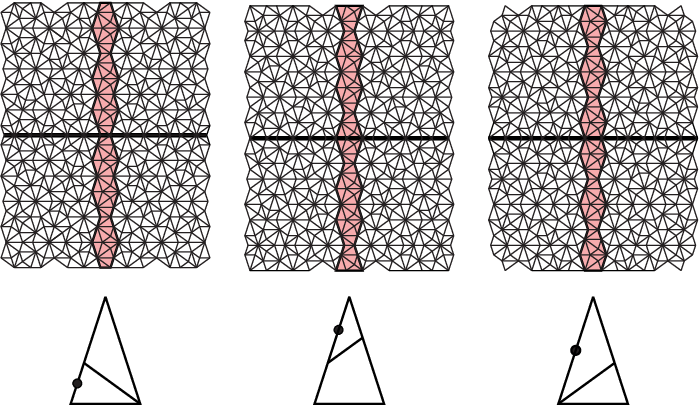}
\caption{Symmetric Conway worms}
\label{fig:threepoints}
\end{figure}

Let $F$ be a $\mathfrak{P}$-bisection, and consider the corresponding partial permutation of the tiles of Penrose tilings. Let $\mathcal{P}_i$ for $i=1, 2, 3$ be the three symmetric tilings with infinite Conway worms.

Fix one of the half-planes bounded by the axis of symmetry of the tiling in each of the tilings $\mathcal{P}_i$ and call it the \emph{upper half-plane}. The other one is the \emph{lower half-plane}.

Our idea is to count the parity of the number of tiles moved by $F$ from the upper to the lower half-planes of $\mathcal{P}_i$. Of course, this number is usually infinite. However, $F$ moves tiles on a uniformly bounded distance. Since the part of the tiling outside of the Conway worm is symmetric with respect to the worm, if a tile $P$ is sufficiently far from the worm, and $P'$ is the tile symmetric to $P$ with respect to the worm, then $F(P')$ is the tile symmetric to $F(P)$ with respect to the worm. Consequently, tiles that are sufficiently far away from the worm and that are moved to the other half-plane come in pairs. Tiles that are sufficiently far away from the axis of symmetry of $\mathcal{P}_i$ are not moved across it. Consequently, the parity of the number of tiles moved by $F$ from the upper to the lower half-plane of $\mathcal{P}_i$ can be defined as the parity of the number of such tiles inside a sufficiently large square centered around the intersection of the Conway worm with the axis of symmetry of the tiling, since this parity will not depend on the square.

We get three parities, which defines the index map $\Psi\colon C_c(\mathfrak{P};\Z)\arr(\Z/2\Z)^3$.

\begin{theorem}
The group $\mathsf{F}(\mathfrak{P})$ of pattern equivariant permutations of Penrose tilings is finitely generated and virtually simple.
\end{theorem}

\begin{proof} Since the Kellenonk inverse semigroup $\mathcal{K}$ of the Penrose tiling is finitely generated and is a basis of topology for $\mathfrak{P}$, the alternating full group $\mathsf{A}(\mathfrak{P})$ is finitely generated. It is also well know that $\mathfrak{P}$ is a minimal groupoid, i.e., that all Penrose tilings are pairwise locally isomorphic. Consequently, $\mathsf{A}(\mathfrak{P})$ is simple. It remains to show that $\mathsf{A}(\mathfrak{P})$ is a subgroup of finite index in the full group.

The groupoid $\mathfrak{P}$ is almost finite by Theorem~\ref{th:contractingalmostfinite}, so by 
Theorem~\ref{th:Li}, the symmetric full group $\mathsf{S}(\mathfrak{P})$ is equal to the kernel of the index map. In fact, this is not hard to prove directly using the description of the set $\Sigma(F)$ and the associated index map $\Psi$ given above.

The range of the index map is generated by $\Psi(B), \Psi(C)$, and $\Psi(D)$. Since $B, C, D$ belong to the full group, we conclude that $\mathsf{S}(\mathfrak{P})$ is a subgroup of index $8$ in the full group.

The symmetric full group $\mathsf{S}(\mathfrak{P})$ is generated by elements of the form $\tau_F=F\cup F^{-1}\cup(\mathfrak{P}^{(0)}\setminus(F\cup F^{-1})$. By~\cite[Theorem~7.2]{nek:fullgr}, the map $[\be(F)]\mapsto \tau_F\mathsf{A}(\mathfrak{P})$ generates an epimorphism $H_0(\mathfrak{G}; \Z/2\Z)\arr\mathsf{S}(\mathfrak{P})/\mathsf{A}(\mathfrak{P})$. Following the computation of $H_0(\mathfrak{P}; \Z)$, we conclude that $H_0(\mathfrak{G}; \Z/2\Z)$ is isomorphic to $(\Z/2\Z)^2$ with the action of $\varsigma_*$ given by the matrix $\left(\begin{array}{cc} 1 & 1\\ 1 & 0\end{array}\right)$.

Consequently, the index of the alternating full group in the symmetric full group is a divisor of $4$, so the alternating subgroup has index $32, 16$, or $8$ in the full group. 
\end{proof}

It is shown in~\cite{walton:pehomology} that if we consider Penrose tilings up to parallel translations only (and not up to all isometries), then the pattern-equivariant homology $H_k$ for $k=0, 1, 2$ is $\Z^8, \Z^5$, and $\Z$, respectively. If we consider them up to orientation-preserving isometries, then they are $\Z^2\oplus(\Z/5\Z)$, $\Z$, and $\Z$, respectively. Both settings lead to self-similar contracting inverse semigroups (see~\cite{WaltonWhittaker:tilings}), so the associated full groups are also finitely generated and virtually simple.

\providecommand{\href}[2]{#2}

\end{document}